\documentclass[11pt,a4paper]{amsart}
\usepackage{amsfonts,amsmath,amssymb}
\usepackage{hyperref}
\usepackage{latexsym,fullpage,amsfonts,amssymb,amsmath,amscd,graphics}
\usepackage[all]{xy}
\usepackage{amssymb,amsthm,amsxtra}
\usepackage[usenames]{color}
\usepackage{amscd}
\usepackage{amsthm}
\usepackage{amsfonts}
\usepackage{amssymb}
\usepackage{mathrsfs}
\usepackage{url}
\usepackage{bbm}
\usepackage{wasysym}
\usepackage{enumitem}
\usepackage[vcentermath]{youngtab}%
\usepackage{framed}
\usepackage[capitalise]{cleveref}
\usepackage{extarrows}

\usepackage{tikz}
\usetikzlibrary{decorations.markings}
\usetikzlibrary{cd}
\usetikzlibrary{patterns}
\usetikzlibrary{shapes}

\tikzset{anchorbase/.style={baseline={([yshift=-0.5ex]current bounding 
box.center)}}}
\tikzset{wipe/.style={white,line width=4pt}}

\tikzset{->-/.style={decoration={
  markings,
  mark=at position #1 with {\arrow{>}}},postaction={decorate}}}
\tikzset{-<-/.style={decoration={
  markings,
  mark=at position #1 with {\arrow{<}}},postaction={decorate}}}

\tikzset{darkg/.style={green!70!black}}

\theoremstyle{plain}% default
\newtheorem*{theorem*}{Theorem}
\newtheorem*{remark*}{Remark}
\newtheorem*{example*}{Example}
\newtheorem{lemma}{Lemma}[subsection]
\newtheorem{proposition}[lemma]{Proposition}
\newtheorem{corollary}[lemma]{Corollary}
\newtheorem{theorem}[lemma]{Theorem}

\newtheorem*{conjecture*}{Conjecture}

\theoremstyle{definition}
\newtheorem{definition}[lemma]{Definition}

\newtheorem{example}[lemma]{Example}

\theoremstyle{remark}
\newtheorem{remark}[lemma]{Remark}

\newtheorem{notation}[lemma]{Notation}

\newcommand{\Hom}{\operatorname{Hom}}

\newcommand{\eps}{\varepsilon}
\newcommand{\triv}{{\mathbbm 1}}

\newcommand{\id}{\operatorname{Id}}

\newcommand{\End}{\operatorname{End}}

\newcommand{\C}{{\mathbb C}}
\newcommand{\Z}{{\mathbb Z}}

\newcommand{\V}{\mathbf{V}}

\newcommand{\abs}[1]{\left|{#1}\right|}

\newcommand{\uRep}{\underline{{\rm Rep}}}

\newcommand{\Rep}{\mathrm{Rep}}

\newcommand{\A}{{\mathcal A}}

\newcommand{\F}{{\mathbb F}}

\newcommand{\InnaA}[1]{{\color{black}{#1}}}

\newcommand{\urep}{\Rep^{\InnaA{ab}}(\underline{GL}_{t}(\F_q))}
\newcommand{\kar}[1]{\InnaA{\underline{\Rep}(GL_{#1}(\F_q))}}

\newcommand{\repinf}{Rep(GL_{\infty}(\F_q))}
\newcommand{\Vect}{\mathrm{Vect}_{\C}}
\newcommand{\one}{{\mathbf{1}}}
\newcommand{\z}{{\dot{0}}}

\newcommand{\comment}[1]{}
\newcommand{\subscript}[2]{$#1 _ #2$}

\def\quotient#1#2{%
    \raise1ex\hbox{$#1$}\Big/\lower1ex\hbox{$#2$}%
}

\begin{document}

\title{Deligne categories and representations of the finite general linear group, part 1: universal property}
\date{\today}
  \author{Inna Entova-Aizenbud, Thorsten Heidersdorf}
% \address{}

 \address{I. E.: Department of Mathematics, Ben Gurion University of the Negev, Beer-Sheva, Israel}
 \email{entova@bgu.ac.il}
 \address{T. H.: Mathematisches Institut Universit\"at Bonn, Germany}
 \email{heidersdorf.thorsten@gmail.com}


\begin{abstract} We study the Deligne \InnaA{interpolation} categories $\kar{t}$ for $t\in \mathbb{C}$\InnaA{, first introduced by F. Knop}. These categories interpolate the categories of finite dimensional complex representations of the finite general linear group $GL_n(\mathbb{F}_q)$. We describe the morphism spaces in this category via generators and relations. 

We show that the generating object of this category (\InnaA{an} analogue of the representation $\C\F_q^n$ of $GL_n(\mathbb{F}_q)$) carries the structure of a Frobenius algebra with a compatible $\F_q$-linear structure; we call such objects $\mathbb{F}_q$-linear Frobenius spaces, and show that $\kar{t}$ is the universal symmetric monoidal category generated by such an $\mathbb{F}_q$-linear Frobenius space of categorical dimension $t$. 

In the second part of the paper, we prove a similar universal property for a category of representations of $GL_{\infty}(\mathbb{F}_q)$.
\end{abstract}

\maketitle

\section{Introduction} 
\subsection{}
An important theme in representation theory is to study representations in large rank. One incarnation of this is due to Deligne who constructed families of universal rigid symmetric monoidal categories $\uRep(GL_t)$, $\uRep(O_t)$, $\uRep(S_t)$, $t \in \mathbb{C}$ (see \cite{Deligne-Milne, Del07}) which capture stabilization phenomena that occur when the rank of the group goes to infinity. These categories interpolate the representation categories $\Rep(GL_n(\mathbb{C}))$, $\Rep(O_n(\mathbb{C}))$, $\Rep(S_n)$ ($n \in \mathbb{N}$) in the sense that these \emph{Deligne categories} admit for integral parameters $t \in \mathbb{N}$ the usual representation categories as quotients. Knop (see \cite{K}) extended Deligne's construction and defined a vast number of interpolating categories which capture the phenomenon of stabilization with respect to rank. Among Knop's examples are the categories of finite dimensional complex representations of the finite general linear group $GL_n(\F_q)$.

The current article is the first paper in a series in which we study \InnaA{Knop's interpolation categories for $GL_n(\F_q)$, which we denote by $\kar{t}$}. 

In order to define this category, we first define a ``skeletal'' subcategory $\InnaA{\mathcal{T}(\underline{GL}_t)}$ generated by a self-dual object $[1]$, which is the analogue of the $GL_n(\F_q)$-representation $\C\F_q^n$ (this is analogous to the construction of $\uRep(S_t)$ in \cite{Del07}). The objects in this skeletal subcategory are of the form $[k]$ for $k \in \mathbb{N}$ and should be thought of as the $k$-th tensor powers of the generating object $\V_t := [1]$. We assign to every subspace $R \subset \F_q^{l+k}$ a morphism $[l] \to [k]$. The Hom space $\Hom_{\mathcal{T}(\underline{GL}_t)} ([l], [k])$ is given by the span of the set $$ \{f_R: R \subset \F_q^{l+k} \; \text{ linear subspace }\}$$ and the composition of morphisms of the form $f_R$ depends polynomially on a fixed parameter $t \in\mathbb{C}$. The rule $[k] \otimes [l] = [k + l]$ turns this skeletal subcategory into a symmetric monoidal category. Its additive Karoubi envelope is baptized $\kar{t}$; this idea is essentially the same as in Knop's construction from \cite{K}, and we show in Section \ref{ssec:knop} that the categories obtained in both constructions are equivalent.

One major difficulty of the $\kar{t}$-case compared to the $S_t$-case is to obtain a more explicit description of the morphism spaces $\Hom([k],[s])$. The first crucial observation is that the generating object $[1]$ is an $\mathbb{F}_q$-linear Frobenius space in $\kar{t}$; this roughly means that $[1]$ carries simultaneously a commutative Frobenius algebra structure and an $\mathbb{F}_q$-module structure with some obvious compatibility conditions (see Definition \ref{def:Frob_linear_space}). 

Let us explain the origin of this structure. \InnaA{Given a symmetric monoidal category $\mathcal{T}$, one may consider the category of commutative algebra objects in $\mathcal{T}$, and its opposite: the category of affine schemes in $\mathcal{T}$.} 
The categories $\Rep(GL_n(\F_q))$ are ``generated'' by the finite \InnaA{affine} scheme $X_n:=\F^n_q$ on which $GL_n(\F_q)$ acts. It is \InnaA{an affine scheme} in  $\Rep(GL_n(\F_q))$, and the standard representation $\C\F_q^n$ of $GL_n(\F_q)$ is isomorphic to $O(X_n)^*$. Furthermore, the \InnaA{affine} scheme $X_n$ has an $\F_q$-linear structure, which induces a corresponding $\F_q$-linear structure structure on $\C\F_q^n$, compatible with the coalgebra structure on $\C\F_q^n=O(X_n)^*$. Together, these structures make the representation $\C\F_q^n$ into an $\mathbb{F}_q$-linear Frobenius space in $\Rep(GL_n(\F_q))$. We show that the category $\Rep(GL_n(\F_q))$ is generated by this space, and extend this statement to the Deligne categories $\kar{t}$.

In order to show this we first describe Hom spaces in $\Rep(GL_n(\mathbb{F}_q))$ by generators. By Corollary \ref{cor:specialization} the category $\kar{t}$ interpolates the finite-dimensional complex representations of $GL_n(\mathbb{F}_q)$ in the sense that for values $t = q^n$, $n \in \Z_{\geq 0}$, we have a full, essentially surjective SM functor
 $$F_n: \kar{t=q^n} \longrightarrow \Rep(GL_n(\F_q)).$$
Since the functor $F_n$ is fully faithful on small tensor powers, we may check the relations for morphisms in $\kar{t}$ by looking at their image under the functor $F_n$.

The main result of this article is that the category $\kar{t}$ satisfies the following universal property:

\begin{theorem}\label{introthrm:univ_prop_Del} (see Theorem \ref{thrm:univ_prop_Del}) 
 Let $\mathcal{C}$ be a Karoubi additive rigid SM category, and let $\V $ be an $\F_q$-linear Frobenius space in $\mathcal{C}$. Let $t = \dim(\V)$. Then there exists a SM functor
 $$ F_{\V}: \kar{t} \to \mathcal{C}, \;\; \V_t \longmapsto \V.$$
 \end{theorem}

Let us give some details on the construction \InnaA{of} such a functor. First of all, we need to send $[k]$ to $\V^{\otimes k}$. We then need to know where to send morphisms $[k] \to [s]$ in $\kar{t}$.

Let us discuss the example $\mathcal{C} = \Rep(GL_n(\F_q))$ with its $\F_q$-linear Frobenius algebra $\V_n = \mathbb{C}\F_q^n$ first. To any $R \subset \F_q^{k+s}$ we associate a morphism $f_R: \V_n^{\otimes k} \to \V_n^{\otimes k}$. 
We show that the set $\{ f_R| R \subset \F_q^{s+k}\}$ spans the space $\Hom_{GL_n(\F_q)}(\V_n^{\otimes s}, \V_n^{\otimes k})$ (see Lemma \ref{lem:basis}) and is a basis if $n\geq s+k$. It is then easy to see (see Corollary \ref{cor:specialization}) that for $t=q^n$ the functor $F_{\V}: \kar{q^n} \to \Rep(GL_n(\F_q))$ agrees with the specialization functor $F_n$ from above when $\V =\V_n \in \Rep(GL_n(\F_q))$.

For a general $\mathcal{C}$ we define morphisms  $\widetilde{f}_R:\V^{\otimes s}\to \V^{\otimes r}$, $R\subset \F_q^{s+k}$ which serve as analogues of the maps $f_R \in \kar{t}$ described above\comment{ in an outburst of creative energy}. We then prove that the morphisms $\widetilde{f}_R \in \mathcal{C}$ satisfy the same relations on their compositions and tensor products as their counterparts in $\kar{t}$. This allows us to construct a symmetric monoidal functor from the Deligne category into $\mathcal{C}$.

We also state other related universal properties, e.g. for $\Rep(GL_n(\F_q))$ itself.  As a Corollary we obtain:

\begin{corollary} 
Let $\mathcal{C}$ be a pre-Tannakian symmetric tensor category, and let $\V$ be an $\F_q$-linear Frobenius space in $\mathcal{C}$ annihilated by some exterior power. Let $n\in \mathbb{Z}_{\geq 0}$ such that $\dim \V = q^n$. Then the SM functor $F_{\V}: \kar{t=q^n} \to \mathcal{C}$ seen in \ref{thrm:univ_prop_Del} factors through the functor $F_n: \kar{t=q^n} \to \Rep(GL_n(\F_q))$, and we have a natural SM isomorphism $$ F_{\V} \longrightarrow \widetilde{F}_{\V} \circ F_n. $$ 
\end{corollary}

In a subsequent article, we will show that for a pre-Tannakian  symmetric tensor category $\mathcal{C}$, and an $\F_q$-linear Frobenius space $\V$ in $\mathcal{C}$ {\it not} annihilated by any exterior power, the functor $F_{\V}$ factors through a certain tensor category $\urep$ in which $\kar{t}$ sits as a full SM subcategory. This category $\urep$ is called the {\it abelian envelope} of $\kar{t}$. We expect that this abelian envelope is related to the category of $VI$-modules of Putman and Sam (see \cite{PS}) and algebraic representations $\repinf$ of $GL_{\infty}(\mathbb{F}_q)$ (cf. \cite{N2}) similarly to the situation for symmetric groups (see \cite{BEAH}). The abelian envelope of $\kar{t}$ in the sense of \cite{EHS} has also been constructed in recent work of Harman-Snowden (see \cite{HS}).

In this article we establish a first link between algebraic representations of the group $GL_{\infty}(\mathbb{F}_q)$ (as defined in Section \ref{def:repinf}) and $\kar{t}$. Let $\mathcal{T}(GL_{\infty}(\mathbb{F}_q))$ denote the full subcategory of $\repinf$ consisting of tensor powers of $\C\mathbb{F}_q^{\infty}$. We establish in Proposition \ref{prop:rep_infty_inj_objects} that there exists a (non-full) symmetric monoidal embedding of $\mathcal{T}(GL_{\infty}(\mathbb{F}_q))$ into $\kar{t}$. The image of this embedding is called $\mathcal{I}_{\infty}$; it contains all objects $\kar{t}$ as well as linear combinations of morphisms $f_R:[k]\to[s]$, $R\subset \F_q^{k+s}$ such that the projection $R\to \F_q^s$ is onto. Both $\mathcal{T}(GL_{\infty}(\mathbb{F}_q))$ and $\mathcal{I}_{\infty}$ do not possess a good notion of duality anymore. In particular, the object $\C\mathbb{F}_q^{\infty} \in \mathcal{T}(GL_{\infty}(\mathbb{F}_q))$ is no longer an $\F_q$-linear Frobenius space. We therefore define in Definition \ref{def:semi_Frob_space} the weaker notion of an $\F_q$-linear semi-Frobenius space which is essentially an $\F_q$-linear Frobenius space which does not necessarily possess a unit (and thus might not \InnaA{be a} rigid \InnaA{object}). An example of this notion is given by the object $\C\mathbb{F}_q^{\infty} \in \mathcal{T}(GL_{\infty}(\mathbb{F}_q))$, and by any $\F_q$-linear Frobenius space.

\begin{theorem}(see Theorem \ref{thrm:inf_univ_prop})
  Let $\mathcal{C}$ be a symmetric monoidal $\C$-linear category, and let $\V $ be an $\F_q$-linear semi-Frobenius space in $\mathcal{C}$. Then there exists a symmetric monoidal $\C$-linear functor
 $$ \Gamma_{\V}: \mathcal{I}_{\infty} \to \mathcal{C}, \;\; [1] \longmapsto \V.$$
\end{theorem}

The proof is very similar to the proof of Theorem \ref{introthrm:univ_prop_Del}. The main difficulty lies in the construction of analogues of the morphisms $\widetilde{f}_R$ in this context. 

\subsection{Structure of the paper}
In Section \ref{sec:classical_endom} we define the spanning set $\{ f_R| R \subset \F_q^{s+k}\}$ for $\Hom_{GL_n(\F_q)}(\V_n^{\otimes s}, \V_n^{\otimes k})$, and show in Section \ref{ssec:composition} that the composition of two morphisms from this set is up to scalar again a morphism in this set. We also determine, in Section \ref{ssec:gen_morph}, a set of generators for the morphisms in $\Rep(GL_n(\F_q))$.

In Section \ref{sec:Deligne_def}, we give an explicit construction of the Deligne categories $\kar{t}$, $t\in \C$ and show that for $t=q^n$ there exists an essentially surjective, full symmetric monoidal functor $F_n: \kar{t=q^n} \longrightarrow \Rep(GL_n(\F_q))$.

In Section \ref{sec:Frob_line_spaces} we define $\F_q$-linear Frobenius spaces as Frobenius algebra objects in a $\C$-linear rigid symmetric monoidal category $\mathcal{C}$ equipped with certain maps which essentially describe an $\F_q$-linear structure on these objects; we show that the objects $\C\F_q^n \in \Rep(GL_n(\F_q))$ and $[1]\in \kar{t=q^n}$ are $\F_q$-linear Frobenius spaces, and in Section \ref{ssec:string_diagr} we introduce string diagrams which are a useful visual tool to manage morphisms between tensor powers of an $\F_q$-linear Frobenius space. We also construct important examples of such morphisms: namely, given such a space $\V$ in a $\C$-linear symmetric monoidal category $\mathcal{C}$, we define for any matrix $A\in Mat_{s\times k}(\F_q)$ a map $\mu_A: \V^{\otimes k} \to \V^{\otimes s}$. 

In Section \ref{sec:linear_algebra_for_Frob_space}, we study the maps $\mu_A$ and show that their compositions and tensor products can be conveniently described using linear algebra over $\F_q$. 

In Section \ref{sec:morph_fR} we define morphisms $\widetilde{f}_R:\V^{\otimes s}\to \V^{\otimes r}$, $R\subset \F_q^{s+k}$ in $\mathcal{C}$, which will serve as analogues of the maps $f_R$ defined in Sections \ref{sec:classical_endom}, \ref{sec:Deligne_def}. We will then prove that they satisfy the same relations on their compositions and tensor products as their counterparts in Section \ref{sec:Deligne_def}.

In Section \ref{sec:univ_prop} we use the results on the morphisms $\widetilde{f}_R$ to construct a SM functor from the Deligne category into a category $\mathcal{C}$ which contains a $t$-dimensional $\F_q$-linear Frobenius space. Any $t$-dimensional $\F_q$-linear Frobenius space in $\mathcal{C}$ gives rise to such a functor. The representation category $\Rep(GL_n(\F_q))$ satisfies a similar universal property with respect to $q^n$-dimensional $\F_q$-linear Frobenius spaces $\V$ annihilated by some exterior power.

In Section \ref{sec:repinf} we study the category $\mathcal{T}(GL_{\infty}(\F_q))$, the full subcategory of $GL_{\infty}(\F_q)-mod$ whose objects are tensor powers of $\V_{\infty} = \C\F_q^{\infty}$ where $\F_q^{\infty} = \bigcup_{n> 0} \F_q^n$. The representation $\V_{\infty}$ is an $\F_q$-linear semi-Frobenius space, and we show that $\mathcal{T}(GL_{\infty}(\F_q))$ satisfies a universal property with respect to $\F_q$-linear semi-Frobenius spaces. For this we need analogs of the morphisms $\widetilde{f}_R$. These morphisms $\hat{f}_R:\V^{\otimes k} \to \V^{\otimes l}$ are then defined and studied in Section \ref{ssec:semi-diagrams}.

%%%%%%%%%%%%%%%%%%%%%%%%%

\section{Notation}\label{sec:notn}
Our main base field throughout the paper is $\C$, but since we work with complex representations of $GL_n(\F_q)$, we will also use vector spaces over $\F_q$.

All functors and all categories in this paper will be $\C$-linear.

\subsection{Finite linear groups}\label{ssec:notn_gen_lin_finite} Let $q \in \mathbb{Z}_{>0}$ be a power of a prime.

To distinguish the vector spaces over the two fields, we denote by $\z$ the zero vector in a vector space over $\F_q$, and by $0$ the zero vector in a complex vector space. 

Given a vector space $V$ over $\F_q$, we denote by $v\dot{+}w \in V$ the sum of vectors $v, w\in V$ and by $\dot{a}v$ multiplication of $v$ by a scalar $a\in \F_q$; the operations over $\C$ will be denoted by the usual notation. For larger sums, we denote by $\dot{\sum}$ sums over $\F_q$ and by $\sum$ sums over $\C$.

\begin{definition}

\mbox{}

 \begin{itemize}
  \item We denote by $GL_n(\F_q) $ the group of all invertible $n \times n$-
matrices over the finite field $\F_q$. 
\item Let $A\in Mat_{d\times m}(\F_q)$. We denote by $Row(A) \subset \F_q^m$ the span of rows of $A$. 
\item
For any $v_1, \ldots, v_s \in \F_q^n$, write $\vec{v} = \begin{bmatrix}
v_1 \\v_2 \\ \vdots \\ v_s
\end{bmatrix} \in \left(\F_q^n\right)^{\oplus s}$. For any matrix $M \in Mat_{k\times s}(\F_q)$ we can now define $M\vec{v}$ as an element of $\left(\F_q^n\right)^{\oplus k}$ (we will write it again as a column vector).
\comment{\item For $1 \leq m \leq n$, let $P_{m, n} 
\subset GL_n(\F_q)$ denote the parabolic subgroup of matrices of the form 
$\begin{bmatrix}
                                   A &B\\ 0 &D
                                  \end{bmatrix}$.}
\item We will denote $\V_{n} := \C\F_q^{n}$.
\item $I_d$ stands for the identity matrix of size $d \times d$ over $\F_q$.
 \end{itemize}
% $\Vc_{n} := \kk(\F_q^{n}\setminus \{\z\})$.

\end{definition}

%The group $GL_n(\F_q)$ acts transitively on the set $\F_q^{n}\setminus 
%\{\z\}$.
%\begin{definition}\label{def:delta_n}
%\mbox{}
%\begin{itemize}
% \item Let $$\Delta_m^{n} := \kk Inj_{\F_q}(\F_q^m, \F_q^{n}) = \Ind_{GL_m(\F
% _q) \times P_{m, n}}^{GL_m(\F
% _q) \times GL_n(\F
% _q)} \kk[GL_m(\F
% _q)]$$ with 
%obvious action of $GL_m(
% \F_q) \times GL_{n}(\F_q)$.
% \item Denote: \[I_{\lambda}^n := \Hom_{GL_k(\F_q)}(\lambda, \Delta_k^{n})\] 
%where 
%$\lambda \in Irr(GL_k(\F_q))$.
%\end{itemize}
%
% 
% 
% 
%
%\end{definition}
%
%\begin{example}
% Clearly, $\Delta_0^{n} = \kk$ (the trivial representation) and $\Delta_1^{n} = 
%\Vc_{n}$. 
%\end{example}

\subsection{SM categories and SM functors} Let $k$ be any field. We adopt the same notion of a $k$-linear symmetric monoidal (SM) category as \cite{Deligne-Milne}. In particular we have a binatural family of braiding morphisms $\gamma_{XY}:X\otimes Y\stackrel{\sim}{\to} Y\otimes X$ which satisfy the constraints of \cite{Deligne-Milne}. For an object $X\in \mathcal{C}$, let $X^*$ denotes its dual (if it exists). If every object has a dual, we call $\mathcal{C}$ {\it rigid}.

\medskip
A $k$-linear functor $F:\mathcal{C}\to \mathcal{C}'$ between two SM categories $\mathcal{C}$ and $\mathcal{C}'$ is symmetric monoidal ({\it SM} for short) if it is equipped with natural isomorphisms $c_{XY}^F:F(X)\otimes F(Y)\stackrel{\sim}{\to}F(X\otimes Y)$ and $\one\stackrel{\sim}{\to}F(\one)$ satisfying the usual  compatibility conditions \cite[Definition 1.8]{Deligne-Milne}. If $F$ is a monoidal functor and $X^*$ the dual of $X$, then $F(X^\ast)$ is a dual of $F(X)$.

%In particular we have a commutative diagram
%$$\xymatrix{
%F(X)\otimes F(Y)\ar[rr]^{\gamma_{F(X)F(Y)}}\ar[d]^{c_{XY}^F}&&F(Y)\otimes F(X)\ar[d]^{c_{YX}^F}\\
%F(X\otimes Y)\ar[rr]^{F(\gamma_{XY})}&& F(Y\otimes X).}$$

\medskip
We say that a $k$-linear SM category $\mathcal{C}$ is a tensor category if

\begin{enumerate} 
\item[(i)] $\mathcal{C}$ is abelian;
\item [(ii)] the canonical morphism $k \to \End(\one)$ is an isomorphism;
\item [(iii)] $\mathcal{C}$ is rigid;
\item [(iv)] every object in $\mathcal{C}$ has finite length.
\end{enumerate}

We also call such categories `pre-Tannakian categories' as in \cite{CO2}. If such a category has in addition a faithful $k$-linear tensor functor to $Vec_k$, the category of finite dimensional $k$ vector spaces, we will call it Tannakian.

\section{Endomorphisms of \texorpdfstring{$\V_n^{\otimes k}$}{tensor powers of Vn}}\label{sec:classical_endom}
\subsection{Basis of endomorphisms}\label{ssec:basis_of_endomorphisms}
Fix $n\geq 1$. Let $G = GL_n(\F_q)$, $V = \F_q^n$ (seen as a vector space over $\F_q$).

Let $s, k\in \Z_{\geq 0}$. 
Recall that $$\V_n^{\otimes s} = \C V^{\times s}.$$ We will consider a natural basis for $\V_n^{\otimes s}$ given by the elements of $V^{\times s}$. For $v_1, \ldots, v_s \in V$, we write $(v_1| \ldots| v_s) \in V^{\times s}$.

Let $R \subset \F_q^{s+k} = \F_q^s \times \F_q^k$ be a linear $\F_q$-subspace. We define a $G$-invariant subspace $$R^{\perp}_{\InnaA{V}} = \{(v_1| \ldots| v_{s+k}) \in V^{\times (s+k)} \; \rvert \; \forall u = (u_1, \ldots, u_{s+k}) \in R, \; \dot{\sum}_i u_i v_i =0  \} $$ in $V^{\times (s+k)}$. This allows us to define a map
$$f_R:\V_n^{\otimes s} \to \V_n^{\otimes k}, \;\;  v_1\otimes \ldots \otimes v_s \mapsto \sum_{\substack{(w_1|\ldots|w_k)\in V^{\times k} \\ (v_1| \ldots| v_s|w_1|\ldots|w_k) \in R^{\perp}_{\InnaA{V}}}} w_1 \otimes\ldots\otimes w_k. $$

\begin{example}
 \begin{enumerate}
  \item For $R=\{\z\} \subset \F_q^{s+k}$ we have: $$f_{\{\z\}} :\V_n^{\otimes s} \to \V_n^{\otimes k}, \;\; v_1\otimes \ldots \otimes v_s \mapsto \sum_{(w_1|\ldots|w_k)\in V^{\times k}} w_1\otimes\ldots\otimes w_k. $$ For instance, for $s=k=1$, we obtain: $f_{\{\z\}} :\V_n \to \V_n$, $v\mapsto \sum_{w \in V} w$ for any $v\in V$.
  \item For $R=\F_q^{s+k}$ we have: $$f_{\F_q^{s+k}} :\V_n^{\otimes s} \to \V_n^{\otimes k}, \;\; v_1\otimes \ldots\otimes v_s \mapsto \begin{cases}
   \z \otimes\ldots\otimes \z &\text{ if } \; v_1 =\ldots = v_s =\z\\                                                                                                                         
   0  &\text{ else }   
   \end{cases}. $$ For instance, for $s=0$, $k=1$, we obtain: $f_{\F_q^{1}} :\C \to \V_n$, $1\mapsto \z$.
   \item For $s=k=1$ and $R=span_{\F_q}(1,0)$ we obtain the map $f_R: \V_n \to \V_n$, $v \to \z$ for any $v\in V\setminus\{\z\}$, and $\z \mapsto \sum_{w\in V} w$.
   \item For $s=k=1$ and $b\in \F_q$, let $R=span_{\F_q}(-b, 1)$. We obtain the map $f_R: \V_n \to \V_n$, $v \to \dot{b} v$ for any $v\in V$.
 \end{enumerate}

\end{example}

\begin{lemma}\label{lem:basis}
 The set $\{ f_R| R \subset \F_q^{s+k}\}$ spans the space $\Hom_G(\V_n^{\otimes s}, \V_n^{\otimes k})$. Furthermore, if $n\geq s+k$ then this set is a basis of $\Hom_G(\V_n^{\otimes s}, \V_n^{\otimes k})$.
\end{lemma}

\begin{proof}
The space $\Hom_G(\V_n^{\otimes s}, \V_n^{\otimes k})$ has a basis $\{\breve{f}_O | O \in V^{\times (s+k)}//G\}$ which is parameterized by the orbits of $G$ on $V^{\times s} \times V^{\times k} = V^{\times (s+k)}$. Here $\breve{f}_O$ is defined as follows: for $(v_1|\ldots |v_s) \in V^{\times s}$, $$\breve{f}_O(v_1\otimes \ldots \otimes v_s) =  \sum_{\substack{(w_1|\ldots|w_k)\in V^{\times k} \\ (v_1| \ldots| v_s|w_1|\ldots|w_k) \in O}} w_1 \otimes\ldots\otimes w_k.$$
The orbits $O \in V^{\times (s+k)}//G$ are in bijection with subspaces $R \subset \F_q^{s+k}$. This bijection can be constructed as follows: a subspace $R \subset \F_q^{s+k}$ such that $\dim R \geq s+k-n$ corresponds to the orbit
$$O_R:= \{(v_1| \ldots| v_s|w_1|\ldots|w_k)\in R^{\perp}_{\InnaA{V}}\,\rvert \; \dim span\{v_1, \ldots, v_s,w_1,\ldots,w_k\} = s+k-\dim R\}$$
If $\dim R < s+k-n$, let us write $O_R := \emptyset$. Thus $\{\breve{f}_{O_R}| \;R \subset \F_q^{s+k}\}$ is a spanning set of $\Hom_G(\V_n^{\otimes s}, \V_n^{\otimes k})$, and it is a basis if $s+k\leq n$. 

Note that for any $(v_1| \ldots| v_s|w_1|\ldots|w_k)\in R^{\perp}_{\InnaA{V}}$, we have: $$\dim span\{v_1, \ldots, v_s,w_1,\ldots,w_k\} \leq s+k-\dim R.$$

Hence $$f_R = \sum_{R \subset S \subset \F_q^{s+k}} \breve{f}_{O_S}$$ and so $\{f_R | \;R \subset \F_q^{s+k}\}$ is a spanning set of $\Hom_G(\V_n^{\otimes s}, \V_n^{\otimes k})$, and it is a basis if $s+k\leq n$.
\end{proof}

\subsection{Composition}\label{ssec:composition}
Let $f_R:\V_n^{\otimes l} \to \V_n^{\otimes k}$, $f_S:\V_n^{\otimes k} \to \V_n^{\otimes s}$ be two endomorphisms as above corresponding to subspaces $R \subset \F_q^{l+k}, S\subset \F_q^{k+s}$.

Consider the subspaces $R\times \{\vec{0}\},  \{\vec{0}\}\times S \subset \F_q^{l+k+s}$. We will denote them $(R, 0)$ and $(0, S)$.  

We will also denote by $(\F_q^l ,0, \F_q^s)$ the image of $\F_q^l \times \F_q^s \hookrightarrow \F_q^{l+k+s}$: namely, this is the subspace of $\F_q^{l+k+s}$ consisting of vectors whose coordinates $l+1, \ldots, l+k$ are zero.

\begin{definition}\label{def:composition_star_d_R_S}
 Denote $$S\circledast R :=  \left((R, 0)+ (0,S)\right) \cap (\F_q^l ,0, \F_q^s),$$ and denote by $S\star R$ the (isomorphic) image of $S\circledast R$ under the identification $ \F_q^{l+s} \cong (\F_q^l ,0, \F_q^s)$. Let
 $$d(R, S) =  k - \dim \left(\substack{\text{ projection of } (R, 0)+ (0,S) \\ \text{ on the subspace } \F_q^k \subset \F_q^{l+k+s}}\right)$$
 where the subspace $\F_q^k \subset \F_q^{l+k+s}$ consists of the vectors in $\F_q^{l+k+s}$ whose coordinates $1,\ldots, l$, and $ l+k+1,\ldots, l+k+s$ are zero.
\end{definition}

\begin{proposition}\label{prop:composition}
 The composition $f_S \circ f_R: \V_n^{\otimes s} \to \V_n^{\otimes l}$ equals
$$ q^{n \cdot d(R, S)} f_{S \star R}.$$
\end{proposition}

\begin{proof}
Let $v_1, \ldots, v_s \in V$. We have:
$$(f_S \circ f_R)(v_1 \otimes \ldots\otimes v_s) = \sum_{\substack{(u_1| \ldots| u_l) \in V^{\times l}, \;\; (w_1|\ldots|w_k)\in V^{\times k} \\ (v_1| \ldots| v_s|w_1|\ldots|w_k) \in R^{\perp}_{\InnaA{V}} \, \text{ and } \\
(w_1|\ldots|w_k|u_1| \ldots| u_l) \in S^{\perp}_{\InnaA{V}} }} u_1 \otimes\ldots\otimes u_l$$
Clearly any summand in this sum satisfies:
$$(v_1| \ldots| v_s|w_1|\ldots|w_k|u_1| \ldots| u_l) \in (R,0)^{\perp}_{\InnaA{V}} \cap (0,S)^{\perp}_{\InnaA{V}} = \left( (R,0) + (0,S) \right)^{\perp}_{\InnaA{V}}$$ for some $(w_1|\ldots|w_k)\in V^{\times k}$ and so
$(v_1| \ldots| v_s|u_1| \ldots| u_l) \in (S\star R)^{\perp}_{\InnaA{V}}$. Vice versa, for every $(u_1| \ldots| u_l) \in V^{\times l}$ such that $(v_1| \ldots| v_s|u_1| \ldots| u_l) \in (S\star R)^{\perp}_{\InnaA{V}}$, we have: the term $u_1 \otimes\ldots\otimes u_l$ appears in the sum above with multiplicity which is equal to the number of elements in the set
$$Z=\{(w_1|\ldots|w_k)\in V^{\times k} \;\rvert\; (v_1| \ldots| v_s|w_1|\ldots|w_k) \in R^{\perp}_{\InnaA{V}} \, \text{ and } \\
(w_1|\ldots|w_k|u_1| \ldots| u_l) \in S^{\perp}_{\InnaA{V}}\}$$
The set $Z \subset V^{\times k}$ is thus an affine $\F_q$-space, described by a system of linear equations over $\F_q$ with variables $w_1,\ldots,w_k\in V$. The rank of this system is precisely the dimension of the projection of $ (R,0) +(0,S)$ on the subspace $\F_q^k \subset \F_q^{s+k+l}$, that is, $k-d(R, S)$. Then the size of $Z$ is $\abs{V}^{d(R, S)} = q^{n d(R, S)}$ which completes the proof of the statement.

\end{proof}

\subsection{Orthogonal indexing}\label{ssec:knop_indexing}

In \cite{K, K2}, there appears another indexing of the basis $\{f_R: R\subset \F_q^{s+k}\}$ for the $Hom$-space $Hom_{GL_n(\F_q)}(\V_n^{\otimes s},\V_n^{\otimes k})$: namely, the morphism which we denoted by $f_R$ would correspond in the notation of \cite{K, K2} to the subspace
\begin{equation}\label{eq:orthog}
 \{(v_1, \ldots, v_{s+k}) \in \F_q^{s+k} \; \rvert \; \forall u = (u_1, \ldots, u_{s+k}) \in R, \; \dot{\sum}_i u_i v_i =0  \}.
\end{equation}

The goal of this subsection is to establish the compatibility between the bases of morphisms used in this paper and the bases used in \cite{K, K2}.

In this Subsection we will denote the subspace \eqref{eq:orthog} of $\F_q^{s+k} $ by $R^{\perp}$ \InnaA{(so that $R^{\perp}:=R^{\perp}_{\F_q^1}$)}. Note that for any $R, S \subset \F_q^{s+k}$, we have: 
$$(R^{\perp})^{\perp} = R, \;\; (R+S)^{\perp} = R^\perp \cap S^\perp. $$

\begin{remark}
 Considering the isomorphism $V^{\times (s+k)} = V\otimes_{\F_q} \F_q^{s+k}$, the subspace $R^{\perp}_{\InnaA{V}} \subset V^{\times (s+k)}$ we considered before would be given by $V\otimes_{\F_q} R^{\perp}$. 
\end{remark}

To see \InnaA{what} the composition of morphisms \InnaA{looks like} in the alternative indexing, we will use the following construction.

Let $f_R:\V_n^{\otimes l} \to \V_n^{\otimes k}$, $f_S:\V_n^{\otimes k} \to \V_n^{\otimes s}$ correspond to subspaces $R \subset \F_q^{l+k}, S\subset \F_q^{k+s}$.

Consider the diagram 
$$\xymatrix{ &{} &{} &{R^{\perp} \times_{\F_q^k} S^{\perp}} \ar[ld] \ar[rd]&{}&{}\\&{} &R^{\perp} \ar[ld] \ar[rd] &{} &S^{\perp} \ar[ld]\ar[rd] &{}\\ &\F_q^l &{} &\F_q^k &{} &\F_q^s} $$
where the bottom arrows are compositions of embeddings $R^{\perp} \hookrightarrow \F_q^{l+k}, S^{\perp}\hookrightarrow \F_q^{k+s}$ with the projections on the subspaces $\F_q^l, \F_q^k, \F_q^s$. 
\begin{notation}\label{notn:Knop}
    We denote by $S^{\perp}\diamond R^{\perp}$ the image of the natural map $R^{\perp} \times_{\F_q^k} S^{\perp} \to \F_q^l\times\F_q^s$, and by $e(S^{\perp}, R^{\perp})$ the dimension of the kernel of this map.
\end{notation}

\begin{lemma}\label{lem:Knop_basis_vs_ours}
 We have: $S^{\perp}\diamond R^{\perp} = (S\star R)^{\perp}$ and $e(S^{\perp}, R^{\perp})=d(R, S)$.
\end{lemma}
\begin{proof}

Consider the subspace $S\circledast R \subset \F_q^{l+k+s}$ we defined in Section \ref{ssec:composition}. Then $$(S\circledast R)^\perp = \left((R, 0)+(0, S)\right)^{\perp} + \InnaA{(\F_q^l, 0, \F_q^s)}^\perp = (R, 0)^{\perp}\cap(0, S)^{\perp} + (0, \F_q^k, 0).$$
Let us use this equality to describe $(S\star R)^{\perp}$. Under the indentification $(\F_q^l, 0, \F_q^s)\cong \F_q^l\times\F_q^s$, the space $(S\star R)^{\perp}$ is identified with the subspace of $(\F_q^l, 0, \F_q^s)$ which is orthogonal to $S\circledast R$, so it is given by the intersection $$(S\circledast R)^\perp \cap (\F_q^l, 0, \F_q^s) = (R, 0)^{\perp}\cap(0, S)^{\perp} \cap (\F_q^l, 0, \F_q^s).$$
Thus
$$(S\star R)^{\perp} = \{(v, u)\in \F_q^l\times\F_q^s \;|\; \exists w\in \F_q^k \,:\, (v,w)\in R^\perp, \,(w,u)\in S^\perp\}.$$
On the other hand, the map $R^{\perp} \times_{\F_q^k} S^{\perp} \to \F_q^l\times\F_q^s$ is the composition of the embedding $$\iota:R^{\perp} \times_{\F_q^k} S^{\perp} \hookrightarrow  \F_q^{l+k+s}, \;\; ((v, w), (w, u)) \longmapsto (v, w, u)$$ and the projection $\F_q^{l+k+s} \twoheadrightarrow \F_q^l\times\F_q^s$. This proves that $S^{\perp}\diamond R^{\perp}=(S\star R)^{\perp}$.

Now, $e(S^{\perp}, R^{\perp})$ equals
$$ \dim Ker\left( R^{\perp} \times_{\F_q^k} S^{\perp} \to \F_q^l\times\F_q^s \right) = \dim \left(Im(\iota) \cap (0, \F_q^k, 0) \right) = \dim  \left( (R, 0)^{\perp}\cap(0, S)^{\perp} \cap (0, \F_q^k, 0) \right),$$
while 
$$d(R, S) =  k - \dim \left(\substack{\text{ projection of } (R, 0)+ (0,S) \\ \text{ on the subspace } \F_q^k \subset \F_q^{l+k+s}}\right) = \dim Coker \left(\substack{\text{ projection of } (R, 0)+ (0,S) \\ \text{ on the subspace } \F_q^k \subset \F_q^{l+k+s}}\right).$$

Denote: $U:= (R, 0)+(0, S)$, $W:= (\F_q^l, 0, \F_q^s)$ subspaces of $\F_q^{l+k+s}$. Then $W^{\perp} = (0, \F_q^k, 0)$. Let $p:\F_q^{l+k+s} \to W^{\perp}$ denote the natural projection.

We have: $e(S^{\perp}, R^{\perp}) = \dim U^{\perp}\cap W^{\perp}$, while 
\begin{align*}
 d(R, S) &= \dim Coker(p\vert_{U}) = \dim W^{\perp} - \dim Im(p\vert_{U}) = \dim W^{\perp} - \dim U +\dim Ker(p\vert_{U}) = \\
 &= \dim W^{\perp} - \dim U +\dim U\cap W = \dim W^{\perp} -\dim U^{\perp} \cap W \\ &= \dim U^{\perp} \cap W^{\perp} = e(S^{\perp}, R^{\perp}).
\end{align*}

\end{proof}

\subsection{Generating morphisms}\label{ssec:gen_morph}
We will now give a list of generating morphisms in the full subcategory $\mathcal{T}(GL_n(\F_q))$ of $\Rep(GL_n(\F_q))$ generated by tensor powers of $\V_n$.

The following proposition will follow from the results of Section \ref{sec:univ_prop}. Namely, we will construct a Deligne category which will allow $\Rep(GL_n(\F_q))$ as a quotient, and show that the analogues of the maps below generate the morphisms in the Deligne category. 

However, for motivational purposes, we include here a direct proof of the proposition.
\begin{proposition}\label{prop:gen_morphisms}
Consider the following morphisms:
 \begin{enumerate}
  \item Morphisms 
\begin{align*}
&\eps:=f_{\{\z\}}:\triv \to \V_n, \;\;,  1 \mapsto \sum_{v\in V} v, \; \; \eps^*:=f_{\{\z\}}:\V_n \to \triv, \;\;, v \to 1   \;\; \text{ for } v\in V,\\ 
&m:=f_{\{(a+b,-a,-b)|a,b\in \F_q\}}: \V_n \otimes \V_n \to \V_n , \;\; v\otimes w \mapsto \delta_{v, w} v \;\; \text{ for } v,w\in V, \\
&m^*:=f_{\{(a+b,-a,-b)|a,b\in \F_q\}}: \V_n  \to \V_n\otimes \V_n , \;\; v \mapsto v\otimes v \;\; \text{ for } v\in V,  \\
&\sigma:=f_{\{(a,b,-b, -a)|a,b\in \F_q\}}: \V_n \otimes \V_n \to \V_n \otimes \V_n, \;\; v\otimes w \mapsto w\otimes v \;\; \text{ for } v,w\in V.
\end{align*}

  \item Morphisms 
  \begin{align*}
  &z:=f_{\F_q^1}: \triv \to \V_n,\;\;  1 \mapsto \z, \\
&\forall a\in \F_q, \;  \mu_a:=f_{\{(-ab,b)|b\in \F_q\}}: \V_n \to \V_n, \; \; v \mapsto \dot{a}v  \;\; \text{ for } v\in V,
\\& \dot{+}:=f_{\{(b,b,-b)|b\in \F_q\}}:\V_n \otimes \V_n \to \V_n, \;\; v\otimes w \mapsto v\dot{+}w   \;\; \text{ for } v,w\in V.           
\end{align*}

 \end{enumerate}
These morphisms generate the space $\left(\bigoplus_{k,l \geq 0} \Hom_G(\V_n^{\otimes k}, 
\V_n^{\otimes l}), \otimes, \circ \right)$.
\end{proposition}
The morphisms described in Proposition \ref{prop:gen_morphisms} can be interpreted as follows: \begin{itemize}
  \item The morphisms $\mu_a, \dot{+}$ give $\V_n$ the structure of an (infinite) dimensional vector space over $\F_q$, with $z:\triv \to \V_n$ defining the zero vector $\z \in \V_n$; \item The rest of the morphisms make $\V_n$ into a commutative Frobenius algebra in $\Rep(GL_n(\F_q))$. The algebra is self-dual, via the pairings 
  \begin{align*}
&ev:=f_{\{(a,-a)|a\in \F_q\}}: \V_n \otimes \V_n \to \triv, \;\; v\otimes w \mapsto \delta_{v, w}   \;\; \text{ for } v,w\in V,\\
&coev:=f_{\{(a,-a)|a\in \F_q\}}: \triv \to \V_n\otimes \V_n,\;\; 1 \mapsto \sum_{v\in V} v\otimes v. 
\end{align*}
where  $$ev = \eps^* \circ m, \;\; coev = m^*\circ \eps.$$
  \begin{align*}
                                                                      \end{align*}
In particular, we may replace  
$m, \eps$ with $coev, ev$ in the list above. Indeed, 
$$m =  (ev \otimes \id ) \circ  ( \id \otimes m^*), \;\; \eps =  (\eps^* \otimes \id) \circ coev .$$
                                                                                    \end{itemize}

\begin{remark}
We have: $\mu_0 = z \circ \eps^*$.
\end{remark}
\begin{proof}

Let $R \subset \F_q^{k+l}$ be an $\F_q$-subspace. We will show that $f_R: 
\V_n^{\otimes k} \to \V_n^{\otimes l}$ can be obtained from the morphisms 
listed in Proposition \ref{prop:gen_morphisms} using tensor products and 
composition.

We already saw that the morphisms $ev, coev$ can be obtained 
from the generating morphisms.

Using the maps $\sigma$, $ev$ and $coev$, we can reduce the question to the case when 
$f_R$ sends 
$v_1\otimes \ldots \otimes v_k$ ($v_1, \ldots, v_k \in V$) to the tensor 
product of $l$ vectors 
$$ \left( \dot{\sum}_{j=1}^{k} \dot{A}_{j1} v_j \right) \otimes  \left( 
\dot{\sum}_{j=1}^{k} \dot{A}_{j2} v_j \right) \otimes \ldots \otimes \left( 
\dot{\sum}_{j=1}^{k} \dot{A}_{jl} v_j \right).$$ for some fixed $A \in 
Mat_{k\times l}(\F_q)$. 

In other words, 
$R=\{(Au, -u)| u \in \F_q^l\}$ for some $A \in Mat_{k\times l}(\F_q)$.

If $l=0$, then $f_R  = (\eps^*)^{\otimes k}$ and we are done.

If $k=0$, then $f_R  = z^{\otimes l}$ and we are done.
So from now on we assume that 
$l, k\neq 0$.

Next, we can present our morphism $f_R$ as composition 
$$ f_R = \left(\otimes_{i=1}^l \mu_{A_i} \right)\circ w \circ \left((m^*)^{it}\right)^{\otimes k}$$
where 
\begin{itemize}
 \item $\left((m^*)^{it}\right)^{\otimes k}$ is given by $$\left((m^*)^{it}\right)^{\otimes k}: \V_n^{\otimes k} \to 
\V_n^{\otimes kl}, \; \; v_1\otimes \ldots \otimes v_k \mapsto  v_1^{\otimes l} 
\otimes \ldots \otimes v_k^{\otimes l},\ \text{ for } v_1, \ldots, v_k \in V $$ 
 \item $w: \V_n^{\otimes kl}\to \V_n^{\otimes kl}$ is an endomorphism given by 
a permutation of the tensor factors, 
so that $w \circ \left((m^*)^{it}\right)^{\otimes k}$ is the morphism
$$\V_n^{\otimes k} \to 
\V_n^{\otimes kl}, \; \; v_1\otimes \ldots \otimes v_k \mapsto  (v_1 
\otimes \ldots \otimes v_k)^{\otimes l},\ \text{ for } v_1, \ldots, v_k \in V. 
$$ Given a pure tensor $v_0 \otimes \ldots 
\otimes v_{kl-1}$, $w$ sends it to a pure tensor $v'_0 \otimes \ldots \otimes 
v'_{kl-1}$ such that $v'_{r+qk} :=v_{q + rl}$ for $0\leq q < l$, $0\leq r <k$.
 \item $\mu_{A_i}: \V_n^{\otimes k} \to \V_n$ is given by $$v_1 \otimes \ldots 
\otimes v_k \mapsto  \dot{\sum}_j \dot{A}_{ji} v_j.$$
\end{itemize}

Clearly, $\left((m^*)^{it}\right)^{\otimes k}$ is obtained from $m^*$ through tensor products and 
iterated compositions, while $w$ is obtained from $\sigma$ through tensor products 
(with $\id$) and compositions. So it remains to check that for each $[a_1, 
\ldots, a_k]\in Mat_{1\times k}(\F_q)$, the map $$\mu_{[a_1, \ldots, a_k]} :\V_n^{\otimes k} 
\to \V_n, \;\;v_1 \otimes \ldots \otimes v_k \mapsto  \dot{\sum}_{j=1}^k 
\dot{a}_{j} v_j$$ can be obtained from the morphisms listed in Proposition 
\ref{prop:gen_morphisms}. Indeed, this is done by composing the map 
$$\otimes_{j=1}^k \mu_{a_j}: \V_n^{\otimes k} \to  \V_n^{\otimes k},\;\; v_1 
\otimes \ldots \otimes v_k \mapsto  \dot{a}_1 v_1 \otimes \ldots \otimes 
\dot{a}_k v_k$$ with iterative applications of the map $$\dot{+} \otimes \id: 
\V_n^{\otimes s} \to \V_n^{\otimes s-1},\;\; v_1 \otimes \ldots \otimes v_s 
\mapsto  (v_1 \dot{+} v_2) \otimes \ldots \otimes v_s$$ for $s=k, k-1,\ldots, 
2$. This completes the proof of the proposition.
\end{proof}

\section{Construction of the Karoubian Deligne category}\label{sec:Deligne_def}
Fix $t\in \C$. We construct the Deligne category $\kar{t}$ in this section, the analogue of $\underline{\Rep}(S_t)$. The original construction is due to Knop (\cite{K, K2}), but his definition is a bit different and we only show in Section \ref{ssec:knop} that our categories are equivalent. It also appears in work of Meir (\cite{M}) (although there only the semisimplified version is considered) and has \InnaA{recently been studied} by Harman and Snowden \cite{HS}.

\medskip
We define $\mathcal{T}(\underline{GL}_t)$ as a category whose objects are non-negative integers, denoted by $[k]$.
Define $$Rel_{s, k} = \{ R \subset \F_q^{s+k} \; \text{ linear subspace }\}, \quad \; \Hom_{\mathcal{T}(\underline{GL}_t)} ([s], [k]) = \C Rel_{s, k}$$
with composition of morphisms given as follows: for $R \in Rel_{s, k}$, $S \in Rel_{k, l}$ we set $$S \circ R := t^{d(R, S)} S \star R.$$

We will also define the tensor structure of $\mathcal{T}(\underline{GL}_t)$ by setting $[l]\otimes [k]:=[l+k]$, and tensor products of morphisms are given as follows: for $\F_q$-linear subspaces $R_1 \subset \F_q^{r_1}$ and $R_2\subset \F_q^{r_2}$, set $f_{R_1} \otimes f_{R_2} :=f_{R_1\times R_2}$ where $R_1\times R_2 \subset \F_q^{r_1+r_2}$ is considered as an $\F_q$-linear subspace.

Each object is then self-dual via the evaluation and coevaluation maps
$$f_R:[k]\otimes [k] \to \triv, \;\; f_R: \triv \to [k]\otimes [k]$$
where $R \subset \F_q^{2k}$ is given by $R=\{(a_1, \ldots, a_k, -a_k, \ldots, -a_1)|a_1, \ldots, a_k \in \F_q\}.$

The symmetric structure is given by the morphisms
$\sigma=f_{R'}: [l]\otimes [k] \to  [k]\otimes [l]$ where $$ R' = \{(a_1, \ldots, a_l, b_1, \ldots, b_k,-b_1, \ldots, -b_k,-a_1, \ldots, -a_l)|\, \forall i, j, \; a_i,  b_j \in \F_q\} \subset  \F_q^{2k+2l}$$

It is easy to see that we obtain a rigid symmetric monoidal category in this way.

\medskip

We define $\kar{t}$ as the Karoubi additive envelope of $\mathcal{T}(\underline{GL}_t)$. The reasons for this  notation will only become fully clear in an upcoming article. In the $S_t$-case the category $\underline{\Rep}(S_t)$ embeds into a tensor category $\underline{\Rep}^{ab}(S_t)$, the abelian envelope \cite{CO}. The original Deligne category then coincides with the full subcategory of tilting objects in $\underline{\Rep}(S_t)$. We expect the same statement in the $\kar{t}$-case.

The following result is a straightforward corollary of the results in Section \ref{sec:classical_endom}:

\begin{proposition}\label{prop:functor_F_n}
For values $t = q^n$, $n \in \Z_{\geq 0}$, we have a full SM functor
 $$F_n: \mathcal{T}(\underline{GL}_t) \longrightarrow \Rep(GL_n(\F_q))$$
 sending the object $[k]$ to $\V_n^{\otimes k}$ and sending the morphism $f_R:[l] \to [k]$ to the morphism $f_R \in \Hom_{GL_n(\F_q)}(\V_n^{\otimes l},\V_n^{\otimes k})$ as defined in Section \ref{sec:classical_endom}.
\end{proposition}

This functor extends naturally to the Karoubi additive envelope $\kar{t=q^n}$ of $\mathcal{T}(\underline{GL}_t)$, and we obtain:

\begin{corollary}\label{cor:specialization}
 For values $t = q^n$, $n \in \Z_{\geq 0}$, we have a full, essentially surjective SM functor
 $$F_n: \kar{t=q^n} \longrightarrow \Rep(GL_n(\F_q)).$$
\end{corollary}
\begin{proof}
 To show that this functor is essentially surjective, recall the category $\Rep(GL_n(\F_q))$ is generated by the objects $\V_n^{\otimes k}$, $k\geq 0$ under taking finite direct sums and direct  summands.
\end{proof}

\begin{remark}\label{rmk:F_n_ff}
 By Lemma \ref{lem:basis}, the functor $F_n$ is fully faithful on the full subcategory $\kar{t}^{\leq \lfloor n/2\rfloor }$ of $\kar{t}$ generated by the objects $[k]$ where $ 0\leq k\leq \lfloor n/2\rfloor$. Analogous results in the $S_t$, $O_t$ and $GL_t$-case were proven in \cite{Del07}.
\end{remark}

\section{Frobenius linear spaces}\label{sec:Frob_line_spaces}

Throughout this section, let $\mathcal{C}$ be a $\C$-linear rigid SM category. We introduce the notion of a Frobenius linear space which is required to state the universal property of $\kar{t}$. The crucial assertions for the proof are Lemma \ref{lem:tensor_prod_f} and Proposition \ref{prop:composition_phi_R}.

\subsection{Frobenius linear spaces: definition}\label{ssec:Frob_obj}
\begin{definition}\label{def:Frob_linear_space}
 Let $\V\in\mathcal{C}$ be an object equipped with the following 
structures:
 \begin{enumerate}[label=(\subscript{Str}{{\arabic*}})]
%   \item $\V = \triv \oplus \Vc$ for some object $\Vc \in \mathcal{C}$. We denote the embedding $\triv \to \V$ by $z$.
  \item \label{rel:Frob_alg} $\V$ is equipped with the structure of a Frobenius algebra object in 
$\mathcal{C}$. That is, $\V$ is equipped with maps $m:\V^{\otimes 2}\to \V$, $\eps:\triv \to \V$, $m^*:\V\to \V^{\otimes 2}$, $\eps^*:\V \to \triv$ such that the following conditions hold:
\begin{enumerate}[label=(\subscript{Fr}{{\arabic*}})]
\item\label{rel:Frob_alg1} $\V$ is a commutative unital algebra object with 
multiplication $m$ and unit $\eps$, and a 
cocommutative counital coalgebra object
with comultiplication $m^*$ and 
counit $\eps^*$,

 \item\label{rel:Frob_alg2} Frobenius Relations: $m^* \circ m = (\id \otimes m)\circ (m^* 
\otimes \id) =(m \otimes \id) \circ  (\id \otimes m^*) $, and 

 Speciality Relation: $m\circ m^*=\id$.
\end{enumerate}

  \item\label{rel:F_q_lin} $\V$ is a module over the field $\F_q$. In other words, $\V$ is equipped 
with
maps \begin{align*}
      &\dot{+}: \V \otimes \V \to \V, \\ &\mu:(\F_q, 
\cdot) \to (End_{\mathcal{C}}(\V), \circ),\;\; a \mapsto \mu_a, \\ &z :\triv 
\to \V
     \end{align*}
 satisfying the following conditions:
  \begin{enumerate}[label=(\subscript{Lin}{{\arabic*}})]
  \item\label{rel:F_q_lin_plus_ass_comm} $\dot{+}$ is associative and commutative: $\dot{+} \circ (\dot{+} \otimes \id) = \dot{+} \circ (\id \otimes \dot{+})$, $\dot{+} \circ \sigma = \dot{+}$ where $\sigma\in \End( 
\V \otimes \V)$ is the symmetry morphism.
   \item\label{rel:F_q_lin_zero} $z$ serves as ``the embedding of the $\dot{0}$ vector'': $\dot{+} \circ (z \otimes \id) = \id, \dot{+} \circ (\id \otimes z) = \id$,
\item\label{rel:F_q_lin_mu} Properties of the maps $\mu_a$: for all $a, b\in \F_q$, $\mu_a \circ \mu_b = \mu_{ab}$ and $\mu_1 = \id$, $\mu_0=z\circ \eps^*$.
\item\label{rel:F_q_lin_plus_lin_distr} Linearity of $\mu_a$ with respect to 
$\dot{+}$: for any $a, b \in \F_q$, 
$\mu_{a+b} = \dot{+} \circ (\mu_a \otimes \mu_b) \circ m^*$.

Distributivity of $\mu_a$: for any $a \in \F_q$, $\mu_a \circ 
\dot{+} = \dot{+} \circ (\mu_a \otimes \mu_a)$.

   \end{enumerate}

 \end{enumerate}
 Assume furthermore that the above structures satisfy the following 
compatibility relations:
\begin{enumerate}[label=(\subscript{Rel}{{\arabic*}})]
\item\label{rel:mu_coalg_mor} For each $a \in \F_q^{\times}$, $\mu_a$ is a morphism of coalgebras, i.e. the 
following diagrams are commutative:
 $$ \xymatrix{&\V \ar[r]^{m^*} &\V \otimes \V \\ &\V \ar[u]^{\mu_a} 
\ar[r]^{m^*} &{\V \otimes \V} \ar[u]_{\mu_a \otimes \mu_a}} \;\;  
\xymatrix{&\V \ar[r]_{\eps^*} &{\triv} \\ &\V \ar[u]^{\mu_a} 
\ar[ru]_{\eps^*}  
&{} } $$ Furthermore, $\mu_a$ respects the multiplication structure, i.e. the 
following diagram is commutative: $$\xymatrix{&\V 
\otimes \V \ar[r]^{m} &\V  \\  &{\V \otimes \V} \ar[u]^{\mu_a \otimes \mu_a}
\ar[r]^{m} &\V \ar[u]_{\mu_a}} $$
 \item\label{rel:z_coalg_mor} The map $z$ is a morphism of coalgebras, i.e. the 
following diagrams are commutative:
 $$ \xymatrix{&\V \ar[r]^{m^*} &\V \otimes \V \\ &\triv \ar[u]^{z} 
\ar[r]^{\id} &{\triv \otimes \triv} \ar[u]_{z \otimes z}} \;\;  
\xymatrix{&\V \ar[r]_{\eps^*} &{\triv} \\ &\triv  \ar[u]^{z} \ar[ru]_{\id}  
&{} }$$
Furthermore, $z$ respects the multiplication structure, i.e. the 
following diagram is commutative:
$$\xymatrix{&\V 
\otimes \V \ar[r]^{m} &\V  \\  &{\triv \otimes \triv} \ar[u]^{z 
\otimes z}
\ar[r]^{\id} &\triv \ar[u]_{z}}$$
 
\item\label{rel:plus_coalg_mor} The morphism $\dot{+}$ is a morphism of coalgebras, i.e. the 
following diagrams are commutative:
 $$ \xymatrix{&\V^{\otimes 2} \ar[d]^{\dot{+}} \ar[rrr]^{(\id \circ \sigma 
\circ \id )\circ (m^* \otimes m^*)} &{}  &{} &\V^{\otimes 4} \ar[d]_{\dot{+} 
\otimes \dot{+}} \\ &\V  
\ar[rrr]^{m^*} &{} &{} &{\V \otimes \V} } \;\;  
\xymatrix{&\V \ar[r]_{\eps^*} &{\triv} \\ &{\V \otimes \V} \ar[u]^{\dot{+}} 
\ar[ru]_{\eps^*\otimes \eps^*}  
&{} } $$ 
Here $\sigma: \V^{\otimes 2} \to  \V^{\otimes 2}$ is the symmetry morphism.
\item\label{rel:cancellation_axiom} Cancellation Axiom: $$\xymatrix{&\V^{ \otimes 3} 
\ar[d]^{\sigma\otimes \id} \ar[rr]^{\id \otimes m^*\otimes \id } &{} &\V^{ 
\otimes 4} \ar[r]^{\dot{+} \otimes 
\dot{+}} &\V\otimes \V \ar[d]^{m}  \\  &\V^{\otimes 3} \ar[rr]_{\id \otimes 
m}&{} 
 &\V\otimes \V \ar[r]_{\dot{+}} &\V} $$
\end{enumerate}

 Such an object $\V$ will be called {\it an $\F_q$-linear Frobenius space} in $\mathcal{C}$.
\end{definition}

\begin{example}
 We have seen in Section \ref{sec:classical_endom} that the object $\V_n$ in $\Rep(GL_n(\F_q))$ is equipped with the required morphisms $(m, \eps, m^*, \eps^*, \dot{+}, \mu, z)$, and it is of dimension $q^n$. We have already seen that these satisfy the Frobenius algebra relations \ref{rel:Frob_alg1}, \ref{rel:Frob_alg2} and it is straightforward that the relations \ref{rel:F_q_lin} (``being a module over $\F_q$'') are satisfied as well.
 
 The compatibility relations state that the Frobenius algebra structure is compatible with the $\F_q$-module structure on $\V_n$.
 
  Intuitively, the morphisms $(m, \eps, m^*, \eps^*)$ serve as logical tools, allowing us to ``see'' the vectors $v\in V$ in the space $\V_n$ and to write down ``equations'' with these vectors $v\in V$. 
 The map $\eps^*:\V_n \to \C$ which sends each $v\in V$ to $1$ serves to ``verify'' whether a given vector belongs to $V$, the map $m$ serves to ``compare'' vectors $v, w\in V$ and check whether they are equal, and the map $m^*$ which allows us to ``double'' a given vector $v\in V$ in order to be used multiple times in an equation.
 
 Let us illustrate the compatibility relations in this case.
 Relation \ref{rel:mu_coalg_mor} translates to:
 \begin{align*}
   \forall v,w \in V, \forall a\in \F_q^{\times}, \;\; &\left((\mu_a\otimes \mu_a)\circ m^*\right)(v) = \dot{a}v\otimes \dot{a}v = m^*(\dot{a}v), \;\; \eps^*(\dot{a}v)=1, \\  \;\; 
   &m(\dot{a}v\otimes \dot{a}w) = \delta_{\dot{a}v,\dot{a}w} \dot{a} v =  \delta_{v,w} \dot{a} v =\dot{a}m(v\otimes w).
 \end{align*}
 Intuitively, these conditions state that $\dot{a}v \in V$, that $\dot{a}v=\dot{a}w \; \Leftrightarrow \; v=w$ for any $a\in \F_q^{\times}$, and that $\dot{a}v\otimes \dot{a}v$ is both a ``doubling'' of the vector $\dot{a}v \in V$ and a multiplication of two copies of $v$ by the same scalar $a$.
 
 Relation \ref{rel:z_coalg_mor} translates to:
 \begin{align*}
   &m^*(\dot{0}) = \dot{0} \otimes \dot{0} = (z^*\otimes z^*)(1), \;\; \eps^*(\dot{0})=1, \\
   &m(\dot{0}\otimes \dot{0}) = \dot{0}=z^*(1).
 \end{align*}
In other words, this condition states that $\dot{0} \in V$, that $\dot{0}=\dot{0} $ and that ``doubling'' the vector $\dot{0} \in V$ is the same as taking the image of $z^*\otimes z^*$.
 
 Relation \ref{rel:plus_coalg_mor} translates to the following: for any $v,w \in V$, applying the map $\left(\dot{+} \otimes \dot{+}\right)\left(\id \otimes \sigma \otimes \id\right)\left(m^*\otimes m^*\right)$ to $v\otimes w$ gives $(v\dot{+}  w)\otimes (v \dot{+} w) $, which is also $m^* (v\dot{+} w)$. Intuitively, this means that ``doubling'' $v \dot{+} w$ is the same as ``doubling'' $v$, ``doubling'' $w$ and then adding the copies separately.
 
 The second part of the condition just states that
 \begin{align*}
   \forall v,w \in V,  \;\; &\eps^*(v)\otimes \eps^*(w) = 1 = \eps^*(v \dot{+} w),
 \end{align*}
 which intuitively means that for $v, w\in V$, $v \dot{+} w \in V$ as well.
 
 Relation \ref{rel:cancellation_axiom} (``Cancellation Axiom'') translates to:
  \begin{align*}
   \forall v,w, u \in V, \;\; \left(m\circ \left(\dot{+} \otimes \dot{+}\right)\right)(v\otimes u \otimes u\otimes w) =  \delta_{v\dot{+}u, u\dot{+}w} (v\dot{+} u) = \dot{+}(u\otimes \delta_{v, w} v)
 \end{align*}
 In other words, in this setting the Cancellation Axiom is equivalent to the statement
 \begin{align*}
   \forall v,w, u \in V, \;\; v=w 
\Leftrightarrow v\dot{+}u=w\dot{+}u,
 \end{align*}
which justifies its name (see also Lemma \ref{lem:eq_axiom_corollary}).

\end{example}

The following is the main example we will consider:
\begin{example}
 \InnaA{Let $t\in \C$.} Consider the object $[1]$ in $\kar{t}$; by abuse of notation, we will denote it by $\V_t$, and often write $\V_t^{\otimes k}$ instead of $[k]$.

Just like we proved in Proposition \ref{prop:gen_morphisms}, the following 
morphisms generate the space $\left(\bigoplus_{m, k 
\geq 0} \Hom_{\InnaA{\kar{t}}}(\V_t^{\otimes m}, \V_t^{\otimes k}), \otimes, \circ 
\right)$:
 \begin{enumerate}
  \item Morphisms 
\begin{align*}
&\eps:=f_{\{\z\}}:\triv \to \V_t, \;\;  &m:=f_{\{(a+b,-a,-b)|a,b\in \F_q\}}: \V_t \otimes \V_t \to \V_t \\
&\eps^*:=f_{\{\z\}}:\V_t \to \triv, \;\;  &m^*:=f_{\{(a+b,-a,-b)|a,b\in \F_q\}}: \V_t \to \V_t \otimes \V_t \\
&\sigma:=f_{\{(a,b,-b, -a)|a,b\in \F_q\}}: \V_t \otimes \V_t \to \V_t \otimes \V_t
\end{align*}

  \item Morphisms 
  \begin{align*}
  &z:=f_{\F_q^1}: \triv \to \V_t,\;\;  1 \mapsto \z, \\
&\forall a\in \F_q, \;  \mu_a:=f_{\{(-ab,b)|b\in \F_q\}}: \V_t \to \V_t, \;\; & \dot{+}:=f_{\{(b,b,-b)|b\in \F_q\}}:\V_t \otimes \V_t \to \V_t.           
\end{align*}
\end{enumerate}
To show that these maps satisfy the conditions of Definition \ref{def:Frob_linear_space}, \InnaA{we first note that in the setting of the category $\kar{t}$, the required relations on the morphisms $\eps,m, \eps^*,m^*,\sigma, z, \mu_a, \dot{+}$ are equivalent to the condition that certain polynomials in $t$ are constant and equal to zero. Therefore it is enough to establish that the required relations hold for an infinite set of values of $t$: for example, for $t=q^{2n}$ where $n\in \Z_{\geq 1}$.}

Recall that by Remark \ref{rmk:F_n_ff}, for $t=q^{2n}$, the functor $F_{2n}:\kar{t} \to \Rep(GL_{2n}(\F_q))$ is fully faithful on the full subcategory $\kar{t}^{\leq n}$ generated by the objects $[k]$, $ 0\leq k\leq n$. The morphisms described in Definition \ref{def:Frob_linear_space} belong to $\kar{t}^{\leq 3}$. Thus it is enough to check that the relations in Definition \ref{def:Frob_linear_space} hold in $\Rep(GL_6(\F_q))$, which we have seen already.

This makes $\V_t = [1]$ an $\F_q$-linear Frobenius space of categorical dimension 
$t$ in $\kar{t}$.

\end{example}

\begin{remark} \InnaA{In \cite{KS}, Khovanov and Sazdanovic constructed $\C$-linear Karoubi additive symmetric monoidal categories $Kob_{\alpha}$ attached to any power series $\alpha$ with coefficients $\alpha = (\alpha_0,\alpha_1,\alpha_2,\ldots)$ with $\alpha_i \in \C$. These were later studied, under the name $DCob_{\alpha}$, by  Khovanov, Ostrik and Kononov in \cite{KOK}. The categories} $DCob_{\alpha}$ all contain a distinguished commutative Frobenius algebra which in some sense realizes the power series $\alpha$.  Since $\kar{t}$ contains a distinguished commutative Frobenius algebra, it is tempting to ask whether it is of the form $DCob_{\alpha}$ for some power series $\alpha$. This is not the case. Computing the power series $\alpha$ for $[1]$ and parameter $t = q^n$ yields the constant power series $\alpha  = (q^n, q^n, \ldots)$. In this case $DCob_{\alpha} \cong \underline{\Rep}(S_{t =q^n})$. This result should be expected since the power series $\alpha$ only depends on the Frobenius structure and not on the additional $\mathbb{F}_q$-module structure.
\end{remark}

\begin{remark}
 The relation $\eps^* \circ z = \id$ appearing in \ref{rel:z_coalg_mor} implies that $z^* \circ z = \id$, so $\triv$ is a direct summand of $\V$. 
Indeed, 
$$z^* \circ z = ev \circ (z \otimes z) = \eps^* \circ m \circ (z \otimes z) = 
\eps^* \circ z =\id.$$
\end{remark}

We now introduce some notation which will be used throughout the paper:

\begin{notation}
We denote: 
\begin{align*}
ev_{\V^{\otimes k}} &= ev \circ (\id \otimes ev \otimes \id) \circ (\id \otimes ev \otimes \id) \circ \ldots \\ coev_{\V^{\otimes k}} &=\ldots\circ(\id \otimes coev \otimes \id) \circ (\id \otimes coev \otimes \id) \circ coev 
\end{align*}
as well as the morphisms \begin{align*}
&\overline{ev}_{\V^{\otimes k}} = {ev}_{\V^{\otimes k}} \circ (\id\otimes w), &\overline{coev}_{\V^{\otimes k}} =(\id \otimes w)\circ {coev}_{\V^{\otimes k}}
\end{align*}
where $w:\V^{\otimes k}\to \V^{\otimes k}$ stands for the composition of symmetry morphisms $\sigma$ which corresponds to the longest element in the symmetric group $S_k$.
\end{notation} 
The morphisms $\overline{ev}_{\V^{\otimes k}}$, $\overline{coev}_{\V^{\otimes k}}$ make $\V^{\otimes k}$ self-dual. 

\begin{example}
For $\mathcal{C} = \Rep(GL_n(\F_q))$, $\V:=\V_n$, the map $ev_{\V_n^{\otimes k}}: \V_n^{\otimes k} \to \V_n$ takes $v_1 \otimes \ldots \otimes v_k$ (where $v_1, \ldots, v_k \in V$) to $1$ if $v_1=v_2=\ldots=v_k$, and to $0$ otherwise.

Similarly, the map $coev_{\V_n^{\otimes k}}: \V_n\to \V_n^{\otimes k}$ takes $v$ to $v \otimes \ldots \otimes v$ whenever $v\in V$.
\end{example}

\subsection{String diagrams}\label{ssec:string_diagr}
Let us consider an $\F_q$-linear Frobenius space $\V$ in $\mathcal{C}$. 

In order to study the homomorphisms $\V^{\otimes s}\to \V^{\otimes k}$, we 
will draw string diagrams to encode morphisms generated by the morphisms $m, \eps, m^*, \eps^*, \sigma, z, \dot{+}$ and $\mu_a$ ($a\in \F_q$). A morphism $\V^{\otimes s}\to \V^{\otimes k}$ for given $s,k \geq 0$ will be drawn as a diagram located in a horizontal strip of the plane, connecting $s$ endpoints (denoted by vertical strands) on the lower border of the strip with $k$ endpoints on the upper border of the strip. The diagram consists of a collection of strings with decorations in the spirit of \cite{S}.

We say that two diagrams are equivalent when they encode equal morphisms.

Composition corresponds to vertical concatenation of diagrams: given two diagrams $D$ and $D'$ corresponding to morphisms $f: \V^{\otimes s}\to \V^{\otimes k}$ and $f': \V^{\otimes k}\to \V^{\otimes l}$, we draw the diagram for $f'\circ f$ by stacking $D'$ on top of $D$ and identifying the $k$ upper endpoints of $D$ with the $k$ lower  endpoints of $D'$.

Drawing strings next to each other horizontally denotes the monoidal product, and having no strings at all denotes an endomorphism of $\triv$. Reflection with respect to a horizontal axis sends a morphism $f$ to $f^*$ due to the self-duality of $\V$ (see Lemma \ref{lem:self-dual} and also Remark \ref{rmk:m_m_star_duality} below).

The morphisms described in Definition \ref{def:Frob_linear_space} will be drawn as follows:
\begin{equation*}
m \;=\;  
\begin{tikzpicture}[anchorbase,scale=1]
\draw[-] (-0.4,0)--(0.4,0);
\draw[-] (0,0)--(0,0.3);
\draw[-] (-0.4,0)--(-0.4,-0.3);
\draw[-] (0.4,0)--(0.4,-0.3);
\end{tikzpicture}\;,\quad, 
m^* \;=\;  \begin{tikzpicture}[anchorbase,scale=1]
\draw[-] (-0.4,0)--(0.4,0);
\draw[-] (0,0)--(0,-0.3);
\draw[-] (-0.4,0)--(-0.4,0.3);
\draw[-] (0.4,0)--(0.4,0.3);
\end{tikzpicture}\;,\quad
\eps^* \;=\;
 \begin{tikzpicture}[anchorbase,scale=1]
\draw[-] (0,0)--(0,-0.5);
\node at (0,0) {$\bullet$};
\end{tikzpicture}, \quad 
\eps \;=\;
 \begin{tikzpicture}[anchorbase,scale=1]
\draw[-] (0,0)--(0,-0.5);
\node at (0,-0.5) {$\bullet$};
\end{tikzpicture}, \quad 
z \;=\;
 \begin{tikzpicture}[anchorbase,scale=1]
\draw[-] (0,0)--(0,-0.4);
\node at (0,-0.5) {$\circ$};
\end{tikzpicture}, \quad 
\dot{+}\;=\;
\begin{tikzpicture}[anchorbase,scale=1]
\draw[-] (-0.5,0)--(-0.2,0);%\dot{+}
\draw[-] (0.2,0)--(0.5,0);
\draw[-] (-0.2,-0.2)--(-0.2,0.2)--(0.2,0.2)--(0.2,-0.2)--(-0.2,-0.2);
\node at (-0,-0) {$\dot{+}$};
\draw[-] (0,0.2)--(0,0.45);
\draw[-] (-0.5,0)--(-0.5,-0.4);
\draw[-] (0.5,0)--(0.5,-0.4);
\end{tikzpicture}\;, \quad 
\mu_a \;=\;
 \begin{tikzpicture}[anchorbase,scale=1]
 \draw[-] (0, -0.6)--(0, -0.3);
 \draw[-] (0, 0.6)--(0, 0.3);
 
 \node[draw,circle] at (0,0) {$\scriptstyle a$};  
 \end{tikzpicture}
\end{equation*}
We also draw the morphisms $\sigma:=\sigma_{\V, \V}$, $ev :=\eps^*\circ m$, $coev:=m^*\circ \eps$, $z^*=ev\circ (\id \otimes z)$ as 
\begin{equation*}  
\sigma \;=\;\begin{tikzpicture}[anchorbase,scale=1]
\draw[-] (0,0)--(0.5,-0.5);
\draw[-] (0.5,0)--(0,-0.5);
\end{tikzpicture}\;,\quad ev \;=\;\begin{tikzpicture}[anchorbase,scale=1]
\draw[-] (0,0)--(0.8,0);
\draw[-] (0,0)--(0,-0.3);
\draw[-] (0.8,0)--(0.8,-0.3);
\end{tikzpicture}\;,\quad coev \;=\;\begin{tikzpicture}[anchorbase,scale=1]
\draw[-] (0,-0.3)--(0.8,-0.3);
\draw[-] (0,0)--(0,-0.3);
\draw[-] (0.8,0)--(0.8,-0.3);
\end{tikzpicture}\;,\quad z^* \;=\;\begin{tikzpicture}[anchorbase,scale=1]
\draw[-] (0,0)--(0,0.45);
\node at (0,0.5) {$\circ$};
\end{tikzpicture}
\;.
\end{equation*}

 Let us write the relations listed in Definition \ref{def:Frob_linear_space} diagrammatically.
 
 \begin{enumerate}[label=(\subscript{DFrob}{{\arabic*}})]
\item\label{itm:str_rel_diag_bialg} Relation \ref{rel:Frob_alg1}: commutativity and associativity relations on $(m, \eps)$, cocommutativity and coassociativity relations on $(m^*, \eps^*)$:
\begin{equation*}
\begin{tikzpicture}[anchorbase,scale=0.6]

\draw[-] (0.5,0)--(0.5,-0.5);
\draw[-] (1.5,0)--(1.5,-0.5);
\draw[-] (0.5,-0.5)--(1.5,-0.5);

\draw[-] (1,-0.5)--(1,-1);
\draw[-] (-0.4,0)--(-0.4,-1);

\draw[-] (-0.4,-1)--(1,-1);
\draw[-] (0.3,-1)--(0.3,-1.5);
\end{tikzpicture} \quad = \quad 
\begin{tikzpicture}[anchorbase,scale=0.6]

\draw[-] (-0.5,0)--(-0.5,-0.5);
\draw[-] (-1.5,0)--(-1.5,-0.5);
\draw[-] (-0.5,-0.5)--(-1.5,-0.5);

\draw[-] (-1,-0.5)--(-1,-1);
\draw[-] (0.4,0)--(0.4,-1);

\draw[-] (0.4,-1)--(-1,-1);
\draw[-] (-0.3,-1)--(-0.3,-1.5);
\end{tikzpicture} 
   \quad,\quad \quad 
   \begin{tikzpicture}[anchorbase,scale=0.6]

\draw[-] (0.5,0)--(0.5,0.5);
\draw[-] (1.5,0)--(1.5,0.5);
\draw[-] (0.5,0.5)--(1.5,0.5);

\draw[-] (1,0.5)--(1,1);
\draw[-] (-0.4,0)--(-0.4,1);

\draw[-] (-0.4,1)--(1,1);
\draw[-] (0.3,1)--(0.3,1.5);
\end{tikzpicture} \quad = \quad 
\begin{tikzpicture}[anchorbase,scale=0.6]

\draw[-] (-0.5,0)--(-0.5,0.5);
\draw[-] (-1.5,0)--(-1.5,0.5);
\draw[-] (-0.5,0.5)--(-1.5,0.5);

\draw[-] (-1,0.5)--(-1,1);
\draw[-] (0.4,0)--(0.4,1);

\draw[-] (0.4,1)--(-1,1);
\draw[-] (-0.3,1)--(-0.3,1.5);
\end{tikzpicture}    \quad,\quad \quad 
\begin{tikzpicture}[anchorbase,scale=0.6]
\draw[-] (0,-1.5)--(0,-1);
\draw[-] (-0.5,-1)--(0.5,-1);
\draw[-] (-0.5,-1)--(-0.5,-0.5);
\draw[-] (0.5,-1)--(0.5,-0.5);
\draw[-] (-0.5,-0.5)--(0.5,0);
\draw[-] (0.5,-0.5)--(-0.5,0);
\end{tikzpicture} \quad  = \quad 
 \begin{tikzpicture}[anchorbase,scale=0.6]

\draw[-] (0,-1.5)--(0,-1);
\draw[-] (-0.5,-1)--(0.5,-1);
\draw[-] (-0.5,-1)--(-0.5,0);
\draw[-] (0.5,-1)--(0.5,0);
\end{tikzpicture}
\end{equation*}

\begin{equation*} 
\begin{tikzpicture}[anchorbase,scale=0.6]
\draw[-] (0,1.5)--(0,1);
\draw[-] (-0.5,1)--(0.5,1);
\draw[-] (-0.5,1)--(-0.5,0.5);
\draw[-] (0.5,1)--(0.5,0.5);
\draw[-] (-0.5,0.5)--(0.5,0);
\draw[-] (0.5,0.5)--(-0.5,0);
\end{tikzpicture} \quad  = \quad 
 \begin{tikzpicture}[anchorbase,scale=0.6]

\draw[-] (0,1.5)--(0,1);
\draw[-] (-0.5,1)--(0.5,1);
\draw[-] (-0.5,1)--(-0.5,0);
\draw[-] (0.5,1)--(0.5,0);
\end{tikzpicture}   \quad,\quad \quad
 \begin{tikzpicture}[anchorbase,scale=0.6]

\draw[-] (0,-1.5)--(0,-1);
\draw[-] (-0.5,-1)--(0.5,-1);
\draw[-] (-0.5,-1)--(-0.5,-0.5);
\draw[-] (0.5,-1)--(0.5,-0.5);
\node at (0.5,-0.5) {$\bullet$};
\end{tikzpicture} \quad  = \quad 
 \begin{tikzpicture}[anchorbase,scale=0.6]

\draw[-] (0,-1.5)--(0,-1);
\draw[-] (-0.5,-1)--(0.5,-1);
\draw[-] (-0.5,-1)--(-0.5,-0.5);
\draw[-] (0.5,-1)--(0.5,-0.5);
\node at (-0.5,-0.5) {$\bullet$};
\end{tikzpicture} \quad  = \quad 
\begin{tikzpicture}[anchorbase,scale=0.3]
\draw[-] (0,-1.5)--(0,1.5);
\end{tikzpicture}  \quad  = \quad  
\begin{tikzpicture}[anchorbase,scale=0.6]
\draw[-] (0,1.5)--(0,1);
\draw[-] (-0.5,1)--(0.5,1);
\draw[-] (-0.5,1)--(-0.5,0.5);
\draw[-] (0.5,1)--(0.5,0.5);
\node at (0.5,0.5) {$\bullet$};
\end{tikzpicture} \quad  = \quad 
 \begin{tikzpicture}[anchorbase,scale=0.6]
\draw[-] (0,1.5)--(0,1);
\draw[-] (-0.5,1)--(0.5,1);
\draw[-] (-0.5,1)--(-0.5,0.5);
\draw[-] (0.5,1)--(0.5,0.5);
\node at (-0.5,0.5) {$\bullet$};
\end{tikzpicture}
\end{equation*} 

  \item\label{itm:str_rel_diag_Frob} Relation \ref{rel:Frob_alg2}: Frobenius Relations and the Special Relation:
 
\begin{equation*}
\begin{tikzpicture}[anchorbase,scale=0.3]
\draw[-] (-2,1)--(-2,-2);
\draw[-] (-2,1)--(0,1);
\draw[-] (-1,2)--(-1,1);
\draw[-] (0,1)--(0,-1);
\draw[-] (0,-1)--(2,-1);
\draw[-] (1,-1)--(1,-2);
\draw[-] (2,2)--(2,-1);
\end{tikzpicture} \quad  = \quad \begin{tikzpicture}[anchorbase,scale=0.3]
\draw[-] (-2,2)--(-2,-1);
\draw[-] (-2,-1)--(0,-1);
\draw[-] (-1,-1)--(-1,-2);
\draw[-] (0,1)--(0,-1);
\draw[-] (0,1)--(2,1);
\draw[-] (1,1)--(1,2);
\draw[-] (2,1)--(2,-2);
\end{tikzpicture} \quad = \quad \begin{tikzpicture}[anchorbase,scale=0.3]
\draw[-] (-1,2)--(-1,1);
\draw[-] (-1,-1)--(-1,-2);
\draw[-] (-1,-1)--(1,-1);
\draw[-] (0,-1)--(0,1);
\draw[-] (-1,1)--(1,1);
\draw[-] (1,2)--(1,1);
\draw[-] (1,-1)--(1,-2);

\end{tikzpicture} \quad,\quad \quad
 \begin{tikzpicture}[anchorbase,scale=0.3]

\draw[-] (0,-2)--(0,-1);
\draw[-] (-1,-1)--(1,-1);
\draw[-] (-1,-1)--(-1,1);
\draw[-] (1,-1)--(1,1);
\draw[-] (-1,1)--(1,1);
\draw[-] (0,1)--(0,2);

\end{tikzpicture} \quad  = \quad \begin{tikzpicture}[anchorbase,scale=0.3]
\draw[-] (0,-2)--(0,2);
\end{tikzpicture}  
\end{equation*} 
\end{enumerate}
\begin{enumerate}[label=({DLin})]
\item\label{itm:str_rel_Lin} The relations on the morphisms $\dot{+}$, $\mu_a$, $z$ ($a\in \F_q$) described in \ref{rel:F_q_lin}: for any $a, b \in \F_q$, 
\begin{enumerate}[label=(\subscript{DLin}{{\arabic*}})]
 \item Relations \ref{rel:F_q_lin_plus_ass_comm}: associativity and commutativity of $\dot{+}$:
 \begin{equation*}
 \begin{tikzpicture}[anchorbase,scale=1.3]
\draw[-] (-1,0.2)--(-1,0.4); 
\draw[-] (-1.2,-0.2)--(-1.2,0.2)--(-0.8,0.2)--(-0.8,-0.2)--(-1.2,-0.2);
\node at (-1,0) {$\dot{+}$};
\draw[-] (-1.5,0)--(-1.2,0);
\draw[-] (-0.8,0)--(-0.5,0);
\draw[-] (-0.5,0)--(-0.5,-0.3);
\draw[-] (-1.5,0)--(-1.5,-0.8);
\draw[-] (-0.7,-0.3)--(-0.7,-0.7)--(-0.3,-0.7)--(-0.3,-0.3)--(-0.7,-0.3);
\node at (-0.5,-0.5) {$\dot{+}$};
\draw[-] (-0.9,-0.5)--(-0.7,-0.5);
\draw[-] (-0.3,-0.5)--(-0.1,-0.5);
\draw[-] (-0.9,-0.5)--(-0.9,-0.8);
\draw[-] (-0.1,-0.5)--(-0.1,-0.8);
    
 \end{tikzpicture}
 \quad=\quad   \begin{tikzpicture}[anchorbase,scale=1.3]
\draw[-] (0,0.2)--(0,0.4); 
\draw[-] (-0.2,-0.2)--(-0.2,0.2)--(0.2,0.2)--(0.2,-0.2)--(-0.2,-0.2);
\node at (0,0) {$\dot{+}$};
\draw[-] (-0.5,0)--(-0.2,0);
\draw[-] (0.2,0)--(0.5,0);
\draw[-] (-0.5,0)--(-0.5,-0.3);
\draw[-] (0.5,0)--(0.5,-0.8);
\draw[-] (-0.7,-0.3)--(-0.7,-0.7)--(-0.3,-0.7)--(-0.3,-0.3)--(-0.7,-0.3);
\node at (-0.5,-0.5) {$\dot{+}$};
\draw[-] (-0.9,-0.5)--(-0.7,-0.5);
\draw[-] (-0.3,-0.5)--(-0.1,-0.5);
\draw[-] (-0.9,-0.5)--(-0.9,-0.8);
\draw[-] (-0.1,-0.5)--(-0.1,-0.8);
\end{tikzpicture}\quad, \quad\quad 
\begin{tikzpicture}[anchorbase,scale=1.3]
\draw[-] (0,0.2)--(0,0.4); 
\draw[-] (-0.2,-0.2)--(-0.2,0.2)--(0.2,0.2)--(0.2,-0.2)--(-0.2,-0.2);
\node at (0,0) {$\dot{+}$};
\draw[-] (-0.5,0)--(-0.2,0);
\draw[-] (0.2,0)--(0.5,0);
\draw[-] (-0.5,0)--(-0.5,-0.3);
\draw[-] (0.5,0)--(0.5,-0.3);
\draw[-] (-0.5,-0.3)--(0.5,-0.8);
\draw[-] (0.5,-0.3)--(-0.5,-0.8);
 \end{tikzpicture}
 \quad=\quad   \begin{tikzpicture}[anchorbase,scale=1.3]
\draw[-] (0,0.2)--(0,0.4); 
\draw[-] (-0.2,-0.2)--(-0.2,0.2)--(0.2,0.2)--(0.2,-0.2)--(-0.2,-0.2);
\node at (0,0) {$\dot{+}$};
\draw[-] (-0.5,0)--(-0.2,0);
\draw[-] (0.2,0)--(0.5,0);
\draw[-] (-0.5,0)--(-0.5,-0.8);
\draw[-] (0.5,0)--(0.5,-0.8);
\end{tikzpicture}
\end{equation*}
\item Relation \ref{rel:F_q_lin_zero}: properties of the zero map $z$: 
\begin{equation*}
 \begin{tikzpicture}[anchorbase,scale=1.3]
\draw[-] (0,0.2)--(0,0.4); 
\draw[-] (-0.2,-0.2)--(-0.2,0.2)--(0.2,0.2)--(0.2,-0.2)--(-0.2,-0.2);
\node at (0,0) {$\dot{+}$};
\draw[-] (-0.5,0)--(-0.2,0);
\draw[-] (0.2,0)--(0.5,0);
\draw[-] (0.5,0)--(0.5,-0.5);
\draw[-] (-0.5,0)--(-0.5,-0.27);
\node at (-0.5,-0.32) {$\circ$};
\end{tikzpicture}\quad = \quad
 \begin{tikzpicture}[anchorbase,scale=1.3]
\draw[-] (0,-0.5)--(0,0.4); 
 \end{tikzpicture}
 \quad=\quad   \begin{tikzpicture}[anchorbase,scale=1.3]
\draw[-] (0,0.2)--(0,0.4); 
\draw[-] (-0.2,-0.2)--(-0.2,0.2)--(0.2,0.2)--(0.2,-0.2)--(-0.2,-0.2);
\node at (0,0) {$\dot{+}$};
\draw[-] (-0.5,0)--(-0.2,0);
\draw[-] (0.2,0)--(0.5,0);
\draw[-] (-0.5,0)--(-0.5,-0.5);
\draw[-] (0.5,0)--(0.5,-0.27);
\node at (0.5,-0.32) {$\circ$};
\end{tikzpicture}
\end{equation*}
 \item Relations \ref{rel:F_q_lin_mu}: $\mu_{ab}=\mu_a \circ \mu_b$, $\mu_1=\id$, $\mu_0 = z\circ \eps^*$:
 \begin{equation*}
  \begin{tikzpicture}[anchorbase,scale=1]
 \draw[-] (0, -1)--(0, -0.8);
 \draw[-] (0, 0.8)--(0, 1);
  \node[draw,circle] at (0,-0.5) {$\scriptstyle b$}; 
 \node[draw,circle] at (0,0.5) {$\scriptstyle a$};  
  \draw[-] (0, -0.2)--(0, 0.2);
 \end{tikzpicture}\quad = \quad
 \begin{tikzpicture}[anchorbase,scale=1]
 \draw[-] (0, -0.8)--(0, -0.33);
 \draw[-] (0, 0.8)--(0, 0.33);
 \node[draw,circle] at (0,0) {$\scriptstyle ab$};  
 \end{tikzpicture}
  \quad,\quad \quad
 \begin{tikzpicture}[anchorbase,scale=1]
\draw[-] (0, -0.8)--(0, -0.3);
 \draw[-] (0, 0.8)--(0, 0.3); 
 \node[draw,circle] at (0,0) {$\scriptstyle 1$};  
 \end{tikzpicture}
 \quad  = \quad 
 \begin{tikzpicture}[anchorbase,scale=1]
\draw[-] (0,-0.8)--(0,0.8);
\end{tikzpicture}    \quad,\quad \quad
 \begin{tikzpicture}[anchorbase,scale=1]
\draw[-] (0, -0.8)--(0, -0.3);
 \draw[-] (0, 0.8)--(0, 0.3); 
 \node[draw,circle] at (0,0) {$\scriptstyle 0$};  
 \end{tikzpicture}
 \quad  = \quad 
 \begin{tikzpicture}[anchorbase,scale=1]
\draw[-] (0,-0.8)--(0,-0.3);
\node at (0,-0.3) {$\bullet$};
\node at (0,0.27) {$\circ$};
\draw[-] (0,0.8)--(0,0.34);
\end{tikzpicture} 
\end{equation*}

\item Relation \ref{rel:F_q_lin_plus_lin_distr}: linearity of $\mu_a$ with respect to 
$\dot{+}$ and distributivity of $\mu_a$:
\begin{equation*}
 \begin{tikzpicture}[anchorbase,scale=1.5]
%m^* 
\draw[-] (0.3,0.2) -- (1.5, 0.2);
\draw[-] (1.5,0.2) -- (1.5, 0.38);
\draw[-] (0.9,0) --(0.9, 0.2);
\draw[-] (0.3,0.2)--(0.3,0.38);

%-1
\draw[-] (1.5,0.83) -- (1.5, 1.1);%right
\node[draw,circle] at (1.5,0.6) {${\scriptstyle b}$};

\draw[-] (0.3,0.83)--(0.3, 1.1);%left
\node[draw,circle] at (0.3,0.6) {${\scriptstyle a}$};

\draw[-] (0.3,1.1)--(0.7,1.1);
\draw[-] (1.1,1.1)--(1.5,1.1);
\draw[-] (0.7,0.9)--(0.7,1.3)--(1.1,1.3)--(1.1,0.9)--(0.7,0.9);
\node at (0.9,1.1) {$\dot{+}$};

\draw[-] (0.9,1.3)--(0.9,1.45);
\end{tikzpicture} \quad = \quad
 \begin{tikzpicture}[anchorbase,scale=1.5]
 \draw[-] (0, -0.8)--(0, -0.3);
 \draw[-] (0, 0.65)--(0, 0.3);
 
 \node[draw,circle] at (0,0) {$\scriptstyle a+b$};  
 \end{tikzpicture}  \quad ,  \quad \quad
 \begin{tikzpicture}[anchorbase,scale=1.5]
\draw[-] (1.5,0) -- (1.5, 0.28);
\draw[-] (0.3,0)--(0.3,0.28);
\draw[-] (1.5,0.73) -- (1.5, 1.1);%right
\node[draw,circle] at (1.5,0.5) {${\scriptstyle a}$};

\draw[-] (0.3,0.73)--(0.3, 1.1);%left
\node[draw,circle] at (0.3,0.5) {${\scriptstyle a}$};

\draw[-] (0.3,1.1)--(0.7,1.1);
\draw[-] (1.1,1.1)--(1.5,1.1);
\draw[-] (0.7,0.9)--(0.7,1.3)--(1.1,1.3)--(1.1,0.9)--(0.7,0.9);
\node at (0.9,1.1) {$\dot{+}$};

\draw[-] (0.9,1.3)--(0.9,1.45);
\end{tikzpicture} \quad = \quad
\begin{tikzpicture}[anchorbase,scale=1.5]
\draw[-] (1.5,0.5) -- (1.5, 0);%right
\draw[-] (0.3,0.5)--(0.3, 0);%left

\draw[-] (0.3,0.5)--(0.7,0.5);
\draw[-] (1.1,0.5)--(1.5,0.5);

\draw[-] (0.7,0.3)--(0.7,0.7)--(1.1,0.7)--(1.1,0.3)--(0.7,0.3);
\node at (0.9,0.5) {$\dot{+}$};

\draw[-] (0.9,0.7)--(0.9,0.9);

\node[draw,circle] at (0.9,1.1) {${\scriptstyle a}$};
\draw[-] (0.9,1.3)--(0.9,1.45);
\end{tikzpicture}
\end{equation*}
\end{enumerate}
\end{enumerate}

\begin{enumerate}[label=(\subscript{DRel}{{\arabic*}})]
\item\label{itm:str_rel_diag_mu_coalg} Relation \ref{rel:mu_coalg_mor} (``$\mu_a$ is a coalgebra and algebra morphism for $a\in \F_q^{\times}$''): 
\begin{equation*}
 \begin{tikzpicture}[anchorbase,scale=1]
  \draw[-] (-0.5, 0.6)--(-0.5, 0.3);
   \draw[-] (0.5, 0.6)--(0.5, 0.3);
 \draw[-] (-0.5, 0.3)--(0.5, 0.3);
\draw[-] (0, 0.3)--(0, -0.1);
 \draw[-] (0, -0.8)--(0, -0.7); 
 \node[draw,circle] at (0,-0.4) {$\scriptstyle a$};  
 \end{tikzpicture}
 \quad  = \quad 
 \begin{tikzpicture}[anchorbase,scale=1]
 \draw[-] (-0.5, 0.6)--(-0.5, 0.4);
   \draw[-] (0.5, 0.6)--(0.5, 0.4);
    \node[draw,circle] at (-0.5,0.1) {$\scriptstyle a$}; 
     \node[draw,circle] at (0.5,0.1) {$\scriptstyle a$}; 
     
 \draw[-] (-0.5, -0.2)--(-0.5, -0.6);
  \draw[-] (0.5, -0.2)--(0.5, -0.6);
  \draw[-] (-0.5, -0.6)--(0.5, -0.6);
\draw[-] (0, -0.6)--(0, -0.8);
\end{tikzpicture} 
   \quad,\quad \quad
 \begin{tikzpicture}[anchorbase,scale=1]
\draw[-] (0, -0.7)--(0, -0.3);
 \draw[-] (0, 0.7)--(0, 0.3); 
 \node[draw,circle] at (0,0) {$\scriptstyle a$};  
 \node at (0,0.7) {$\bullet$};
 \end{tikzpicture}
 \quad  = \quad 
 \begin{tikzpicture}[anchorbase,scale=1]
\draw[-] (0,-0.7)--(0,0.7);
 \node at (0,0.7) {$\bullet$};
\end{tikzpicture} 
  \quad,\quad \quad
 \begin{tikzpicture}[anchorbase,scale=1]
  \draw[-] (-0.5, -0.6)--(-0.5, -0.3);
   \draw[-] (0.5, -0.6)--(0.5, -0.3);
 \draw[-] (-0.5, -0.3)--(0.5, -0.3);
\draw[-] (0, -0.3)--(0, 0.15);
 \draw[-] (0, 0.8)--(0, 0.65); 
 \node[draw,circle] at (0,0.4) {$\scriptstyle a$};  
 \end{tikzpicture}
 \quad  = \quad 
 \begin{tikzpicture}[anchorbase,scale=1]
 \draw[-] (-0.5, -0.6)--(-0.5, -0.4);
   \draw[-] (0.5, -0.6)--(0.5, -0.4);
    \node[draw,circle] at (-0.5,-0.1) {$\scriptstyle a$}; 
     \node[draw,circle] at (0.5,-0.1) {$\scriptstyle a$}; 
     
 \draw[-] (-0.5, 0.2)--(-0.5, 0.6);
  \draw[-] (0.5, 0.2)--(0.5, 0.6);
  \draw[-] (-0.5, 0.6)--(0.5, 0.6);
\draw[-] (0, 0.6)--(0, 0.8);
\end{tikzpicture}  
\end{equation*}
\item\label{itm:str_rel_diag_z_coalg} Relation \ref{rel:z_coalg_mor} (``$z$ is a coalgebra and algebra morphism''):
\begin{equation*}
 \begin{tikzpicture}[anchorbase,scale=1]
  \draw[-] (-0.5, 0.6)--(-0.5, 0.3);
   \draw[-] (0.5, 0.6)--(0.5, 0.3);
 \draw[-] (-0.5, 0.3)--(0.5, 0.3);
\draw[-] (0, 0.3)--(0, -0.14);
\node at (0,-0.2) {$\circ$};
 \end{tikzpicture}
 \quad  = \quad 
 \begin{tikzpicture}[anchorbase,scale=1]
\draw[-] (-0.5, 0.6)--(-0.5, -0.14);
   \draw[-] (0.5, 0.6)--(0.5, -0.14);
\node at (-0.5,-0.2) {$\circ$};
\node at (0.5,-0.2) {$\circ$};
\end{tikzpicture} 
   \quad,\quad \quad
 \begin{tikzpicture}[anchorbase,scale=1]
\draw[-] (0, -0.64)--(0, 0.1); 
\node at (0,-0.7) {$\circ$};
 \node at (0,0.1) {$\bullet$};
 \end{tikzpicture}
 \quad  = \quad \id_{\triv}
  \quad,\quad \quad
  \begin{tikzpicture}[anchorbase,scale=1]
  \draw[-] (-0.5, -0.6)--(-0.5, -0.3);
   \draw[-] (0.5, -0.6)--(0.5, -0.3);
 \draw[-] (-0.5, -0.3)--(0.5, -0.3);
\draw[-] (0, -0.3)--(0, 0.15);
\node at (0,0.2) {$\circ$};
 \end{tikzpicture}
 \quad  = \quad 
 \begin{tikzpicture}[anchorbase,scale=1]
\draw[-] (-0.5, -0.6)--(-0.5, 0.15);
   \draw[-] (0.5,-0.6)--(0.5, 0.15);
\node at (-0.5,0.2) {$\circ$};
\node at (0.5,0.2) {$\circ$};
\end{tikzpicture} 
\end{equation*}

\item\label{itm:str_rel_diag_plus_coalg} Relation \ref{rel:plus_coalg_mor} (``$\dot{+}$ is a coalgebra morphism''):
 \begin{equation*}
 \begin{tikzpicture}[anchorbase,scale=1.2]
   \draw[-] (0.4,-0.5)--(0.4,-0.2);%m^* top
     \draw[-] (-0.4,-0.5)--(-0.4,-0.2);
  \draw[-] (-0.4,-0.5)--(0.4,-0.5);

 \draw[-] (0,-0.8)--(0,-0.5);
 
 %dot{+} bottom
\draw[-] (-0.2,-1.2)--(-0.2,-0.8)--(0.2,-0.8)--(0.2,-1.2)--(-0.2,-1.2);
\node at (0, -1) {$\dot{+}$};
\draw[-] (-0.4,-1)--(-0.2,-1);
\draw[-] (0.2,-1)--(0.4,-1);
  
\draw[-] (-0.4,-1) -- (-0.4, -1.5);
\draw[-] (0.4,-1) -- (0.4, -1.5);
    
 \end{tikzpicture}
 \quad=\quad   \begin{tikzpicture}[anchorbase,scale=1.3]

%leftmost part:
\draw[-] (-0.2,-0.2)--(-0.2,0.2)--(0.2,0.2)--(0.2,-0.2)--(-0.2,-0.2);
\node at (0,0) {$\dot{+}$};
\draw[-] (-0.4,0)--(-0.4,-0.8);
\draw[-] (0,0.2)--(0,0.45);
\draw[-] (0.4,0)--(0.4,-0.3);

\draw[-] (-0.4,0)--(-0.2,0);
\draw[-] (0.2,0)--(0.4,0);

\draw[-] (1.8,-0.2)--(1.8,0.2)--(2.2,0.2)--(2.2,-0.2)--(1.8,-0.2);
\node at (2,0) {$\dot{+}$};
\draw[-] (1.6,0)--(1.6,-0.3);
\draw[-] (2,0.2)--(2,0.45);
\draw[-] (2.4,0)--(2.4,-0.8);

\draw[-] (1.6,0)--(1.8,0);
\draw[-] (2.2,0)--(2.4,0);

%\id\otimes \sigma \id
\draw[-] (1.6,-0.3)--(0.4,-0.8);
\draw[-] (0.4,-0.3)--(1.6,-0.8);

%m^* \otimes m^*
\draw[-] (-0.4, -0.8) -- (0.4,-0.8);
\draw[-] (1.6, -0.8) -- (2.4,-0.8);
\draw[-] (-0, -0.8) -- (0,-1.1);
\draw[-] (2, -0.8) -- (2,-1.1);
\end{tikzpicture} 
\quad,\quad \quad
\begin{tikzpicture}[anchorbase,scale=1]
\draw[-] (-0.5,0)--(-0.2,0);%\dot{+}
\draw[-] (0.2,0)--(0.5,0);
\draw[-] (-0.2,-0.2)--(-0.2,0.2)--(0.2,0.2)--(0.2,-0.2)--(-0.2,-0.2);
\node at (-0,-0) {$\dot{+}$};
\draw[-] (0,0.2)--(0,0.5);
\node at (0,0.5) {$\bullet$};
\draw[-] (-0.5,0)--(-0.5,-0.5);
\draw[-] (0.5,0)--(0.5,-0.5);
\end{tikzpicture} \quad  = \quad  \begin{tikzpicture}[anchorbase,scale=1]
\draw[-] (-0.5,0.5)--(-0.5,-0.5);
\draw[-] (0.5,0.5)--(0.5,-0.5);
\node at (-0.5,0.5) {$\bullet$};
\node at (0.5,0.5) {$\bullet$};
\end{tikzpicture} 
\end{equation*}  
 \item\label{itm:str_rel_diag_cancel_axiom} Relation \ref{rel:cancellation_axiom} (``the Cancellation Axiom''):
 
\begin{equation*}
\begin{tikzpicture}[anchorbase,scale=1.1]
\draw[-] (-0.5,0)--(-0.2,0);%\dot{+}
\draw[-] (0.2,0)--(1,0);
\draw[-] (-0.2,-0.2)--(-0.2,0.2)--(0.2,0.2)--(0.2,-0.2)--(-0.2,-0.2);
\node at (-0,-0) {$\dot{+}$};
\draw[-] (0,0.2)--(0,0.5);
\draw[-] (-0.5,0)--(-0.5,-1);%\id \otimes m
\draw[-] (1,0)--(1,-0.5);
\draw[-] (0.5,-0.5)--(1.5,-0.5);
\draw[-] (0.5,-0.5)--(0.5,-1);
\draw[-] (1.5,-0.5)--(1.5,-1.5);% \sigma\otimes \id
\draw[-] (-0.5,-1)--(0.5,-1.5);
\draw[-] (0.5,-1)--(-0.5,-1.5);
\end{tikzpicture} \quad = \quad \begin{tikzpicture}[anchorbase,scale=1.1]
\draw[-] (0,-0.5)--(0,0);%m
\draw[-] (-1,-0.5)--(1,-0.5);
\draw[-] (-1,-0.5)--(-1,-0.8);
\draw[-] (1,-0.5)--(1,-0.8);
\draw[-] (-1.5,-1)--(-1.2,-1);%\dot{+}
\draw[-] (-0.8,-1)--(-0.5,-1);
\draw[-] (-1.2,-1.2)--(-1.2,-0.8)--(-0.8,-0.8)--(-0.8,-1.2)--(-1.2,-1.2);
\node at (-1,-1) {$\dot{+}$};
\draw[-] (0.5,-1)--(0.8,-1);%\dot{+}
\draw[-] (1.2,-1)--(1.5,-1);
\draw[-] (0.8,-0.8)--(0.8,-1.2)--(1.2,-1.2)--(1.2,-0.8)--(0.8,-0.8);
\node at (1,-1) {$\dot{+}$};
\draw[-] (-1.5,-1)--(-1.5,-2);%\id \otimes m^*\otimes \id
\draw[-] (-0.5,-1)--(-0.5,-1.5);
\draw[-] (0.5,-1)--(0.5,-1.5);
\draw[-] (1.5,-1)--(1.5,-2);
\draw[-] (-0.5,-1.5)--(0.5,-1.5);
\draw[-] (0,-1.5)--(0,-2);
\end{tikzpicture} 
\end{equation*}  
 \end{enumerate}

Here are some special morphisms and their string diagrams, which we will often use:
\begin{definition}[Special string diagrams]\label{def:special_string_diag}
 Let $\V$ be an $\F_q$-linear Frobenius  
space in $\mathcal{C}$. We will use the following notation for iterative 
compositions: for $r,d \geq 2$, 
\begin{align*}
    &(m^*)^{it}: \V \to \V^{\otimes d}, \;\; (m^*)^{it}:=\ldots \circ 
(m^* \otimes 
\id_{\V^{\otimes 2}}) \circ (m^* \otimes 
\id_{\V}) \circ m^*\\
&(\dot{+})^{it}: \V^{\otimes r} \to \V, \;\; (\dot{+})^{it}:=
 \dot{+}\circ (\dot{+} \otimes 
\id_{\V})\circ (\dot{+} \otimes 
\id_{\V^{\otimes 2}}) \circ\ldots
\end{align*}
For $k=d=1$, we will set $(m^*)^{it} = \id_{\V} = (\dot{+})^{it}$.

Due to the coassociativity of $m^*$, we may draw $(m^*)^{it}$ \InnaA{as a diagram with $d$ strands in the upper part:}

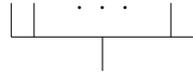
\begin{equation*}
\begin{tikzpicture}[anchorbase,scale=1.5]
\draw[-] (-0.8,0)--(0.8,0);
\draw[-] (-0.8,0)--(-0.8,0.3);
\draw[-] (-0.6,0)--(-0.6,0.3);
\draw[-] (0.6,0)--(0.6,0.3);
\draw[-] (0.8,0)--(0.8,0.3);
\draw[-] (0,0)--(0,-0.3);
\node at (0,0.25) {$\cdot$};
\node at (0.2,0.25) {$\cdot$};
\node at (-0.2,0.25) {$\cdot$};
\end{tikzpicture}
\end{equation*}

Due to the associativity of $\dot{+}$, draw $(\dot{+})^{it}$ \InnaA{as a diagram with $r$ strands in the lower part:}

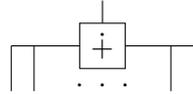
\begin{equation*}
\begin{tikzpicture}[anchorbase,scale=1.5]
\draw[-] (-0.8,0)--(-0.2,0);
\draw[-] (0.2,0)--(0.8,0);
\draw[-] (-0.2,-0.2)--(-0.2,0.2)--(0.2,0.2)--(0.2,-0.2)--(-0.2,-0.2);
\node at (-0,-0) {$\dot{+}$};
\draw[-] (-0.8,0)--(-0.8,-0.4);
\draw[-] (-0.6,0)--(-0.6,-0.4);
\draw[-] (0.6,0)--(0.6,-0.4);
\draw[-] (0.8,0)--(0.8,-0.4);
\draw[-] (0,0.2)--(0,0.4);
\node at (0,-0.35) {$\cdot$};
\node at (0.2,-0.35) {$\cdot$};
\node at (-0.2,-0.35) {$\cdot$};
\end{tikzpicture}
\end{equation*}
Let $A \in Mat_{r\times d}(\F_q)$. We will denote by 
$\mu_A: \V^{\otimes d} \to 
\V^{\otimes r}$ the composition
$$\V^{\otimes d}  \xrightarrow{\left((m^*)^{it}\right)^{\otimes d}} \V^{\otimes 
rd}  \xrightarrow{w} \V^{\otimes 
rd}  \xrightarrow{\bigotimes_{1\leq i \leq r, 1\leq j \leq d} \mu_{A_{i,j}}}  
\V^{\otimes rd}  
\xrightarrow{\left((\dot{+})^{it}\right)^{\otimes r}} \V^{\otimes 
r} $$

where $w$ is a permutation of factors as in the proof of Proposition 
\ref{prop:gen_morphisms} and $\bigotimes_{1\leq i \leq r, 1\leq j \leq d} 
\mu_{A_{i,j}}$ means that we multiply the $(i+d(j-1))$-th factor by $A_{i,j}$ 
for each $i,j$. We will draw $\mu_A$ diagramatically as follows:
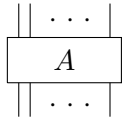
\begin{equation*}
\begin{tikzpicture}[anchorbase,scale=1.5]
\draw[-] (-0.5,-0.2)--(-0.5,0.2)--(0.5,0.2)--(0.5,-0.2)--(-0.5,-0.2);
\node at (-0,-0) {$A$};
\draw[-] (-0.4,-0.2)--(-0.4,-0.5);
\node at (-.47,.58) {};
\draw[-] (0.4,-0.2)--(0.4,-0.5);
\draw[-] (-0.4,0.2)--(-0.4,0.5);
\draw[-] (-0.3,0.2)--(-0.3,0.5);
\draw[-] (-0.3,-0.2)--(-0.3,-0.5);
\draw[-] (0.4,0.2)--(0.4,0.5);
\node at (-0.1,0.35) {$\cdot$};
\node at (0.05,0.35) {$\cdot$};
\node at (0.2,0.35) {$\cdot$};
\node at (-0.1,-0.4) {$\cdot$};
\node at (0.05,-0.4) {$\cdot$};
\node at (0.2,-0.4) {$\cdot$}; 
\end{tikzpicture}
\end{equation*}

\end{definition}
If $r=1$ and our matrix $A = [a_1, \ldots, a_d]$ is a row matrix, then  
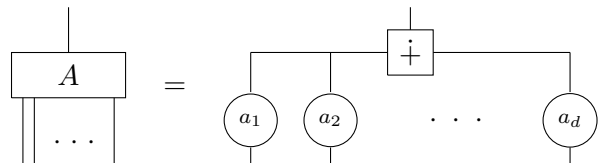
\begin{equation*}
 \begin{tikzpicture}[anchorbase,scale=1.5]
\draw[-] (-0.5,-0.2)--(-0.5,0.2)--(0.5,0.2)--(0.5,-0.2)--(-0.5,-0.2);
\node at (-0,-0) {${A}$};
\draw[-] (-0.4,-0.2)--(-0.4,-0.8);
\draw[-] (0.4,-0.2)--(0.4,-0.8);
\draw[-] (0,0.2)--(0,0.6);
\draw[-] (-0.3,-0.2)--(-0.3,-0.8);
\node at (-0.1,-0.6) {$\cdot$};
\node at (0.05,-0.6) {$\cdot$};
\node at (0.2,-0.6) {$\cdot$}; 
\end{tikzpicture} \quad
= \quad\begin{tikzpicture}[anchorbase,scale=1.5]
\draw[-] (-1.4,0)--(-0.2,0);
\draw[-] (0.2,0)--(1.4,0);
\draw[-] (-0.2,-0.2)--(-0.2,0.2)--(0.2,0.2)--(0.2,-0.2)--(-0.2,-0.2);
\node at (-0,-0) {$\dot{+}$};
\draw[-] (-1.4,0)--(-1.4,-0.35);
\draw[-] (-0.7,0)--(-0.7,-0.35);
% \draw[-] (0.7,0)--(0.7,-0.27);
\draw[-] (1.4,0)--(1.4,-0.35);
\draw[-] (0,0.2)--(0,0.4);
\node at (0.2,-0.6) {$\cdot$};
\node at (0.4,-0.6) {$\cdot$};
\node at (0.6,-0.6) {$\cdot$};
\draw[-] (-1.4,-0.84)--(-1.4,-1);
\draw[-] (-0.7,-0.84)--(-0.7,-1);
% \draw[-] (0.7,-1)--(0.7,-1.5);
\draw[-] (1.4,-0.84)--(1.4,-1);
\node[draw,circle] at (-1.4,-0.6) {${\scriptstyle a_1}$};
\node[draw,circle] at (-0.7,-0.6) {${\scriptstyle a_{2}}$};
% \node[draw,circle] at (0.7,-0.6) {$a_{d-1}$};
\node[draw,circle] at (1.4,-0.6) {${\scriptstyle a_{d}}$};
\end{tikzpicture}
\end{equation*}

For a general matrix $A \in Mat_{r\times d}(\F_q)$, if we denote the rows of 
$A$ by $A_1, A_2, \ldots, A_r$, we have the following equality:
\begin{equation}\label{dg:matrix_mult}
\begin{tikzpicture}[anchorbase,scale=1.5]
\draw[-] (-0.5,-0.2)--(-0.5,0.2)--(0.5,0.2)--(0.5,-0.2)--(-0.5,-0.2);
\node at (-0,-0) {$A$};
\draw[-] (-0.4,-0.2)--(-0.4,-1);
\draw[-] (0.4,-0.2)--(0.4,-1);
\draw[-] (-0.4,0.2)--(-0.4,0.5);
\draw[-] (-0.3,0.2)--(-0.3,0.5);
\draw[-] (-0.3,-0.2)--(-0.3,-1);
\draw[-] (0.4,0.2)--(0.4,0.5);
\node at (-0.1,0.35) {$\cdot$};
\node at (0.05,0.35) {$\cdot$};
\node at (0.2,0.35) {$\cdot$};
\node at (-0.1,-0.8) {$\cdot$};
\node at (0.05,-0.8) {$\cdot$};
\node at (0.2,-0.8) {$\cdot$}; 
\end{tikzpicture} \quad
= \quad
\begin{tikzpicture}[anchorbase,scale=1.5]%leftmost part:
\draw[-] (-0.5,-0.2)--(-0.5,0.2)--(0.5,0.2)--(0.5,-0.2)--(-0.5,-0.2);
\node at (0,0) {${A_1}$};
\draw[-] (-0.4,-0.2)--(-0.4,-0.7);
\draw[-] (0,0.2)--(0,0.5);
\draw[-] (0.4,-0.2)--(0.4,-0.5);
\node at (-0.2,-0.4) {$\cdot$};
\node at (-0,-0.4) {$\cdot$};
\node at (0.2,-0.4) {$\cdot$};%%%
\draw[-] (1.5,-0.2)--(1.5,0.2)--(2.5,0.2)--(2.5,-0.2)--(1.5,-0.2);
\node at (2,0) {${A_2}$};
\draw[-] (1.6,-0.2)--(1.6,-0.7);
\draw[-] (2,0.2)--(2,0.5);
\draw[-] (2.4,-0.2)--(2.4,-0.5);
\node at (1.8,-0.4) {$\cdot$};
\node at (2,-0.4) {$\cdot$};
\node at (2.2,-0.4) {$\cdot$};%%%intermediate dots:
\node at (3.9,0) {$\cdot$};
\node at (4,0) {$\cdot$};
\node at (4.1,0) {$\cdot$};%%% rightmost part:
\draw[-] (5.5,-0.2)--(5.5,0.2)--(6.5,0.2)--(6.5,-0.2)--(5.5,-0.2);
\node at (6,0) {${A_r}$};
\draw[-] (5.6,-0.2)--(5.6,-0.7);
\draw[-] (6,0.2)--(6,0.5);
\draw[-] (6.4,-0.2)--(6.4,-0.5);
\node at (5.8,-0.4) {$\cdot$};
\node at (6,-0.4) {$\cdot$};
\node at (6.2,-0.4) {$\cdot$};%%% Bottom part
\draw[-] (-0.4,-0.7)--(5.6,-0.7);
\draw[-] (0.4,-0.5)--(1.5, -0.5);
\draw[-] (1.7, -0.5)--(5.5, -0.5);
\draw[-] (5.7,-0.5)--(6.4,-0.5);
\draw[-] (3, -0.7)--(3,-1);
\draw[-] (3.4, -0.5)--(3.4, -0.65);
\draw[-] (3.4, -0.75)--(3.4, -1);
\node at (3.2,-1) {$\ldots$};
\end{tikzpicture}
\end{equation}

Here each intersection of the strands of the form $\begin{tikzpicture}[anchorbase]
\draw[-] (-0.2,0)--(-0.1,0);
\draw[-] (0.1,0)--(0.2,0);
\draw[-] (0,0.2)--(0,-0.2);
\end{tikzpicture}$ or $\begin{tikzpicture}[anchorbase]
\draw[-] (-0.2,0)--(0.2,0);
\draw[-] (0,0.2)--(0, 0.1);
\draw[-] (0,-0.1)--(0,-0.2);
\end{tikzpicture}$ means that we apply the symmetry morphism $\sigma: \V\otimes \V \to \V\otimes \V$ on the relevant factors.

\begin{example}\label{ex:matrix_mult_classical}
For $\mathcal{C} = \Rep(GL_n(\F_q))$, $\V:=\V_n$, and any matrix $A\in Mat_{r\times d}(\F_q)$, the operator $\mu_A:\V_n^{\otimes d}\to \V_n^{\otimes r}$ is given by
$$ v_1\otimes \ldots \otimes v_d \longmapsto \left(\dot{\sum}_{j=1}^{d} \dot{A}_{1,j} v_j \right)\otimes  \ldots  \otimes \left(\dot{\sum}_{j=1}^{d} \dot{A}_{r,j} v_j \right) \in \V_n^{\otimes r}$$
for any $v_1, \ldots, v_d \in V$.

Let us give this map an interpretation in terms of $\F_q$-linear vector spaces.

Recall that $V=\F_q^n$ denotes the basis of $\V_n$ and write the elements of $V^{\times d} = V^{\dot{\oplus} d}$ as column vectors whose entries are given by elements of $V$. Then $\mu_A :\V_n^{\otimes d} \to \V_n^{\otimes r} $ corresponds to the $\F_q$-linear operator 
\begin{align*}
V^{\times d} &\longrightarrow V^{\times r}, \;\; 
\begin{bmatrix}
v_1\\ \vdots\\ v_{d}
\end{bmatrix}  \longmapsto A \begin{bmatrix}
v_1 \\ \vdots\\ v_{d}
\end{bmatrix}.
\end{align*}
\end{example}

\begin{example}
For $A=\begin{bmatrix} 1 &\ldots &1
\end{bmatrix} \in Mat_{1\times k}(\F_q)$, the morphism $\mu_{A}$ is just $(\dot{+})^{it}$. Similarly, for $B=\begin{bmatrix} 1 \\ \vdots \\1
\end{bmatrix} \in Mat_{k\times 1}(\F_q)$, the morphism $\mu_{B}$ is just $(m^*)^{it}$. 
\end{example}

\section{Composing generating morphisms}\label{sec:linear_algebra_for_Frob_space}

In this section, we give some results concerning the morphisms from Definition \ref{def:Frob_linear_space} and the morphisms $\mu_A$ for a general $\C$-linear rigid SM category $\mathcal{C}$ and an $\F_q$-linear Frobenius space $\V$ in $\mathcal{C}$. We will also provide examples for most of the statements in the case $\mathcal{C} = \Rep(GL_n(\F_q))$, $\V:=\V_n$, using the notation of Section \ref{sec:classical_endom}.

The main purpose of this section is to prove the following statements, given in Propositions \ref{prop:compos_mu_A} and \ref{prop:tensor_prod_mu_B}: for any $B \in Mat_{ s\times k}(\F_q), A \in Mat_{k \times l} (\F_q), A'\in Mat_{k' \times l'} (\F_q)$, we have $\mu_B \circ \mu_A = \mu_{BA}$ and $\mu_A \otimes \mu_A' = \mu_{\begin{bmatrix} A &0 \\ 0 &A'
\end{bmatrix}}$; thus the morphisms $\mu_A$ are direct analogues of ``multiplication by a matrix''. Section \ref{ssec:muA_comp} is devoted to proving the first equality, and Section \ref{ssec:muA_tens_prod} is devoted to proving the second equality. 

In Sections \ref{ssec:transp_mu_A} and \ref{ssec:duality_mu_A} we establish several results related to the equality $(\mu_A)^*= \mu_{A^{-1}}$ for invertible square matrices $A$ (the equality itself is proved in Section \ref{ssec:duality_mu_A}).

In Section \ref{sec:repinf} we will deal with a setting very similar to the current one, but lacking the morphism $\eps$ and thus the morphism $coev$ as well, destroying the self-duality of the object $\V$. For the benefit of that section, we will not use the morphisms $\eps$, $coev$ in Sections \ref{ssec:muA_tens_prod},  \ref{ssec:muA_comp} and \ref{ssec:transp_mu_A}, so that the results in these sections can be applied to the setting of Section \ref{sec:repinf} as well.

\subsection{Tensor product of morphisms \texorpdfstring{$\mu_A$}{muA}}\label{ssec:muA_tens_prod}

\begin{lemma}\label{lem:horizontal stacking_with_zero}
Let $d, r_1, r_2\geq 0$, and let $A\in Mat_{d\times r_1}(\F_q)$.
Denote by $\dot{0}$ the zero matrix in $Mat_{d\times r_2}(\F_q)$. Then $$\mu_{\begin{bmatrix} A &0 \end{bmatrix}} = \mu_A \otimes (\eps^*)^{\otimes r_2},\; \mu_{\begin{bmatrix} 0 &A\end{bmatrix}} = (\eps^*)^{\otimes r_2} \otimes \mu_A.$$ 
In terms of diagrams, this means: 
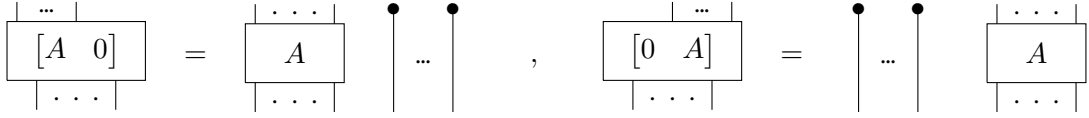
\begin{equation*}
 \begin{tikzpicture}[anchorbase,scale=1.3]
%leftmost part:
\draw[-] (-0.7,-0.2)--(-0.7,0.4)--(0.7,0.4)--(0.7,-0.2)--(-0.7,-0.2);
\node at (0,0.1) {${\begin{bmatrix} A &0\end{bmatrix}}$};
\draw[-] (-0.4,-0.2)--(-0.4,-0.5);
\draw[-] (0.4,-0.2)--(0.4,-0.5);
\node at (-0.2,-0.4) {$\cdot$};
\node at (-0,-0.4) {$\cdot$};
\node at (0.2,-0.4) {$\cdot$};
\draw[-] (-0.6,0.4)--(-0.6,0.6);
\draw[-] (0,0.4)--(0,0.6);
\node at (-0.25,0.5) {$\cdot$};
\node at (-0.3,0.5) {$\cdot$};
\node at (-0.35,0.5) {$\cdot$};%%%

\end{tikzpicture}
\quad = \quad
\begin{tikzpicture}[anchorbase, scale=1.3]
%leftmost part:
\draw[-] (-0.5,-0.2)--(-0.5,0.4)--(0.5,0.4)--(0.5,-0.2)--(-0.5,-0.2);
\node at (0,0.1) {${A}$};
\draw[-] (-0.4,-0.2)--(-0.4,-0.5);
\draw[-] (0.4,-0.2)--(0.4,-0.5);
\node at (-0.2,-0.4) {$\cdot$};
\node at (-0,-0.4) {$\cdot$};
\node at (0.2,-0.4) {$\cdot$};
\draw[-] (-0.4,0.4)--(-0.4,0.6);
\draw[-] (0.4,0.4)--(0.4,0.6);
\node at (-0.2,0.5) {$\cdot$};
\node at (-0,0.5) {$\cdot$};
\node at (0.2,0.5) {$\cdot$};

%%%
\node at (1, 0.55) {$\bullet$};
\node at (1.6, 0.55) {$\bullet$};
\draw[-] (1,-0.5)--(1,0.55);
\draw[-] (1.6,-0.5)--(1.6,0.55);

%%% left intermediate dots:
\node at (1.25,0) {$\cdot$};
\node at (1.3,0) {$\cdot$};
\node at (1.35,0) {$\cdot$};

\end{tikzpicture}
\quad\quad , \quad \quad
 \begin{tikzpicture}[anchorbase,scale=1.3]
%%% rightmost part:
\draw[-] (3.3,-0.2)--(3.3,0.4)--(4.7,0.4)--(4.7,-0.2)--(3.3,-0.2);
\node at (4,0.1) {${\begin{bmatrix} 0 &A\end{bmatrix}}$};
\draw[-] (3.6,-0.2)--(3.6,-0.5);
\draw[-] (4.4,-0.2)--(4.4,-0.5);
\node at (3.8,-0.4) {$\cdot$};
\node at (4,-0.4) {$\cdot$};
\node at (4.2,-0.4) {$\cdot$};
\draw[-] (4,0.4)--(4,0.6);
\draw[-] (4.6,0.4)--(4.6,0.6);
\node at (4.35,0.5) {$\cdot$};
\node at (4.3,0.5) {$\cdot$};
\node at (4.25,0.5) {$\cdot$};%%%

\end{tikzpicture}
\quad=\quad
 \begin{tikzpicture}[anchorbase,scale=1.3]
%%% right part:
\draw[-] (3.5,-0.2)--(3.5,0.4)--(4.5,0.4)--(4.5,-0.2)--(3.5,-0.2);
\node at (4,0.1) {${A}$};
\draw[-] (3.6,-0.2)--(3.6,-0.5);
\draw[-] (4.4,-0.2)--(4.4,-0.5);
\node at (3.8,-0.4) {$\cdot$};
\node at (4,-0.4) {$\cdot$};
\node at (4.2,-0.4) {$\cdot$};
\draw[-] (3.6,0.4)--(3.6,0.6);
\draw[-] (4.4,0.4)--(4.4,0.6);
\node at (3.8,0.5) {$\cdot$};
\node at (4,0.5) {$\cdot$};
\node at (4.2,0.5) {$\cdot$};

%%%
\node at (2.2, 0.55) {$\bullet$};
\node at (2.8, 0.55) {$\bullet$};
\draw[-] (2.2,-0.5)--(2.2,0.55);
\draw[-] (2.8,-0.5)--(2.8,0.55);

%%% right intermediate dots:
\node at (2.45,0) {$\cdot$};
\node at (2.5,0) {$\cdot$};
\node at (2.55,0) {$\cdot$};

\end{tikzpicture}
\end{equation*}

\end{lemma}
\begin{proof}
We will prove the first equality; the second one is proved in exactly the same manner. 
Clearly, it is enough to prove this for $d=1$: that is, we need to check that $$\mu_{[a_1 \, \ldots \, a_{r_1}\, 0\, \ldots\, 0]} = \mu_{[a_1 \, \ldots\, a_{r_1}]} \otimes (\eps^*)^{\otimes r_2}$$ for any $A = [a_1 \,\ldots\, a_{r_1}] \in Mat_{1\times r_1}(\F_q)$.
Let us draw the corresponding diagram:
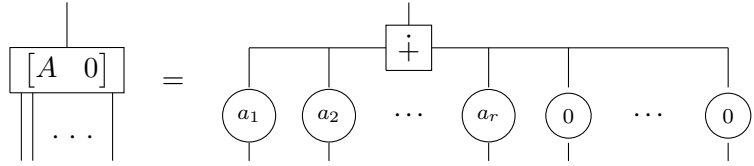
\begin{equation*}
 \begin{tikzpicture}[anchorbase,scale=1.5]
\draw[-] (-0.5,-0.2)--(-0.5,0.2)--(0.5,0.2)--(0.5,-0.2)--(-0.5,-0.2);
\node at (-0,-0) {${\begin{bmatrix}
A &0
\end{bmatrix}}$};
\draw[-] (-0.4,-0.2)--(-0.4,-0.8);
\draw[-] (0.4,-0.2)--(0.4,-0.8);
\draw[-] (0,0.2)--(0,0.6);
\draw[-] (-0.3,-0.2)--(-0.3,-0.8);
\node at (-0.1,-0.6) {$\cdot$};
\node at (0.05,-0.6) {$\cdot$};
\node at (0.2,-0.6) {$\cdot$}; 
\end{tikzpicture} \quad
= \quad\begin{tikzpicture}[anchorbase,scale=1.5]
\draw[-] (-1.4,0)--(-0.2,0);
\draw[-] (0.2,0)--(2.8,0);
\draw[-] (-0.2,-0.2)--(-0.2,0.2)--(0.2,0.2)--(0.2,-0.2)--(-0.2,-0.2);
\node at (-0,-0) {$\dot{+}$};
\draw[-] (0,0.2)--(0,0.4);

\draw[-] (-1.4,0)--(-1.4,-0.35);
\draw[-] (-0.7,0)--(-0.7,-0.35);
\draw[-] (0.7,0)--(0.7,-0.35);
\draw[-] (1.4,0)--(1.4,-0.35);
\draw[-] (2.8,0)--(2.8,-0.35);

\node at (-0.1,-0.6) {$\cdot$};
\node at (0,-0.6) {$\cdot$};
\node at (0.1,-0.6) {$\cdot$};

\node at (2,-0.6) {$\cdot$};
\node at (2.1,-0.6) {$\cdot$};
\node at (2.2,-0.6) {$\cdot$};

\draw[-] (-1.4,-0.84)--(-1.4,-1);
\draw[-] (-0.7,-0.84)--(-0.7,-1);
\draw[-] (0.7,-0.84)--(0.7,-1);
\draw[-] (1.4,-0.84)--(1.4,-1);
\draw[-] (2.8,-0.84)--(2.8,-1);

\node[draw,circle] at (-1.4,-0.6) {${\scriptstyle a_1}$};
\node[draw,circle] at (-0.7,-0.6) {${\scriptstyle a_{2}}$};
 \node[draw,circle] at (0.7,-0.6) {${\scriptstyle a_{r}}$};
\node[draw,circle] at (1.4,-0.6) {${\scriptstyle 0}$};
\node[draw,circle] at (2.8,-0.6) {${\scriptstyle 0}$};
\end{tikzpicture}
\end{equation*}
Recall that from Relation \ref{rel:F_q_lin_mu}, we have:
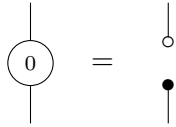
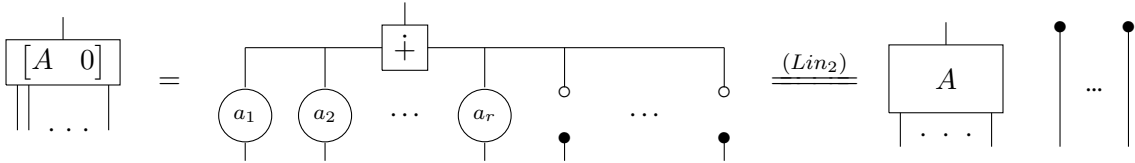
\begin{equation*} 
\begin{tikzpicture}[anchorbase,scale=1]
\draw[-] (0, -0.8)--(0, -0.3);
 \draw[-] (0, 0.8)--(0, 0.3); 
 \node[draw,circle] at (0,0) {$\scriptstyle 0$};  
 \end{tikzpicture}
 \quad  = \quad 
 \begin{tikzpicture}[anchorbase,scale=1]
\draw[-] (0,-0.8)--(0,-0.3);
\node at (0,-0.3) {$\bullet$};
\node at (0,0.27) {$\circ$};
\draw[-] (0,0.8)--(0,0.34);
\end{tikzpicture} 
\end{equation*}
Hence
\begin{equation*}
 \begin{tikzpicture}[anchorbase,scale=1.5]
\draw[-] (-0.5,-0.2)--(-0.5,0.2)--(0.5,0.2)--(0.5,-0.2)--(-0.5,-0.2);
\node at (-0,-0) {${\begin{bmatrix}
A &0
\end{bmatrix}}$};
\draw[-] (-0.4,-0.2)--(-0.4,-0.6);
\draw[-] (0.4,-0.2)--(0.4,-0.6);
\draw[-] (0,0.2)--(0,0.4);
\draw[-] (-0.3,-0.2)--(-0.3,-0.6);
\node at (-0.1,-0.6) {$\cdot$};
\node at (0.05,-0.6) {$\cdot$};
\node at (0.2,-0.6) {$\cdot$}; 
\end{tikzpicture} \quad
= \quad\begin{tikzpicture}[anchorbase,scale=1.5]
\draw[-] (-1.4,0)--(-0.2,0);
\draw[-] (0.2,0)--(2.8,0);
\draw[-] (-0.2,-0.2)--(-0.2,0.2)--(0.2,0.2)--(0.2,-0.2)--(-0.2,-0.2);
\node at (-0,-0) {$\dot{+}$};
\draw[-] (0,0.2)--(0,0.4);

\draw[-] (-1.4,0)--(-1.4,-0.35);
\draw[-] (-0.7,0)--(-0.7,-0.35);
\draw[-] (0.7,0)--(0.7,-0.35);
\draw[-] (1.4,0)--(1.4,-0.35);
\draw[-] (2.8,0)--(2.8,-0.35);

\node at (-0.1,-0.6) {$\cdot$};
\node at (0,-0.6) {$\cdot$};
\node at (0.1,-0.6) {$\cdot$};

\node at (2,-0.6) {$\cdot$};
\node at (2.1,-0.6) {$\cdot$};
\node at (2.2,-0.6) {$\cdot$};

\draw[-] (-1.4,-0.84)--(-1.4,-1);
\draw[-] (-0.7,-0.84)--(-0.7,-1);
\draw[-] (0.7,-0.84)--(0.7,-1);
\draw[-] (1.4,-0.84)--(1.4,-1);
\draw[-] (2.8,-0.84)--(2.8,-1);

\node[draw,circle] at (-1.4,-0.6) {${\scriptstyle a_1}$};
\node[draw,circle] at (-0.7,-0.6) {${\scriptstyle a_{2}}$};
 \node[draw,circle] at (0.7,-0.6) {${\scriptstyle a_{r}}$};
 
 \node at (1.4,-0.4) {${\circ}$};
\node at (2.8,-0.4) {${\circ}$};

\node at (1.4,-0.8) {${\bullet}$};
\node at (2.8,-0.8) {${\bullet}$};
\end{tikzpicture} \quad \InnaA{\xlongequal{\text{\ref{rel:F_q_lin_zero}}}} \quad \begin{tikzpicture}[anchorbase, scale=1.5]
%leftmost part:
\draw[-] (-0.5,-0.2)--(-0.5,0.4)--(0.5,0.4)--(0.5,-0.2)--(-0.5,-0.2);
\node at (0,0.1) {${A}$};
\draw[-] (-0.4,-0.2)--(-0.4,-0.5);
\draw[-] (0.4,-0.2)--(0.4,-0.5);
\node at (-0.2,-0.4) {$\cdot$};
\node at (-0,-0.4) {$\cdot$};
\node at (0.2,-0.4) {$\cdot$};

\draw[-] (0,0.4)--(0,0.6);

%%%
\node at (1, 0.55) {$\bullet$};
\node at (1.6, 0.55) {$\bullet$};
\draw[-] (1,-0.5)--(1,0.55);
\draw[-] (1.6,-0.5)--(1.6,0.55);

%%% left intermediate dots:
\node at (1.25,0) {$\cdot$};
\node at (1.3,0) {$\cdot$};
\node at (1.35,0) {$\cdot$};

\end{tikzpicture}
\end{equation*}

\end{proof}

\begin{proposition}\label{prop:tensor_prod_mu_B}
Let $A \in Mat_{d_1\times r_1}(\F_q)$, $B \in Mat_{d_2\times r_2}(\F_q)$ be two matrices, and denote: $C = \begin{bmatrix}                                                    A &0\\ 
0 &B 
\end{bmatrix}
$. Then $\mu_{A}\otimes \mu_{B} = \mu_C$.
\end{proposition}

\begin{example}
 For $\mathcal{C} = \Rep(GL_n(\F_q))$, $\V:=\V_n$, this lemma states that for any $v_1, \ldots, v_{r_1}, w_1, \ldots, w_{r_2} \in V$, we have: $$\mu_A(v_1, \ldots, v_{r_1}) \otimes \mu_B (w_1, \ldots, w_{r_2}) = \mu_{\begin{bmatrix}                                                    A &0\\ 
0 &B 
\end{bmatrix}} (v_1, \ldots, v_{r_1}, w_1, \ldots, w_{r_2}).$$
Indeed, in the setting of Example \ref{ex:matrix_mult_classical}, the map $\mu_A \otimes \mu_B :\V_n^{\otimes r_1+r_2} \to \V_n^{\otimes d_1+d_2} $ corresponds to the $\F_q$-linear operator 
\begin{align*}
V^{\times (r_1+r_2)} =V^{\times r_1} \dot{\oplus} V^{\times r_2}  &\longrightarrow V^{\times d_1} \dot{\oplus} V^{\times d_2} = V^{\times (d_1+d_2)}\\
\begin{bmatrix}
v_1\\ \vdots \\ v_{r_1}
\end{bmatrix} \dot{\oplus} \begin{bmatrix}
w_1\\ \vdots \\ w_{r_2}
\end{bmatrix} &\longmapsto A \begin{bmatrix}
v_1\\ \vdots\\ v_{r_1}
\end{bmatrix}
 \dot{\oplus} B \begin{bmatrix}
w_1\\ \vdots\\ w_{r_2}
\end{bmatrix}
\end{align*} which is exactly the multiplication of 
$\begin{bmatrix}
v_1\\ \vdots \\ v_{r_1} \\w_1\\ \vdots \\ w_{r_2}
\end{bmatrix}$ by $\begin{bmatrix}                                                    A &0\\ 
0 &B 
\end{bmatrix}$.
\end{example}

\begin{proof}[Proof of Proposition \ref{prop:tensor_prod_mu_B}]

Consider the morphism $\mu_C$. The morphism $\mu_C $ is given by the diagram
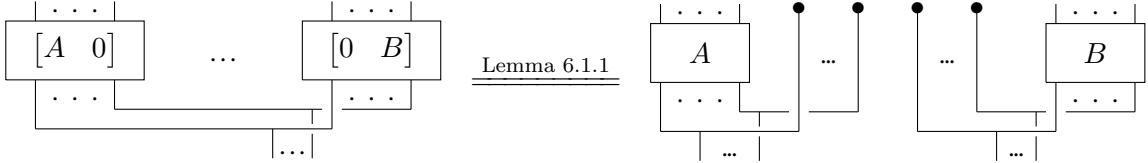
\begin{equation*}
 \begin{tikzpicture}[anchorbase,scale=1.3]
%leftmost part:
\draw[-] (0.3,-0.2)--(0.3,0.4)--(1.7,0.4)--(1.7,-0.2)--(0.3,-0.2);
\node at (1,0.1) {${\begin{bmatrix} A &0\end{bmatrix}}$};
\draw[-] (0.6,-0.2)--(0.6,-0.7);
\draw[-] (1.4,-0.2)--(1.4,-0.5);
\node at (0.8,-0.4) {$\cdot$};
\node at (1,-0.4) {$\cdot$};
\node at (1.2,-0.4) {$\cdot$};
\draw[-] (0.6,0.4)--(0.6,0.6);
\draw[-] (1.4,0.4)--(1.4,0.6);
\node at (0.8,0.5) {$\cdot$};
\node at (1,0.5) {$\cdot$};
\node at (1.2,0.5) {$\cdot$};%%%

%%%intermediate dots:
\node at (2.4,0) {$\cdot$};
\node at (2.5,0) {$\cdot$};
\node at (2.6,0) {$\cdot$};

%%% rightmost part:
\draw[-] (3.3,-0.2)--(3.3,0.4)--(4.7,0.4)--(4.7,-0.2)--(3.3,-0.2);
\node at (4,0.1) {${\begin{bmatrix} 0 &B\end{bmatrix}}$};
\draw[-] (3.6,-0.2)--(3.6,-0.7);
\draw[-] (4.4,-0.2)--(4.4,-0.5);
\node at (3.8,-0.4) {$\cdot$};
\node at (4,-0.4) {$\cdot$};
\node at (4.2,-0.4) {$\cdot$};
\draw[-] (3.6,0.4)--(3.6,0.6);
\draw[-] (4.4,0.4)--(4.4,0.6);
\node at (3.8,0.5) {$\cdot$};
\node at (4,0.5) {$\cdot$};
\node at (4.2,0.5) {$\cdot$};
%%% Bottom part
\draw[-] (0.6,-0.7)--(3.6,-0.7);
\draw[-] (1.4,-0.5)--(3.5, -0.5);
\draw[-] (3.7,-0.5)--(4.4,-0.5);
\draw[-] (3, -0.7)--(3,-1);
\draw[-] (3.4, -0.5)--(3.4, -0.65);
\draw[-] (3.4, -0.75)--(3.4, -1);
\node at (3.1,-0.95) {$\cdot$};
\node at (3.2,-0.95) {$\cdot$};
\node at (3.3,-0.95) {$\cdot$};
\end{tikzpicture}
\InnaA{\quad \xlongequal{\text{\cref{lem:horizontal stacking_with_zero}}}\quad}
 \begin{tikzpicture}[anchorbase,scale=1.3]
%leftmost part:
\draw[-] (-0.5,-0.2)--(-0.5,0.4)--(0.5,0.4)--(0.5,-0.2)--(-0.5,-0.2);
\node at (0,0.1) {${A}$};
\draw[-] (-0.4,-0.2)--(-0.4,-0.7);
\draw[-] (0.4,-0.2)--(0.4,-0.5);
\node at (-0.2,-0.4) {$\cdot$};
\node at (-0,-0.4) {$\cdot$};
\node at (0.2,-0.4) {$\cdot$};
\draw[-] (-0.4,0.4)--(-0.4,0.6);
\draw[-] (0.4,0.4)--(0.4,0.6);
\node at (-0.2,0.5) {$\cdot$};
\node at (-0,0.5) {$\cdot$};
\node at (0.2,0.5) {$\cdot$};

%%%
\node at (1, 0.55) {$\bullet$};
\node at (1.6, 0.55) {$\bullet$};
\draw[-] (1,-0.7)--(1,0.55);
\draw[-] (1.6,-0.5)--(1.6,0.55);

%%% left intermediate dots:
\node at (1.25,0) {$\cdot$};
\node at (1.3,0) {$\cdot$};
\node at (1.35,0) {$\cdot$};

%%% Bottom left part
\draw[-] (-0.4,-0.7)--(1,-0.7);
\draw[-] (0.4,-0.5)--(0.9, -0.5);
\draw[-] (1.1,-0.5)--(1.6,-0.5);
\draw[-] (0, -0.7)--(0,-1);
\draw[-] (0.6, -0.5)--(0.6, -0.65);
\draw[-] (0.6, -0.75)--(0.6, -1);
\node at (0.25,-0.95) {$\cdot$};
\node at (0.3,-0.95) {$\cdot$};
\node at (0.35,-0.95) {$\cdot$};

%%% right part:
\draw[-] (3.5,-0.2)--(3.5,0.4)--(4.5,0.4)--(4.5,-0.2)--(3.5,-0.2);
\node at (4,0.1) {${B}$};
\draw[-] (3.6,-0.2)--(3.6,-0.7);
\draw[-] (4.4,-0.2)--(4.4,-0.5);
\node at (3.8,-0.4) {$\cdot$};
\node at (4,-0.4) {$\cdot$};
\node at (4.2,-0.4) {$\cdot$};
\draw[-] (3.6,0.4)--(3.6,0.6);
\draw[-] (4.4,0.4)--(4.4,0.6);
\node at (3.8,0.5) {$\cdot$};
\node at (4,0.5) {$\cdot$};
\node at (4.2,0.5) {$\cdot$};

%%%
\node at (2.2, 0.55) {$\bullet$};
\node at (2.8, 0.55) {$\bullet$};
\draw[-] (2.2,-0.7)--(2.2,0.55);
\draw[-] (2.8,-0.5)--(2.8,0.55);

%%% right intermediate dots:
\node at (2.45,0) {$\cdot$};
\node at (2.5,0) {$\cdot$};
\node at (2.55,0) {$\cdot$};

%%% Bottom right part
\draw[-] (2.2,-0.7)--(3.6,-0.7);
\draw[-] (2.8,-0.5)--(3.5, -0.5);
\draw[-] (3.7,-0.5)--(4.4,-0.5);
\draw[-] (3, -0.7)--(3,-1);
\draw[-] (3.4, -0.5)--(3.4, -0.65);
\draw[-] (3.4, -0.75)--(3.4, -1);
\node at (3.15,-0.95) {$\cdot$};
\node at (3.2,-0.95) {$\cdot$};
\node at (3.25,-0.95) {$\cdot$};
\end{tikzpicture}
\end{equation*}
By Frobenius coalgebra relations (see  \ref{rel:Frob_alg}), we conclude that this is just $\mu_{A}\otimes \mu_{B}$ as required.

\end{proof}

\subsection{Composing morphisms \texorpdfstring{$\mu_A$}{muA}}\label{ssec:muA_comp}
For the result about composition of morphisms $\mu_A$, we will need some auxiliary lemmas.

\begin{lemma}\label{lem:mu_identity}
Let $k,l \geq 1$. Consider the matrix $\begin{bmatrix}
I_l \\ \vdots \\ I_l
\end{bmatrix} \in Mat_{kl\times l}(\F_q)$. We have: $$ \mu_{\begin{bmatrix}
I_l \\ \vdots \\ I_l
\end{bmatrix}} = w \circ \left((m^*)^{it}\right)^{\otimes l} $$
 where $(m^*)^{it}: \V \to \V^{\otimes k}$ and $w: \V^{\otimes kl} \to \V^{\otimes kl}$ is the permutation of tensor factors sending the factor $s + rl$ to the position $r+sk$ for $0\leq s < l$, $0\leq r <k$. In particular, $\mu_{I_l} = \id_{\V^{\otimes l}}$.
\end{lemma}
\begin{proof}
By Lemma \ref{lem:horizontal stacking_with_zero}, we have:
$\mu_{\begin{bmatrix}
0 &\ldots &0 &1 &0 &\ldots &0
\end{bmatrix}} = (\eps^*)^{\otimes r_1} \otimes \id \otimes (\eps^*)^{\otimes r_2}$ where $r_1$ is the number of zeroes in the left part of the matrix, and $r_2$ is the number of zeroes in the right part of the matrix. Using the definition of $\mu_A$ (see Diagram \eqref{dg:matrix_mult}), and the fact that $\eps^*$ is a counit for $m^*$ (see Relation \ref{rel:Frob_alg}), we obtain the required statement.
\end{proof}
This Lemma immediately implies:
\begin{corollary}\label{cor:vert_stacking}
Let $A\in Mat_{s_1\times l}(\F_q), A' \in Mat_{s_2\times l}(\F_q)$. Then $$\mu_{\begin{bmatrix} A\\A'
\end{bmatrix}} = \left(\mu_A \otimes \mu_{A'}\right) \circ \mu_{\begin{bmatrix}
I_l \\ I_l
\end{bmatrix}}.$$
\end{corollary}

\begin{lemma}\label{lem:interchanging_plus_and_comult}
 Let $k, l\geq 2$. Then the composition $(m^*)^{it} \circ (\dot{+})^{it} $ of $(\dot{+})^{it}: \V^{\otimes l} \to \V$ and $(m^*)^{it}: \V \to \V^{\otimes k}$ equals 
 $$\left((\dot{+})^{it}\right)^{\otimes k} \circ w \circ \left((m^*)^{it}\right)^{\otimes l} $$
 where $w: \V^{\otimes kl} \to \V^{\otimes kl}$ is as in Lemma \ref{lem:mu_identity}. 
 
 Diagrammatically, this can be drawn as
 \begin{equation*}
 \begin{tikzpicture}[anchorbase,scale=1.3]
   \draw[-] (0.4,-0.5)--(0.4,-0.2);%m^* top
     \draw[-] (-0.4,-0.5)--(-0.4,-0.2);
  \draw[-] (-0.4,-0.5)--(0.4,-0.5);
  
  \node at (-0.1,-0.3) {$\cdot$};
\node at (-0,-0.3) {$\cdot$};
\node at (0.1,-0.3) {$\cdot$};

 \draw[-] (0,-0.8)--(0,-0.5);
 
 %dot{+} bottom
\draw[-] (-0.2,-1.2)--(-0.2,-0.8)--(0.2,-0.8)--(0.2,-1.2)--(-0.2,-1.2);
\node at (0, -1) {$\dot{+}$};
\draw[-] (-0.4,-1)--(-0.2,-1);
\draw[-] (0.2,-1)--(0.4,-1);
  
\draw[-] (-0.4,-1) -- (-0.4, -1.5);
\draw[-] (0.4,-1) -- (0.4, -1.5);
    
\node at (-0.1,-1.45) {$\cdot$};
\node at (-0,-1.45) {$\cdot$};
\node at (0.1,-1.45) {$\cdot$};
 \end{tikzpicture}
 \quad=\quad
  \begin{tikzpicture}[anchorbase,scale=1.3]

%leftmost part:
\draw[-] (-0.2,-0.2)--(-0.2,0.2)--(0.2,0.2)--(0.2,-0.2)--(-0.2,-0.2);
\node at (0,0) {$\dot{+}$};
\draw[-] (-0.4,0)--(-0.4,-0.7);
\draw[-] (0,0.2)--(0,0.4);
\draw[-] (0.4,0)--(0.4,-0.5);

\draw[-] (-0.4,0)--(-0.2,0);
\draw[-] (0.2,0)--(0.4,0);

\node at (-0.1,-0.4) {$\cdot$};
\node at (-0,-0.4) {$\cdot$};
\node at (0.1,-0.4) {$\cdot$};%%%
\draw[-] (1.8,-0.2)--(1.8,0.2)--(2.2,0.2)--(2.2,-0.2)--(1.8,-0.2);
\node at (2,0) {$\dot{+}$};
\draw[-] (1.6,0)--(1.6,-0.7);
\draw[-] (2,0.2)--(2,0.4);
\draw[-] (2.4,0)--(2.4,-0.5);

\draw[-] (1.6,0)--(1.8,0);
\draw[-] (2.2,0)--(2.4,0);

\node at (1.9,-0.4) {$\cdot$};
\node at (2,-0.4) {$\cdot$};
\node at (2.1,-0.4) {$\cdot$};

%%%intermediate dots:
\node at (3.9,0) {$\cdot$};
\node at (4,0) {$\cdot$};
\node at (4.1,0) {$\cdot$};

%%% rightmost part:
\draw[-] (5.8,-0.2)--(5.8,0.2)--(6.2,0.2)--(6.2,-0.2)--(5.8,-0.2);
\node at (6,0) {$\dot{+}$};
\draw[-] (5.6,0)--(5.6,-0.7);
\draw[-] (6,0.2)--(6,0.4);
\draw[-] (6.4,0)--(6.4,-0.5);

\draw[-] (5.6,0)--(5.8,0);
\draw[-] (6.2,0)--(6.4,0);

\node at (5.9,-0.4) {$\cdot$};
\node at (6,-0.4) {$\cdot$};
\node at (6.1,-0.4) {$\cdot$};%%% Bottom part
\draw[-] (-0.4,-0.7)--(5.6,-0.7);
\draw[-] (0.4,-0.5)--(1.5, -0.5);
\draw[-] (1.7, -0.5)--(5.5, -0.5);
\draw[-] (5.7,-0.5)--(6.4,-0.5);
\draw[-] (3, -0.7)--(3,-1);
\draw[-] (3.4, -0.5)--(3.4, -0.65);
\draw[-] (3.4, -0.75)--(3.4, -1);
%bottom dots:
\node at (3.1,-0.95) {$\cdot$};
\node at (3.2,-0.95) {$\cdot$};
\node at (3.3,-0.95) {$\cdot$};

\end{tikzpicture} 
 \end{equation*}

\end{lemma}

\begin{remark}
 In terms of morphisms $\mu_A$, this reads: $ \mu_{\begin{bmatrix}
 1 \\
 \vdots \\
 1 
 \end{bmatrix}} \circ \mu_{\begin{bmatrix}
 1 &\ldots &1
 \end{bmatrix}}=\mu_{\begin{bmatrix}
 1  &\ldots &1\\
 &\ldots \\
 1  &\ldots &1
 \end{bmatrix}}$.
\end{remark}

\begin{example}
For $\mathcal{C} = \Rep(GL_n(\F_q))$, $\V:=\V_n$, the statement of this Lemma translates to the fact that for any $v_1, \ldots, v_l \in V$, the element
$$ (v_1\dot{+}\ldots \dot{+} v_l) \otimes (v_1\dot{+}\ldots \dot{+} v_l)\otimes  \ldots  \otimes (v_1\dot{+}\ldots \dot{+} v_l) \in \V_n^{\otimes k}$$ can be obtained from $v_1\otimes\ldots \otimes v_l \in \V_n^{\otimes l}$ in two ways. The first way is by applying iteratively the map $m^*$ to $v_1\dot{+}\ldots \dot{+} v_l \in \V$. The second way is by applying $\dot{+}$ to the factors of the tensor
$$ (v_1\otimes v_2 \otimes\ldots \otimes  v_l)\ldots \otimes (v_1 \otimes v_2\otimes  \ldots \otimes v_l) \in \V_n^{\otimes kl}$$
which in turn is obtained by permutation of factors from the tensor $$ (v_1\otimes v_1 \otimes \ldots \otimes v_1) \otimes  \ldots \otimes (v_l\otimes v_l\otimes \ldots \otimes v_l). $$ The latter is just $\left((m^*)^{it}\right)^{\otimes l}(v_1\otimes\ldots \otimes v_l)$.
\end{example}

\begin{proof}[Proof of Lemma \ref{lem:interchanging_plus_and_comult}]
The proof is by double induction on the pair $(k, l)$. For $(k=2,l=2)$, this statement is just the Relation \ref{rel:plus_coalg_mor}, stating that $\dot{+}$ is a morphism of coalgebras. This relation tells us that
\begin{equation}\label{eq:plus_coalg_mor}
 m^*\circ \dot{+} =( \dot{+} \otimes \dot{+})\circ (\id \otimes \sigma \otimes \id) \circ (m^*\otimes m^*)
\end{equation}
(see also\ref{itm:str_rel_diag_plus_coalg} for a diagrammatic version of this equality).

Now, assume the statement holds for $(k=2, l)$ for some $l\geq 2$. Let us prove it for $(k=2, l+1)$. 
Indeed, in that case by induction assumption
$$m^* \circ (\dot{+})^{it} = \left((\dot{+})^{it}\right)^{\otimes 2} \circ w \circ \left((m^*)^{it}\right)^{\otimes l} \circ (\id\otimes \dot{+}). $$ Applying \eqref{eq:plus_coalg_mor}, we see that the right hand side equals $\left((\dot{+})^{it}\right)^{\otimes 2} \circ w' \circ \left((m^*)^{it}\right)^{\otimes l}$. Here $(\dot{+})^{it}$ denotes, by abuse of notation, the map $\V^{\otimes l+1} \to \V$, and $w':\V^{\otimes 2(l+1)} \to \V^{\otimes 2(l+1)}$ is a morphism permuting the tensor factors, obtained by precomposing $w$ with a permutation $\V^{\otimes 2(l+1)} \to \V^{\otimes 2(l+1)}$ swapping the factors $2l, 2l+1$. This proves the statement in the case $k=2$.

Next, assume the statement holds for some $(k, l)$, $k,l \geq 2$, and for the pair $(2,l)$. Let us prove it for $k+1, l$. By induction assumption for $(k, l)$ and for $(2,l)$, we have:
$$(m^*\otimes \id)\circ (m^*)^{it} \circ (\dot{+})^{it}: \V^{\otimes l}\to \V^{\otimes k+1}$$ equals 
\begin{align*}
(m^*\otimes \id)\circ \left((\dot{+})^{it}\right)^{\otimes k} \circ w \circ \left((m^*)^{it}\right)^{\otimes l}  =   \left((\dot{+})^{it}\right)^{\otimes k+1} \circ (w'\otimes \id) \circ \left((m^*)^{\otimes l}\otimes \id \right) \circ w \circ\left((m^*)^{it}\right)^{\otimes l}                                                                                                \end{align*}
where $w':\V^{\otimes 2l}\to \V^{\otimes 2l}$ is defined analogously to $w$ (for $k=2$) and the rightmost $(m^*)^{it}$ stands, by abuse of notation, for the map $\V\to\V^{\otimes k+1}$. 
Due to the functoriality of $\sigma$ and thus of $w$, we may write
$$\left((m^*)^{\otimes l}\otimes \id \right) \circ w \circ\left((m^*)^{it}\right)^{\otimes l} $$ as a product of a ``permutation'' morphism $w'': \V^{\otimes (k+1)l}\to \V^{\otimes (k+1)l}$ (permuting the tensor factors) with $ \left((m^*)^{it}\right)^{\otimes l}$. It is easy to see that the composition $(w'\otimes \id) \circ w''$ is precisely the permutation $\V^{\otimes (k+1)l}\to \V^{\otimes (k+1)l}$ which we needed, and the proof is complete.
\end{proof}

\begin{lemma}\label{lem:interchanging_mu_A_and_comult}
 Let $k, l\geq 2$, and let $A=[a_1, \ldots, a_l] \in Mat_{1\times l}(\F_q)$. Then the composition $(m^*)^{it} \circ \mu_A $ of $\mu_A: \V^{\otimes l} \to \V$ and $(m^*)^{it}: \V \to \V^{\otimes k}$ equals 
 $$\mu_A^{\otimes k} \circ w \circ \left((m^*)^{it}\right)^{\otimes l} = \mu_A^{\otimes k} \circ \mu_{\begin{bmatrix} I_l \\ \vdots \\ I_l
 \end{bmatrix}}$$
 where $w: \V^{\otimes kl} \to \V^{\otimes kl}$ is as in Lemma \ref{lem:interchanging_plus_and_comult}.
 Diagrammatically, this can be drawn as
 \begin{equation*}
 \begin{tikzpicture}[anchorbase,scale=1.3]
   \draw[-] (0.4,-0.5)--(0.4,-0.2);%m^* top
     \draw[-] (-0.4,-0.5)--(-0.4,-0.2);
  \draw[-] (-0.4,-0.5)--(0.4,-0.5);
  
  \node at (-0.1,-0.35) {$\cdot$};
\node at (-0,-0.35) {$\cdot$};
\node at (0.1,-0.35) {$\cdot$};

 \draw[-] (0,-0.8)--(0,-0.5);
 
 %dot{+} bottom
\draw[-] (-0.5,-1.2)--(-0.5,-0.8)--(0.5,-0.8)--(0.5,-1.2)--(-0.5,-1.2);
\node at (0, -1) {$A$};

\draw[-] (-0.4,-1.2) -- (-0.4, -1.5);
\draw[-] (0.4,-1.2) -- (0.4, -1.5);
    
\node at (-0.1,-1.45) {$\cdot$};
\node at (-0,-1.45) {$\cdot$};
\node at (0.1,-1.45) {$\cdot$};
 \end{tikzpicture}
 \quad=\quad
  \begin{tikzpicture}[anchorbase,scale=1.3]

%leftmost part:
\draw[-] (-0.5,-0.2)--(-0.5,0.2)--(0.5,0.2)--(0.5,-0.2)--(-0.5,-0.2);
\node at (0,0) {$A$};
\draw[-] (-0.4,-0.2)--(-0.4,-0.7);
\draw[-] (0,0.2)--(0,0.4);
\draw[-] (0.4,-0.2)--(0.4,-0.5);

\node at (-0.1,-0.4) {$\cdot$};
\node at (-0,-0.4) {$\cdot$};
\node at (0.1,-0.4) {$\cdot$};%%%
\draw[-] (1.5,-0.2)--(1.5,0.2)--(2.5,0.2)--(2.5,-0.2)--(1.5,-0.2);
\node at (2,0) {$A$};
\draw[-] (1.6,-0.2)--(1.6,-0.7);
\draw[-] (2,0.2)--(2,0.4);
\draw[-] (2.4,-0.2)--(2.4,-0.5);

\node at (1.9,-0.4) {$\cdot$};
\node at (2,-0.4) {$\cdot$};
\node at (2.1,-0.4) {$\cdot$};

%%%intermediate dots:
\node at (3.9,0) {$\cdot$};
\node at (4,0) {$\cdot$};
\node at (4.1,0) {$\cdot$};

%%% rightmost part:
\draw[-] (5.5,-0.2)--(5.5,0.2)--(6.5,0.2)--(6.5,-0.2)--(5.5,-0.2);
\node at (6,0) {$A$};
\draw[-] (5.6,-0.2)--(5.6,-0.7);
\draw[-] (6,0.2)--(6,0.4);
\draw[-] (6.4,-0.2)--(6.4,-0.5);

\node at (5.9,-0.4) {$\cdot$};
\node at (6,-0.4) {$\cdot$};
\node at (6.1,-0.4) {$\cdot$};%%% Bottom part
\draw[-] (-0.4,-0.7)--(5.6,-0.7);
\draw[-] (0.4,-0.5)--(1.5, -0.5);
\draw[-] (1.7, -0.5)--(5.5, -0.5);
\draw[-] (5.7,-0.5)--(6.4,-0.5);
\draw[-] (3, -0.7)--(3,-1);
\draw[-] (3.4, -0.5)--(3.4, -0.65);
\draw[-] (3.4, -0.75)--(3.4, -1);
%bottom dots:
\node at (3.1,-0.95) {$\cdot$};
\node at (3.2,-0.95) {$\cdot$};
\node at (3.3,-0.95) {$\cdot$};

\end{tikzpicture} 
 \end{equation*}

\end{lemma}
\begin{remark}
 In terms of morphism $\mu_A$, this reads: $ \mu_{\begin{bmatrix}
 1 \\
 \vdots \\
 1 
 \end{bmatrix}} \circ \mu_{A}=\mu_{\begin{bmatrix}
 a_1  &\ldots &a_l\\
 &\ldots \\
 a_1  &\ldots &a_l
 \end{bmatrix}}$.
\end{remark}

\begin{example}
For $\mathcal{C} = \Rep(GL_n(\F_q))$, $\V:=\V_n$, and any matrix $A\in Mat_{1\times l}(\F_q)$, this Lemma states that for any $v_1, \ldots, v_l \in V$, the element
$$ v_1\otimes \ldots \otimes v_d \longmapsto \left(\dot{\sum}_{j=1}^{d} \dot{A}_{1,j} v_j \right)\otimes  \ldots  \otimes \left(\dot{\sum}_{j=1}^{d} \dot{A}_{1,j} v_j \right) \in \V_n^{\otimes k}$$
is obtained by applying $$\mu_A^{\otimes k}$$ (see Example \ref{ex:matrix_mult_classical}) to the factors of the tensor
$$ (v_1\otimes v_2 \otimes\ldots \otimes  v_l)\ldots \otimes (v_1 \otimes v_2\otimes  \ldots \otimes v_l) \in \V_n^{\otimes kl}$$
which in turn is obtained by permutation of factors from the tensor $$ v_1\otimes v_1 \otimes \ldots \otimes v_1 \otimes  \ldots \otimes v_l\otimes v_l\otimes \ldots \otimes v_l. $$ The latter is just $\left((m^*)^{it}\right)^{\otimes l}(v_1\otimes\ldots \otimes v_l)$.
\end{example}

\begin{proof}[Proof of Lemma \ref{lem:interchanging_mu_A_and_comult}]
 Let us use Lemma \ref{lem:interchanging_plus_and_comult}:
 \begin{align*}
  (m^*)^{it} \circ \mu_A = (m^*)^{it} \circ (\dot{+})^{it}  \circ \bigotimes_{i=1}^l \mu_{a_{i}} = \left((\dot{+})^{it}\right)^{\otimes k} \circ w \circ \left((m^*)^{it}\right)^{\otimes l}  \circ \bigotimes_{i=1}^l \mu_{a_{i}}
 \end{align*}

 By Relation \ref{rel:mu_coalg_mor}, stating that $\mu_a$ is a morphism of coalgebras, we have: $$\left((m^*)^{it}\right)^{\otimes l}  \circ \bigotimes_{i=1}^l \mu_{a_{i}} = \left(\bigotimes_{i=1}^l \mu_{a_{i}}\right)^{\otimes k} \circ \left((m^*)^{it}\right)^{\otimes l}. $$
 Furthermore, since the symmetry morphism $\sigma:\V\otimes \V \to \V\otimes \V$ is functorial, so is $w$, hence $$w\circ \left(\bigotimes_{i=1}^l \mu_{a_{i}}\right)^{\otimes k} = \left(\bigotimes_{i=1}^l \mu_{a_{i}}\right)^{\otimes k} \circ w$$
 and we conclude that 
  \begin{align*}
  (m^*)^{it} \circ \mu_A &= \left((\dot{+})^{it}\right)^{\otimes k} \circ w \circ \left((m^*)^{it}\right)^{\otimes l}  \circ \bigotimes_{i=1}^l \mu_{a_{i}} = \left((\dot{+})^{it}\right)^{\otimes k}\circ \left(\bigotimes_{i=1}^l \mu_{a_{i}}\right)^{\otimes k}  \circ w \circ \left((m^*)^{it}\right)^{\otimes l}   \\ &=
  \mu_A^{\otimes k} \circ w \circ \left((m^*)^{it}\right)^{\otimes l} 
 \end{align*}
 as required.
\end{proof}
From Lemma \ref{lem:interchanging_mu_A_and_comult} we conclude:
\begin{lemma}\label{lem:mu_compos_aux}
For any $ A\in Mat_{k\times l}(\F_q)$, let $\begin{bmatrix}
I_k \\ \vdots \\ I_k
\end{bmatrix} \in Mat_{rk\times k}(\F_q)$. Then we have: $$\mu_{\begin{bmatrix}
I_k \\ \vdots \\ I_k
\end{bmatrix}}\circ \mu_A = \mu_{\begin{bmatrix}
A \\ \vdots \\ A
\end{bmatrix}} = (\mu_A)^{\otimes r} \circ \mu_{ \begin{bmatrix}
I_l \\ \vdots \\ I_l
\end{bmatrix}} $$ where the rightmost matrix has $rl$ rows.
\end{lemma}

\begin{proof}
The second equality follows from Corollary \ref{cor:vert_stacking}. 

For the first equality, we argue as follows. Let $A_1, \ldots, A_k$ denote the rows of $A$. By Lemma \ref{lem:mu_identity} and  Corollary \ref{cor:vert_stacking}, we have:
$$\mu_{\begin{bmatrix}
I_k \\ \vdots \\ I_k
\end{bmatrix}}\circ \mu_A =  w_1 \circ \left((m^*)^{it_1}\right)^{\otimes k} \circ \left(\bigotimes_{i=1}^k \mu_{A_i} \right) \circ w_2 \circ \left((m^*)^{it_2}\right)^{\otimes l}$$
 where $(m^*)^{it_1}: \V \to \V^{\otimes r}$, $(m^*)^{it_2}: \V \to \V^{\otimes k}$ and $w_1: \V^{\otimes rk} \to \V^{\otimes rk}$, $w_2: \V^{\otimes kl} \to \V^{\otimes kl}$ as in Lemma \ref{lem:mu_identity}.
 Now, Lemma \ref{lem:interchanging_mu_A_and_comult} applied to the right hand side gives the desired result.
\end{proof}

We are now ready to show that the operators $\mu_A$ behave as ``multiplication by matrix $A$'':
\begin{proposition}\label{prop:compos_mu_A}
For any $B \in Mat_{r \times k}(\F_q), A\in Mat_{k\times l}(\F_q)$, we have:
$$\mu_B \circ \mu_A = 
\mu_{BA}. $$
\end{proposition}
\begin{proof}
 First, assume that $r=k=1$ (so $B$ is just a scalar, $b$, and $A$ is a row matrix). By the Relations \ref{rel:F_q_lin} from Definition \ref{def:Frob_linear_space} describing the ``$\F_q$-module structure on $\V$'', we have:
 $$\mu_b \circ \mu_A = \mu_b \circ (\dot{+})^{it} \circ (\mu_{a_{11}} \otimes \ldots \otimes \mu_{a_{1l}}) =  (\dot{+})^{it} \circ ((\mu_b \circ \mu_{a_{11}}) \otimes \ldots \otimes(\mu_b \circ \mu_{a_{1l}}))  = (\dot{+})^{it} \circ (\mu_{ba_{11}} \otimes \ldots \otimes \mu_{ba_{1l}}) $$ which clearly equals $\mu_{bA}$.
 
 Next, let $k, l$ be any, and consider the composition of 
 $(\dot{+})^{it}: \V^{\otimes k} \to \V$ with $\mu_A$. Since the composition $(\dot{+})^{it} \circ ((\dot{+})^{it} \otimes \ldots \otimes (\dot{+})^{it}): \V^{\otimes kl} \to \V$ is just $(\dot{+})^{it}: \V^{\otimes kl} \to \V$, we conclude that 
 $(\dot{+})^{it} \circ \mu_A = \mu_{[1, \ldots, 1]A}$. 
 
 Now, assume that $r=1$ and $k, l$ are arbitrary. Then $B=[b_1,\ldots, b_k]$ is a row matrix, and $$\mu_B \circ \mu_A = (\dot{+})^{it} \circ \bigotimes_{i=1}^k \mu_{b_{i}} \circ \bigotimes_{i=1}^k \mu_{A_i} \circ (m^*)^{it}  = (\dot{+})^{it} \circ \bigotimes_{i=1}^k \mu_{b_{i}A_i} \circ (m^*)^{it} $$
where $A_1, \ldots, A_k$ are the rows of $A$:
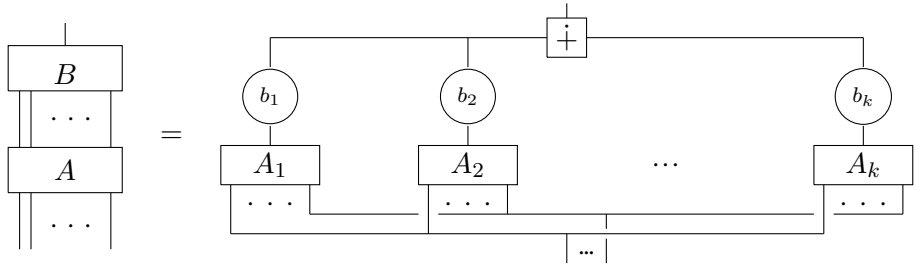
\begin{equation*}\begin{tikzpicture}[anchorbase,scale=1.5]
\draw[-] (-0.5,-0.1)--(-0.5,0.3)--(0.5,0.3)--(0.5,-0.1)--(-0.5,-0.1);
\node at (0,0.05) {$B$};
\draw[-] (0,0.3)--(0,0.5);

\draw[-] (-0.4,-0.1)--(-0.4,-0.6);
\draw[-] (0.4,-0.1)--(0.4,-0.6);
\draw[-] (-0.3,-0.1)--(-0.3,-0.6);

\draw[-] (-0.5,-1)--(-0.5,-0.6)--(0.5,-0.6)--(0.5,-1)--(-0.5,-1);
\node at (0,-0.8) {$A$};

\draw[-] (-0.4,-1)--(-0.4,-1.5);
\draw[-] (0.4,-1)--(0.4,-1.5);
\draw[-] (-0.3,-1)--(-0.3,-1.5);

\node at (-0.1,-0.35) {$\cdot$};
\node at (0.05,-0.35) {$\cdot$};
\node at (0.2,-0.35) {$\cdot$};

\node at (-0.1,-1.3) {$\cdot$};
\node at (0.05,-1.3) {$\cdot$};
\node at (0.2,-1.3) {$\cdot$}; 
\end{tikzpicture} \quad
= \quad
\begin{tikzpicture}[anchorbase,scale=1.3]
%upper sum
\draw[-] (2.8,1.1)--(2.8,1.5)--(3.2,1.5)--(3.2,1.1)--(2.8,1.1);
\node at (3,1.3) {$\dot{+}$};
\draw[-] (3,1.5)--(3,1.65);
\draw[-] (0,1.3)--(2.8, 1.3);
\draw[-] (3.2, 1.3)--(6,1.3);
%scalars
\node[draw,circle] at (0,0.7) {$\scriptstyle b_1$};
\node[draw,circle] at (2,0.7) {$\scriptstyle  b_{2}$};
\node[draw,circle] at (6,0.7) {$\scriptstyle  b_{k}$};
\draw[-] (0,1)--(0,1.3);
\draw[-] (2,1)--(2,1.3);
\draw[-] (6,1)--(6,1.3);
%leftmost part:
\draw[-] (-0.5,-0.2)--(-0.5,0.2)--(0.5,0.2)--(0.5,-0.2)--(-0.5,-0.2);
\node at (0,0) {${A_1}$};
\draw[-] (-0.4,-0.2)--(-0.4,-0.7);
\draw[-] (0,0.2)--(0,0.4);
\draw[-] (0.4,-0.2)--(0.4,-0.5);
\node at (-0.2,-0.4) {$\cdot$};
\node at (-0,-0.4) {$\cdot$};
\node at (0.2,-0.4) {$\cdot$};%%%
\draw[-] (1.5,-0.2)--(1.5,0.2)--(2.5,0.2)--(2.5,-0.2)--(1.5,-0.2);
\node at (2,0) {${A_2}$};
\draw[-] (1.6,-0.2)--(1.6,-0.7);
\draw[-] (2,0.2)--(2,0.4);
\draw[-] (2.4,-0.2)--(2.4,-0.5);
\node at (1.8,-0.4) {$\cdot$};
\node at (2,-0.4) {$\cdot$};
\node at (2.2,-0.4) {$\cdot$};%%%intermediate dots:
\node at (3.9,0) {$\cdot$};
\node at (4,0) {$\cdot$};
\node at (4.1,0) {$\cdot$};%%% rightmost part:
\draw[-] (5.5,-0.2)--(5.5,0.2)--(6.5,0.2)--(6.5,-0.2)--(5.5,-0.2);
\node at (6,0) {${A_k}$};
\draw[-] (5.6,-0.2)--(5.6,-0.7);
\draw[-] (6,0.2)--(6,0.4);
\draw[-] (6.4,-0.2)--(6.4,-0.5);
\node at (5.8,-0.4) {$\cdot$};
\node at (6,-0.4) {$\cdot$};
\node at (6.2,-0.4) {$\cdot$};%%% Bottom part
\draw[-] (-0.4,-0.7)--(5.6,-0.7);
\draw[-] (0.4,-0.5)--(1.5, -0.5);
\draw[-] (1.7, -0.5)--(5.5, -0.5);
\draw[-] (5.7,-0.5)--(6.4,-0.5);
\draw[-] (3, -0.7)--(3,-1);
\draw[-] (3.4, -0.5)--(3.4, -0.65);
\draw[-] (3.4, -0.75)--(3.4, -1);
\node at (3.15,-0.9) {$\cdot$};
\node at (3.2,-0.9) {$\cdot$};
\node at (3.25,-0.9) {$\cdot$};
\end{tikzpicture}
\end{equation*}
The computation above for $\mu_{b_i}\circ \mu_A$ then implies that $\mu_B\circ \mu_A = \mu_{BA}$ in this case.

Now we tackle the most general case.

We write $B_1, \ldots, B_r$ for the rows of $B$. Then by Corollary \ref{cor:vert_stacking} and Lemma \ref{lem:mu_compos_aux} we have:
\begin{align*}
    \mu_B\circ \mu_A &= \left( \bigotimes_{i=1}^r \mu_{B_i} \right) \circ \mu_{\begin{bmatrix} I_k \\ \vdots  \\I_k
    \end{bmatrix}} \circ \mu_A = \left( \bigotimes_{i=1}^r \mu_{B_i} \right)  \circ (\mu_A)^{\otimes r} \circ \mu_{\begin{bmatrix} I_l \\ \vdots  \\I_l
    \end{bmatrix}} = \left( \bigotimes_{i=1}^r \mu_{B_i \, A} \right) \circ \mu_{\begin{bmatrix} I_l \\ \vdots  \\I_l
    \end{bmatrix}}
\end{align*}

Here the last equality follows from the case $r=1$. The right hand side is now $\mu_{BA}$ by Corollary \ref{cor:vert_stacking}. 
Diagrammatically, this means:
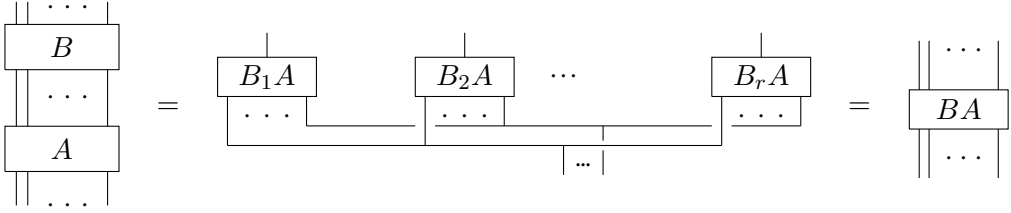
\begin{equation*}\begin{tikzpicture}[anchorbase,scale=1.5]
\draw[-] (-0.5,-0.1)--(-0.5,0.3)--(0.5,0.3)--(0.5,-0.1)--(-0.5,-0.1);
\node at (0,0.1) {$B$};

\draw[-] (-0.4,0.3)--(-0.4,0.5);
\draw[-] (-0.3,0.3)--(-0.3,0.5);
\draw[-] (0.4,0.3)--(0.4,0.5);
\node at (-0.1,0.45) {$\cdot$};
\node at (0.05,0.45) {$\cdot$};
\node at (0.2,0.45) {$\cdot$};

\draw[-] (-0.4,-0.1)--(-0.4,-0.6);
\draw[-] (0.4,-0.1)--(0.4,-0.6);
\draw[-] (-0.3,-0.1)--(-0.3,-0.6);

\draw[-] (-0.5,-1)--(-0.5,-0.6)--(0.5,-0.6)--(0.5,-1)--(-0.5,-1);
\node at (0,-0.8) {$A$};

\draw[-] (-0.4,-1)--(-0.4,-1.3);
\draw[-] (0.4,-1)--(0.4,-1.3);
\draw[-] (-0.3,-1)--(-0.3,-1.3);

\node at (-0.1,-0.35) {$\cdot$};
\node at (0.05,-0.35) {$\cdot$};
\node at (0.2,-0.35) {$\cdot$};

\node at (-0.1,-1.3) {$\cdot$};
\node at (0.05,-1.3) {$\cdot$};
\node at (0.2,-1.3) {$\cdot$}; 
\end{tikzpicture} \quad
= \quad
\begin{tikzpicture}[anchorbase,scale=1.3]
%leftmost part:
\draw[-] (-0.5,-0.2)--(-0.5,0.2)--(0.5,0.2)--(0.5,-0.2)--(-0.5,-0.2);
\node at (0,0) {${B_1 A}$};
\draw[-] (-0.4,-0.2)--(-0.4,-0.7);
\draw[-] (0,0.2)--(0,0.45);
\draw[-] (0.4,-0.2)--(0.4,-0.5);
\node at (-0.2,-0.4) {$\cdot$};
\node at (-0,-0.4) {$\cdot$};
\node at (0.2,-0.4) {$\cdot$};%%%
\draw[-] (1.5,-0.2)--(1.5,0.2)--(2.5,0.2)--(2.5,-0.2)--(1.5,-0.2);
\node at (2,0) {${B_2 A}$};
\draw[-] (1.6,-0.2)--(1.6,-0.7);
\draw[-] (2,0.2)--(2,0.45);
\draw[-] (2.4,-0.2)--(2.4,-0.5);
\node at (1.8,-0.4) {$\cdot$};
\node at (2,-0.4) {$\cdot$};
\node at (2.2,-0.4) {$\cdot$};%%%intermediate dots:
\node at (2.9,0) {$\cdot$};
\node at (3,0) {$\cdot$};
\node at (3.1,0) {$\cdot$};%%% rightmost part:
\draw[-] (4.5,-0.2)--(4.5,0.2)--(5.5,0.2)--(5.5,-0.2)--(4.5,-0.2);
\node at (5,0) {${B_r A}$};
\draw[-] (4.6,-0.2)--(4.6,-0.7);
\draw[-] (5,0.2)--(5,0.45);
\draw[-] (5.4,-0.2)--(5.4,-0.5);
\node at (4.8,-0.4) {$\cdot$};
\node at (5,-0.4) {$\cdot$};
\node at (5.2,-0.4) {$\cdot$};%%% Bottom part
\draw[-] (-0.4,-0.7)--(4.6,-0.7);
\draw[-] (0.4,-0.5)--(1.5, -0.5);
\draw[-] (1.7, -0.5)--(4.5, -0.5);
\draw[-] (4.7,-0.5)--(5.4,-0.5);
\draw[-] (3, -0.7)--(3,-1);
\draw[-] (3.4, -0.5)--(3.4, -0.65);
\draw[-] (3.4, -0.75)--(3.4, -1);
\node at (3.15,-0.9) {$\cdot$};
\node at (3.2,-0.9) {$\cdot$};
\node at (3.25,-0.9) {$\cdot$};
\end{tikzpicture} \quad
= \quad
\begin{tikzpicture}[anchorbase,scale=1.3]
\draw[-] (-0.5,-0.2)--(-0.5,0.2)--(0.5,0.2)--(0.5,-0.2)--(-0.5,-0.2);
\node at (-0,-0) {${BA}$};
\draw[-] (-0.4,-0.2)--(-0.4,-0.7);
\draw[-] (0.4,-0.2)--(0.4,-0.7);
\draw[-] (-0.4,0.2)--(-0.4,0.7);
\draw[-] (-0.3,0.2)--(-0.3,0.7);
\draw[-] (0.4,0.2)--(0.4,0.7);
\draw[-] (-0.3,-0.2)--(-0.3,-0.7);
\node at (-0.1,-0.5) {$\cdot$};
\node at (0.05,-0.5) {$\cdot$};
\node at (0.2,-0.5) {$\cdot$};
\node at (-0.1,0.6) {$\cdot$};
\node at (0.05,0.6) {$\cdot$};
\node at (0.2,0.6) {$\cdot$};
 \end{tikzpicture}
\end{equation*}
 This concludes the proof of the proposition.
\end{proof}

Finally, the following two results will be useful for Section \ref{sec:repinf}:
\begin{corollary}\label{cor:horizontal stacking}
Let $A\in Mat_{d\times r_1}(\F_q), B\in Mat_{d\times r_2}(\F_q)$. Then we have:
$\mu_{ \begin{bmatrix} A &B\end{bmatrix}} = \mu_{\begin{bmatrix} I_{d} &I_{d} \end{bmatrix}} \circ (\mu_A \otimes \mu_B)$.
\end{corollary}

\begin{proof}
By Proposition \ref{prop:compos_mu_A}, 
$$ \begin{bmatrix} A &B\end{bmatrix} = \begin{bmatrix} I_d &I_d\end{bmatrix}\begin{bmatrix} A &0\\ 0 &B\end{bmatrix} \; \Longrightarrow 
\; \mu_{ \begin{bmatrix} A &B\end{bmatrix}} = \mu_{\begin{bmatrix} I_{d} &I_{d} \end{bmatrix}} \circ \mu_{\begin{bmatrix} A &0\\ 0 &B\end{bmatrix}}. $$
Proposition \ref{prop:tensor_prod_mu_B} now implies the required statement. 
\end{proof}

\subsection{Morphisms \texorpdfstring{$\mu_A$}{muA} for invertible matrices}\label{ssec:transp_mu_A}
The main goal of this section is to show that for $A\in GL_d(\F_q)$, $(z^*)^{\otimes d} \mu_A = (z^*)^{\otimes d}$ (this is proved in Corollary \ref{cor:regular_system_of_eq}); this statement is the analogue of the linear algebra statement ``The only solution to the system of linear equations $Ax=0$ for $A\in GL_d(\F_q)$ is the solution $x=0$''. 

Along the way, we prove several related statements (consequences of the Cancellation Axiom \ref{rel:cancellation_axiom}) which will be very useful to us later on.
\begin{lemma}\label{lem:eq_axiom_corollary}
 We have the following equality of morphisms $\V^{\otimes 3}\to \triv$: 
 \begin{equation*}
\begin{tikzpicture}[anchorbase,scale=1.1]
\draw[-] (0,0.5)--(2,0.5);%ev
\draw[-] (-0.5,0)--(-0.2,0);%\dot{+}
\draw[-] (0.2,0)--(0.5,0);
\draw[-] (-0.2,-0.2)--(-0.2,0.2)--(0.2,0.2)--(0.2,-0.2)--(-0.2,-0.2);
\node at (-0,-0) {$\dot{+}$};
\draw[-] (0,0.2)--(0,0.5);
\draw[-] (-0.5,0)--(-0.5,-1);
\draw[-] (-0.5,0)--(-0.5,-1);

\draw[-] (0.5,0)--(0.5,-0.5);%m^*
\draw[-] (0.5, -0.5)--(1.5,-0.5);
\draw[-] (1, -0.5)--(1,-1);

\draw[-] (1.5,0)--(1.8,0);%\dot{+}
\draw[-] (2.2,0)--(2.5,0);
\draw[-] (1.8,-0.2)--(1.8,0.2)--(2.2,0.2)--(2.2,-0.2)--(1.8,-0.2);
\node at (2,0) {$\dot{+}$};
\draw[-] (2,0.2)--(2,0.5);
\draw[-] (1.5,0)--(1.5,-0.5);
\draw[-] (2.5,0)--(2.5,-1);
\end{tikzpicture} \quad = \quad \begin{tikzpicture}[anchorbase,scale=1.1]
\draw[-] (0.5,-0.5)--(0.5,-1);%middle part
\node at (0.5, -0.5) {$\bullet$};

\draw[-] (-0.5,0.5)--(1.5,0.5);
\draw[-] (-0.5,0.5)--(-0.5,-1);
\draw[-] (1.5,0.5)--(1.5,-1);

\end{tikzpicture}
\end{equation*} 
\end{lemma}
\begin{example}
For $\mathcal{C} = \Rep(GL_n(\F_q))$, $\V:=\V_n$, this statement translates to
$$\forall v,u,w\in V, \;\; \delta_{v\dot{+} u, w \dot{+} u} = \delta_{v,w}.$$
\end{example}
\begin{proof}
Let us compute the left hand side of the above equation, using the definition $ev := \eps^*\circ m$:
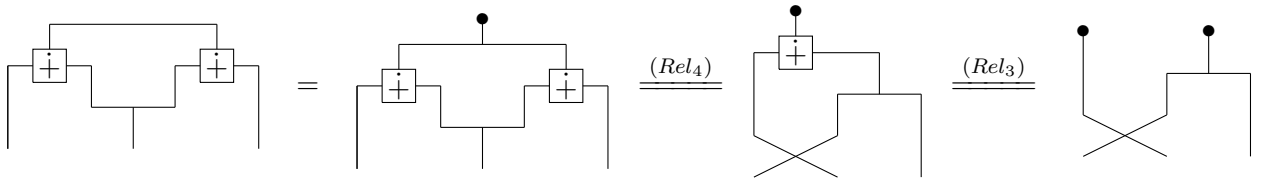
\begin{equation*}
\begin{tikzpicture}[anchorbase,scale=1.1]
\draw[-] (0,0.5)--(2,0.5);%ev
\draw[-] (-0.5,0)--(-0.2,0);%\dot{+}
\draw[-] (0.2,0)--(0.5,0);
\draw[-] (-0.2,-0.2)--(-0.2,0.2)--(0.2,0.2)--(0.2,-0.2)--(-0.2,-0.2);
\node at (-0,-0) {$\dot{+}$};
\draw[-] (0,0.2)--(0,0.5);
\draw[-] (-0.5,0)--(-0.5,-1);
\draw[-] (-0.5,0)--(-0.5,-1);

\draw[-] (0.5,0)--(0.5,-0.5);%m^*
\draw[-] (0.5, -0.5)--(1.5,-0.5);
\draw[-] (1, -0.5)--(1,-1);

\draw[-] (1.5,0)--(1.8,0);%\dot{+}
\draw[-] (2.2,0)--(2.5,0);
\draw[-] (1.8,-0.2)--(1.8,0.2)--(2.2,0.2)--(2.2,-0.2)--(1.8,-0.2);
\node at (2,0) {$\dot{+}$};
\draw[-] (2,0.2)--(2,0.5);
\draw[-] (1.5,0)--(1.5,-0.5);
\draw[-] (2.5,0)--(2.5,-1);
\end{tikzpicture} \quad = \quad \begin{tikzpicture}[anchorbase,scale=1.1]
\draw[-] (0,0.5)--(2,0.5);%ev= \esp^* \circ m
\draw[-] (1,0.5)--(1,0.8);
\node at (1, 0.8) {$\bullet$};

\draw[-] (-0.5,0)--(-0.2,0);%\dot{+}
\draw[-] (0.2,0)--(0.5,0);
\draw[-] (-0.2,-0.2)--(-0.2,0.2)--(0.2,0.2)--(0.2,-0.2)--(-0.2,-0.2);
\node at (-0,-0) {$\dot{+}$};
\draw[-] (0,0.2)--(0,0.5);
\draw[-] (-0.5,0)--(-0.5,-1);
\draw[-] (-0.5,0)--(-0.5,-1);

\draw[-] (0.5,0)--(0.5,-0.5);%m^*
\draw[-] (0.5, -0.5)--(1.5,-0.5);
\draw[-] (1, -0.5)--(1,-1);

\draw[-] (1.5,0)--(1.8,0);%\dot{+}
\draw[-] (2.2,0)--(2.5,0);
\draw[-] (1.8,-0.2)--(1.8,0.2)--(2.2,0.2)--(2.2,-0.2)--(1.8,-0.2);
\node at (2,0) {$\dot{+}$};
\draw[-] (2,0.2)--(2,0.5);
\draw[-] (1.5,0)--(1.5,-0.5);
\draw[-] (2.5,0)--(2.5,-1);
\end{tikzpicture}  \quad \InnaA{\xlongequal{\text{\ref{rel:cancellation_axiom}}}}  \quad \begin{tikzpicture}[anchorbase,scale=1.1]
\draw[-] (-0.5,0)--(-0.2,0);%\dot{+}
\draw[-] (0.2,0)--(1,0);
\draw[-] (-0.2,-0.2)--(-0.2,0.2)--(0.2,0.2)--(0.2,-0.2)--(-0.2,-0.2);
\node at (-0,-0) {$\dot{+}$};
\draw[-] (0,0.2)--(0,0.5);%\eps^*
\node at (0, 0.5) {$\bullet$};
\draw[-] (-0.5,0)--(-0.5,-1);%\id \otimes m
\draw[-] (1,0)--(1,-0.5);
\draw[-] (0.5,-0.5)--(1.5,-0.5);
\draw[-] (0.5,-0.5)--(0.5,-1);
\draw[-] (1.5,-0.5)--(1.5,-1.5);% \sigma\otimes \id
\draw[-] (-0.5,-1)--(0.5,-1.5);
\draw[-] (0.5,-1)--(-0.5,-1.5);
\end{tikzpicture}
\quad \InnaA{\xlongequal{\text{\ref{rel:plus_coalg_mor}}}} \quad \begin{tikzpicture}[anchorbase,scale=1.1]
\node at (-0.5, 0) {$\bullet$};
\node at (1, 0) {$\bullet$};
\draw[-] (-0.5,0)--(-0.5,-1);%\id \otimes m
\draw[-] (1,0)--(1,-0.5);
\draw[-] (0.5,-0.5)--(1.5,-0.5);
\draw[-] (0.5,-0.5)--(0.5,-1);
\draw[-] (1.5,-0.5)--(1.5,-1.5);% \sigma\otimes \id
\draw[-] (-0.5,-1)--(0.5,-1.5);
\draw[-] (0.5,-1)--(-0.5,-1.5);
\end{tikzpicture}
\end{equation*} 
Here the second equality uses the Cancellation Axiom \ref{rel:cancellation_axiom}, \ref{itm:str_rel_diag_cancel_axiom},
and the third equality uses Relation \ref{rel:plus_coalg_mor}, \ref{itm:str_rel_diag_plus_coalg} stating that $\eps^*\circ \dot{+}=\eps^*\otimes \eps^*$.

By the definition of $ev = \eps^*\circ m$, we have:
 
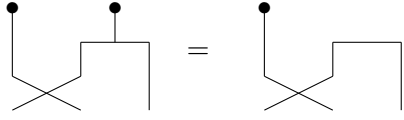
\begin{equation*}
 \begin{tikzpicture}[anchorbase,scale=0.9]
\node at (-0.5, 0) {$\bullet$};
\node at (1, 0) {$\bullet$};
\draw[-] (-0.5,0)--(-0.5,-1);%\id \otimes m
\draw[-] (1,0)--(1,-0.5);
\draw[-] (0.5,-0.5)--(1.5,-0.5);
\draw[-] (0.5,-0.5)--(0.5,-1);
\draw[-] (1.5,-0.5)--(1.5,-1.5);% \sigma\otimes \id
\draw[-] (-0.5,-1)--(0.5,-1.5);
\draw[-] (0.5,-1)--(-0.5,-1.5);
\end{tikzpicture}
 \quad = \quad \begin{tikzpicture}[anchorbase,scale=0.9]
\node at (-0.5, 0) {$\bullet$};
\draw[-] (-0.5,0)--(-0.5,-1);%\id \otimes m
\draw[-] (0.5,-0.5)--(1.5,-0.5);
\draw[-] (0.5,-0.5)--(0.5,-1);
\draw[-] (1.5,-0.5)--(1.5,-1.5);% \sigma\otimes \id
\draw[-] (-0.5,-1)--(0.5,-1.5);
\draw[-] (0.5,-1)--(-0.5,-1.5);
\end{tikzpicture}
\end{equation*} 

Using the fact that the symmetry morphism is functorial, we conclude that the right hand side is just $ ev\circ (\id \otimes \eps^* \otimes\id) $ and thus the statement is proved.
\end{proof}
\begin{lemma}\label{lem:transferring_plus}
 We have the following equality of morphisms $\V^{\otimes 3}\to \triv$: $$ev \circ (\dot{+} \otimes \id) = ev \circ (\id \otimes \dot{+}) \circ (\id \otimes \mu_{-1} \otimes \id).$$
 Diagrammatically, this is drawn as 
 \begin{equation*}
\begin{tikzpicture}[anchorbase,scale=1.3]
\draw[-] (-0.5,0)--(-0.2,0);%\dot{+}
\draw[-] (0.2,0)--(0.5,0);
\draw[-] (-0.2,-0.2)--(-0.2,0.2)--(0.2,0.2)--(0.2,-0.2)--(-0.2,-0.2);
\node at (-0,-0) {$\dot{+}$};
\draw[-] (0,0.2)--(0,0.5);

\draw[-] (-0.5,0)--(-0.5,-1);
\draw[-] (0.5,0)--(0.5,-1);

\draw[-] (0,0.5)--(1,0.5);
\draw[-] (1,0.5)--(1,-1);
\end{tikzpicture} \quad  = \quad  \begin{tikzpicture}[anchorbase,scale=1.3]
\draw[-] (-0.5,0)--(-0.2,0);%\dot{+}
\draw[-] (0.2,0)--(0.5,0);
\draw[-] (-0.2,-0.2)--(-0.2,0.2)--(0.2,0.2)--(0.2,-0.2)--(-0.2,-0.2);
\node at (-0,-0) {$\dot{+}$};
\draw[-] (0,0.2)--(0,0.5);
\draw[-] (-0.5,0)--(-0.5,-0.2);
\draw[-] (-0.5,-0.8)--(-0.5,-1);
\node[draw,circle] at (-0.5,-0.5) {${\scriptstyle -1}$};

\draw[-] (0.5,0)--(0.5,-1);

\draw[-] (0,0.5)--(-1,0.5);
\draw[-] (-1,0.5)--(-1,-1);
\end{tikzpicture} 
\end{equation*}  
\end{lemma}
\begin{example}
For $\mathcal{C} = \Rep(GL_n(\F_q))$, $\V:=\V_n$, this statement translates to
$$\forall v,u,w\in V, \;\; \delta_{v\dot{+} u, w } = \delta_{v,w\dot{-} u}.$$
\end{example}

\begin{proof}[Proof of Lemma \ref{lem:eq_axiom_corollary}]
We will give the proof in terms of diagrams:

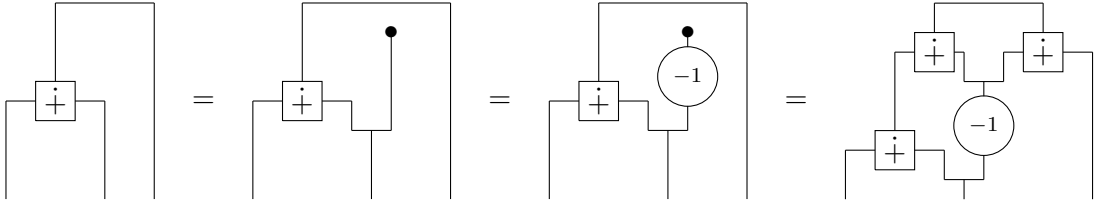
\begin{equation*}
\begin{tikzpicture}[anchorbase,scale=1.3]
\draw[-] (-0.5,0)--(-0.2,0);%\dot{+}
\draw[-] (0.2,0)--(0.5,0);
\draw[-] (-0.2,-0.2)--(-0.2,0.2)--(0.2,0.2)--(0.2,-0.2)--(-0.2,-0.2);
\node at (-0,-0) {$\dot{+}$};

\draw[-] (0,0.2)--(0,1);

\draw[-] (-0.5,0)--(-0.5,-1);
\draw[-] (0.5,0)--(0.5,-1);

\draw[-] (0,1)--(1,1);
\draw[-] (1,1)--(1,-1);
\end{tikzpicture} \quad  = \quad
\begin{tikzpicture}[anchorbase,scale=1.3]
\draw[-] (-0.5,0)--(-0.2,0);%\dot{+}
\draw[-] (0.2,0)--(0.5,0);
\draw[-] (-0.2,-0.2)--(-0.2,0.2)--(0.2,0.2)--(0.2,-0.2)--(-0.2,-0.2);
\node at (-0,-0) {$\dot{+}$};

\draw[-] (0,0.2)--(0,1);

\draw[-] (-0.5,0)--(-0.5,-1);
\draw[-] (0.5,0)--(0.5,-0.3);
%m^*
\draw[-] (0.5,-0.3)--(0.9,-0.3);
\draw[-] (0.9,0.7)--(0.9,-0.3);
\node at (0.9,0.7) {$\bullet$};

\draw[-] (0.7,-0.3)--(0.7,-1);

\draw[-] (0,1)--(1.5,1);
\draw[-] (1.5,1)--(1.5,-1);
\end{tikzpicture} \quad  =  \quad
\begin{tikzpicture}[anchorbase,scale=1.3]
\draw[-] (-0.5,0)--(-0.2,0);%\dot{+}
\draw[-] (0.2,0)--(0.5,0);
\draw[-] (-0.2,-0.2)--(-0.2,0.2)--(0.2,0.2)--(0.2,-0.2)--(-0.2,-0.2);
\node at (-0,-0) {$\dot{+}$};

\draw[-] (0,0.2)--(0,1);

\draw[-] (-0.5,0)--(-0.5,-1);
\draw[-] (0.5,0)--(0.5,-0.3);
%m^*
\draw[-] (0.5,-0.3)--(0.9,-0.3);
\node at (0.9,0.7) {$\bullet$};
\draw[-] (0.7,-0.3)--(0.7,-1);

%-1
\draw[-] (0.9,0.7)--(0.9,0.55);
\draw[-] (0.9,-0.05)--(0.9,-0.3);
\node[draw,circle] at (0.9,0.25) {${\scriptstyle -1}$};

\draw[-] (0,1)--(1.5,1);
\draw[-] (1.5,1)--(1.5,-1);
\end{tikzpicture} \quad  = \quad
\begin{tikzpicture}[anchorbase,scale=1.3]
\draw[-] (-0.5,0)--(-0.2,0);%\dot{+} lower
\draw[-] (0.2,0)--(0.5,0);
\draw[-] (-0.2,-0.2)--(-0.2,0.2)--(0.2,0.2)--(0.2,-0.2)--(-0.2,-0.2);
\node at (-0,-0) {$\dot{+}$};

\draw[-] (0,0.2)--(0,1);

\draw[-] (-0.5,0)--(-0.5,-0.5);
\draw[-] (0.5,0)--(0.5,-0.3);
%m^* lower
\draw[-] (0.5,-0.3)--(0.9,-0.3);
\draw[-] (0.7,-0.3)--(0.7,-0.5);
%m^* upper
\draw[-] (0.7,0.7) -- (1.1, 0.7);
\draw[-] (0.7,0.7) -- (0.7, 1);
\draw[-] (1.1, 0.7)--(1.1, 1);
%-1
\draw[-] (0.9,0.7)--(0.9,0.55);
\draw[-] (0.9,-0.05)--(0.9,-0.3);
\node[draw,circle] at (0.9,0.25) {${\scriptstyle -1}$};

%two upper \dot{+}
%right
\draw[-] (2,1)--(1.7, 1);
\draw[-] (1.1,1)--(1.3, 1);
\draw[-] (1.3,0.8)--(1.3,1.2)--(1.7,1.2)--(1.7,0.8)--(1.3,0.8);
\node at (1.5,1) {$\dot{+}$};
\draw[-] (2,1)--(2,-0.5);
%left
\draw[-] (0.6,1)--(0.7, 1);
\draw[-] (0,1)--(0.2, 1);
\draw[-] (0.2,0.8)--(0.2,1.2)--(0.6,1.2)--(0.6,0.8)--(0.2,0.8);
\node at (0.4,1) {$\dot{+}$};

\draw[-] (0.4,1.5)--(1.5,1.5);
\draw[-] (0.4,1.2)--(0.4,1.5);
\draw[-] (1.5,1.2)--(1.5,1.5);
\end{tikzpicture}
\end{equation*}  
Here the first equality follows from Frobenius relations \ref{rel:Frob_alg2}, \ref{itm:str_rel_diag_Frob}, the second follows from the equality $\eps^* \circ \mu_{-1}=\eps^*$ (Relation \ref{rel:mu_coalg_mor}, \ref{itm:str_rel_diag_mu_coalg}), and the third follows from Lemma \ref{lem:eq_axiom_corollary}.

\InnaA{We now apply} Relation \ref{rel:mu_coalg_mor}, \ref{itm:str_rel_diag_mu_coalg} (implying that $m^*\circ \mu_{-1} = \mu_{-1}^{\otimes 2}\circ m^*$) \InnaA{and Relation \ref{itm:str_rel_Lin} (stating that $\dot{+}$ is associative)} to the right hand side of this equation:
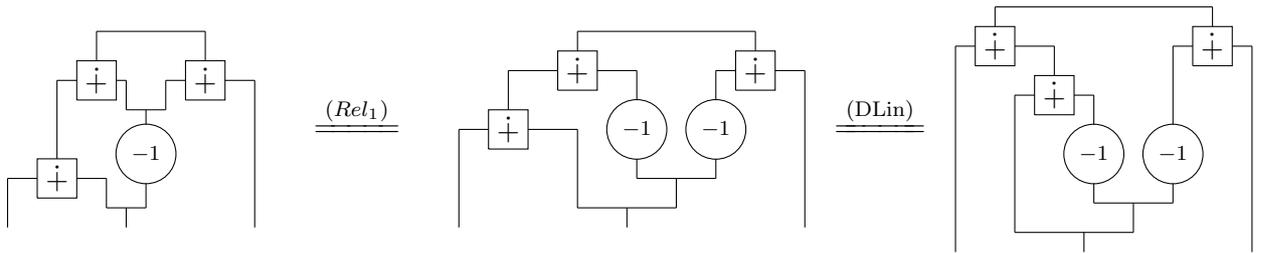
\begin{equation*}
\begin{tikzpicture}[anchorbase,scale=1.3]
\draw[-] (-0.5,0)--(-0.2,0);%\dot{+} lower
\draw[-] (0.2,0)--(0.5,0);
\draw[-] (-0.2,-0.2)--(-0.2,0.2)--(0.2,0.2)--(0.2,-0.2)--(-0.2,-0.2);
\node at (-0,-0) {$\dot{+}$};

\draw[-] (0,0.2)--(0,1);

\draw[-] (-0.5,0)--(-0.5,-0.5);
\draw[-] (0.5,0)--(0.5,-0.3);
%m^* lower
\draw[-] (0.5,-0.3)--(0.9,-0.3);
\draw[-] (0.7,-0.3)--(0.7,-0.5);
%m^* upper
\draw[-] (0.7,0.7) -- (1.1, 0.7);
\draw[-] (0.7,0.7) -- (0.7, 1);
\draw[-] (1.1, 0.7)--(1.1, 1);
%-1
\draw[-] (0.9,0.7)--(0.9,0.55);
\draw[-] (0.9,-0.05)--(0.9,-0.3);
\node[draw,circle] at (0.9,0.25) {${\scriptstyle -1}$};

%two upper \dot{+}
%right
\draw[-] (2,1)--(1.7, 1);
\draw[-] (1.1,1)--(1.3, 1);
\draw[-] (1.3,0.8)--(1.3,1.2)--(1.7,1.2)--(1.7,0.8)--(1.3,0.8);
\node at (1.5,1) {$\dot{+}$};
\draw[-] (2,1)--(2,-0.5);
%left
\draw[-] (0.6,1)--(0.7, 1);
\draw[-] (0,1)--(0.2, 1);
\draw[-] (0.2,0.8)--(0.2,1.2)--(0.6,1.2)--(0.6,0.8)--(0.2,0.8);
\node at (0.4,1) {$\dot{+}$};

\draw[-] (0.4,1.5)--(1.5,1.5);
\draw[-] (0.4,1.2)--(0.4,1.5);
\draw[-] (1.5,1.2)--(1.5,1.5);
\end{tikzpicture} \quad  \InnaA{\quad \xlongequal{\text{\ref{rel:mu_coalg_mor}}} \quad} \quad 
\begin{tikzpicture}[anchorbase,scale=1.3]
\draw[-] (-0.5,0.5)--(-0.2,0.5);%\dot{+} lower
\draw[-] (0.2,0.5)--(0.7,0.5);
\draw[-] (-0.2,0.3)--(-0.2,0.7)--(0.2,0.7)--(0.2,0.3)--(-0.2,0.3);
\node at (0,0.5) {$\dot{+}$};

\draw[-] (0,0.7)--(0,1.1);

\draw[-] (-0.5,0.5)--(-0.5,-0.5);
\draw[-] (0.7,0.5)--(0.7,-0.3);
%m^* lower
\draw[-] (0.7,-0.3)--(1.7,-0.3);
\draw[-] (1.2,-0.3)--(1.2,-0.5);
\draw[-] (1.7,-0.3)--(1.7,0);
%m^* upper
\draw[-] (1.3,0) -- (2.1, 0);
\draw[-] (1.3,0) -- (1.3, 0.2);
\draw[-] (2.1,0)--(2.1, 0.2);
%-1
\draw[-] (1.3,0.8) -- (1.3, 1.1);
\draw[-] (2.1,0.8)--(2.1, 1.1);
% \draw[-] (0.9,0)--(0.9,-0.3);
\node[draw,circle] at (1.3,0.5) {${\scriptstyle -1}$};
\node[draw,circle] at (2.1,0.5) {${\scriptstyle -1}$};
%two upper \dot{+}
%right
\draw[-] (2.3,1.1)--(2.1, 1.1);
\draw[-] (2.7,1.1)--(3, 1.1);
\draw[-] (2.3,0.9)--(2.3,1.3)--(2.7,1.3)--(2.7,0.9)--(2.3,0.9);
\node at (2.5,1.1) {$\dot{+}$};
\draw[-] (3,1.1)--(3,-0.5);
%left
\draw[-] (0.9,1.1)--(1.3, 1.1);
\draw[-] (0,1.1)--(0.5, 1.1);
\draw[-] (0.5,0.9)--(0.5,1.3)--(0.9,1.3)--(0.9,0.9)--(0.5,0.9);
\node at (0.7,1.1) {$\dot{+}$};

%upper ev
\draw[-] (0.7,1.5)--(2.5,1.5);
\draw[-] (0.7,1.3)--(0.7,1.5);
\draw[-] (2.5,1.3)--(2.5,1.5);
\end{tikzpicture}
\InnaA{\quad \xlongequal{\text{\ref{itm:str_rel_Lin}}} \quad}
\begin{tikzpicture}[anchorbase,scale=1.3]

%m^* lower
\draw[-] (0.5,-0.3)--(1.7,-0.3);
\draw[-] (1.2,-0.3)--(1.2,-0.5);
\draw[-] (1.7,-0.3)--(1.7,0);
%m^* upper
\draw[-] (1.3,0) -- (2.1, 0);
\draw[-] (1.3,0) -- (1.3, 0.2);
\draw[-] (2.1,0)--(2.1, 0.2);
%-1
\draw[-] (1.3,0.8) -- (1.3, 1.1);%left
\node[draw,circle] at (1.3,0.5) {${\scriptstyle -1}$};

\draw[-] (2.1,0.8)--(2.1, 1.6);%right
\node[draw,circle] at (2.1,0.5) {${\scriptstyle -1}$};

%\dot{+} lower
\draw[-] (0.5,1.1)--(0.7,1.1);
\draw[-] (1.1,1.1)--(1.3,1.1);
\draw[-] (0.7,0.9)--(0.7,1.3)--(1.1,1.3)--(1.1,0.9)--(0.7,0.9);
\node at (0.9,1.1) {$\dot{+}$};

\draw[-] (0.5,1.1)--(0.5,-0.3);

\draw[-] (0.9,1.3)--(0.9,1.6);

%two upper \dot{+}
%right
\draw[-] (2.3,1.6)--(2.1, 1.6);
\draw[-] (2.7,1.6)--(2.9, 1.6);
\draw[-] (2.3,1.4)--(2.3,1.8)--(2.7,1.8)--(2.7,1.4)--(2.3,1.4);
\node at (2.5,1.6) {$\dot{+}$};
\draw[-] (2.9,1.6)--(2.9,-0.5);
%left
\draw[-] (0.5,1.6)--(0.9, 1.6);
\draw[-] (-0.1,1.6)--(0.1, 1.6);
\draw[-] (0.1,1.4)--(0.1,1.8)--(0.5,1.8)--(0.5,1.4)--(0.1,1.4);
\node at (0.3,1.6) {$\dot{+}$};
\draw[-] (-0.1,1.6)--(-0.1, -0.5);%left leg

%upper ev
\draw[-] (0.3,2)--(2.5,2);
\draw[-] (0.3,1.8)--(0.3,2);
\draw[-] (2.5,1.8)--(2.5,2);
\end{tikzpicture} 
\end{equation*}
Applying the coassociativity of $m^*$ (see Relation \ref{rel:Frob_alg1}, \ref{itm:str_rel_diag_bialg}), we obtain:
\begin{equation*}
\begin{tikzpicture}[anchorbase,scale=1.3]

%m^* lower
\draw[-] (0.5,-0.3)--(1.7,-0.3);
\draw[-] (1.2,-0.3)--(1.2,-0.5);
\draw[-] (1.7,-0.3)--(1.7,0);
%m^* upper
\draw[-] (1.3,0) -- (2.1, 0);
\draw[-] (1.3,0) -- (1.3, 0.2);
\draw[-] (2.1,0)--(2.1, 0.2);
%-1
\draw[-] (1.3,0.8) -- (1.3, 1.1);%left
\node[draw,circle] at (1.3,0.5) {${\scriptstyle -1}$};

\draw[-] (2.1,0.8)--(2.1, 1.6);%right
\node[draw,circle] at (2.1,0.5) {${\scriptstyle -1}$};

%\dot{+} lower
\draw[-] (0.5,1.1)--(0.7,1.1);
\draw[-] (1.1,1.1)--(1.3,1.1);
\draw[-] (0.7,0.9)--(0.7,1.3)--(1.1,1.3)--(1.1,0.9)--(0.7,0.9);
\node at (0.9,1.1) {$\dot{+}$};

\draw[-] (0.5,1.1)--(0.5,-0.3);

\draw[-] (0.9,1.3)--(0.9,1.6);

%two upper \dot{+}
%right
\draw[-] (2.3,1.6)--(2.1, 1.6);
\draw[-] (2.7,1.6)--(3, 1.6);
\draw[-] (2.3,1.4)--(2.3,1.8)--(2.7,1.8)--(2.7,1.4)--(2.3,1.4);
\node at (2.5,1.6) {$\dot{+}$};
\draw[-] (3,1.6)--(3,-0.5);
%left
\draw[-] (0.2,1.6)--(0.9, 1.6);
\draw[-] (-0.7,1.6)--(-0.2, 1.6);
\draw[-] (-0.2,1.4)--(-0.2,1.8)--(0.2,1.8)--(0.2,1.4)--(-0.2,1.4);
\node at (0,1.6) {$\dot{+}$};
\draw[-] (-0.7,1.6)--(-0.7, -0.5);%left leg

%upper ev
\draw[-] (0,2)--(2.5,2);
\draw[-] (0,1.8)--(0,2);
\draw[-] (2.5,1.8)--(2.5,2);
\end{tikzpicture} \quad  = \quad
\begin{tikzpicture}[anchorbase,scale=1.3]

%m^* lower
\draw[-] (0.9,-0.3)--(2.1,-0.3);
\draw[-] (1.5,-0.3)--(1.5,-0.5);
\draw[-] (0.9,-0.3)--(0.9,0);
\draw[-] (2.1,-0.3)--(2.1, 0.2);

%m^* upper
\draw[-] (0.5,0) -- (1.3, 0);
\draw[-] (1.3,0) -- (1.3, 0.2);

%-1
\draw[-] (1.3,0.8) -- (1.3, 1.1);%left
\node[draw,circle] at (1.3,0.5) {${\scriptstyle -1}$};

\draw[-] (2.1,0.8)--(2.1, 1.6);%right
\node[draw,circle] at (2.1,0.5) {${\scriptstyle -1}$};

%\dot{+} lower
\draw[-] (0.5,1.1)--(0.7,1.1);
\draw[-] (1.1,1.1)--(1.3,1.1);
\draw[-] (0.7,0.9)--(0.7,1.3)--(1.1,1.3)--(1.1,0.9)--(0.7,0.9);
\node at (0.9,1.1) {$\dot{+}$};

\draw[-] (0.5,1.1)--(0.5,0);

\draw[-] (0.9,1.3)--(0.9,1.6);

%two upper \dot{+}
%right
\draw[-] (2.3,1.6)--(2.1, 1.6);
\draw[-] (2.7,1.6)--(3, 1.6);
\draw[-] (2.3,1.4)--(2.3,1.8)--(2.7,1.8)--(2.7,1.4)--(2.3,1.4);
\node at (2.5,1.6) {$\dot{+}$};
\draw[-] (3,1.6)--(3,-0.5);
%left
\draw[-] (0.2,1.6)--(0.9, 1.6);
\draw[-] (-0.7,1.6)--(-0.2, 1.6);
\draw[-] (-0.2,1.4)--(-0.2,1.8)--(0.2,1.8)--(0.2,1.4)--(-0.2,1.4);
\node at (0,1.6) {$\dot{+}$};
\draw[-] (-0.7,1.6)--(-0.7, -0.5);%left leg

%upper ev
\draw[-] (0,2)--(2.5,2);
\draw[-] (0,1.8)--(0,2);
\draw[-] (2.5,1.8)--(2.5,2);
\end{tikzpicture}
\end{equation*}
Now, by Relations \ref{rel:F_q_lin_mu}, \ref{rel:F_q_lin_plus_lin_distr}, \ref{rel:F_q_lin_mu} (see also their diagrammatic versions \ref{itm:str_rel_Lin}), we have:

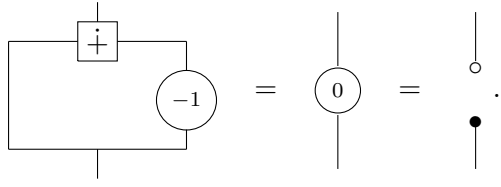
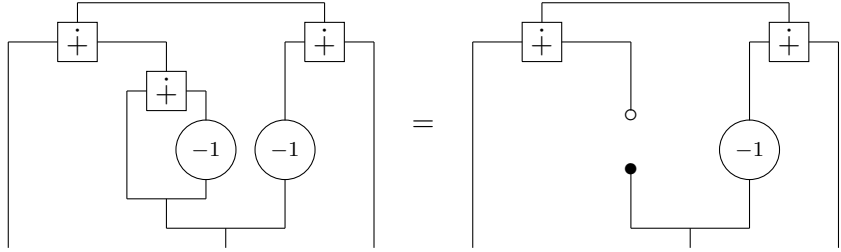
\begin{equation*}
 \begin{tikzpicture}[anchorbase,scale=1.3]
%m^* 
\draw[-] (0,0) -- (1.8, 0);
\draw[-] (1.8,0) -- (1.8, 0.2);
\draw[-] (0.9,0) --(0.9, -0.3);
\draw[-] (0,0)--(0,1.1);

%-1
\draw[-] (1.8,0.8) -- (1.8, 1.1);%right
\node[draw,circle] at (1.8,0.5) {${\scriptstyle -1}$};

\draw[-] (0,1.1)--(0.7,1.1);
\draw[-] (1.1,1.1)--(1.8,1.1);
\draw[-] (0.7,0.9)--(0.7,1.3)--(1.1,1.3)--(1.1,0.9)--(0.7,0.9);
\node at (0.9,1.1) {$\dot{+}$};

\draw[-] (0.9,1.3)--(0.9,1.5);
\end{tikzpicture} \quad = \quad
 \begin{tikzpicture}[anchorbase,scale=1.3]
 \draw[-] (0, -0.8)--(0, -0.25);
 \draw[-] (0, 0.8)--(0, 0.25);
 
 \node[draw,circle] at (0,0) {$\scriptstyle 0$};  
 \end{tikzpicture}
 \quad = \quad  \begin{tikzpicture}[anchorbase,scale=1.3]
 \draw[-] (0, -0.8)--(0, -0.33);
 \draw[-] (0, 0.8)--(0, 0.3);
 
 \node at (0,0.22) {$\circ$};
 \node at (0,-0.33) {$\bullet$}; 
 \end{tikzpicture}.
\end{equation*}
Hence

\begin{equation*}
\begin{tikzpicture}[anchorbase,scale=1.3]

%m^* lower
\draw[-] (0.9,-0.3)--(2.1,-0.3);
\draw[-] (1.5,-0.3)--(1.5,-0.5);
\draw[-] (0.9,-0.3)--(0.9,0);
\draw[-] (2.1,-0.3)--(2.1, 0.2);

%m^* upper
\draw[-] (0.5,0) -- (1.3, 0);
\draw[-] (1.3,0) -- (1.3, 0.2);

%-1
\draw[-] (1.3,0.8) -- (1.3, 1.1);%left
\node[draw,circle] at (1.3,0.5) {${\scriptstyle -1}$};

\draw[-] (2.1,0.8)--(2.1, 1.6);%right
\node[draw,circle] at (2.1,0.5) {${\scriptstyle -1}$};

%\dot{+} lower
\draw[-] (0.5,1.1)--(0.7,1.1);
\draw[-] (1.1,1.1)--(1.3,1.1);
\draw[-] (0.7,0.9)--(0.7,1.3)--(1.1,1.3)--(1.1,0.9)--(0.7,0.9);
\node at (0.9,1.1) {$\dot{+}$};

\draw[-] (0.5,1.1)--(0.5,0);

\draw[-] (0.9,1.3)--(0.9,1.6);

%two upper \dot{+}
%right
\draw[-] (2.3,1.6)--(2.1, 1.6);
\draw[-] (2.7,1.6)--(3, 1.6);
\draw[-] (2.3,1.4)--(2.3,1.8)--(2.7,1.8)--(2.7,1.4)--(2.3,1.4);
\node at (2.5,1.6) {$\dot{+}$};
\draw[-] (3,1.6)--(3,-0.5);
%left
\draw[-] (0.2,1.6)--(0.9, 1.6);
\draw[-] (-0.7,1.6)--(-0.2, 1.6);
\draw[-] (-0.2,1.4)--(-0.2,1.8)--(0.2,1.8)--(0.2,1.4)--(-0.2,1.4);
\node at (0,1.6) {$\dot{+}$};
\draw[-] (-0.7,1.6)--(-0.7, -0.5);%left leg

%upper ev
\draw[-] (0,2)--(2.5,2);
\draw[-] (0,1.8)--(0,2);
\draw[-] (2.5,1.8)--(2.5,2);
\end{tikzpicture} \quad  = \quad 
\begin{tikzpicture}[anchorbase,scale=1.3]

%m^* lower
\draw[-] (0.9,-0.3)--(2.1,-0.3);
\draw[-] (1.5,-0.3)--(1.5,-0.5);

\draw[-] (2.1,-0.3)--(2.1, 0.2);

\draw[-] (0.9,-0.3)--(0.9, 0.3);
\node at (0.9, 0.3) {$\bullet$};
\node at (0.9, 0.85) {$\circ$};
\draw[-] (0.9,0.9)--(0.9,1.6);

%-1

\draw[-] (2.1,0.8)--(2.1, 1.6);%right
\node[draw,circle] at (2.1,0.5) {${\scriptstyle -1}$};

%two upper \dot{+}
%right
\draw[-] (2.3,1.6)--(2.1, 1.6);
\draw[-] (2.7,1.6)--(3, 1.6);
\draw[-] (2.3,1.4)--(2.3,1.8)--(2.7,1.8)--(2.7,1.4)--(2.3,1.4);
\node at (2.5,1.6) {$\dot{+}$};
\draw[-] (3,1.6)--(3,-0.5);
%left
\draw[-] (0.2,1.6)--(0.9, 1.6);
\draw[-] (-0.7,1.6)--(-0.2, 1.6);
\draw[-] (-0.2,1.4)--(-0.2,1.8)--(0.2,1.8)--(0.2,1.4)--(-0.2,1.4);
\node at (0,1.6) {$\dot{+}$};
\draw[-] (-0.7,1.6)--(-0.7, -0.5);%left leg

%upper ev
\draw[-] (0,2)--(2.5,2);
\draw[-] (0,1.8)--(0,2);
\draw[-] (2.5,1.8)--(2.5,2);
\end{tikzpicture}
\end{equation*}
We now use the relations ``$z$ is the zero element for $\dot{+}$'' (Relation \ref{rel:F_q_lin_zero}, \ref{itm:str_rel_Lin}) and ``$\eps^*$ is the counit of $m^*$'' (Relation \ref{rel:Frob_alg1}, \ref{itm:str_rel_diag_bialg}) to show that the diagram on the right is just 
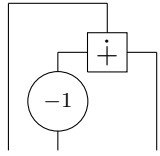
\begin{equation*}
 \begin{tikzpicture}[anchorbase,scale=1.3]
\draw[-] (-0.5,0)--(-0.2,0);%\dot{+}
\draw[-] (0.2,0)--(0.5,0);
\draw[-] (-0.2,-0.2)--(-0.2,0.2)--(0.2,0.2)--(0.2,-0.2)--(-0.2,-0.2);
\node at (-0,-0) {$\dot{+}$};
\draw[-] (0,0.2)--(0,0.5);
\draw[-] (-0.5,0)--(-0.5,-0.2);
\draw[-] (-0.5,-0.8)--(-0.5,-1);
\node[draw,circle] at (-0.5,-0.5) {${\scriptstyle -1}$};

\draw[-] (0.5,0)--(0.5,-1);

\draw[-] (0,0.5)--(-1,0.5);
\draw[-] (-1,0.5)--(-1,-1);
\end{tikzpicture} 
\end{equation*}  
and the proof of the lemma is complete.
\end{proof}

\begin{lemma}\label{lem:dual_scalar_mult}
 For any $a\in \F_q^{\times}$, we have $ev \circ (\mu_a \otimes \mu_a)=ev$.
\end{lemma}

\begin{proof}
Using Relations \ref{rel:mu_coalg_mor} which state that $\mu_a$ commutes with both counit and multiplication (that is, $\eps^*\circ \mu_a = \eps^*$ and $m\circ (\mu_a\otimes \mu_a) = \mu_a \circ m$), we obtain that 
$$ev \circ (\mu_a\otimes \mu_a) =\eps^*\circ m\circ (\mu_a\otimes \mu_a) = \eps^*\circ \mu_a \circ m = \eps^*\circ m
=ev.$$ 
\end{proof}

\begin{lemma}\label{lem:transp_mu_A}
For any $A\in GL_k(\F_q)$, the map $$\overline{ev}_{\V^{\otimes k}} \circ (\mu_A\otimes \mu_A): \V^{\otimes k}\otimes \V^{\otimes k} \to \triv$$ is equal to $\overline{ev}_{\V^{\otimes k}}$.
\end{lemma}
\begin{example}
 Let $\mathcal{C} = \Rep(GL_n(\F_q))$, $\V:=\V_n$, $k=2$ and 
 $A = \begin{bmatrix}
 1 &1\\
 0 &1                                                                 \end{bmatrix}
 $. Then $\mu_A(v_1\otimes v_2)= (v_1 \dot{+} v_2)\otimes v_2$ for any $v_1, v_2\in V$. Recall that $\{v_1\otimes v_2| \, v_1, v_2\in V\}$ forms a basis in $\V_n^{\otimes 2}$, so the map $\mu_A^*$ takes $v'_1\otimes v'_2$ to $(v'_1\dot{-}v'_2)\otimes v'_2$ for any $v'_1, v'_2\in V$. Hence $\mu_A^* = \mu_{A^{-1}}$ where $A = \begin{bmatrix}
 1 &-1\\
 0 &1                                                                 \end{bmatrix}
 $. This implies the statement of the lemma in the setting of our example.
\end{example}

\begin{proof}[Proof of Lemma \ref{lem:transp_mu_A}]
Let us write $A$ as a product of elementary matrices (these are matrices obtained from the identity matrix by one single elementary row operation). Recall that for this it is enough to use the following elementary matrices:
\begin{itemize}
 \item $E_{R_i \longleftrightarrow R_{i+1}}$ for $1\leq i\leq k$, the permutation matrix for a simple reflection. 
 \item $E_{R_i \longleftarrow a R_i}$ for $a\neq 0$ and for $1\leq i\leq k$. This is the diagonal matrix with $a$ in position $(i, i)$ and $1$'s on the rest of the diagonal.
 \item $E_{R_i \longleftarrow  R_i+R_{i+1}}$ for $1\leq i< k$. This matrix differs from the matrix $\id_{k\times k}$ by a single entry $(i, i+1)$, which is $1$.
\end{itemize}
Due to Proposition \ref{prop:compos_mu_A}, it is enough to prove the equality $\overline{ev}_{\V^{\otimes k}} \circ (\mu_A\otimes \mu_A) = \overline{ev}_{\V^{\otimes k}}$ for the elementary matrices $A$ as above.
We now do this case-by-case:

\InnaA{{\bf  Let $A=E_{R_i \longleftrightarrow R_{i+1}}$ for $1\leq i\leq k$.}} In that case, the statement boils down to proving that 
 $\overline{ev}_{\V^{\otimes 2}} \circ (\sigma\otimes \sigma) = ev_{\V^{\otimes 2}}$, where $\sigma: \V^{\otimes 2} \to \V^{\otimes 2}$ is the symmetry morphism. But this is true for any special commutative Frobenius algebra in a rigid symmetric monoidal category, so there is nothing to check.
 
\InnaA{{\bf Let $A=E_{R_i \longleftarrow a R_i}$ for $a\neq 0$ and $1\leq i\leq k$.}} In that case, the statement boils down to proving that 
 $ev\circ (\mu_a\otimes \mu_a) = ev$, which is proved in Lemma \ref{lem:dual_scalar_mult}.
 
\InnaA{{\bf Let $A=E_{R_i \longleftarrow  R_i+R_{i+1}}$ for $1\leq i < k$.}} In that case, the statement boils down to proving that $\overline{ev}_{\V^{\otimes 2}} \circ (\mu_A\otimes \mu_A) = \overline{ev}_{\V^{\otimes 2}}$ for 
 $A = \begin{bmatrix}
              1 &1\\
              0 &1                                                                                                                                                                                                       \end{bmatrix}$. Equivalently, we need to prove that 
$ev_{\V^{\otimes 2}} \circ (\mu_A\otimes \mu_{A^t}) = ev_{\V^{\otimes 2}}$.
Diagramatically, $\mu_A$ looks like this:
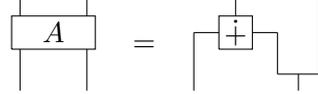
\begin{equation*}\begin{tikzpicture}[anchorbase,scale=1.1]
\draw[-] (-0.5,-0.2)--(-0.5,0.2)--(0.5,0.2)--(0.5,-0.2)--(-0.5,-0.2);
\node at (-0,-0) {${A}$};
\draw[-] (-0.4,-0.2)--(-0.4,-0.7);
\draw[-] (0.4,-0.2)--(0.4,-0.7);
\draw[-] (-0.4,0.2)--(-0.4,0.4);
\draw[-] (0.4,0.2)--(0.4,0.4);
\end{tikzpicture} \quad = \quad
\begin{tikzpicture}[anchorbase,scale=1.1]
\draw[-] (-0.5,0)--(-0.2,0);%\dot{+}
\draw[-] (0.2,0)--(0.5,0);
\draw[-] (-0.2,-0.2)--(-0.2,0.2)--(0.2,0.2)--(0.2,-0.2)--(-0.2,-0.2);
\node at (-0,-0) {$\dot{+}$};
\draw[-] (0,0.2)--(0,0.4);
\draw[-] (-0.5,0)--(-0.5,-0.7);
\draw[-] (0.5,0)--(0.5,-0.5);%m^*
\draw[-] (1,0.4)--(1,-0.5);
\draw[-] (0.5, -0.5)--(1,-0.5);
\draw[-] (0.75, -0.5)--(0.75,-0.7);
\end{tikzpicture} 
\end{equation*}

\InnaA{So we need to prove the following equality:}
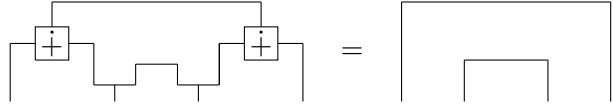
\begin{equation*}
\begin{tikzpicture}[anchorbase,scale=1.1]
\draw[-] (0,0.5)--(2.5,0.5);%ev
\draw[-] (-0.5,0)--(-0.2,0);%\dot{+}
\draw[-] (0.2,0)--(0.5,0);
\draw[-] (-0.2,-0.2)--(-0.2,0.2)--(0.2,0.2)--(0.2,-0.2)--(-0.2,-0.2);
\node at (-0,-0) {$\dot{+}$};
\draw[-] (0,0.2)--(0,0.5);
\draw[-] (-0.5,0)--(-0.5,-0.7);
\draw[-] (0.5,0)--(0.5,-0.5);%m^*
\draw[-] (0.5, -0.5)--(1,-0.5);
\draw[-] (0.75, -0.5)--(0.75,-0.7);

\draw[-] (1,-0.25)--(1,-0.5);%m
\draw[-] (1.5,-0.25)--(1.5,-0.5);
\draw[-] (1,-0.25)--(1.5,-0.25);

\draw[-] (2,0)--(2,-0.5);%m^*
\draw[-] (1.5,-0.5)--(2,-0.5);
\draw[-] (1.75,-0.5)--(1.75,-0.7);

\draw[-] (2,0)--(2.3,0);%\dot{+}
\draw[-] (2.7,0)--(3,0);
\draw[-] (2.3,-0.2)--(2.3,0.2)--(2.7,0.2)--(2.7,-0.2)--(2.3,-0.2);
\node at (2.5,0) {$\dot{+}$};
\draw[-] (2.5,0.2)--(2.5,0.5);
\draw[-] (2,0)--(2,-0.5);
\draw[-] (3,0)--(3,-0.7);
\end{tikzpicture} \quad= \quad\begin{tikzpicture}[anchorbase,scale=1.1]
\draw[-] (0,0.5)--(2.5,0.5);%ev
\draw[-] (0,-0.7)--(0,0.5);
\draw[-] (2.5,-0.7)--(2.5,0.5);
\draw[-] (0.75,-0.2)--(1.75,-0.2);%ev
\draw[-] (0.75,-0.7)--(0.75,-0.2);
\draw[-] (1.75,-0.7)--(1.75,-0.2);
\end{tikzpicture} 
\end{equation*}  
Let us use the Frobenius Relations \ref{rel:Frob_alg2} \InnaA{(see also \ref{itm:str_rel_diag_Frob}), and the fact that $\eps^*$ is a counit of $m^*$ (see \ref{rel:Frob_alg}). These imply}:
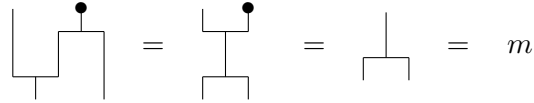
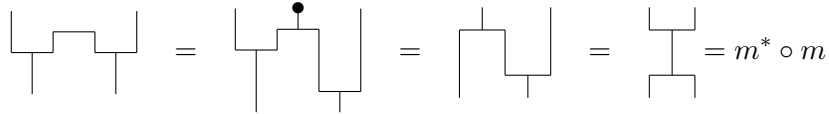
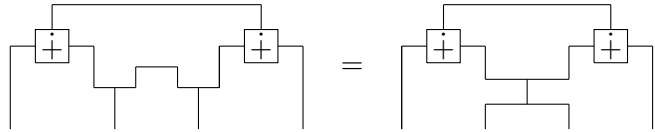
\begin{equation*}
\begin{tikzpicture}[anchorbase,scale=0.3]
\draw[-] (-2,2)--(-2,-1);
\draw[-] (-2,-1)--(0,-1);
\draw[-] (-1,-1)--(-1,-2);
\draw[-] (0,1)--(0,-1);
\draw[-] (0,1)--(2,1);
\draw[-] (1,1)--(1,2);
\draw[-] (2,1)--(2,-2);
\node at (1, 2) {$\bullet$};
\end{tikzpicture}\quad = \quad \begin{tikzpicture}[anchorbase,scale=0.3]
\draw[-] (-1,2)--(-1,1);
\draw[-] (-1,-1)--(-1,-2);
\draw[-] (-1,-1)--(1,-1);
\draw[-] (0,-1)--(0,1);
\draw[-] (-1,1)--(1,1);
\draw[-] (1,2)--(1,1);
\draw[-] (1,-1)--(1,-2);
\node at (1, 2) {$\bullet$};
\end{tikzpicture} \quad = \quad \begin{tikzpicture}[anchorbase,scale=0.3]
\draw[-] (-1,-1)--(-1,-2);
\draw[-] (-1,-1)--(1,-1);
\draw[-] (0,-1)--(0,1);
\draw[-] (1,-1)--(1,-2);
\end{tikzpicture}  \quad = \quad m
\end{equation*} 
\InnaA{and so}
\begin{equation*}
\begin{tikzpicture}[anchorbase,scale=1.1]
\draw[-] (0.5,0)--(0.5,-0.5);%m^*
\draw[-] (0.5, -0.5)--(1,-0.5);
\draw[-] (0.75, -0.5)--(0.75,-1);

\draw[-] (1,-0.25)--(1,-0.5);%m
\draw[-] (1.5,-0.25)--(1.5,-0.5);
\draw[-] (1,-0.25)--(1.5,-0.25);

\draw[-] (2,0)--(2,-0.5);%m^*
\draw[-] (1.5,-0.5)--(2,-0.5);
\draw[-] (1.75,-0.5)--(1.75,-1);
\end{tikzpicture} \quad= \quad \begin{tikzpicture}[anchorbase,scale=1.1]
\draw[-] (0.5,0)--(0.5,-0.5);%m^*
\draw[-] (0.5, -0.5)--(1,-0.5);
\draw[-] (0.75, -0.5)--(0.75,-1.25);

\draw[-] (1,-0.25)--(1,-0.5);%m
\draw[-] (1.5,-0.25)--(1.5,-0.5);
\draw[-] (1,-0.25)--(1.5,-0.25);

\draw[-] (1.25,-0.25)--(1.25,0);%eps^*
\node at (1.25, 0) {$\bullet$};

\draw[-] (2,0)--(2,-1);%m^*
\draw[-] (1.5,-0.5)--(1.5,-1);
\draw[-] (1.5,-1)--(2,-1);
\draw[-] (1.75,-1)--(1.75,-1.25);
\end{tikzpicture} \quad=  \quad \begin{tikzpicture}[anchorbase,scale=0.3]
\draw[-] (-2,1)--(-2,-2);
\draw[-] (-2,1)--(0,1);
\draw[-] (-1,2)--(-1,1);
\draw[-] (0,1)--(0,-1);
\draw[-] (0,-1)--(2,-1);
\draw[-] (1,-1)--(1,-2);
\draw[-] (2,2)--(2,-1);
\end{tikzpicture} \quad = \quad \begin{tikzpicture}[anchorbase,scale=0.3]
\draw[-] (-1,2)--(-1,1);
\draw[-] (-1,-1)--(-1,-2);
\draw[-] (-1,-1)--(1,-1);
\draw[-] (0,-1)--(0,1);
\draw[-] (-1,1)--(1,1);
\draw[-] (1,2)--(1,1);
\draw[-] (1,-1)--(1,-2);

\end{tikzpicture} =m^*\circ m
\end{equation*}
Thus we obtain:
\begin{equation*}
\begin{tikzpicture}[anchorbase,scale=1.1]
\draw[-] (0,0.5)--(2.5,0.5);%ev
\draw[-] (-0.5,0)--(-0.2,0);%\dot{+}
\draw[-] (0.2,0)--(0.5,0);
\draw[-] (-0.2,-0.2)--(-0.2,0.2)--(0.2,0.2)--(0.2,-0.2)--(-0.2,-0.2);
\node at (-0,-0) {$\dot{+}$};
\draw[-] (0,0.2)--(0,0.5);
\draw[-] (-0.5,0)--(-0.5,-1);

\draw[-] (0.5,0)--(0.5,-0.5);%m^*
\draw[-] (0.5, -0.5)--(1,-0.5);
\draw[-] (0.75, -0.5)--(0.75,-1);

\draw[-] (1,-0.25)--(1,-0.5);%m
\draw[-] (1.5,-0.25)--(1.5,-0.5);
\draw[-] (1,-0.25)--(1.5,-0.25);

\draw[-] (2,0)--(2,-0.5);%m^*
\draw[-] (1.5,-0.5)--(2,-0.5);
\draw[-] (1.75,-0.5)--(1.75,-1);

\draw[-] (2,0)--(2.3,0);%\dot{+}
\draw[-] (2.7,0)--(3,0);
\draw[-] (2.3,-0.2)--(2.3,0.2)--(2.7,0.2)--(2.7,-0.2)--(2.3,-0.2);
\node at (2.5,0) {$\dot{+}$};
\draw[-] (2.5,0.2)--(2.5,0.5);
\draw[-] (2,0)--(2,-0.5);
\draw[-] (3,0)--(3,-1);
\end{tikzpicture} \quad= \quad\begin{tikzpicture}[anchorbase,scale=1.1]
\draw[-] (0,0.5)--(2,0.5);%ev
\draw[-] (-0.5,0)--(-0.2,0);%\dot{+}
\draw[-] (0.2,0)--(0.5,0);
\draw[-] (-0.2,-0.2)--(-0.2,0.2)--(0.2,0.2)--(0.2,-0.2)--(-0.2,-0.2);
\node at (-0,-0) {$\dot{+}$};
\draw[-] (0,0.2)--(0,0.5);
\draw[-] (-0.5,0)--(-0.5,-1);
\draw[-] (-0.5,0)--(-0.5,-1);

\draw[-] (0.5,0)--(0.5,-0.4);%m^*
\draw[-] (0.5, -0.4)--(1.5,-0.4);
\draw[-] (1, -0.4)--(1,-0.7);
%m
\draw[-] (0.5,-0.7)--(1.5,-0.7);
\draw[-] (0.5,-0.7)--(0.5,-1);
\draw[-] (1.5,-0.7)--(1.5,-1);

\draw[-] (1.5,0)--(1.8,0);%\dot{+}
\draw[-] (2.2,0)--(2.5,0);
\draw[-] (1.8,-0.2)--(1.8,0.2)--(2.2,0.2)--(2.2,-0.2)--(1.8,-0.2);
\node at (2,0) {$\dot{+}$};
\draw[-] (2,0.2)--(2,0.5);
\draw[-] (1.5,0)--(1.5,-0.4);
\draw[-] (2.5,0)--(2.5,-1);
\end{tikzpicture} 
\end{equation*}  

We now apply Lemma \ref{lem:eq_axiom_corollary} to the right hand side of this equality and obtain:
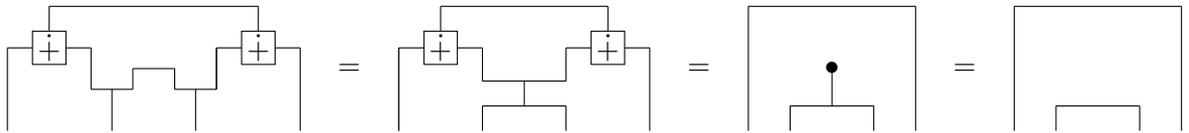
\begin{equation*}\begin{tikzpicture}[anchorbase,scale=1.1]
\draw[-] (0,0.5)--(2.5,0.5);%ev
\draw[-] (-0.5,0)--(-0.2,0);%\dot{+}
\draw[-] (0.2,0)--(0.5,0);
\draw[-] (-0.2,-0.2)--(-0.2,0.2)--(0.2,0.2)--(0.2,-0.2)--(-0.2,-0.2);
\node at (-0,-0) {$\dot{+}$};
\draw[-] (0,0.2)--(0,0.5);
\draw[-] (-0.5,0)--(-0.5,-1);

\draw[-] (0.5,0)--(0.5,-0.5);%m^*
\draw[-] (0.5, -0.5)--(1,-0.5);
\draw[-] (0.75, -0.5)--(0.75,-1);

\draw[-] (1,-0.25)--(1,-0.5);%m
\draw[-] (1.5,-0.25)--(1.5,-0.5);
\draw[-] (1,-0.25)--(1.5,-0.25);

\draw[-] (2,0)--(2,-0.5);%m^*
\draw[-] (1.5,-0.5)--(2,-0.5);
\draw[-] (1.75,-0.5)--(1.75,-1);

\draw[-] (2,0)--(2.3,0);%\dot{+}
\draw[-] (2.7,0)--(3,0);
\draw[-] (2.3,-0.2)--(2.3,0.2)--(2.7,0.2)--(2.7,-0.2)--(2.3,-0.2);
\node at (2.5,0) {$\dot{+}$};
\draw[-] (2.5,0.2)--(2.5,0.5);
\draw[-] (2,0)--(2,-0.5);
\draw[-] (3,0)--(3,-1);
\end{tikzpicture} \quad= \quad\begin{tikzpicture}[anchorbase,scale=1.1]
\draw[-] (0,0.5)--(2,0.5);%ev
\draw[-] (-0.5,0)--(-0.2,0);%\dot{+}
\draw[-] (0.2,0)--(0.5,0);
\draw[-] (-0.2,-0.2)--(-0.2,0.2)--(0.2,0.2)--(0.2,-0.2)--(-0.2,-0.2);
\node at (-0,-0) {$\dot{+}$};
\draw[-] (0,0.2)--(0,0.5);
\draw[-] (-0.5,0)--(-0.5,-1);
\draw[-] (-0.5,0)--(-0.5,-1);

\draw[-] (0.5,0)--(0.5,-0.4);%m^*
\draw[-] (0.5, -0.4)--(1.5,-0.4);
\draw[-] (1, -0.4)--(1,-0.7);
%m
\draw[-] (0.5,-0.7)--(1.5,-0.7);
\draw[-] (0.5,-0.7)--(0.5,-1);
\draw[-] (1.5,-0.7)--(1.5,-1);

\draw[-] (1.5,0)--(1.8,0);%\dot{+}
\draw[-] (2.2,0)--(2.5,0);
\draw[-] (1.8,-0.2)--(1.8,0.2)--(2.2,0.2)--(2.2,-0.2)--(1.8,-0.2);
\node at (2,0) {$\dot{+}$};
\draw[-] (2,0.2)--(2,0.5);
\draw[-] (1.5,0)--(1.5,-0.4);
\draw[-] (2.5,0)--(2.5,-1);
\end{tikzpicture}  \quad= \quad\begin{tikzpicture}[anchorbase,scale=1.1]
\draw[-] (0,0.5)--(2,0.5);%ev
\draw[-] (0,0.5)--(0,-1);
\draw[-] (2,0.5)--(2,-1);

\draw[-] (0.5,-0.7)--(1.5,-0.7);%m
\draw[-] (0.5,-0.7)--(0.5,-1);
\draw[-] (1.5,-0.7)--(1.5,-1);

\draw[-] (1,-0.7)--(1,-0.25);%\eps^*
\node at (1,-0.25) {$\bullet$};
\end{tikzpicture}  \quad= \quad\begin{tikzpicture}[anchorbase,scale=1.1]
\draw[-] (0,0.5)--(2,0.5);%ev
\draw[-] (0,0.5)--(0,-1);
\draw[-] (2,0.5)--(2,-1);

\draw[-] (0.5,-0.7)--(1.5,-0.7);%ev
\draw[-] (0.5,-0.7)--(0.5,-1);
\draw[-] (1.5,-0.7)--(1.5,-1);

\end{tikzpicture}
\end{equation*}  

which concludes the proof of the lemma.
\end{proof}

\begin{corollary}\label{cor:regular_system_of_eq}
 For any $A \in GL_{k}(\F_q)$, we have: $(z^*)^{\otimes k} \circ \mu_A = 
(z^*)^{\otimes k}$.
\end{corollary}
\begin{example}
 Let $\mathcal{C} = \Rep(GL_n(\F_q))$, $\V:=\V_n$. This corollary states that for any $v_1, \ldots, v_k \in V$ and any $A \in GL_{k}(\F_q)$, the element $$\mu_A \left(v_1\otimes \ldots \otimes v_k \right)= \left(\dot{\sum}_{j=1}^{k} \dot{A}_{1,j} v_j \right)\otimes  \ldots  \otimes \left(\dot{\sum}_{j=1}^{k} \dot{A}_{k,j} v_j \right) \in \V_n^{\otimes k}$$ equals $\dot{0}\otimes \ldots \otimes \dot{0}$ iff $v_1 = \ldots = v_k = \dot{0}$. In this setting, it is easy to see that this statement is true: indeed, $\mu_A \left(v_1\otimes \ldots \otimes v_k \right)= \dot{0}\otimes \ldots \otimes \dot{0}$ iff the system of equalities 
 $$\begin{cases}
    \dot{\sum}_{j=1}^{k} \dot{A}_{1,j} v_j = \dot{0}\\
    \dot{\sum}_{j=1}^{k} \dot{A}_{2,j} v_j = \dot{0}\\
    \ldots\\
    \dot{\sum}_{j=1}^{k} \dot{A}_{k,j} v_j = \dot{0}
   \end{cases}$$
holds, which, as we know from linear algebra, happens exactly when $v_1 = \ldots = v_k = \dot{0}$ (due to the fact that $A$ is invertible). 
\end{example}

\begin{proof}
Recall that $z^* = ev\circ (\id\otimes z): \V\to \triv$. Thus $$(z^*)^{\otimes d} \circ \mu_A = \overline{ev}_{\V^{\otimes k}} \circ (\mu_A \otimes z^{\otimes k}) = \overline{ev}_{\V^{\otimes k}} \circ (\mu_A \otimes (\mu_A \circ z^{\otimes k})) $$ The second equality follows from the fact that $z^{\otimes k} = \mu_A \circ z^{\otimes k}$, which is a direct consequence of Relations \eqref{rel:F_q_lin}). By Lemma \ref{lem:transp_mu_A}, we have: $$\overline{ev}_{\V^{\otimes k}} \circ (\mu_A \otimes (\mu_A \circ z^{\otimes k})) = \overline{ev}_{\V^{\otimes k}} \circ (\mu_A \otimes \mu_{A}) \circ (\id \otimes z^{\otimes k})= \overline{ev}_{\V^{\otimes k}} \circ (\id \otimes z^{\otimes k}) = (z^*)^{\otimes k}.$$
\end{proof}

The following lemma will be useful in Section \ref{ssec:semi-diagrams}.

\begin{lemma}\label{lem:comparing_with_zero}
We have: $$(\id_{\V} \otimes z^*) \circ m^* = m \circ (\id_{\V} \otimes z) = z \circ  z^*.$$
Diagrammatically, this means:

\begin{equation*} 
\begin{tikzpicture}[anchorbase,scale=1]
\draw[-] (-0.5, 0.5)--(-0.5,0)--(0.5,0)--(0.5, 0.5);
\draw[-] (0,0)--(0,-0.5);
\node at (0.5, 0.53) {$\circ$};
\end{tikzpicture} \quad = \quad 
\begin{tikzpicture}[anchorbase,scale=1]
\draw[-] (-0.5, -0.5)--(-0.5,0)--(0.5,0)--(0.5, -0.49);
\draw[-] (0,0)--(0,0.5);
\node at (0.5, -0.53) {$\circ$};
\end{tikzpicture} 
\quad = \quad
\begin{tikzpicture}[anchorbase,scale=1]
\draw[-] (0,-0.55)--(0,-0.16);
\node at (0, -0.13) {$\circ$};
\draw[-] (0,0.55)--(0,0.16);
\node at (0, 0.13) {$\circ$};
\end{tikzpicture} 
\end{equation*}
\end{lemma}
\begin{proof}
We use Frobenius algebra relations \ref{rel:Frob_alg} (see also their diagrammatic form \ref{itm:str_rel_diag_Frob}) to establish the following equalities:
\begin{equation*} 
\begin{tikzpicture}[anchorbase,scale=1]
\draw[-] (-0.5, 0.5)--(-0.5,0)--(0.5,0)--(0.5, 0.5);
\draw[-] (0,0)--(0,-0.5);
\node at (0.5, 0.53) {$\circ$};
\end{tikzpicture} 
\quad = \quad
\begin{tikzpicture}[anchorbase,scale=1]
\draw[-] (-1.5, 0.25)--(-1.5, -0.5)--(-0.5, -0.5)--(-0.5,0)--(0.5,0)--(0.5, -0.72);
\draw[-] (-1,-0.5)--(-1,-0.75);
\draw[-] (0,0)--(0,0.25);
\node at (0.5, -0.78) {$\circ$};
\node at (0, 0.28) {$\bullet$};
\end{tikzpicture}
\quad = \quad
 \begin{tikzpicture}[anchorbase,scale=1]
\draw[-] (-1.5, -0.5)--(-1.5, 0.5)--(-0.5, 0.5)--(-0.5,0)--(0.5,0)--(0.5, 0.75);
\draw[-] (-1,0.5)--(-1,0.75);
\draw[-] (0,0)--(0,-0.49);
\node at (0, -0.53) {$\circ$};
\node at (0.5, 0.78) {$\bullet$};
\end{tikzpicture}
\quad = \quad 
\begin{tikzpicture}[anchorbase,scale=1]
\draw[-] (-0.5, -0.5)--(-0.5,0)--(0.5,0)--(0.5, -0.49);
\draw[-] (0,0)--(0,0.5);
\node at (0.5, -0.53) {$\circ$};
\end{tikzpicture} 
\end{equation*}

The same Frobenius algebra relations imply:

\begin{equation*}
 \begin{tikzpicture}[anchorbase,scale=1]
\draw[-] (-0.5, -0.6)--(-0.5,0)--(0.5,0)--(0.5, -0.5);
\draw[-] (0,0)--(0,0.5);
\node at (0.5, -0.53) {$\circ$};
\end{tikzpicture} \quad = \quad \begin{tikzpicture}[anchorbase,scale=1]
\draw[-] (-1.5, 0.25)--(-1.5, -0.5)--(-0.5, -0.5)--(-0.5,0)--(0.5,0)--(0.5, -0.75);
\draw[-] (-1,-0.5)--(-1,-0.9);
\draw[-] (0,0)--(0,0.45);
\node at (0.5, -0.78) {$\circ$};
\node at (-1.5, 0.28) {$\bullet$};
\end{tikzpicture}
\quad = \quad \begin{tikzpicture}[anchorbase,scale=1]
\draw[-] (-1.5, -0.5)--(-1.5, 0.5)--(-0.5, 0.5)--(-0.5,0)--(0.5,0)--(0.5, 0.85);
\draw[-] (-1,0.5)--(-1,0.75);
\draw[-] (0,0)--(0,-0.49);
\node at (0, -0.53) {$\circ$};
\node at (-1, 0.78) {$\bullet$};
\end{tikzpicture}
\end{equation*}
By Relation \ref{rel:z_coalg_mor} (see also diagrammatic Relation \ref{itm:str_rel_diag_z_coalg}), the right hand side equals
\begin{equation*}
 \begin{tikzpicture}[anchorbase,scale=1]
\draw[-] (-1.5, -0.25)--(-1.5, 0)--(-0.5, 0)--(-0.5,-0.25);
\draw[-] (-1,0)--(-1,0.3);
\node at (-0.5, -0.28) {$\circ$};
\node at (-1, 0.28) {$\bullet$};
\draw[-] (0,0.5)--(0,-0.25);
\node at (0, -0.28) {$\circ$};
\end{tikzpicture}
\quad = \quad 
 \begin{tikzpicture}[anchorbase,scale=1]
\draw[-] (-1.5, -0.3)--(-1.5, 0.5)--(-0.5, 0.5)--(-0.5,-0.25);
\node at (-0.5, -0.28) {$\circ$};
\draw[-] (0,0.5)--(0,-0.25);
\node at (0, -0.28) {$\circ$};
\end{tikzpicture} 
\end{equation*}
Thus we showed that  
$$m \circ (\id_{\V} \otimes z) =  \left(coev \circ (\id_{\V} \otimes  z) \right) \otimes z.$$

Finally, we use the functoriality of tensor product to establish that $$\left(coev \circ (\id_{\V} \otimes  z) \right) \otimes z= z^*\otimes z = z \circ  z^*, $$ and we are done.
\end{proof}

\subsection{Duality}\label{ssec:duality_mu_A}
\begin{lemma}\label{lem:self-dual}
 The morphisms $ev :=\eps^*\circ m$, $coev:=m^*\circ \eps$ make $\V$ a self-dual object in $\mathcal{C}$. The categorical dimension of this object is $\dim(\V)=\eps^*\circ \eps$.
\end{lemma}
\begin{example}
For $\mathcal{C} = \Rep(GL_n(\F_q))$, $\V:=\V_n$, we have $$(\eps^*\circ \eps) (1) = \sum_{v\in V} \eps^*(v) = \sum_{v\in V} 1 = q^n = \dim \V_n.$$
\end{example}

\begin{proof}
The first statement is a standard fact for Frobenius algebras. We only need to check the ``snake relations'' $$(ev\otimes \id) \circ (\id \otimes coev) = \id =(\id\otimes ev) \circ (coev\otimes \id)$$ and they follow directly from the axioms of a Frobenius algebra along with the Frobenius relations (see \ref{rel:Frob_alg1}). For the second statement, recall that 
$$ \dim(\V)=ev\circ \sigma\circ coev = \eps^*\circ m\circ \sigma\circ coev  = \eps^*\circ m\circ coev = \eps^*\circ m \circ m^*\circ \eps= \eps^*\circ \eps.$$
Here the second equality is due to the commutativity of $m$, while the last equality is due to the Speciality Relation (see \ref{rel:Frob_alg2}).
\end{proof}

\begin{remark}\label{rmk:m_m_star_duality}
The self-duality $\V \cong \V^*$ makes the morphisms $m$, $m^*$ dual to each other, and the morphisms $\eps, \eps^*$ dual to each other.
\end{remark}

\begin{lemma}\label{cor:transpose}
 The morphism $(\mu_A)^*: \V^{\otimes k} \to \V^{\otimes k}$ is $\mu_{A^{-1}}$.

\end{lemma}
\begin{proof}
This is a direct corollary of Lemma \ref{lem:transp_mu_A}.
\end{proof}
\begin{corollary}\label{cor:eps_and_mu}
 We have $\mu_a \circ \eps  = \eps$ for any $a\in \F_q^{\times}$.
\end{corollary}
\begin{proof}
 This follows directly from Relation \ref{rel:mu_coalg_mor}, which states that $\eps^* \circ 
\mu_a = \eps^*$, and Lemma \ref{lem:dual_scalar_mult}.
\end{proof}
\begin{example}
For $\mathcal{C} = \Rep(GL_n(\F_q))$, $\V:=\V_n$, the statement of Lemma \ref{lem:dual_scalar_mult} translates to
$\delta_{\dot{a}v, \dot{a}w} = \delta_{v,w}$ for any $a\in \F_q^{\times}$, $v,w\in V$, and Corollary \ref{cor:eps_and_mu} translates to $\dot{a}\sum_{v\in V} v = \sum_{v\in V} v$ for any $a\in \F_q^{\times}$.
\end{example}

\section{Morphisms \texorpdfstring{$\widetilde{f}_R$}{fR}}\label{sec:morph_fR}

\InnaA{Let $\mathcal{C}$ be a $\C$-linear rigid SM category and let $\V$ be an $\F_q$-linear Frobenius space in $\mathcal{C}$.}

In this section we define morphisms $\widetilde{f}_R:\V^{\otimes s}\to \V^{\otimes r}$, $R\subset \F_q^{s+k}$ in $\mathcal{C}$ which will serve as analogues of the maps $f_R$ defined in Sections \ref{sec:classical_endom}, \ref{sec:Deligne_def}. We will then prove that they satisfy the same relations on their compositions and tensor products as their counterparts in Section \ref{sec:Deligne_def}; this will allow us to construct a SM functor from the Deligne category into $\mathcal{C}$ in Section \ref{sec:univ_prop}.

The following definition is the first step toward defining morphisms $\widetilde{f}_R$. 

\begin{definition}\label{def:phi_R}
 Let $R \subset \F_q^r$ be an $\F_q$-linear subspace. 
 Let $d = \dim_{\F_q} R$. Let $B \in Mat_{d\times r}(\F_q)$ be a matrix of rank $d$ such that $Row(B)=R$.
 Consider the morphism
 \begin{align*}
&\phi_R: \V^{\otimes r} \to \triv\\
&\phi_R:\V^{\otimes r}\xrightarrow{\mu_{B}} \V^{\otimes d} \xrightarrow{(z^*)^{\otimes d}} \triv
 \end{align*}
Diagrammatically, we draw $\phi_R$ as follows:
 
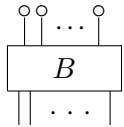
\begin{equation*}
\begin{tikzpicture}[anchorbase,scale=1.5]
\draw[-] (-0.5,-0.2)--(-0.5,0.2)--(0.5,0.2)--(0.5,-0.2)--(-0.5,-0.2);
\node at (-0,-0) {$B$};
\draw[-] (-0.4,-0.2)--(-0.4,-0.5);
\draw[-] (0.4,-0.2)--(0.4,-0.5);
\draw[-] (-0.2,0.2)--(-0.2,0.45);
\draw[-] (-0.35,0.2)--(-0.35,0.45);
\draw[-] (0.3,0.2)--(0.3,0.45);
\draw[-] (-0.3,-0.2)--(-0.3,-0.5);

\node at (-0.35,0.5) {$\circ$};
\node at (-0.2,0.5) {$\circ$};
\node at (0.3,0.5) {$\circ$};

\node at (-0.05,0.35) {$\cdot$};
\node at (0.05,0.35) {$\cdot$};
\node at (0.15,0.35) {$\cdot$};
\node at (-0.1,-0.4) {$\cdot$};
\node at (0.05,-0.4) {$\cdot$};
\node at (0.2,-0.4) {$\cdot$}; 
\end{tikzpicture}
\end{equation*}
\end{definition}

First of all, we check that $\phi_R$ is well-defined:
\begin{lemma}
 The morphism $\phi_R$ does not depend on the choice of the basis in $ R$.
\end{lemma}
\begin{proof}
Let $B'\in Mat_{d\times r}(\F_q)$ be any matrix whose rows are a basis of $R$. Then $B' = A \circ B$ for some $A \in GL_d(\F_q)$, so 
$$ (z^*)^{\otimes d} \circ \mu_{B'} = (z^*)^{\otimes d} \circ \mu_A \circ \mu_{B'}.$$
 By Corollary \ref{cor:regular_system_of_eq}, we have: $(z^*)^{\otimes d} \circ \mu_A = (z^*)^{\otimes d}$ and so $(z^*)^{\otimes d} \circ \mu_{B'} =(z^*)^{\otimes d} \circ \mu_{B}$, as required.
\end{proof}

\begin{example}
 Let $\mathcal{C} = \Rep(GL_n(\F_q))$, $\V:=\V_n$, and $R \subset \F_q^r$ be an $\F_q$-linear subspace. For any $v_1, \ldots, v_r \in V$, the morphism $\phi_R:\V_n^{\otimes r}\to \C$ takes $v_1\otimes \ldots \otimes v_r$ to $1$ iff $(v_1| \ldots| v_r) \in R^\perp_{\InnaA{\F_q^n}}$ (notation as in Section \ref{sec:classical_endom}) and to $0$ otherwise.
\end{example}
It turns out that the tensor product of $\phi_R$'s is again a morphism of the same form. The following is a direct consequence of Proposition \ref{prop:tensor_prod_mu_B}:

\begin{corollary}\label{cor:tensor_prod_phi_R}
 Let $R_1 \subset \F_q^{r_1}$ and $R_2\subset \F_q^{r_2}$ be $\F_q$-linear subspaces. 
 Let $d_i = \dim_{\F_q} R_i$ for $i=1,2$. Let $B^{(i)} \in Mat_{d_i\times r_i}(\F_q)$ be a matrix whose rows form a basis of $R_i$ 
 for $i=1,2$. Then $\phi_{R_1}\otimes \phi_{R_2} =\phi_{S}$ where $S = R_1\times R_2 \subset \F_q^{r_1+r_2}$; a basis for $S$ is given by rows of the block matrix $C = \begin{bmatrix}                                                                                                                                        B^{(1)} &0\\                                                                                                                            0 &B^{(2)}                                                                                                                                           \end{bmatrix}
$. 
\end{corollary}

\begin{example}
 For $\mathcal{C} = \Rep(GL_n(\F_q))$, $\V:=\V_n$, \InnaA{$V:=\F_q^n$}, this corollary states that for any $v_1, \ldots, v_{r_1}, w_1, \ldots, w_{r_2} \in V$, we have: $(v_1| \ldots| v_{r_1}| w_1| \ldots| w_{r_2}) \in (R_1\times R_2)^\perp_{\InnaA{V}}$ iff both $(v_1| \ldots| v_{r_1}) \in (R_1)^\perp_{\InnaA{V}}$ and $(w_1| \ldots| w_{r_2}) \in (R_2)^\perp_{\InnaA{V}}$.
\end{example}

\begin{definition}\label{def:isom_T}
 Let $s, k\geq 0$. Consider the isomorphism
 \begin{align*}
&T: \Hom_{\mathcal{C}}(\V^{\otimes s}, \V^{\otimes k}) \longrightarrow \Hom_{\mathcal{C}}(\V^{\otimes s}\otimes \V^{\otimes k}, \triv)\\
&T(f):\V^{\otimes s}\otimes \V^{\otimes k} \xrightarrow{f\otimes \id} \V^{\otimes k}\otimes \V^{\otimes k} \xrightarrow{ev_{V^{\otimes k}}} \triv,
 \end{align*}
\end{definition}

Let $s, k, l \geq 0$. Consider the operation $\circledast$ defined by 
\begin{align*}
&\circledast:\Hom_{\mathcal{C}}(\V^{\otimes s}\otimes \V^{\otimes k}, \triv) \otimes \Hom_{\mathcal{C}}(\V^{\otimes k}\otimes \V^{\otimes l}, \triv)\longrightarrow \Hom_{\mathcal{C}}(\V^{\otimes s}\otimes \V^{\otimes l}, \triv) \\
&\phi \circledast \psi:\V^{\otimes s}\otimes \V^{\otimes l} \xrightarrow{\id \otimes coev_{V^{\otimes k}} \otimes \id} \V^{\otimes s}\otimes \V^{\otimes k}\otimes \V^{\otimes k}\otimes \V^{\otimes l} \xrightarrow{\phi \otimes \psi} \triv.
\end{align*}

The following lemma is proved in \cite[Lemma 8.5(i)]{Del07}:
\begin{lemma}\label{lem:ast_vs_composition}
 Let $f: \V^{\otimes s}\to \V^{\otimes k}$ and $g:\V^{\otimes k}\to \V^{\otimes l}$. Then $T(g\circ f) = T(f) \circledast T(g)$.
\end{lemma}
Further properties of the isomorphism $T$ can be found in \cite[Lemma 8.5]{Del07}. We also have the following straightforward statement:
\begin{lemma}\label{lem:T_and_tensor_prod}
 Let $f: \V^{\otimes s_1}\to \V^{\otimes k_1}$ and $g:\V^{\otimes s_2}\to \V^{\otimes k_2}$. Then $T(g\otimes f) = T(g) \otimes T(f)$.
\end{lemma}

\begin{definition}\label{def:f_R_in_C}
Let $R \subset \F_q^{s+k}$. We define $\widetilde{f}_R:\V^{\otimes s}\to \V^{\otimes k}$ as $\widetilde{f}_R:=T^{-1}(\phi_R)$.
\end{definition}

\begin{example}

\mbox{}

 \begin{enumerate}
  \item  For $\mathcal{C} = \Rep(GL_n(\F_q))$, $\V:=\V_n$, the morphism $\widetilde{f}_R:=T^{-1}(\phi_R)$ is just $f_R: \V_n^{\otimes s}\to \V_n^{\otimes k}$ as defined in Section \ref{sec:classical_endom}:
  $$v_1\otimes \ldots\otimes v_s \longmapsto \sum_{\substack{(w_1|\ldots|w_k)\in V^{\times k} \\ (v_1| \ldots| v_s|w_1|\ldots|w_k) \in R^{\perp}_{\InnaA{V}} }} w_1 \otimes\ldots\otimes w_k$$ for any $v_1, \ldots, v_s \in V$.
  \item Similarly, for $\mathcal{C} := \kar{t}$ and $\V:=\V_t=[1]$, the morphism $\widetilde{f}_R:=T^{-1}(\phi_R)$ is just $f_R: [s]\to [k]$ as defined in Section \ref{sec:Deligne_def}.
 \end{enumerate}

\end{example}

As can be expected from this example, we have, in general:
\begin{lemma}\label{lem:tensor_prod_f}
 Let $R_1 \subset \F_q^{r_1+s_1}$ and $R_2\subset \F_q^{r_2+s_2}$ be $\F_q$-linear subspaces. 
 Let $d_i = \dim_{\F_q} R_i$ for $i=1,2$. Let $B^{(i)} \in Mat_{d_i\times r_i}(\F_q)$ be a matrix whose rows form a basis of $R_i$ 
 for $i=1,2$. Then $\widetilde{f}_{R_1}\otimes \widetilde{f}_{R_2} =\widetilde{f}_{S}$ where $S = R_1\times R_2 \subset \F_q^{r_1+r_2}$; a basis for $S$ is given by the block matrix $C = \begin{bmatrix}   
 B^{(1)} &0\\                                                        0 &B^{(2)}            
 \end{bmatrix}
$. 
\end{lemma}
\begin{proof}
 Since $T$ is an isomorphism, it is enough to check that 
 $$T\left(\widetilde{f}_{R_1}\otimes \widetilde{f}_{R_2}\right) =T\left(\widetilde{f}_{S}\right)$$
 But by Lemma \ref{lem:T_and_tensor_prod}, we have:
$$T\left(\widetilde{f}_{R_1}\otimes \widetilde{f}_{R_2}\right) =T\left(\widetilde{f}_{R_1}\right)\otimes T \left( \widetilde{f}_{R_2}\right) = \phi_{R_1} \otimes \phi_{R_2} = \phi_S= T\left(\widetilde{f}_{S}\right)$$
 (the equality $\phi_{R_1} \otimes \phi_{R_2} = \phi_S$ was proved in Corollary \ref{cor:tensor_prod_phi_R}).
 
\end{proof}

The following proposition generalizes Proposition \ref{prop:composition}:

\begin{proposition}\label{prop:composition_phi_R}
 Let $R \subset \F_q^{s+k}$ and $S \subset \F_q^{k+l}$ be $\F_q$-linear subspaces. Then $$\widetilde{f}_S \circ \widetilde{f}_R = (\dim \V)^{d(R, S)} \widetilde{f}_{S \star R}$$ where $d(R,S)$ and $S\star R$ are defined in Definition \ref{def:composition_star_d_R_S}.
\end{proposition}
\begin{proof}
By Lemma \ref{lem:ast_vs_composition}, it is enough to prove that 
\begin{equation}\label{eq:compos_phi_R}
 \phi_R \circledast \phi_S = (\dim \V)^{d(R, S)} \phi_{S \star R}.
\end{equation}

Let $B \in Mat_{d_1\times (s+k)}(\F_q)$ be a matrix whose rows form a basis of $R$ and let  
$C \in Mat_{d_2\times (k+l)}(\F_q)$ be a matrix whose rows form a basis of $S$.
 
We prove the equality \eqref{eq:compos_phi_R} in several steps, which we outline below. After that, we will provide the missing details for each step.
\begin{enumerate}
 \item[{\bf Step 1:}] We show that $$\phi_R \circledast \phi_S =  (z^*)^{\otimes (d_1+d_2)}\circ \mu_D \circ (\id_{\V^{\otimes s}} \otimes \eps^{\otimes k} \otimes \id_{\V^{\otimes l}})$$ for the matrix $D \in Mat_{(d_1+d_2) \times (s+k+l)}$ defined as $$D = 
 { \arraycolsep=5pt\left[\begin{array}{cccccc}                                                                                                                                       &&B &|&0& \\ \hline                                                                                                                      &0 &| &C &&   
 \end{array}                                                                                                                                         \right] }$$ (note that this is not a block matrix!). Then $$Row(D) = R\times_{\F_q^k} S =(R,0) + (0,S) \subset \F_q^{s+k+l},$$ though the rows of $D$ do not have to be linearly independent.
 
  \item[{\bf Step 2:}] We show that $(z^*)^{\otimes (d_1+d_2)}\circ \mu_D  = (z^*)^{\otimes (d_1+d_2)}\circ \mu_{D'} $ for any $D'\in Mat_{(d_1+d_2) \times (s+k+l)}$ which is row-equivalent to $D$ (that is, $Row(D) = Row(D')$). 
  
   \item[{\bf Step 3:}] In particular, the equality in Step 2 is true for a matrix $D'$ obtained from $D$ as follows: using elementary row operations, we can bring $D$ to a matrix of the form
  $$D' = \begin{bmatrix}                                                                                                                                                                                                                                                               &G&| &M &|&H&  \\\hline    
  &E&|&0 &|&F&  \\\hline                                                                                                                        &0&| &0 &|&0&                                                                                                                                            \end{bmatrix}$$
  Here the middle $k$ columns form a matrix in row echelon form, which we denote by $\begin{bmatrix}                                                                                                                                                                                                                                                               M   \\   
  0 \\                                                                                                                     0                                                                                                                                            \end{bmatrix}$
  where $M$ has full row rank.
The middle part $\begin{bmatrix}                                                                                                                                        &E&|&0 &|&F&                                                                                                                                             \end{bmatrix}$ has linearly independent rows whose middle $k$ coordinates are zero (we denote the number of such rows by $r$). The bottom part consists of zero rows. 
  
  We denote the $r\times (s+l)$ matrix $\begin{bmatrix}                                                                                                                                        &E&|&F&                                                                                                                                             \end{bmatrix}$ by $B\star C$.
  
  Here is how these new matrices relate to the original statement: the rows of $B\star C$ are clearly linearly independent and span the subspace $R\star S \subset \F_q^{s+l}$. The rank of $M$ is the dimension of the projection of $ R\times_{\F_q^k} S$ on the subspace $\F_q^k \subset \F_q^{s+k+l}$, so $k-rk(M) = d(R,S)$.
  
  \item[{\bf Step 4:}] We show that $$(z^*)^{\otimes (d_1+d_2)}\circ \mu_{D'} \circ (\id_{\V^{\otimes s}} \otimes \eps^{\otimes k} \otimes \id_{\V^{\otimes l}})=(\dim \V)^{d(R,S)} (z^*)^{\otimes r}\circ \mu_{B\star C}. $$
  The left hand side has been shown to be equal to $\phi_R \circledast \phi_S$, while the right hand side is $(\dim \V)^{d(R,S)} \phi_{R\star S}$, which proves the required statement.
\end{enumerate}
We now provide the missing details in each step:

\InnaA{{\bf Step 1:}}
 Recall that $\phi_R \circledast \phi_S = (\phi_R \otimes \phi_S) \circ \left(\id_{\V^{\otimes s}} \otimes coev_{\V^{\otimes k}} \otimes \id_{\V^{\otimes l}}\right)$. Diagrammatically, $\phi_R\circledast \phi_S$ can be drawn as
 
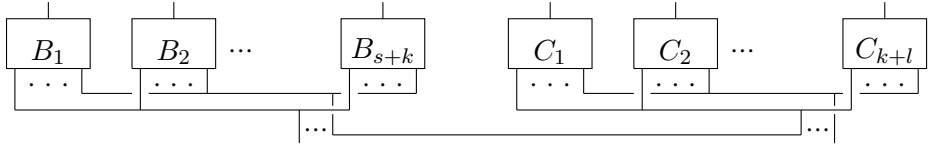
\begin{equation*}
\begin{tikzpicture}[anchorbase,scale=1.1]
%first part:
\draw[-] (-0.5,-0.2)--(-0.5,0.4)--(0.5,0.4)--(0.5,-0.2)--(-0.5,-0.2);
\node at (0,0) {${B_1}$};
\draw[-] (-0.4,-0.2)--(-0.4,-0.7);
\draw[-] (0,0.4)--(0,0.6);
\draw[-] (0.4,-0.2)--(0.4,-0.5);
\node at (-0.2,-0.4) {$\cdot$};
\node at (-0,-0.4) {$\cdot$};
\node at (0.2,-0.4) {$\cdot$};%%%

\draw[-] (1,-0.2)--(1,0.4)--(2,0.4)--(2,-0.2)--(1,-0.2);
\node at (1.5,0) {${B_2}$};
\draw[-] (1.1,-0.2)--(1.1,-0.7);
\draw[-] (1.5,0.4)--(1.5,0.6);
\draw[-] (1.9,-0.2)--(1.9,-0.5);
\node at (1.3,-0.4) {$\cdot$};
\node at (1.5,-0.4) {$\cdot$};
\node at (1.7,-0.4) {$\cdot$};%%%intermediate dots:
\node at (2.2,0) {$\cdot$};
\node at (2.3,0) {$\cdot$};
\node at (2.4,0) {$\cdot$};

\draw[-] (3.5,-0.2)--(3.5,0.4)--(4.5,0.4)--(4.5,-0.2)--(3.5,-0.2);
\node at (4,0) {${B_{s+k}}$};
\draw[-] (3.6,-0.2)--(3.6,-0.7);
\draw[-] (4,0.4)--(4,0.6);
\draw[-] (4.4,-0.2)--(4.4,-0.5);
\node at (3.8,-0.4) {$\cdot$};
\node at (4,-0.4) {$\cdot$};
\node at (4.2,-0.4) {$\cdot$};%%% Bottom part
\draw[-] (-0.4,-0.7)--(3.6,-0.7);
\draw[-] (0.4,-0.5)--(1, -0.5);
\draw[-] (1.2, -0.5)--(3.5, -0.5);
\draw[-] (3.7,-0.5)--(4.4,-0.5);
\draw[-] (3, -0.7)--(3,-1.1);
\draw[-] (3.4, -0.5)--(3.4, -0.65);
\draw[-] (3.4, -0.75)--(3.4, -1);
\node at (3.1,-0.95) {$\cdot$};
\node at (3.2,-0.95) {$\cdot$};
\node at (3.3,-0.95) {$\cdot$};

\draw[-] (5.5,-0.2)--(5.5,0.4)--(6.5,0.4)--(6.5,-0.2)--(5.5,-0.2);
\node at (6,0) {${C_1}$};
\draw[-] (5.6,-0.2)--(5.6,-0.7);
\draw[-] (6,0.4)--(6,0.6);
\draw[-] (6.4,-0.2)--(6.4,-0.5);
\node at (5.8,-0.4) {$\cdot$};
\node at (6,-0.4) {$\cdot$};
\node at (6.2,-0.4) {$\cdot$};%%%

\draw[-] (7,-0.2)--(7,0.4)--(8,0.4)--(8,-0.2)--(7,-0.2);
\node at (7.5,0) {${C_2}$};
\draw[-] (7.1,-0.2)--(7.1,-0.7);
\draw[-] (7.5,0.4)--(7.5,0.6);
\draw[-] (7.9,-0.2)--(7.9,-0.5);
\node at (7.3,-0.4) {$\cdot$};
\node at (7.5,-0.4) {$\cdot$};
\node at (7.7,-0.4) {$\cdot$};%%%intermediate dots:
\node at (8.2,0) {$\cdot$};
\node at (8.3,0) {$\cdot$};
\node at (8.4,0) {$\cdot$};

%%% rightmost part:
\draw[-] (9.5,-0.2)--(9.5,0.4)--(10.5,0.4)--(10.5,-0.2)--(9.5,-0.2);
\node at (10,0) {${C_{k+l}}$};
\draw[-] (9.6,-0.2)--(9.6,-0.7);
\draw[-] (10,0.4)--(10,0.6);
\draw[-] (10.4,-0.2)--(10.4,-0.5);
\node at (9.8,-0.4) {$\cdot$};
\node at (10,-0.4) {$\cdot$};
\node at (10.2,-0.4) {$\cdot$};%%% Bottom part
\draw[-] (5.6,-0.7)--(9.6,-0.7);
\draw[-] (6.4,-0.5)--(7, -0.5);
\draw[-] (7.2, -0.5)--(9.5, -0.5);
\draw[-] (9.7,-0.5)--(10.4,-0.5);
\draw[-] (9, -0.7)--(9,-1);
\draw[-] (9.4, -0.5)--(9.4, -0.65);
\draw[-] (9.4, -0.75)--(9.4, -1.1);
\node at (9.1,-0.95) {$\cdot$};
\node at (9.2,-0.95) {$\cdot$};
\node at (9.3,-0.95) {$\cdot$};
% connection
\draw[-] (3.4, -1)--(9, -1);
\end{tikzpicture} 
\end{equation*}
where we have $s$ string bottom endpoints on the left, $l$ string bottom endpoints on the right, and $k$ ``cups'' in between.
This diagram can be drawn alternatively as
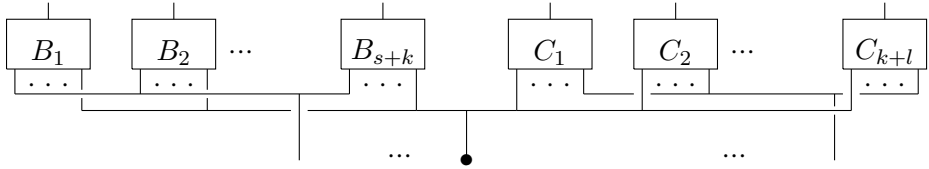
\begin{equation*}
\begin{tikzpicture}[anchorbase,scale=1.1]
%first part:
\draw[-] (-0.5,-0.2)--(-0.5,0.4)--(0.5,0.4)--(0.5,-0.2)--(-0.5,-0.2);
\node at (0,0) {${B_1}$};
\draw[-] (-0.4,-0.2)--(-0.4,-0.5);
\draw[-] (0,0.4)--(0,0.6);
\draw[-] (0.4,-0.2)--(0.4,-.45);
\draw[-] (0.4,-0.55)--(0.4,-.7);
\node at (-0.2,-0.4) {$\cdot$};
\node at (-0,-0.4) {$\cdot$};
\node at (0.2,-0.4) {$\cdot$};%%%

\draw[-] (1,-0.2)--(1,0.4)--(2,0.4)--(2,-0.2)--(1,-0.2);
\node at (1.5,0) {${B_2}$};
\draw[-] (1.1,-0.2)--(1.1,-0.5);
\draw[-] (1.5,0.4)--(1.5,0.6);
\draw[-] (1.9,-0.2)--(1.9,-.45);
\draw[-] (1.9, -0.55)--(1.9,-.7);
\node at (1.3,-0.4) {$\cdot$};
\node at (1.5,-0.4) {$\cdot$};
\node at (1.7,-0.4) {$\cdot$};%%%intermediate dots:
\node at (2.2,0) {$\cdot$};
\node at (2.3,0) {$\cdot$};
\node at (2.4,0) {$\cdot$};

\draw[-] (3.5,-0.2)--(3.5,0.4)--(4.5,0.4)--(4.5,-0.2)--(3.5,-0.2);
\node at (4,0) {${B_{s+k}}$};
\draw[-] (3.6,-0.2)--(3.6,-0.5);
\draw[-] (4,0.4)--(4,0.6);
\draw[-] (4.4,-0.2)--(4.4,-.7);
\node at (3.8,-0.4) {$\cdot$};
\node at (4,-0.4) {$\cdot$};
\node at (4.2,-0.4) {$\cdot$};%%% Bottom part
\draw[-] (-0.4,-0.5)--(3.6,-0.5);
\draw[-] (0.4,-.7)--(2.9, -.7);
\draw[-] (3.1, -.7)--(4.4,-.7);
\draw[-] (3, -0.5)--(3,-1.3);

\node at (4.1,-1.25) {$\cdot$};
\node at (4.2,-1.25) {$\cdot$};
\node at (4.3,-1.25) {$\cdot$};

\draw[-] (5.5,-0.2)--(5.5,0.4)--(6.5,0.4)--(6.5,-0.2)--(5.5,-0.2);
\node at (6,0) {${C_1}$};
\draw[-] (5.6,-0.2)--(5.6,-0.7);
\draw[-] (6,0.4)--(6,0.6);
\draw[-] (6.4,-0.2)--(6.4,-0.5);
\node at (5.8,-0.4) {$\cdot$};
\node at (6,-0.4) {$\cdot$};
\node at (6.2,-0.4) {$\cdot$};%%%

\draw[-] (7,-0.2)--(7,0.4)--(8,0.4)--(8,-0.2)--(7,-0.2);
\node at (7.5,0) {${C_2}$};
\draw[-] (7.1,-0.2)--(7.1,-0.7);
\draw[-] (7.5,0.4)--(7.5,0.6);
\draw[-] (7.9,-0.2)--(7.9,-0.5);
\node at (7.3,-0.4) {$\cdot$};
\node at (7.5,-0.4) {$\cdot$};
\node at (7.7,-0.4) {$\cdot$};%%%intermediate dots:
\node at (8.2,0) {$\cdot$};
\node at (8.3,0) {$\cdot$};
\node at (8.4,0) {$\cdot$};

%%% rightmost part:
\draw[-] (9.5,-0.2)--(9.5,0.4)--(10.5,0.4)--(10.5,-0.2)--(9.5,-0.2);
\node at (10,0) {${C_{k+l}}$};
\draw[-] (9.6,-0.2)--(9.6,-0.7);
\draw[-] (10,0.4)--(10,0.6);
\draw[-] (10.4,-0.2)--(10.4,-0.5);
\node at (9.8,-0.4) {$\cdot$};
\node at (10,-0.4) {$\cdot$};
\node at (10.2,-0.4) {$\cdot$};%%% Bottom part
\draw[-] (5.6,-0.7)--(9.6,-0.7);
\draw[-] (6.4,-0.5)--(7, -0.5);
\draw[-] (7.2, -0.5)--(9.5, -0.5);
\draw[-] (9.7,-0.5)--(10.4,-0.5);
% \draw[-] (9, -0.7)--(9,-1);
\draw[-] (9.4, -0.5)--(9.4, -0.65);
\draw[-] (9.4, -0.75)--(9.4, -1.3);
\node at (8.1,-1.25) {$\cdot$};
\node at (8.2,-1.25) {$\cdot$};
\node at (8.3,-1.25) {$\cdot$};

\draw[-] (5, -.7)--(5, -1.3);
\node at (5, -1.3) {$\bullet$};

% connection
\draw[-] (4.4,-0.7)--(5.6, -0.7);
\end{tikzpicture} 
\end{equation*}
Here in the bottom we have $s$ strands on the left, $l$ strands on the right, and $k$ ``tails'' of the form $\begin{tikzpicture}[anchorbase,scale=1]
\draw[-] (0, 0.5) --(0,0);
\node at (0, 0) {$\bullet$};
\end{tikzpicture}  $ in the middle (these stand for $\eps^{\otimes k}$). This shows that the equality

 $$\phi_R \circledast \phi_S =  (z^*)^{\otimes (d_1+d_2)}\circ \mu_D \circ (\id_{\V^{\otimes s}} \otimes \eps^{\otimes k} \otimes \id_{\V^{\otimes l}})$$ 
 holds.

\InnaA{{\bf Step 2:}} We need to show that $(z^*)^{\otimes d}\circ \mu_X  = (z^*)^{\otimes d}\circ \mu_{X'} $ for any matrices $X$, $X'$ of same size (with $d$ rows) which are row-equivalent. Indeed, for any two such matrices there exists $A \in GL_d(\F_q)$ such that $X'=AX$, so by Proposition \ref{prop:compos_mu_A}, we have: $\mu_{X'} = \mu_A \circ \mu_X$. Thus
  \begin{align*}
  (z^*)^{\otimes d}\circ \mu_{X'} =  (z^*)^{\otimes d}\circ \mu_A \circ \mu_X  = (z^*)^{\otimes d}\circ \mu_X
  \end{align*}
where the last equality is due to Corollary \ref{cor:regular_system_of_eq}.
  
\InnaA{{\bf Step 3:}} In this step there is nothing to prove. 
  
\InnaA{{\bf Step 4:}}
  Let us denote by $D'_1, D'_2, \ldots, D'_{d_1+d_2}$ the rows of the matrix $D'$, and let $$ \widetilde{M} := \begin{bmatrix}                                                                                                                                        &G&|&M &|&H&                                                                                                                                             \end{bmatrix}.$$ We will denote by $\widetilde{r}$ the number of rows in $\widetilde{M}$ (so $\widetilde{r} = rk(\widetilde{M})=rk(M)=k-d(R, S)$). 
  
  Consider the morphism $(z^*)^{\otimes (d_1+d_2)}\circ \mu_{D'} \circ (\id_{\V^{\otimes s}} \otimes \eps^{\otimes k} \otimes \id_{\V^{\otimes l}})$. Diagrammatically, we can draw it as follows:
  
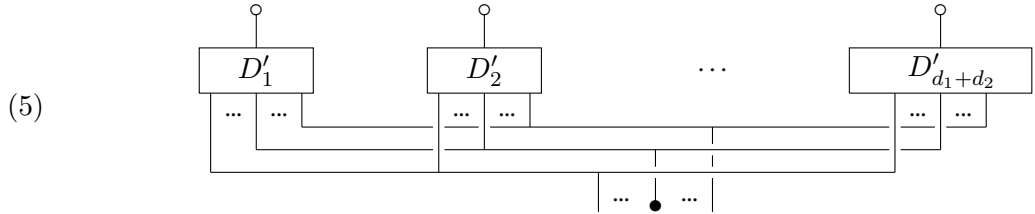
\begin{equation}\label{eq:Step_4_compos_proof}
\begin{tikzpicture}[anchorbase,scale=1.5]%leftmost part:
\draw[-] (-0.5,-0.2)--(-0.5,0.2)--(0.5,0.2)--(0.5,-0.2)--(-0.5,-0.2);
\node at (0,0) {${D'_1}$};

\draw[-] (-0.4,-0.2)--(-0.4,-0.9);
\draw[-] (0.4,-0.2)--(0.4,-0.5);
\draw[-] (0,-0.2)--(0,-0.7);

\draw[-] (0,0.2)--(0,0.5);
\node at (0,0.53) {$\circ$};

\node at (-0.25,-0.4) {$\cdot$};
\node at (-0.2,-0.4) {$\cdot$};
\node at (-0.15,-0.4) {$\cdot$};

\node at (0.25,-0.4) {$\cdot$};
\node at (0.2,-0.4) {$\cdot$};
\node at (0.15,-0.4) {$\cdot$};

%%%
\draw[-] (1.5,-0.2)--(1.5,0.2)--(2.5,0.2)--(2.5,-0.2)--(1.5,-0.2);
\node at (2,0) {${D'_2}$};

\draw[-] (1.6,-0.2)--(1.6,-0.9);
\draw[-] (2.4,-0.2)--(2.4,-0.5);
\draw[-] (2,-0.2)--(2,-0.7);

\draw[-] (2,0.2)--(2,0.5);
\node at (2,0.53) {$\circ$};

\node at (1.85,-0.4) {$\cdot$};
\node at (1.8,-0.4) {$\cdot$};
\node at (1.75,-0.4) {$\cdot$};

\node at (2.15,-0.4) {$\cdot$};
\node at (2.2,-0.4) {$\cdot$};
\node at (2.25,-0.4) {$\cdot$};
%%%intermediate dots:
\node at (3.9,0) {$\cdot$};
\node at (4,0) {$\cdot$};
\node at (4.1,0) {$\cdot$};%%% rightmost part:
\draw[-] (5.2,-0.2)--(5.2,0.2)--(6.8,0.2)--(6.8,-0.2)--(5.2,-0.2);
\node at (6.1,0) {${D'_{d_1+d_2}}$};

\draw[-] (5.6,-0.2)--(5.6,-0.9);
\draw[-] (6.4,-0.2)--(6.4,-0.5);
\draw[-] (6,-0.2)--(6,-0.7);

\draw[-] (6,0.2)--(6,0.5);
\node at (6,0.53) {$\circ$};

\node at (5.75,-0.4) {$\cdot$};
\node at (5.8,-0.4) {$\cdot$};
\node at (5.85,-0.4) {$\cdot$};

\node at (6.15,-0.4) {$\cdot$};
\node at (6.2,-0.4) {$\cdot$};
\node at (6.25,-0.4) {$\cdot$};

%%% Bottom part
\draw[-] (-0.4,-0.9)--(5.6,-0.9);%bottom horizontal line

%middle horizontal
\draw[-] (0,-0.7)--(1.55, -0.7);
\draw[-] (1.65,-0.7)--(5.55, -0.7);
\draw[-] (5.65,-0.7)--(6,-0.7);

%top horizontal line
\draw[-] (0.4,-0.5)--(1.55, -0.5);
\draw[-] (1.65, -0.5)--(1.95, -0.5);
\draw[-] (2.05, -0.5)--(5.55, -0.5);
\draw[-] (5.65,-0.5)--(5.95,-0.5);
\draw[-] (6.05,-0.5)--(6.4,-0.5);

\draw[-] (3, -0.9)--(3,-1.25);

\draw[-] (3.5, -0.7)--(3.5, -0.85);
\draw[-] (3.5, -0.95)--(3.5, -1.2);

\node at (3.5,-1.2) {$\bullet$};

\draw[-] (4, -0.5)--(4, -0.65);
\draw[-] (4, -0.75)--(4, -0.85);
\draw[-] (4, -0.95)--(4, -1.25);

%bottom dots
\node at (3.15,-1.15) {$\cdot$};
\node at (3.2,-1.15) {$\cdot$};
\node at (3.25,-1.15) {$\cdot$};

%bottom dots
\node at (3.75,-1.15) {$\cdot$};
\node at (3.8,-1.15) {$\cdot$};
\node at (3.85,-1.15) {$\cdot$};
\end{tikzpicture}
\end{equation}

Again, in the bottom of the diagram we have $s$ strands on the left, $l$ strands on the right, and $k$ ``tails'' of the form $\begin{tikzpicture}[anchorbase,scale=1]
\draw[-] (0, 0.5) --(0,0);
\node at (0, 0) {$\bullet$};
\end{tikzpicture}  $ in the middle (these stand for $\eps^{\otimes k}$).

Let us look at each of the rows $D'_i$ separately. For any row $D'_i$ having the form $[a_1, \ldots, a_s, 0, \ldots, 0, a'_1, \ldots a'_l]$, set $$D''_i := [a_1, \ldots, a_s, a'_1, \ldots a'_l].$$ 

By the properties of the zero map $z$, we have: $\dot{+}\circ(\id \otimes z) =  \dot{+}\circ(z\otimes \id) = \id $ (\ref{rel:F_q_lin_zero}). So we obtain: $$\mu_{D'_i} = \mu_{D''_i} \,\circ  \,\left(\id_{\V^{\otimes s}}\otimes (\eps^*)^{\otimes k} \otimes \id_{\V^{\otimes l}}  \right).$$ Diagrammatically, this is drawn as 
 
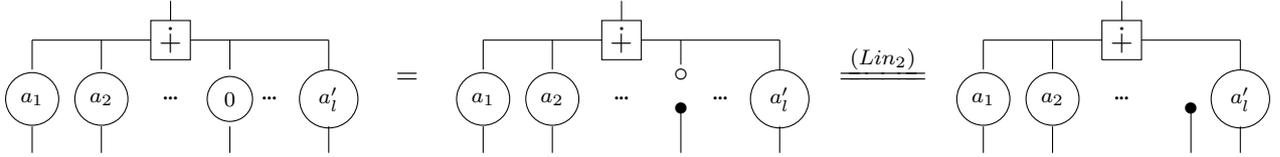
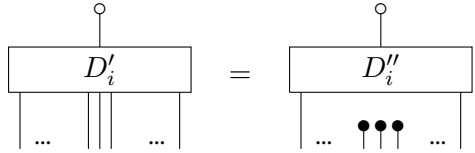
\begin{equation*}
 \begin{tikzpicture}[anchorbase,scale=1.3]
\draw[-] (-1.4,0)--(-0.2,0);
\draw[-] (0.2,0)--(1.6,0);
\draw[-] (-0.2,-0.2)--(-0.2,0.2)--(0.2,0.2)--(0.2,-0.2)--(-0.2,-0.2);
\node at (-0,-0) {$\dot{+}$};
\draw[-] (0,0.2)--(0,0.4);

\draw[-] (-1.4,0)--(-1.4,-0.3);
\draw[-] (-0.7,0)--(-0.7,-0.3);
\draw[-] (0.6,0)--(0.6,-0.3);
\draw[-] (1.6,0)--(1.6,-0.3);

\node at (-0.05,-0.6) {$\cdot$};
\node at (0,-0.6) {$\cdot$};
\node at (0.05,-0.6) {$\cdot$};

\node at (0.95,-0.6) {$\cdot$};
\node at (1,-0.6) {$\cdot$};
\node at (1.05,-0.6) {$\cdot$};

\draw[-] (-1.4,-0.88)--(-1.4,-1.15);
\draw[-] (-0.7,-0.88)--(-0.7,-1.15);
\draw[-] (0.6,-0.88)--(0.6,-1.15);
\draw[-] (1.6,-0.88)--(1.6,-1.15);

\node[draw,circle] at (-1.4,-0.6) {$ {\scriptstyle a_1}$};
\node[draw,circle] at (-0.7,-0.6) {${\scriptstyle a_2}$};
\node[draw,circle] at (0.6,-0.6) {$ \scriptstyle 0$};
\node[draw,circle]  at (1.6,-0.6) {${\scriptstyle a'_l}$};
\end{tikzpicture} \quad = \quad 
 \begin{tikzpicture}[anchorbase,scale=1.3]
\draw[-] (-1.4,0)--(-0.2,0);
\draw[-] (0.2,0)--(1.6,0);
\draw[-] (-0.2,-0.2)--(-0.2,0.2)--(0.2,0.2)--(0.2,-0.2)--(-0.2,-0.2);
\node at (-0,-0) {$\dot{+}$};
\draw[-] (0,0.2)--(0,0.4);

\draw[-] (-1.4,0)--(-1.4,-0.3);
\draw[-] (-0.7,0)--(-0.7,-0.3);

\draw[-] (0.6,0)--(0.6,-0.25);
\draw[-] (1.6,0)--(1.6,-0.3);

\node at (-0.05,-0.6) {$\cdot$};
\node at (0,-0.6) {$\cdot$};
\node at (0.05,-0.6) {$\cdot$};

\node at (0.95,-0.6) {$\cdot$};
\node at (1,-0.6) {$\cdot$};
\node at (1.05,-0.6) {$\cdot$};

\draw[-] (-1.4,-0.88)--(-1.4,-1.15);
\draw[-] (-0.7,-0.88)--(-0.7,-1.15);
\draw[-] (0.6,-0.7)--(0.6,-1.15);
\draw[-] (1.6,-0.88)--(1.6,-1.15);

\node[draw,circle] at (-1.4,-0.6) {$ {\scriptstyle a_1}$};
\node[draw,circle] at (-0.7,-0.6) {${\scriptstyle a_2}$};

\node[draw,circle]  at (1.6,-0.6) {${\scriptstyle a'_l}$};

\node at (0.6,-0.35) {$\circ$};

\node at (0.6,-0.7) {$\bullet$};
\end{tikzpicture} 
\quad \InnaA{ \xlongequal{\text{\ref{rel:F_q_lin_zero}}}} \quad
  \begin{tikzpicture}[anchorbase,scale=1.3]
\draw[-] (-1.4,0)--(-0.2,0);
\draw[-] (0.2,0)--(1.2,0);
\draw[-] (-0.2,-0.2)--(-0.2,0.2)--(0.2,0.2)--(0.2,-0.2)--(-0.2,-0.2);
\node at (-0,-0) {$\dot{+}$};
\draw[-] (0,0.2)--(0,0.4);

\draw[-] (-1.4,0)--(-1.4,-0.3);
\draw[-] (-0.7,0)--(-0.7,-0.3);

\draw[-] (1.2,0)--(1.2,-0.3);

\node at (-0.05,-0.6) {$\cdot$};
\node at (0,-0.6) {$\cdot$};
\node at (0.05,-0.6) {$\cdot$};

\draw[-] (-1.4,-0.88)--(-1.4,-1.15);
\draw[-] (-0.7,-0.88)--(-0.7,-1.15);

\draw[-] (1.2,-0.88)--(1.2,-1.15);

\node[draw,circle] at (-1.4,-0.6) {$ {\scriptstyle a_1}$};
\node[draw,circle] at (-0.7,-0.6) {${\scriptstyle a_2}$};

\node[draw,circle]  at (1.2,-0.6) {${\scriptstyle a'_l}$};

\draw[-] (0.7, -0.7)--(0.7, -1.15);
\node at (0.7,-0.7) {$\bullet$};
\end{tikzpicture} 
\end{equation*}

Hence we obtain:

 \begin{equation*}
\begin{tikzpicture}[anchorbase,scale=1.5]%leftmost part:
\draw[-] (-0.8,-0.2)--(-0.8,0.2)--(0.8,0.2)--(0.8,-0.2)--(-0.8,-0.2);
\node at (0,0) {${D'_i}$};

\draw[-] (-0.7,-0.2)--(-0.7,-0.7);
\draw[-] (0.7,-0.2)--(0.7,-0.7);
\draw[-] (0,-0.2)--(0,-0.7);
\draw[-] (0.1,-0.2)--(0.1,-0.7);
\draw[-] (-0.1,-0.2)--(-0.1,-0.7);

\draw[-] (0,0.2)--(0,0.5);
\node at (0,0.53) {$\circ$};

\node at (-0.45,-0.65) {$\cdot$};
\node at (-0.5,-0.65) {$\cdot$};
\node at (-0.55,-0.65) {$\cdot$};

\node at (0.45,-0.65) {$\cdot$};
\node at (0.5,-0.65) {$\cdot$};
\node at (0.55,-0.65) {$\cdot$};
\end{tikzpicture} \quad =  \quad \begin{tikzpicture}[anchorbase,scale=1.5]%leftmost part:
\draw[-] (-0.8,-0.2)--(-0.8,0.2)--(0.8,0.2)--(0.8,-0.2)--(-0.8,-0.2);
\node at (0,0) {${D''_i}$};

\draw[-] (-0.7,-0.2)--(-0.7,-0.7);
\draw[-] (0.7,-0.2)--(0.7,-0.7);

\draw[-] (0,-0.5)--(0,-0.7);
\node at (0,-0.5) {$\bullet$};

\draw[-] (0.15,-0.5)--(0.15,-0.7);
\node at (0.15,-0.5) {$\bullet$};

\draw[-] (-0.15,-0.5)--(-0.15,-0.7);
\node at (-0.15,-0.5) {$\bullet$};

\draw[-] (0,0.2)--(0,0.5);
\node at (0,0.53) {$\circ$};

\node at (-0.45,-0.65) {$\cdot$};
\node at (-0.5,-0.65) {$\cdot$};
\node at (-0.55,-0.65) {$\cdot$};

\node at (0.45,-0.65) {$\cdot$};
\node at (0.5,-0.65) {$\cdot$};
\node at (0.55,-0.65) {$\cdot$};
\end{tikzpicture} 
\end{equation*}
On the right hand side of this equation, in the bottom we have $s$ strands on the left, $l$ strands on the right, and $k$ ``tails'' of the form $\begin{tikzpicture}[anchorbase,scale=1]
\draw[-] (0, 0.3) --(0,0);
\node at (0, 0.3) {$\bullet$};
\end{tikzpicture}  $ in the middle (these stand for $(\eps^*)^{\otimes k}$).
%   We show that $$(z^*)^{\otimes (d_1+d_2)}\circ \mu_{D'} \circ (\id_{\V^{\otimes s}} \otimes \eps^{\otimes k} \otimes \id_{\V^{\otimes l}})=(\dim \V)^{d(R,S)} (z^*)^{\otimes r}\circ \mu_{B\star C}. $$

Now, recall that $\eps^*$ is a counit of $m^*$ (\ref{itm:str_rel_diag_bialg}): $$(\eps^*\otimes \id) \circ m^* = (\id \otimes\eps^* ) \circ m^* = \id.$$ So for any $p\geq 2$, the composition of
$(m^*)^{it}: \V\to \V^{\otimes p}$ and $\id \otimes \eps^*\otimes \id:  \V^{\otimes p} \to  \V^{\otimes (p-1)}$ gives $(m^*)^{it}: \V\to \V^{\otimes p-1}$ (no matter on which factor $\eps^*$ acts).

Diagrammatically, this means that
\begin{equation}\label{eq:Step_4_comult_counit}
\begin{tikzpicture}[anchorbase,scale=1.5]
\draw[-] (-0.8,0)--(0.8,0);
\draw[-] (-0.8,0)--(-0.8,0.3);
\draw[-] (-0.7,0)--(-0.7,0.3);

\draw[-] (0.05,0)--(0.05,0.3);

\draw[-] (0.2,0)--(0.2,0.3);
\node at (0.2,0.3) {$\bullet$};

\draw[-] (0.35,0)--(0.35,0.3);

\draw[-] (0.8,0)--(0.8,0.3);

\draw[-] (0,0)--(0,-0.3);

\node at (-0.15,0.15) {$\cdot$};
\node at (-0.2,0.15) {$\cdot$};
\node at (-0.25,0.15) {$\cdot$};
\node at (0.55,0.15) {$\cdot$};
\node at (0.6,0.15) {$\cdot$};
\node at (0.65,0.15) {$\cdot$};
\end{tikzpicture} \quad = \quad \begin{tikzpicture}[anchorbase,scale=1.5]
\draw[-] (-0.8,0)--(0.8,0);
\draw[-] (-0.8,0)--(-0.8,0.3);
\draw[-] (-0.7,0)--(-0.7,0.3);

\draw[-] (0.05,0)--(0.05,0.3);
\draw[-] (0.35,0)--(0.35,0.3);

\draw[-] (0.8,0)--(0.8,0.3);
\draw[-] (0,0)--(0,-0.3);

\node at (-0.15,0.15) {$\cdot$};
\node at (-0.2,0.15) {$\cdot$};
\node at (-0.25,0.15) {$\cdot$};
\node at (0.55,0.15) {$\cdot$};
\node at (0.6,0.15) {$\cdot$};
\node at (0.65,0.15) {$\cdot$};
\end{tikzpicture}
\end{equation}

Together, these computations mean that for any row $D'_i$ belonging to middle part $\begin{bmatrix}                                                                                                                                        &E&|&0 &|&F&                                                                                                                                             \end{bmatrix}$ of the matrix $D'$, we have:
 
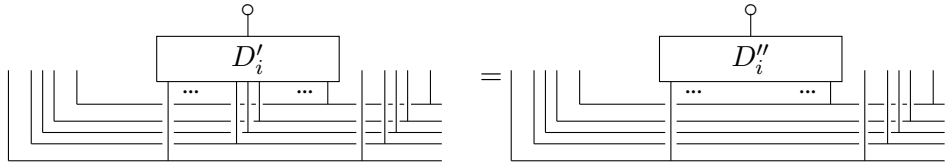
\begin{equation*}
\begin{tikzpicture}[anchorbase,scale=1.5]%center
\draw[-] (-0.8,-0.1)--(-0.8,0.3)--(0.8,0.3)--(0.8,-0.1)--(-0.8,-0.1);
\node at (0,0.1) {${D'_i}$};

\draw[-] (-0.7,-0.1)--(-0.7,-0.8);
\draw[-] (0.7,-0.1)--(0.7,-0.3);

\draw[-] (0.1,-0.1)--(0.1,-0.45);
\draw[-] (0,-0.1)--(0,-0.55);
\draw[-] (-0.1,-0.1)--(-0.1,-0.65);

\draw[-] (0,0.3)--(0,0.5);
\node at (0,0.53) {$\circ$};

\node at (-0.45,-0.2) {$\cdot$};
\node at (-0.5,-0.2) {$\cdot$};
\node at (-0.55,-0.2) {$\cdot$};

\node at (0.45,-0.2) {$\cdot$};
\node at (0.5,-0.2) {$\cdot$};
\node at (0.55,-0.2) {$\cdot$};

%bottom parts 

\draw[-] (-1.5,-0.3)--(-0.75, -0.3);
\draw[-] (-0.65, -0.3)--(-0.15,-0.3);
\draw[-] (-0.07,-0.3)--(-0.03,-0.3);
\draw[-] (0.03,-0.3)--(0.07,-0.3);
\draw[-] (0.13,-0.3)--(0.95,-0.3);
\draw[-] (1.03,-0.3)--(1.17,-0.3);
\draw[-] (1.23,-0.3)--(1.27,-0.3);
\draw[-] (1.33,-0.3)--(1.37, -0.3);
\draw[-] (1.43,-0.3)--(1.7,-0.3);
\draw[-] (1.6,-0.3)--(1.6,0);
\draw[-] (-1.5,0)--(-1.5,-0.3);

\draw[-] (-1.7,-0.45)--(-0.75, -0.45);
\draw[-] (-0.65, -0.45)--(-0.15,-0.45);
\draw[-] (-0.07,-0.45)--(-0.03,-0.45);
\draw[-] (0.05,-0.45)--(0.95,-0.45);
\draw[-] (1.03,-0.45)--(1.17,-0.45);
\draw[-] (1.23,-0.45)--(1.27,-0.45);
\draw[-] (1.33,-0.45)--(1.7,-0.45);
\draw[-] (1.4,-0.45)--(1.4,0);
\draw[-] (-1.7,0)--(-1.7,-0.45);

\draw[-] (-1.8,-0.55)--(-0.75, -0.55);
\draw[-] (-0.65, -0.55)--(-0.15,-0.55);
\draw[-] (-0.05,-0.55)--(0.95,-0.55);
\draw[-] (1.03,-0.55)--(1.17,-0.55);
\draw[-] (1.23,-0.55)--(1.7,-0.55);
\draw[-] (1.3,-0.55)--(1.3,0);
\draw[-] (-1.8,0)--(-1.8,-0.55);

\draw[-] (-1.9,-0.65)--(-0.75, -0.65);
\draw[-] (-0.65, -0.65)--(0.95,-0.65);
\draw[-] (1.03,-0.65)--(1.7,-0.65);
\draw[-] (1.2,-0.65)--(1.2,0);
\draw[-] (-1.9,0)--(-1.9,-0.65);

\draw[-] (-2.1,-0.8)--(1.7,-0.8);
\draw[-] (-2.1,0)--(-2.1,-0.8);
\draw[-] (1,-0.8)--(1,0);
\end{tikzpicture} \quad =  
\begin{tikzpicture}[anchorbase,scale=1.5]%leftmost part:
\draw[-] (-0.8,-0.1)--(-0.8,0.3)--(0.8,0.3)--(0.8,-0.1)--(-0.8,-0.1);
\node at (0,0.1) {${D''_i}$};

\draw[-] (-0.7,-0.1)--(-0.7,-0.8);
\draw[-] (0.7,-0.1)--(0.7,-0.3);

\draw[-] (0,0.3)--(0,0.5);
\node at (0,0.53) {$\circ$};

\node at (-0.45,-0.2) {$\cdot$};
\node at (-0.5,-0.2) {$\cdot$};
\node at (-0.55,-0.2) {$\cdot$};

\node at (0.45,-0.2) {$\cdot$};
\node at (0.5,-0.2) {$\cdot$};
\node at (0.55,-0.2) {$\cdot$};

%bottom parts 

\draw[-] (-1.5,-0.3)--(-0.75, -0.3);
\draw[-] (-0.65, -0.3)--(0.95,-0.3);
\draw[-] (1.03,-0.3)--(1.17,-0.3);
\draw[-] (1.23,-0.3)--(1.27,-0.3);
\draw[-] (1.33,-0.3)--(1.37, -0.3);
\draw[-] (1.43,-0.3)--(1.7,-0.3);
\draw[-] (1.6,-0.3)--(1.6,0);
\draw[-] (-1.5,0)--(-1.5,-0.3);

\draw[-] (-1.7,-0.45)--(-0.75, -0.45);
\draw[-] (-0.65, -0.45)--(0.95,-0.45);
\draw[-] (1.03,-0.45)--(1.17,-0.45);
\draw[-] (1.23,-0.45)--(1.27,-0.45);
\draw[-] (1.33,-0.45)--(1.7,-0.45);
\draw[-] (1.4,-0.45)--(1.4,0);
\draw[-] (-1.7,0)--(-1.7,-0.45);

\draw[-] (-1.8,-0.55)--(-0.75, -0.55);
\draw[-] (-0.65, -0.55)--(0.95,-0.55);
\draw[-] (1.03,-0.55)--(1.17,-0.55);
\draw[-] (1.23,-0.55)--(1.7,-0.55);
\draw[-] (1.3,-0.55)--(1.3,0);
\draw[-] (-1.8,0)--(-1.8,-0.55);

\draw[-] (-1.9,-0.65)--(-0.75, -0.65);
\draw[-] (-0.65, -0.65)--(0.95,-0.65);
\draw[-] (1.03,-0.65)--(1.7,-0.65);
\draw[-] (1.2,-0.65)--(1.2,0);
\draw[-] (-1.9,0)--(-1.9,-0.65);

\draw[-] (-2.1,-0.8)--(1.7,-0.8);
\draw[-] (-2.1,0)--(-2.1,-0.8);
\draw[-] (1,-0.8)--(1,0);
\end{tikzpicture}
\end{equation*}

Similarly, for any zeroes row $D'_i = [0, \ldots, 0]$ of $D'$, we have:

 
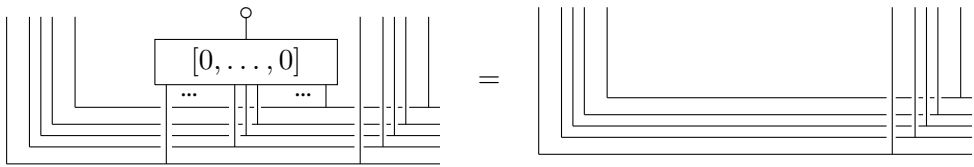
\begin{equation*}
\begin{tikzpicture}[anchorbase,scale=1.5]%center
\draw[-] (-0.8,-0.1)--(-0.8,0.3)--(0.8,0.3)--(0.8,-0.1)--(-0.8,-0.1);
\node at (0,0.1) {${[0,\ldots,0]}$};

\draw[-] (-0.7,-0.1)--(-0.7,-0.8);
\draw[-] (0.7,-0.1)--(0.7,-0.3);

\draw[-] (0.1,-0.1)--(0.1,-0.45);
\draw[-] (0,-0.1)--(0,-0.55);
\draw[-] (-0.1,-0.1)--(-0.1,-0.65);

\draw[-] (0,0.3)--(0,0.5);
\node at (0,0.53) {$\circ$};

\node at (-0.45,-0.2) {$\cdot$};
\node at (-0.5,-0.2) {$\cdot$};
\node at (-0.55,-0.2) {$\cdot$};

\node at (0.45,-0.2) {$\cdot$};
\node at (0.5,-0.2) {$\cdot$};
\node at (0.55,-0.2) {$\cdot$};

%bottom parts 

\draw[-] (-1.5,-0.3)--(-0.75, -0.3);
\draw[-] (-0.65, -0.3)--(-0.15,-0.3);
\draw[-] (-0.07,-0.3)--(-0.03,-0.3);
\draw[-] (0.03,-0.3)--(0.07,-0.3);
\draw[-] (0.13,-0.3)--(0.95,-0.3);
\draw[-] (1.03,-0.3)--(1.17,-0.3);
\draw[-] (1.23,-0.3)--(1.27,-0.3);
\draw[-] (1.33,-0.3)--(1.37, -0.3);
\draw[-] (1.43,-0.3)--(1.7,-0.3);
\draw[-] (1.6,-0.3)--(1.6,0.5);
\draw[-] (-1.5,0.5)--(-1.5,-0.3);

\draw[-] (-1.7,-0.45)--(-0.75, -0.45);
\draw[-] (-0.65, -0.45)--(-0.15,-0.45);
\draw[-] (-0.07,-0.45)--(-0.03,-0.45);
\draw[-] (0.05,-0.45)--(0.95,-0.45);
\draw[-] (1.03,-0.45)--(1.17,-0.45);
\draw[-] (1.23,-0.45)--(1.27,-0.45);
\draw[-] (1.33,-0.45)--(1.7,-0.45);
\draw[-] (1.4,-0.45)--(1.4,0.5);
\draw[-] (-1.7,0.5)--(-1.7,-0.45);

\draw[-] (-1.8,-0.55)--(-0.75, -0.55);
\draw[-] (-0.65, -0.55)--(-0.15,-0.55);
\draw[-] (-0.05,-0.55)--(0.95,-0.55);
\draw[-] (1.03,-0.55)--(1.17,-0.55);
\draw[-] (1.23,-0.55)--(1.7,-0.55);
\draw[-] (1.3,-0.55)--(1.3,0.5);
\draw[-] (-1.8,0.5)--(-1.8,-0.55);

\draw[-] (-1.9,-0.65)--(-0.75, -0.65);
\draw[-] (-0.65, -0.65)--(0.95,-0.65);
\draw[-] (1.03,-0.65)--(1.7,-0.65);
\draw[-] (1.2,-0.65)--(1.2,0.5);
\draw[-] (-1.9,0.5)--(-1.9,-0.65);

\draw[-] (-2.1,-0.8)--(1.7,-0.8);
\draw[-] (-2.1,0.5)--(-2.1,-0.8);
\draw[-] (1,-0.8)--(1,0.5);
\end{tikzpicture} \quad =  \quad  
\begin{tikzpicture}[anchorbase,scale=1.5]%leftmost part:

%bottom parts 

\draw[-] (-1.5,-0.3)--(0.95,-0.3);
\draw[-] (1.03,-0.3)--(1.17,-0.3);
\draw[-] (1.23,-0.3)--(1.27,-0.3);
\draw[-] (1.33,-0.3)--(1.37, -0.3);
\draw[-] (1.43,-0.3)--(1.7,-0.3);
\draw[-] (1.6,-0.3)--(1.6,0.5);
\draw[-] (-1.5,0.5)--(-1.5,-0.3);

\draw[-] (-1.7,-0.45)--(0.95,-0.45);
\draw[-] (1.03,-0.45)--(1.17,-0.45);
\draw[-] (1.23,-0.45)--(1.27,-0.45);
\draw[-] (1.33,-0.45)--(1.7,-0.45);
\draw[-] (1.4,-0.45)--(1.4,0.5);
\draw[-] (-1.7,0.5)--(-1.7,-0.45);

\draw[-] (-1.8,-0.55)--(0.95,-0.55);
\draw[-] (1.03,-0.55)--(1.17,-0.55);
\draw[-] (1.23,-0.55)--(1.7,-0.55);
\draw[-] (1.3,-0.55)--(1.3,0.5);
\draw[-] (-1.8,0.5)--(-1.8,-0.55);

\draw[-] (-1.9,-0.65)--(0.95,-0.65);
\draw[-] (1.03,-0.65)--(1.7,-0.65);
\draw[-] (1.2,-0.65)--(1.2,0.5);
\draw[-] (-1.9,0.5)--(-1.9,-0.65);

\draw[-] (-2.1,-0.8)--(1.7,-0.8);
\draw[-] (-2.1,0.5)--(-2.1,-0.8);
\draw[-] (1,-0.8)--(1,0.5);
\end{tikzpicture}
\end{equation*}

In other words, we can just erase the box corresponding to a zeroes row, together with all the strings leading into in, and obtain an equivalent diagram.

Summing up, Diagram \eqref{eq:Step_4_compos_proof} is equivalent to the following diagram:

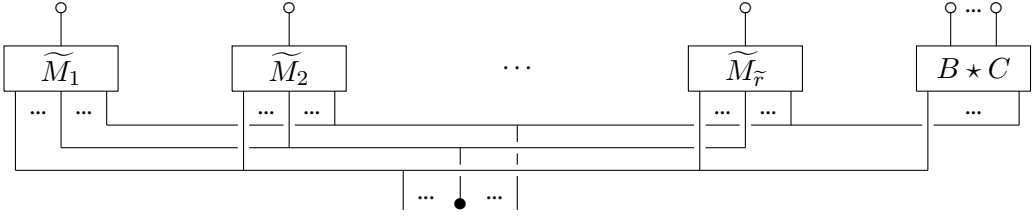
\begin{equation*}
\begin{tikzpicture}[anchorbase,scale=1.5]%leftmost part:
\draw[-] (-0.5,-0.2)--(-0.5,0.2)--(0.5,0.2)--(0.5,-0.2)--(-0.5,-0.2);
\node at (0,0) {${\widetilde{M}_1}$};

\draw[-] (-0.4,-0.2)--(-0.4,-0.9);
\draw[-] (0.4,-0.2)--(0.4,-0.5);
\draw[-] (0,-0.2)--(0,-0.7);

\draw[-] (0,0.2)--(0,0.5);
\node at (0,0.53) {$\circ$};

\node at (-0.25,-0.4) {$\cdot$};
\node at (-0.2,-0.4) {$\cdot$};
\node at (-0.15,-0.4) {$\cdot$};

\node at (0.25,-0.4) {$\cdot$};
\node at (0.2,-0.4) {$\cdot$};
\node at (0.15,-0.4) {$\cdot$};

%%%
\draw[-] (1.5,-0.2)--(1.5,0.2)--(2.5,0.2)--(2.5,-0.2)--(1.5,-0.2);
\node at (2,0) {${\widetilde{M}_2}$};

\draw[-] (1.6,-0.2)--(1.6,-0.9);
\draw[-] (2.4,-0.2)--(2.4,-0.5);
\draw[-] (2,-0.2)--(2,-0.7);

\draw[-] (2,0.2)--(2,0.5);
\node at (2,0.53) {$\circ$};

\node at (1.85,-0.4) {$\cdot$};
\node at (1.8,-0.4) {$\cdot$};
\node at (1.75,-0.4) {$\cdot$};

\node at (2.15,-0.4) {$\cdot$};
\node at (2.2,-0.4) {$\cdot$};
\node at (2.25,-0.4) {$\cdot$};
%%%intermediate dots:
\node at (3.9,0) {$\cdot$};
\node at (4,0) {$\cdot$};
\node at (4.1,0) {$\cdot$};

%%% rightmost part of \widetilde{M}:
\draw[-] (5.5,-0.2)--(5.5,0.2)--(6.5,0.2)--(6.5,-0.2)--(5.5,-0.2);
\node at (6,0) {${\widetilde{M}_{\widetilde{r}}}$};

\draw[-] (5.6,-0.2)--(5.6,-0.9);
\draw[-] (6.4,-0.2)--(6.4,-0.5);
\draw[-] (6,-0.2)--(6,-0.7);

\draw[-] (6,0.2)--(6,0.5);
\node at (6,0.53) {$\circ$};

\node at (5.75,-0.4) {$\cdot$};
\node at (5.8,-0.4) {$\cdot$};
\node at (5.85,-0.4) {$\cdot$};

\node at (6.15,-0.4) {$\cdot$};
\node at (6.2,-0.4) {$\cdot$};
\node at (6.25,-0.4) {$\cdot$};

%%% rightmost part:
\draw[-] (7.5,-0.2)--(7.5,0.2)--(8.5,0.2)--(8.5,-0.2)--(7.5,-0.2);
\node at (8,0) {${B\star C}$};

\draw[-] (7.6,-0.2)--(7.6,-0.9);

\draw[-] (8.4,-0.2)--(8.4,-0.5);

\draw[-] (7.8,0.2)--(7.8,0.5);
\node at (7.8,0.53) {$\circ$};
\node at (7.95,0.5) {$\cdot$};
\node at (8,0.5) {$\cdot$};
\node at (8.05,0.5) {$\cdot$};
\draw[-] (8.2,0.2)--(8.2,0.5);
\node at (8.2,0.53) {$\circ$};

\node at (7.95,-0.4) {$\cdot$};
\node at (8,-0.4) {$\cdot$};
\node at (8.05,-0.4) {$\cdot$};

%%% Bottom part
\draw[-] (-0.4,-0.9)--(7.6,-0.9);%bottom horizontal line

%middle horizontal
\draw[-] (0,-0.7)--(1.55, -0.7);
\draw[-] (1.65,-0.7)--(5.55, -0.7);
\draw[-] (5.65,-0.7)--(6,-0.7);

%top horizontal line
\draw[-] (0.4,-0.5)--(1.55, -0.5);
\draw[-] (1.65, -0.5)--(1.95, -0.5);
\draw[-] (2.05, -0.5)--(5.55, -0.5);
\draw[-] (5.65,-0.5)--(5.95,-0.5);
\draw[-] (6.05,-0.5)--(7.55, -0.5);
\draw[-] (7.65,-0.5)--(8.4,-0.5);

\draw[-] (3, -0.9)--(3,-1.25);

\draw[-] (3.5, -0.7)--(3.5, -0.85);
\draw[-] (3.5, -0.95)--(3.5, -1.2);

\node at (3.5,-1.2) {$\bullet$};

\draw[-] (4, -0.5)--(4, -0.65);
\draw[-] (4, -0.75)--(4, -0.85);
\draw[-] (4, -0.95)--(4, -1.25);

%bottom dots
\node at (3.15,-1.15) {$\cdot$};
\node at (3.2,-1.15) {$\cdot$};
\node at (3.25,-1.15) {$\cdot$};

%bottom dots
\node at (3.75,-1.15) {$\cdot$};
\node at (3.8,-1.15) {$\cdot$};
\node at (3.85,-1.15) {$\cdot$};
\end{tikzpicture}
\end{equation*}
The number of the bottom strands and ``tails'' is as in \eqref{eq:Step_4_compos_proof}. Now let us consider the matrix $M$. Since the matrix $M$ is in row canonical form, there is an entry $1$ in each row of $M$ so that the remaining entries in this column of $M$ are zeroes. 
For each $i$, we denote by $j_i$ the smallest index of the column such that $M_{i, j_i} =1$, $M_{i, j'}=0$ for all $j'<j_i$.

Then the above argument shows that the $j_i$-th factor $\eps^*:\triv \to \V$ (drawn as the $j_i$-th fork with a ``tail'' of the form $\begin{tikzpicture}[anchorbase,scale=1]
\draw[-] (0, 0.5) --(0,0);
\node at (0, 0) {$\bullet$};
\end{tikzpicture}  $) only needs to be connected to the rectangle representing $\mu_{\widetilde{M}_i}$. 

This implies that we can redraw the above diagram as

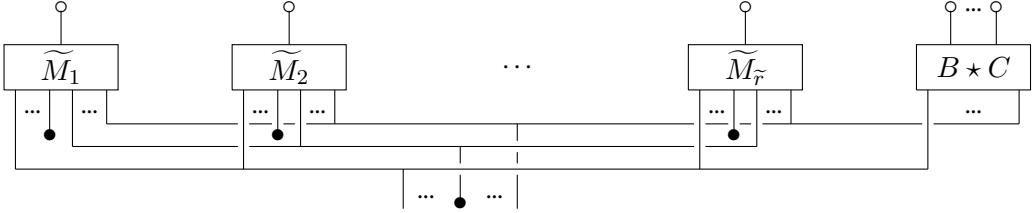
\begin{equation*}
\begin{tikzpicture}[anchorbase,scale=1.5]%leftmost part:
\draw[-] (-0.5,-0.2)--(-0.5,0.2)--(0.5,0.2)--(0.5,-0.2)--(-0.5,-0.2);
\node at (0,0) {${\widetilde{M}_1}$};

\draw[-] (0.4,-0.2)--(0.4,-0.5);
\draw[-] (-0.1,-0.2)--(-0.1,-0.6);
\node at (-0.1,-0.6) {$\bullet$};
\draw[-] (0.1,-0.2)--(0.1,-0.7);
\draw[-] (-0.4,-0.2)--(-0.4,-0.9);

\draw[-] (0,0.2)--(0,0.5);
\node at (0,0.53) {$\circ$};

\node at (-0.25,-0.4) {$\cdot$};
\node at (-0.2,-0.4) {$\cdot$};
\node at (-0.3,-0.4) {$\cdot$};

\node at (0.25,-0.4) {$\cdot$};
\node at (0.2,-0.4) {$\cdot$};
\node at (0.3,-0.4) {$\cdot$};

%%%
\draw[-] (1.5,-0.2)--(1.5,0.2)--(2.5,0.2)--(2.5,-0.2)--(1.5,-0.2);
\node at (2,0) {${\widetilde{M}_2}$};

\draw[-] (1.6,-0.2)--(1.6,-0.9);
\draw[-] (1.9,-0.2)--(1.9,-0.6);
\node at (1.9,-0.6) {$\bullet$};
\draw[-] (2.1,-0.2)--(2.1,-0.7);
\draw[-] (2.4,-0.2)--(2.4,-0.5);

\draw[-] (2,0.2)--(2,0.5);
\node at (2,0.53) {$\circ$};

\node at (1.7,-0.4) {$\cdot$};
\node at (1.8,-0.4) {$\cdot$};
\node at (1.75,-0.4) {$\cdot$};

\node at (2.3,-0.4) {$\cdot$};
\node at (2.2,-0.4) {$\cdot$};
\node at (2.25,-0.4) {$\cdot$};
%%%intermediate dots:
\node at (3.9,0) {$\cdot$};
\node at (4,0) {$\cdot$};
\node at (4.1,0) {$\cdot$};

%%% rightmost part of \widetilde{M}:
\draw[-] (5.5,-0.2)--(5.5,0.2)--(6.5,0.2)--(6.5,-0.2)--(5.5,-0.2);
\node at (6,0) {${\widetilde{M}_{\widetilde{r}}}$};

\draw[-] (5.6,-0.2)--(5.6,-0.9);
\draw[-] (5.9,-0.2)--(5.9,-0.6);
\node at (5.9,-0.6) {$\bullet$};
\draw[-] (6.1,-0.2)--(6.1,-0.7);
\draw[-] (6.4,-0.2)--(6.4,-0.5);

\draw[-] (6,0.2)--(6,0.5);
\node at (6,0.53) {$\circ$};

\node at (5.75,-0.4) {$\cdot$};
\node at (5.8,-0.4) {$\cdot$};
\node at (5.7,-0.4) {$\cdot$};

\node at (6.3,-0.4) {$\cdot$};
\node at (6.2,-0.4) {$\cdot$};
\node at (6.25,-0.4) {$\cdot$};

%%% rightmost part:
\draw[-] (7.5,-0.2)--(7.5,0.2)--(8.5,0.2)--(8.5,-0.2)--(7.5,-0.2);
\node at (8,0) {${B\star C}$};

\draw[-] (7.6,-0.2)--(7.6,-0.9);

\draw[-] (8.4,-0.2)--(8.4,-0.5);

\draw[-] (7.8,0.2)--(7.8,0.5);
\node at (7.8,0.53) {$\circ$};
\node at (7.95,0.5) {$\cdot$};
\node at (8,0.5) {$\cdot$};
\node at (8.05,0.5) {$\cdot$};
\draw[-] (8.2,0.2)--(8.2,0.5);
\node at (8.2,0.53) {$\circ$};

\node at (7.95,-0.4) {$\cdot$};
\node at (8,-0.4) {$\cdot$};
\node at (8.05,-0.4) {$\cdot$};

%%% Bottom part
\draw[-] (-0.4,-0.9)--(7.6,-0.9);%bottom horizontal line

%middle horizontal
\draw[-] (0.1,-0.7)--(1.55, -0.7);
\draw[-] (1.65,-0.7)--(5.55, -0.7);
\draw[-] (5.65,-0.7)--(6.1,-0.7);

%top horizontal line
\draw[-] (0.4,-0.5)--(1.55, -0.5);
\draw[-] (1.65, -0.5)--(1.85, -0.5);
\draw[-] (1.95, -0.5)--(2.05, -0.5);
\draw[-] (2.15, -0.5)--(5.55, -0.5);
\draw[-] (5.65,-0.5)--(5.85,-0.5);
\draw[-] (5.95,-0.5)--(6.05,-0.5);
\draw[-] (6.15,-0.5)--(7.55, -0.5);
\draw[-] (7.65,-0.5)--(8.4,-0.5);

\draw[-] (3, -0.9)--(3,-1.25);

\draw[-] (3.5, -0.7)--(3.5, -0.85);
\draw[-] (3.5, -0.95)--(3.5, -1.2);

\node at (3.5,-1.2) {$\bullet$};

\draw[-] (4, -0.5)--(4, -0.65);
\draw[-] (4, -0.75)--(4, -0.85);
\draw[-] (4, -0.95)--(4, -1.25);

%bottom dots
\node at (3.15,-1.15) {$\cdot$};
\node at (3.2,-1.15) {$\cdot$};
\node at (3.25,-1.15) {$\cdot$};

%bottom dots
\node at (3.75,-1.15) {$\cdot$};
\node at (3.8,-1.15) {$\cdot$};
\node at (3.85,-1.15) {$\cdot$};
\end{tikzpicture}
\end{equation*}

Now in the bottom of the diagram we have $s$ strands on the left, $l$ strands on the right, and $k-\widetilde{r} = d(R, S)$ ``tails'' of the form $\begin{tikzpicture}[anchorbase,scale=1]
\draw[-] (0, 0.5) --(0,0);
\node at (0, 0) {$\bullet$};
\end{tikzpicture}  $ in the middle. The row canonical form of $M$ ensures that for each $i=1, \ldots, \widetilde{r}$ we have: 
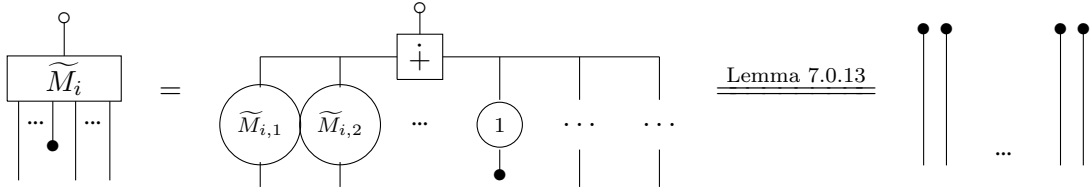
\begin{equation*}
\begin{tikzpicture}[anchorbase,scale=1.5]%leftmost part:
\draw[-] (-0.5,-0.2)--(-0.5,0.2)--(0.5,0.2)--(0.5,-0.2)--(-0.5,-0.2);
\node at (0,0) {${\widetilde{M}_i}$};

\draw[-] (0.4,-0.2)--(0.4,-0.9);
\draw[-] (-0.1,-0.2)--(-0.1,-0.6);
\node at (-0.1,-0.6) {$\bullet$};
\draw[-] (0.1,-0.2)--(0.1,-0.9);
\draw[-] (-0.4,-0.2)--(-0.4,-0.9);

\draw[-] (0,0.2)--(0,0.5);
\node at (0,0.53) {$\circ$};

\node at (-0.25,-0.4) {$\cdot$};
\node at (-0.2,-0.4) {$\cdot$};
\node at (-0.3,-0.4) {$\cdot$};

\node at (0.25,-0.4) {$\cdot$};
\node at (0.2,-0.4) {$\cdot$};
\node at (0.3,-0.4) {$\cdot$};
\end{tikzpicture} \quad =  \quad
\begin{tikzpicture}[anchorbase,scale=1.5]
\draw[-] (-1.4,0)--(-0.2,0);
\draw[-] (0.2,0)--(2.1,0);
\draw[-] (-0.2,-0.2)--(-0.2,0.2)--(0.2,0.2)--(0.2,-0.2)--(-0.2,-0.2);
\node at (-0,-0) {$\dot{+}$};
\draw[-] (-1.4,0)--(-1.4,-0.27);
\draw[-] (-0.7,0)--(-0.7,-0.27);
\draw[-] (0.7,0)--(0.7,-0.38);
\draw[-] (1.4,0)--(1.4,-0.38);
\draw[-] (2.1,0)--(2.1,-0.38);

\draw[-] (0,0.2)--(0,0.4);
\node at (0,0.43) {$\circ$};

\node at (-0.05,-0.6) {$\cdot$};
\node at (0,-0.6) {$\cdot$};
\node at (0.05,-0.6) {$\cdot$};
\draw[-] (-1.4,-0.93)--(-1.4,-1.15);
\draw[-] (-0.7,-0.93)--(-0.7,-1.15);
\draw[-] (0.7,-0.82)--(0.7,-1.05);
\draw[-] (1.4,-0.82)--(1.4,-1.15);
\draw[-] (2.1,-0.82)--(2.1,-1.15);
\node[draw,circle] at (-1.4,-0.6) {$ {\scriptstyle \widetilde{M}_{i,1}}$};
\node[draw,circle] at (-0.7,-0.6) {${\scriptstyle \widetilde{M}_{i,2}}$};
\node[draw,circle] at (0.7,-0.6) {$\scriptstyle 1$};
\node at (1.4,-0.6) {$\ldots$};
\node at (2.1,-0.6) {$\ldots$};
\node at (0.7,-1.05) {$\bullet$};
\end{tikzpicture}
\InnaA{\quad \xlongequal{\text{\cref{lem:eps_star_mu_A}}}\quad}
\begin{tikzpicture}[anchorbase,scale=1.5]
\draw[-] (0.7,0.5)--(0.7,-0.7);
\node at (0.7,0.5) {$\bullet$};

\node at (-0.05,-0.6) {$\cdot$};
\node at (0,-0.6) {$\cdot$};
\node at (0.05,-0.6) {$\cdot$};

\draw[-] (0.5,0.5)--(0.5,-0.7);
\node at (0.5,0.5) {$\bullet$};

\draw[-] (-0.5,0.5)--(-0.5,-0.7);
\node at (-0.5,0.5) {$\bullet$};

\draw[-] (-0.7,0.5)--(-0.7,-0.7);
\node at (-0.7,0.5) {$\bullet$};
\end{tikzpicture}
\end{equation*}

Hence Diagram \eqref{eq:Step_4_compos_proof} is equivalent to the diagram

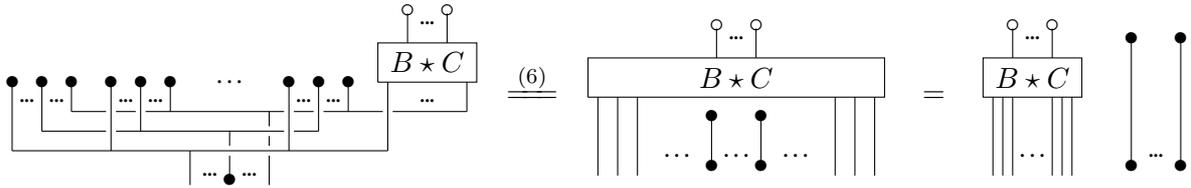
\begin{equation*}
\begin{tikzpicture}[anchorbase,scale=1.3]%leftmost part:

\draw[-] (1.4,-0.2)--(1.4,-0.5);
\draw[-] (1.1,-0.2)--(1.1,-0.7);
\draw[-] (0.8,-0.2)--(0.8,-0.9);
\node at (0.8,-0.2) {$\bullet$};
\node at (1.1,-0.2) {$\bullet$};
\node at (1.4,-0.2) {$\bullet$};

\node at (1,-0.4) {$\cdot$};
\node at (0.9,-0.4) {$\cdot$};
\node at (0.95,-0.4) {$\cdot$};

\node at (1.25,-0.4) {$\cdot$};
\node at (1.2,-0.4) {$\cdot$};
\node at (1.3,-0.4) {$\cdot$};

%%%

\draw[-] (1.8,-0.2)--(1.8,-0.9);
\draw[-] (2.1,-0.2)--(2.1,-0.7);
\draw[-] (2.4,-0.2)--(2.4,-0.5);
\node at (1.8,-0.2) {$\bullet$};
\node at (2.1,-0.2) {$\bullet$};
\node at (2.4,-0.2) {$\bullet$};

\node at (2,-0.4) {$\cdot$};
\node at (1.9,-0.4) {$\cdot$};
\node at (1.95,-0.4) {$\cdot$};

\node at (2.3,-0.4) {$\cdot$};
\node at (2.2,-0.4) {$\cdot$};
\node at (2.25,-0.4) {$\cdot$};
%%%intermediate dots:
\node at (2.9,-0.2) {$\cdot$};
\node at (3,-0.2) {$\cdot$};
\node at (3.1,-0.2) {$\cdot$};

%%% rightmost part of \widetilde{M}:

\draw[-] (3.6,-0.2)--(3.6,-0.9);
\draw[-] (4.2,-0.2)--(4.2,-0.5);
\draw[-] (3.9,-0.2)--(3.9,-0.7);
\node at (3.6,-0.2) {$\bullet$};
\node at (3.9,-0.2) {$\bullet$};
\node at (4.2,-0.2) {$\bullet$};

\node at (3.75,-0.4) {$\cdot$};
\node at (3.8,-0.4) {$\cdot$};
\node at (3.7,-0.4) {$\cdot$};

\node at (4,-0.4) {$\cdot$};
\node at (4.05,-0.4) {$\cdot$};
\node at (4.1,-0.4) {$\cdot$};

%%% rightmost part:
\draw[-] (4.5,-0.2)--(4.5,0.2)--(5.5,0.2)--(5.5,-0.2)--(4.5,-0.2);
\node at (5,0) {${B\star C}$};

\draw[-] (4.6,-0.2)--(4.6,-0.9);

\draw[-] (5.4,-0.2)--(5.4,-0.5);

\draw[-] (4.8,0.2)--(4.8,0.5);
\node at (4.8,0.53) {$\circ$};
\node at (4.95,0.4) {$\cdot$};
\node at (5,0.4) {$\cdot$};
\node at (5.05,0.4) {$\cdot$};
\draw[-] (5.2,0.2)--(5.2,0.5);
\node at (5.2,0.53) {$\circ$};

\node at (4.95,-0.4) {$\cdot$};
\node at (5,-0.4) {$\cdot$};
\node at (5.05,-0.4) {$\cdot$};

%%% Bottom part
\draw[-] (0.8,-0.9)--(4.6,-0.9);%bottom horizontal line

%middle horizontal
\draw[-] (1.1,-0.7)--(1.75, -0.7);
\draw[-] (1.85,-0.7)--(3.55, -0.7);
\draw[-] (3.65,-0.7)--(3.9,-0.7);

%top horizontal line
\draw[-] (1.4,-0.5)--(1.75, -0.5);
\draw[-] (1.85, -0.5)--(2.05, -0.5);
\draw[-] (2.15, -0.5)--(3.55, -0.5);
\draw[-] (3.65,-0.5)--(3.85,-0.5);
\draw[-] (3.95,-0.5)--(4.55, -0.5);
\draw[-] (4.65,-0.5)--(5.4,-0.5);

\draw[-] (2.6, -0.9)--(2.6,-1.25);

\draw[-] (3, -0.7)--(3, -0.85);
\draw[-] (3, -0.95)--(3, -1.2);

\node at (3,-1.2) {$\bullet$};

\draw[-] (3.4, -0.5)--(3.4, -0.65);
\draw[-] (3.4, -0.75)--(3.4, -0.85);
\draw[-] (3.4, -0.95)--(3.4, -1.25);

%bottom dots
\node at (2.85,-1.15) {$\cdot$};
\node at (2.8,-1.15) {$\cdot$};
\node at (2.75,-1.15) {$\cdot$};

%bottom dots
\node at (3.15,-1.15) {$\cdot$};
\node at (3.2,-1.15) {$\cdot$};
\node at (3.25,-1.15) {$\cdot$};
\end{tikzpicture}
\InnaA{\quad \xlongequal{\text{\eqref{eq:Step_4_comult_counit}}}\quad} 
\begin{tikzpicture}[anchorbase,scale=1.3]%leftmost part:

%%% rightmost part:
\draw[-] (2.5,-0.2)--(2.5,0.2)--(5.5,0.2)--(5.5,-0.2)--(2.5,-0.2);
\node at (4,0) {${B\star C}$};

\draw[-] (2.6,-0.2)--(2.6,-1);
\draw[-] (2.8,-0.2)--(2.8,-1);
\draw[-] (3,-0.2)--(3,-1);

\draw[-] (5,-0.2)--(5,-1);
\draw[-] (5.2,-0.2)--(5.2,-1);
\draw[-] (5.4,-0.2)--(5.4,-1);

\draw[-] (3.8,0.2)--(3.8,0.5);
\node at (3.8,0.53) {$\circ$};
\node at (3.95,0.4) {$\cdot$};
\node at (4,0.4) {$\cdot$};
\node at (4.05,0.4) {$\cdot$};
\draw[-] (4.2,0.2)--(4.2,0.5);
\node at (4.2,0.53) {$\circ$};

\draw[-] (3.75, -0.4)--(3.75, -0.9);
\node at (3.75,-0.4) {$\bullet$};
\node at (3.75,-0.9) {$\bullet$};

\draw[-] (4.25, -0.4)--(4.25, -0.9);
\node at (4.25,-0.4) {$\bullet$};
\node at (4.25,-0.9) {$\bullet$};

%bottom dots
\node at (3.3,-0.8) {$\cdot$};
\node at (3.4,-0.8) {$\cdot$};
\node at (3.5,-0.8) {$\cdot$};
%bottom dots
\node at (3.9,-0.8) {$\cdot$};
\node at (4,-0.8) {$\cdot$};
\node at (4.1,-0.8) {$\cdot$};
% 
%bottom dots
\node at (4.5,-0.8) {$\cdot$};
\node at (4.6,-0.8) {$\cdot$};
\node at (4.7,-0.8) {$\cdot$};
\end{tikzpicture} \quad = \quad 
\begin{tikzpicture}[anchorbase,scale=1.3]%leftmost part:

%%% rightmost part:
\draw[-] (4.5,-0.2)--(4.5,0.2)--(5.5,0.2)--(5.5,-0.2)--(4.5,-0.2);
\node at (5,0) {${B\star C}$};

\draw[-] (4.6,-0.2)--(4.6,-1);
\draw[-] (4.7,-0.2)--(4.7,-1);
\draw[-] (4.8,-0.2)--(4.8,-1);

\draw[-] (5.2,-0.2)--(5.2,-1);
\draw[-] (5.3,-0.2)--(5.3,-1);
\draw[-] (5.4,-0.2)--(5.4,-1);

\draw[-] (4.8,0.2)--(4.8,0.5);
\node at (4.8,0.53) {$\circ$};
\node at (4.95,0.4) {$\cdot$};
\node at (5,0.4) {$\cdot$};
\node at (5.05,0.4) {$\cdot$};
\draw[-] (5.2,0.2)--(5.2,0.5);
\node at (5.2,0.53) {$\circ$};

%bottom dots
\node at (4.9,-0.8) {$\cdot$};
\node at (5,-0.8) {$\cdot$};
\node at (5.1,-0.8) {$\cdot$};

%% eps^* \circ eps
\draw[-] (6, 0.4)--(6, -0.9);
\node at (6,0.4) {$\bullet$};
\node at (6,-0.9) {$\bullet$};

\node at (6.2,-0.8) {$\cdot$};
\node at (6.25,-0.8) {$\cdot$};
\node at (6.3,-0.8) {$\cdot$};

\draw[-] (6.5, 0.4)--(6.5, -0.9);
\node at (6.5,0.4) {$\bullet$};
\node at (6.5,-0.9) {$\bullet$};
\end{tikzpicture} 
\end{equation*}
In the bottom of this diagram we have $s+l$ strands on the left, and $k-\widetilde{r} = d(R, S)$ strings of the form $\begin{tikzpicture}[anchorbase,scale=1]
\node at (0, 0.5) {$\bullet$};
\draw[-] (0, 0.5) --(0,0);
\node at (0, 0) {$\bullet$};
\end{tikzpicture}  $ on the right. The diagram 
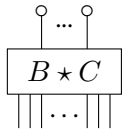
\begin{equation*}
 \begin{tikzpicture}[anchorbase,scale=1.5]%leftmost part:

%%% rightmost part:
\draw[-] (4.5,-0.2)--(4.5,0.2)--(5.5,0.2)--(5.5,-0.2)--(4.5,-0.2);
\node at (5,0) {${B\star C}$};

\draw[-] (4.6,-0.2)--(4.6,-0.5);
\draw[-] (4.7,-0.2)--(4.7,-0.5);
\draw[-] (4.8,-0.2)--(4.8,-0.5);

\draw[-] (5.2,-0.2)--(5.2,-0.5);
\draw[-] (5.3,-0.2)--(5.3,-0.5);
\draw[-] (5.4,-0.2)--(5.4,-0.5);

\draw[-] (4.8,0.2)--(4.8,0.5);
\node at (4.8,0.53) {$\circ$};
\node at (4.95,0.4) {$\cdot$};
\node at (5,0.4) {$\cdot$};
\node at (5.05,0.4) {$\cdot$};
\draw[-] (5.2,0.2)--(5.2,0.5);
\node at (5.2,0.53) {$\circ$};

%bottom dots
\node at (4.9,-0.4) {$\cdot$};
\node at (5,-0.4) {$\cdot$};
\node at (5.1,-0.4) {$\cdot$};
\end{tikzpicture}
\end{equation*}
is the diagram of $\phi_{R\star S}$ by Step 3. Thus we proved that $$(z^*)^{\otimes (d_1+d_2)}\circ \mu_{D'} \circ (\id_{\V^{\otimes s}} \otimes \eps^{\otimes k} \otimes \id_{\V^{\otimes l}}) = \phi_{R\star S} \otimes (\eps^*\circ \eps)^{\otimes d(R, S)}.$$
By Lemma \ref{lem:self-dual}, $\eps^*\circ \eps = (\dim \V) \id_{\triv}$, and we conclude that
$$\phi_R \circledast \phi_S = (z^*)^{\otimes (d_1+d_2)}\circ \mu_{D'} \circ (\id_{\V^{\otimes s}} \otimes \eps^{\otimes k} \otimes \id_{\V^{\otimes l}}) = (\dim \V)^{ d(R, S)}\phi_{R\star S} .$$ This proves the required statement.

\end{proof}
Finally, we prove the lemma used in the proof of Proposition \ref{prop:composition}.
\begin{lemma}\label{lem:eps_star_mu_A}
 Let $a_1, \ldots, a_k \in \F_q$ and assume that $a_i \neq 0$ for some $i\in \{1,\ldots, k\}$. Then
 $$ z^* \circ \mu_{\begin{bmatrix}
 a_1 &\ldots &a_k
 \end{bmatrix}} \circ \left(\id_{\V^{\otimes (i-1)}}\otimes \eps \otimes \id_{\V^{\otimes (k-i)}}  \right)= (\eps^*)^{\otimes (k-1)}$$
 
 Diagrammatically, this can be drawn as
 \begin{equation*}
 \begin{tikzpicture}[anchorbase,scale=1.3]
\draw[-] (-1.4,0)--(-0.2,0);
\draw[-] (0.2,0)--(2.1,0);
\draw[-] (-0.2,-0.2)--(-0.2,0.2)--(0.2,0.2)--(0.2,-0.2)--(-0.2,-0.2);
\node at (-0,-0) {$\dot{+}$};
\draw[-] (-1.4,0)--(-1.4,-0.3);
\draw[-] (-0.7,0)--(-0.7,-0.3);
\draw[-] (0.7,0)--(0.7,-0.3);
\draw[-] (2.1,0)--(2.1,-0.3);

\draw[-] (0,0.2)--(0,0.4);
\node at (0,0.43) {$\circ$};

\node at (-0.05,-0.6) {$\cdot$};
\node at (0,-0.6) {$\cdot$};
\node at (0.05,-0.6) {$\cdot$};

\node at (1.35,-0.6) {$\cdot$};
\node at (1.4,-0.6) {$\cdot$};
\node at (1.45,-0.6) {$\cdot$};

\draw[-] (-1.4,-0.87)--(-1.4,-1.15);
\draw[-] (-0.7,-0.87)--(-0.7,-1.15);
\draw[-] (0.7,-0.87)--(0.7,-1.05);
\draw[-] (2.1,-0.87)--(2.1,-1.15);

\node[draw,circle] at (-1.4,-0.6) {$ {\scriptstyle a_1}$};
\node[draw,circle] at (-0.7,-0.6) {${\scriptstyle a_2}$};
\node[draw,circle] at (0.7,-0.6) {$ {\scriptstyle a_i}$};
\node[draw,circle]  at (2.1,-0.6) {${\scriptstyle a_k}$};
\node at (0.7,-1.08) {$\bullet$};
\end{tikzpicture} \quad = \quad
\begin{tikzpicture}[anchorbase,scale=1.3]
\draw[-] (0.7,0.5)--(0.7,-0.9);
\node at (0.7,0.5) {$\bullet$};

\node at (-0.05,-0.6) {$\cdot$};
\node at (0,-0.6) {$\cdot$};
\node at (0.05,-0.6) {$\cdot$};

\draw[-] (0.5,0.5)--(0.5,-0.9);
\node at (0.5,0.5) {$\bullet$};

\draw[-] (-0.5,0.5)--(-0.5,-0.9);
\node at (-0.5,0.5) {$\bullet$};

\draw[-] (-0.7,0.5)--(-0.7,-0.9);
\node at (-0.7,0.5) {$\bullet$};
\end{tikzpicture}
\end{equation*}
\end{lemma}

\begin{example}
 Let $\mathcal{C} = \Rep(GL_n(\F_q))$, $\V:=\V_n$. Let $a_1, \ldots, a_k \in \F_q$ as in the lemma. The lemma states that for any $v_1, \ldots, v_{i-1}, v_{i+1}, \ldots, v_k \in V$, the equality
 $$\dot{\sum}_{1\leq j\leq k, j\neq i} \dot{a}_j v_j \,\dot{+}\, \dot{a}_i v = \dot{0}$$ holds for exactly one $v\in V$, so the vector $\dot{0}$ appears in the sum $$\sum_{v\in V} \, \left(\dot{\sum}_{1\leq j\leq k, j\neq i} \dot{a}_j v_j \,\dot{+}\, \dot{a}_i v \right)$$ with coefficient $1$, which is precisely $(\eps^*)^{\otimes (k-1)}(v_1\otimes \ldots\otimes v_k)$.
\end{example}

\begin{proof}
First of all, recall that Corollary \ref{cor:eps_and_mu} states: $a_i \neq 0$ implies $\mu_{a_i} \circ \eps = \eps = \mu_{1} \circ \eps$. Hence we may assume that $a_i=1$ from now on.

Next, it is enough to check the statement for $a_1=\ldots =a_k=1$. Indeed, if we show that 
\begin{equation}\label{eq:lin_eq_and_eps_star}
 z^* \circ (\dot{+})^{it} \circ \left(\id_{\V^{\otimes (i-1)}}\otimes \eps \otimes \id_{\V^{\otimes (k-i)}}  \right)= (\eps^*)^{\otimes (k-1)}
\end{equation}
 then we will have
\begin{align*}
&z^* \circ \mu_{[a_1, \ldots, a_k]} \circ \left(\id_{\V^{\otimes (i-1)}}\otimes \eps \otimes \id_{\V^{\otimes (k-i)}}  \right)= z^* \circ (\dot{+})^{it} \circ \left(\bigotimes_{j=1}^k \mu_{a_j}\right) \circ \left(\id_{\V^{\otimes (i-1)}}\otimes \eps \otimes \id_{\V^{\otimes (k-i)}}  \right) \\ &=z^* \circ (\dot{+})^{it}  \circ \left(\id_{\V^{\otimes (i-1)}}\otimes \eps \otimes \id_{\V^{\otimes (k-i)}}  \right)\circ \left(\bigotimes_{j=1}^k \mu_{a_j}\right)=(\eps^*)^{\otimes (k-1)} \circ \left(\bigotimes_{j=1}^k \mu_{a_j}\right) = (\eps^*)^{\otimes (k-1)} 
\end{align*}
The equality between the first and the last lines follows from the assumption $a_i=1$, while the last equality follows from Relation \ref{rel:mu_coalg_mor} stating that $\eps^*\circ \mu_a = \eps^*$ for any $a\in \F_q$ (see also \ref{itm:str_rel_diag_mu_coalg}).

So it is enough to prove Equation \eqref{eq:lin_eq_and_eps_star}, i.e. that 
 \begin{equation}\label{eq:lin_eq_and_eps_star_diagram}
 \begin{tikzpicture}[anchorbase,scale=1.5]
\draw[-] (-1,0)--(-0.2,0);
\draw[-] (0.2,0)--(1.4,0);
\draw[-] (-0.2,-0.2)--(-0.2,0.2)--(0.2,0.2)--(0.2,-0.2)--(-0.2,-0.2);
\node at (-0,-0) {$\dot{+}$};
\draw[-] (-1,0)--(-1,-0.45);
\draw[-] (-0.7,0)--(-0.7,-0.45);
\draw[-] (0.6,0)--(0.6,-0.33);
\draw[-] (1.1,0)--(1.1,-0.45);
\draw[-] (1.4,0)--(1.4,-0.45);

\draw[-] (0,0.2)--(0,0.4);
\node at (0,0.43) {$\circ$};

\node at (-0.05,-0.35) {$\cdot$};
\node at (0,-0.35) {$\cdot$};
\node at (0.05,-0.35) {$\cdot$};

\node at (0.85,-0.35) {$\cdot$};
\node at (0.9,-0.35) {$\cdot$};
\node at (0.95,-0.35) {$\cdot$};

\node at (0.6,-0.33) {$\bullet$};
\end{tikzpicture} \quad = \quad
\begin{tikzpicture}[anchorbase,scale=1.5]
\draw[-] (0.7,0.4)--(0.7,-0.45);
\node at (0.7,0.4) {$\bullet$};

\node at (-0.05,-0.35) {$\cdot$};
\node at (0,-0.35) {$\cdot$};
\node at (0.05,-0.35) {$\cdot$};

\draw[-] (0.5,0.4)--(0.5,-0.45);
\node at (0.5,0.4) {$\bullet$};

\draw[-] (-0.5,0.4)--(-0.5,-0.45);
\node at (-0.5,0.4) {$\bullet$};

\draw[-] (-0.7,0.4)--(-0.7,-0.45);
\node at (-0.7,0.4) {$\bullet$};
\end{tikzpicture}
\end{equation}
We prove the Equality \eqref{eq:lin_eq_and_eps_star} (equivalently, its diagrammatic version \eqref{eq:lin_eq_and_eps_star_diagram}) by induction on $k$.  
First, we prove the base case $k=2$. We need to show that
\begin{equation*}
\begin{tikzpicture}[anchorbase,scale=1]
\draw[-] (-0.5,0)--(-0.2,0);%\dot{+}
\draw[-] (0.2,0)--(0.5,0);
\draw[-] (-0.2,-0.2)--(-0.2,0.2)--(0.2,0.2)--(0.2,-0.2)--(-0.2,-0.2);
\node at (-0,-0) {$\dot{+}$};
\draw[-] (0,0.2)--(0,0.5);
\node at (0,0.53) {$\circ$};
\draw[-] (-0.5,0)--(-0.5,-0.5);
\draw[-] (0.5,0)--(0.5,-0.5);
\node at (0.5,-0.5) {$\bullet$};
\end{tikzpicture} \quad  = \quad  \begin{tikzpicture}[anchorbase,scale=1]
\draw[-] (0.5,0.5)--(0.5,-0.5);
\node at (0.5,0.5) {$\bullet$};
\end{tikzpicture} 
\end{equation*}  
Indeed, we have:

\begin{equation*}
\begin{tikzpicture}[anchorbase,scale=1.5]
\draw[-] (-0.5,0)--(-0.2,0);%\dot{+}
\draw[-] (0.2,0)--(0.5,0);
\draw[-] (-0.2,-0.2)--(-0.2,0.2)--(0.2,0.2)--(0.2,-0.2)--(-0.2,-0.2);
\node at (-0,-0) {$\dot{+}$};
\draw[-] (0,0.2)--(0,0.47);
\node at (0,0.5) {$\circ$};
\draw[-] (-0.5,0)--(-0.5,-0.9);
\draw[-] (0.5,0)--(0.5,-0.9);
\node at (0.5,-0.9) {$\bullet$};
\end{tikzpicture} \quad = \quad 
\begin{tikzpicture}[anchorbase,scale=1.5]
\draw[-] (-0.5,0)--(-0.2,0);%\dot{+}
\draw[-] (0.2,0)--(0.5,0);
\draw[-] (-0.2,-0.2)--(-0.2,0.2)--(0.2,0.2)--(0.2,-0.2)--(-0.2,-0.2);
\node at (-0,-0) {$\dot{+}$};

\draw[-] (0,0.2)--(0,0.5);
\draw[-] (0,0.5)--(1,0.5);
\draw[-] (1, 0.5)--(1, -0.95);
\node at (1,-1) {$\circ$};

\draw[-] (-0.5,0)--(-0.5,-1);
\draw[-] (0.5,0)--(0.5,-1);
\node at (0.5,-1) {$\bullet$};
\end{tikzpicture} 
\quad  = \quad 
\begin{tikzpicture}[anchorbase,scale=1.5]
\draw[-] (-0.5,0)--(-0.2,0);%\dot{+}
\draw[-] (0.2,0)--(0.5,0);
\draw[-] (-0.2,-0.2)--(-0.2,0.2)--(0.2,0.2)--(0.2,-0.2)--(-0.2,-0.2);
\node at (-0,-0) {$\dot{+}$};

\draw[-] (0,0.2)--(0,0.5);
\draw[-] (0,0.5)--(-1,0.5);
\draw[-] (-1, 0.5)--(-1, -1);
\node at (-1,-1) {$\bullet$};

\draw[-] (-0.5,0)--(-0.5,-0.25);
\draw[-] (-0.5,-0.75)--(-0.5,-1);
\node[draw, circle] at (-0.5,-0.5) {${\scriptstyle -1}$};
\draw[-] (0.5,0)--(0.5,-0.95);
\node at (0.5,-1) {$\circ$};

\end{tikzpicture} 
\quad  = \quad 
\begin{tikzpicture}[anchorbase,scale=1.5]

\draw[-] (-1, 0.5)--(-1, -1);
\node at (-1,-1) {$\bullet$};

\draw[-] (-1, 0.5)--(-0.5,0.5);
\draw[-] (-0.5,0.5)--(-0.5,-0.25);
\draw[-] (-0.5,-0.75)--(-0.5,-1);
\node[draw, circle] at (-0.5,-0.5) {${\scriptstyle -1}$};

\end{tikzpicture} 
\quad  = \quad 
\begin{tikzpicture}[anchorbase,scale=1.5]

\node at (-0.5,0.4) {$\bullet$};
\draw[-] (-0.5,0.4)--(-0.5,-0.25);
\draw[-] (-0.5,-0.75)--(-0.5,-1);
\node[draw, circle] at (-0.5,-0.5) {${\scriptstyle -1}$};
\end{tikzpicture} 
\quad  = \quad 
\begin{tikzpicture}[anchorbase,scale=1.5]

\node at (-0.5,0.4) {$\bullet$};
\draw[-] (-0.5,0.4)--(-0.5,-1);
\end{tikzpicture} 
\end{equation*}
Here the first equality is by the definition of $z^*$, the second by Lemma \ref{lem:transferring_plus}. The third equality is by the ``unit of $\dot{+}$'' relation on $\dot{+}$, $z$, as written in Relations \ref{rel:F_q_lin}. The fourth equality is due to the fact that $\eps$ is a unit of $m$ (see \ref{itm:str_rel_diag_bialg}), and the last equality is by the Relation \ref{rel:mu_coalg_mor} stating that $\eps^*\circ\mu_{-1} = \eps^*$ (see \ref{itm:str_rel_diag_mu_coalg}). This proves the case $k=2$.

Now let us consider the general case. Let $k>2$ and assume Equality \eqref{eq:lin_eq_and_eps_star} holds for $k-1$. By associativity of $\dot{+}$ we have:
\begin{equation*}
 \begin{tikzpicture}[anchorbase,scale=1.5]
\draw[-] (-1,0)--(-0.2,0);
\draw[-] (0.2,0)--(1.4,0);
\draw[-] (-0.2,-0.2)--(-0.2,0.2)--(0.2,0.2)--(0.2,-0.2)--(-0.2,-0.2);
\node at (-0,-0) {$\dot{+}$};
\draw[-] (-1,0)--(-1,-0.88);
\draw[-] (-0.7,0)--(-0.7,-0.88);
\draw[-] (0.6,0)--(0.6,-0.83);
\draw[-] (1.1,0)--(1.1,-0.88);
\draw[-] (1.4,0)--(1.4,-0.88);

\draw[-] (0,0.2)--(0,0.4);
\node at (0,0.43) {$\circ$};

\node at (-0.05,-0.8) {$\cdot$};
\node at (0,-0.8) {$\cdot$};
\node at (0.05,-0.8) {$\cdot$};

\node at (0.85,-0.8) {$\cdot$};
\node at (0.9,-0.8) {$\cdot$};
\node at (0.95,-0.8) {$\cdot$};

\node at (0.6,-0.83) {$\bullet$};
\end{tikzpicture} \quad = \quad
 \begin{tikzpicture}[anchorbase,scale=1.5]
\draw[-] (-0.4,0.3)--(-0.2,0.3);%lower plus
\draw[-] (0.2,0.3)--(0.4,0.3);
\draw[-] (-0.2,0.1)--(-0.2,0.5)--(0.2,0.5)--(0.2,0.1)--(-0.2,0.1);
\node at (0,0.3) {$\dot{+}$};

\draw[-] (-0.4,0.3)--(-0.4,-0.1);
\draw[-] (0.4,0.3)--(0.4,-0.1);

 \draw[-] (0,0.5)--(0,0.8);%upper plus
 \draw[-] (-2.6, 0.8)--(-2.2, 0.8);
 \draw[-] (-1.8, 0.8)--(0, 0.8); 
 \draw[-] (-2.2,0.6)--(-2.2,1)--(-1.8,1)--(-1.8,0.6)--(-2.2,0.6);
\node at (-2,0.8) {$\dot{+}$};
 
  \draw[-] (-2.6,-0.1)--(-2.6,0.8);
  \draw[-] (-2.4,-0.1)--(-2.4, 0.8);
 
  \draw[-] (-0.8,-0.1)--(-0.8, 0.8);
 
 \draw[-] (-1.5,0)--(-1.5, 0.8);
 \node at (-1.5,0) {$\bullet$};
 
 \draw[-] (-2,1)--(-2,1.2);%upper circ
 \node at (-2,1.23) {$\circ$};

\node at (-2.05,-0.05) {$\cdot$};
\node at (-2,-0.05) {$\cdot$};
\node at (-1.95,-0.05) {$\cdot$};

\node at (-1.05,-0.05) {$\cdot$};
\node at (-1.1,-0.05) {$\cdot$};
\node at (-1.15,-0.05) {$\cdot$};
\end{tikzpicture}
\InnaA{\quad\xlongequal{\substack{\text{induction }\\\text{assumption}}}\quad}
  \begin{tikzpicture}[anchorbase,scale=1.5]
\draw[-] (-0.4,0.3)--(-0.2,0.3);%lower plus
\draw[-] (0.2,0.3)--(0.4,0.3);
\draw[-] (-0.2,0.1)--(-0.2,0.5)--(0.2,0.5)--(0.2,0.1)--(-0.2,0.1);
\node at (0,0.3) {$\dot{+}$};

\draw[-] (-0.4,0.3)--(-0.4,-0.1);
\draw[-] (0.4,0.3)--(0.4,-0.1);

 \draw[-] (0,0.5)--(0,0.8);

  \draw[-] (-1.8,-0.1)--(-1.8,0.8);
  \draw[-] (-1.6,-0.1)--(-1.6, 0.8);
 
  \draw[-] (-0.6,-0.1)--(-0.6, 0.8);
 
 \node at (-1.8,0.8) {$\bullet$};
 \node at (-1.6,0.8) {$\bullet$};
  \node at (-0.6,0.8) {$\bullet$};
  \node at (0,0.8) {$\bullet$};

\node at (-1.25,0.15) {$\cdot$};
\node at (-1.2,0.15) {$\cdot$};
\node at (-1.15,0.15) {$\cdot$};
\end{tikzpicture}
\end{equation*}
Thus we have shown that 
\begin{equation*}
 z^* \circ (\dot{+})^{it} \circ \left(\id_{\V^{\otimes (i-1)}}\otimes \eps \otimes \id_{\V^{\otimes (k-i)}}  \right)= (\eps^*)^{\otimes (k-2)} \otimes (\eps^* \circ \dot{+}).
\end{equation*}
Next, by Relation \ref{rel:plus_coalg_mor} (see also \ref{itm:str_rel_diag_plus_coalg}) we have:
$ (\eps^* \circ \dot{+}) = \eps^* \otimes \eps^*$, i.e.
\begin{equation*}
\begin{tikzpicture}[anchorbase,scale=1]
\draw[-] (-0.5,0)--(-0.2,0);%\dot{+}
\draw[-] (0.2,0)--(0.5,0);
\draw[-] (-0.2,-0.2)--(-0.2,0.2)--(0.2,0.2)--(0.2,-0.2)--(-0.2,-0.2);
\node at (-0,-0) {$\dot{+}$};
\draw[-] (0,0.2)--(0,0.5);
\node at (0,0.5) {$\bullet$};
\draw[-] (-0.5,0)--(-0.5,-0.5);
\draw[-] (0.5,0)--(0.5,-0.5);
\end{tikzpicture} \quad  = \quad  \begin{tikzpicture}[anchorbase,scale=1]
\draw[-] (-0.3,0.5)--(-0.3,-0.5);
\draw[-] (0.3,0.5)--(0.3,-0.5);
\node at (-0.3,0.5) {$\bullet$};
\node at (0.3,0.5) {$\bullet$};
\end{tikzpicture} 
\end{equation*}  
This proves Equations \eqref{eq:lin_eq_and_eps_star} and its diagram version \eqref{eq:lin_eq_and_eps_star_diagram}, so the proof of the Lemma is complete.
\end{proof}

\section{Universal properties}\label{sec:univ_prop}

This section contains several of the main theorems. We prove the universal property of $\kar{t}$ and of $\Rep(GL_n(\F_q))$ and identify our category with the category studied by Knop in \cite{K, K2}.

\subsection{Universal property of the Deligne category}\label{ssec:univ_prop_Deligne}

\begin{theorem}\label{thrm:univ_prop_Del}
 Let $\mathcal{C}$ be a Karoubi additive rigid SM category, and let $\V $ be an $\F_q$-linear Frobenius space in $\mathcal{C}$. Let $t = \dim(\V)$. Then there exists a SM functor
 $$ F_{\V}: \kar{t} \to \mathcal{C}, \;\; \V_t \longmapsto \V.$$
 
\end{theorem}
\begin{proof}
We may assume $\mathcal{C}$ to be a strict rigid symmetric monoidal category (see \cite{SML}).

 Define a functor $F:\mathcal{T}(\underline{GL}_t) \to \mathcal{C}$ by sending $[k]$ to $\V^{\otimes k}$ and for each linear subspace $R \subset \F_q^{k+l}$, $F$ sends $f_R: [k] \to [l]$ (as defined in Section \ref{sec:Deligne_def}) to $\widetilde{f}_R:\V^{\otimes k} \to \V^{\otimes l}$ (as defined in Definition \ref{def:f_R_in_C}). We then extend $F$ by linearity to the entire $\Hom([k], [l])$.
 
 We claim that this defines a functor: indeed, by Proposition \ref{prop:composition_phi_R} we have: for any linear subspaces $R \subset \F_q^{s+k}, S \subset \F_q^{k+l}$, $F(f_S \circ f_R) = \widetilde{f}_S \circ \widetilde{f}_R$ as required.
 
 Now, to check that this is a monoidal functor, we only need to check that $F(f_{R_1} \otimes f_{R_2}) = \widetilde{f}_{R_1} \otimes \widetilde{f}_{R_2}$ for any linear subspaces $R_1 \subset \F_q^{r_1+s_1}, R_2 \subset \F_q^{r_2+s_2}$. This is proved in Lemma \ref{lem:tensor_prod_f}. By the construction of the Karoubi envelope, the functor $F$ extends uniquely to a functor $ F_X: \kar{t} \to \mathcal{C}$. 
 
 The fact that the functor is symmetric follows from the fact that the symmetry morphism
 $\sigma\in \Hom_{\mathcal{C}}(\V^{\otimes l} \otimes \V^{\otimes k}, \V^{\otimes k} \otimes \V^{\otimes l})$ is given by $\widetilde{f}_{R'}$ where $$ R' = \{(a_1, \ldots, a_l, b_1, \ldots, b_k,-b_1, \ldots, -b_k,-a_1, \ldots, -a_l)|\, \forall i, j, \; a_i,  b_j \in \F_q\} \subset  \F_q^{2k+2l}.$$
\end{proof}

\begin{example}
 Let $\mathcal{C} = \Rep(GL_n(\F_q))$, $\V = \V_n$. The functor $$F_n: \kar{t=q^n} \longrightarrow \Rep(GL_n(\F_q)), \;\;\V_{t=q^n} \longmapsto \V_n$$ is the semisimplification of $\kar{t=q^n}$ seen in Proposition \ref{prop:functor_F_n}.
\end{example}

\begin{corollary}\label{cor:univ_prop_2}
 The functor
 $$Fun^{\otimes}\left(\kar{t}, \mathcal{C}\right) \longrightarrow Frob_{\F_q}(\mathcal{C}, t), \;\; F\longmapsto F(\V_t)$$
 is an equivalence of categories. Here 
 \begin{itemize}
  \item $Frob_{\F_q}(\mathcal{C}, t)$ is the subcategory of $\mathcal{C}$ consisting of $t$-dimensional $\F_q$-linear Frobenius spaces and isomorphisms between them.
  \item $Fun^{\otimes}\left(\kar{t}, \mathcal{C}\right)$ is the category of SM functors $\kar{t} \to \mathcal{C}$ and natural SM isomorphisms.
 \end{itemize}
\end{corollary}

\begin{proof}
 To define the inverse functor $$Frob_{\F_q}(\mathcal{C}, t) \longrightarrow Fun^{\otimes}\left(\kar{t}, \mathcal{C}\right), \;\; \V \longmapsto F_{\V}$$ we just need to explain how an isomorphism $\phi:\V \to \V'$ extends to a natural SM isomorphism $F_{\V} \to F_{\V'}$. Indeed, by the construction of the functor $F_{\V}$, $\phi$ defines isomorphisms $ F_{\V}([s])\longrightarrow F_{\V'}([s])$ for each $s\in \Z_{\geq 0}$, and these isomorphisms commute with the maps $\widetilde{f}_R$ constructed in the proof of Theorem \ref{thrm:univ_prop_Del}. Thus $\phi$ induces a natural SM isomorphism between the restricted functors $ F_{\V}, F_{\V'}: \mathcal{T}(\underline{GL}_t) \rightrightarrows \mathcal{C}$. By the properties of additive Karoubi envelopes, this isomorphism extends uniquely to a natural SM isomorphism $F_{\V} \to F_{\V'}$.
\end{proof}

\begin{corollary} 
Any morphism $[s]\to [k]$ in $\mathcal{T}(\underline{GL}_t)$ (respectively, in $\kar{t}$) can be presented as $\C$-linear combination of compositions \InnaA{and $\otimes$ products} of morphisms $m, m^*, \eps, \eps^*$, $\sigma$ (the symmetry morphism), $z$, $\dot{+}$ and $\mu_a$ for $a \in \F_q^{\times}$. 
\end{corollary}
\begin{proof}
For $\mathcal{C}:=\mathcal{T}(\underline{GL}_t)$, Theorem \ref{thrm:univ_prop_Del} showed that any morphism $f_R$ is a composition of the morphisms above. 

Since in this case $f_R$'s form a basis for $Hom([s], [k])$ (by construction of the Deligne category), the required statement now follows.
\end{proof}

\subsection{Relation to the category defined \texorpdfstring{in \cite{K, K2}}{by Knop}}\label{ssec:knop}

In this section we briefly recall the definition of the category \InnaA{$\kar{t}$} as given in \cite{K, K2} and explain why it is equivalent to our definition. \InnaA{We will use the notation as in \cref{ssec:knop_indexing} (see \cref{notn:Knop}): given $R\subset \F_q^{l+k}, S\subset \F_q^{k+s}$, we denote by $S\diamond R$ the image of the map $R\times_{\F_q^k} S \to \F_q^{l+s}$ and by $e(S, R)$ the dimension of the kernel of this map.}

Define $\mathcal{T}'$ as a category whose objects are non-negative integers; the objects will be denoted by $[k]$, $k\in \Z_{\geq 0}$.
We set $$ \Hom_{\mathcal{T}'} ([s], [k]) = \C Rel_{s, k}$$
where the morphism $[s] \to [k]$ corresponding to $R \in Rel_{s, k}$ is denoted by $h_R$. The composition of morphisms given as follows: for $R \in Rel_{s, k}$, $S \in Rel_{k, l}$ we set $$h_S \circ h_R := t^{e(S, R)} h_{S \diamond R}.$$

The tensor structure of $\mathcal{T}'$ is given by $[l]\otimes [k]:=[l+k]$, and tensor products of morphisms are given as follows: for $\F_q$-linear subspaces $R_1 \subset \F_q^{r_1}$ and $R_2\subset \F_q^{r_2}$, set $h_{R_1} \otimes h_{R_2} :=h_{R_1\times R_2}$ where $R_1\times R_2 \subset \F_q^{r_1+r_2}$ is considered as an $\F_q$-linear subspace. 

Each object is then self-dual via the evaluation and coevaluation maps
$$h_{R}:[k]\otimes [k] \to \triv, \;\; h_{R}: \triv \to [k]\otimes [k]$$
where $R \subset \F_q^{2k}$ is given by $R=\{(a_1, \ldots, a_k, a_k, \ldots, a_1)|a_1, \ldots, a_k \in \F_q\}.$

The symmetric structure is given by the morphisms
$\sigma=h_{R}: [l]\otimes [k] \to  [k]\otimes [l]$ where $$ R = \{(a_1, \ldots, a_l, b_1, \ldots, b_k,b_1, \ldots, b_k,a_1, \ldots, a_l)|\, \forall i, j, \; a_i,  b_j \in \F_q\} \subset  \F_q^{2k+2l}.$$

Let $\mathcal{K}$ be the Karoubi additive envelope of $\mathcal{T}'$. This is the Deligne category corresponding to finite linear groups as defined in \cite{K, K2}.
\begin{lemma}\label{lem:rel_to Knop}
 We have an equivalence of symmetric monoidal categories:
 $$\mathcal{T} \to \mathcal{T}', [k] \mapsto [k], f_R \mapsto h_{R^\perp},$$
 which induces an equivalence of symmetric monoidal categories
 $\kar{t} \cong \mathcal{K}$.
 
\end{lemma}

\begin{proof}
Lemma \ref{lem:Knop_basis_vs_ours} implies that the correspondence $[k] \mapsto [k], f_R \mapsto \InnaA{h}_{R^\perp}$ defines a functor, and the fact that it is monoidal means that $R_1^\perp \times R_2^\perp = (R_1\times R_2)^\perp$ for any $R_1 \subset \F_q^{r_1}$ and $R_2\subset \F_q^{r_2}$ (the latter statement being straightforward). 
To show that this is a symmetric functor, we need to check that symmetry morphisms correspond to symmetry morphisms: that is, to check that for
$$ R = \{(a_1, \ldots, a_l, b_1, \ldots, b_k,-b_1, \ldots, -b_k,-a_1, \ldots, -a_l)|\, \forall i, j, \; a_i,  b_j \in \F_q\} \subset  \F_q^{2k+2l}$$
we have:
$$ R^\perp = \{(a'_1, \ldots, a'_l, b'_1, \ldots, b'_k,b'_1, \ldots, b'_k,a'_1, \ldots, a'_l)|\, \forall i, j, \; a'_i,  b'_j \in \F_q\}.$$

Checking this equality is again straightforward; so we proved that our correspondence defines a symmetric monoidal functor. Clearly, this functor is an equivalence of categories $\mathcal{T} \to \mathcal{T}'$, with an inverse functor given by 
$$\mathcal{T}' \to \mathcal{T}', [k] \mapsto [k], h_R \mapsto f_{R^\perp}.$$
Finally, the functor $\mathcal{T} \to \mathcal{T}'$ induces an additive symmetric monoidal functor 
 $\kar{t} \to \mathcal{K}$, and the fact that the former is an equivalence implies that so is the latter. 
\end{proof}

The equivalence above now allows us to use the results by from \cite{K, K2} concerning the semisimplicity of $\mathcal{K}$:

\begin{theorem}[F. Knop]
 The category $\kar{t}$ is abelian semisimple for $t \notin \{q^n|n\in \Z_{\geq 0}\}$. For $t=q^n$, the category $\kar{t}$ is not abelian nor semisimple, and its semisimplification is $\Rep(GL_n(\F_q))$.
\end{theorem}

%\begin{remark}
%Proposition \ref{prop:gen_morphisms} showed that any morphism $\V_n^{\otimes s} \to \V_n^{\otimes k}$ in $\Rep(GL_n(\F_q))$ can be presented as $\C$-linear combination of compositions of morphisms $ev, coev$ and $\mu_A$ for some matrices $A$. The same proof shows that this statement holds, word for word, for morphisms $[s]\to [k]$ in $\kar{t}$. 
%\end{remark}

\subsection{Universal property of the category \texorpdfstring{$\Rep(GL_n(\F_q))$}{Rep(GLn(Fq))}}\label{ssec:univ_prop_classical}

We begin with a reminder on the universal property of the category $\Rep(S_n)$, where $S_n$ is the symmetric group on $n$ letters. This property follows e.g. from \cite{Del07} \cite{CO2}, but we prove it here for completeness of exposition.

\begin{lemma}\label{lem:S_n_univ_prop}
 Let $\mathcal{C}$ be a pre-Tannakian symmetric tensor category, and let $A$ be a commutative Frobenius object in $\mathcal{C}$ such that $\wedge^N A = 0$ for some $N\geq 1$. Then $\dim A = n$ for some integer $n$, and there exists a SM functor
 $$ \Rep(S_n) \to \mathcal{C},$$
 sending the $n$-dimensional permutation representation of $S_n$ to $A$.

\end{lemma}
\begin{proof}
Let $\mathcal{C}_{A}$ be the full tensor subcategory of $\mathcal{C}$ generated by $A$ (it is the full subcategory whose objects are subquotients of finite direct sums of tensor powers of $A$). It is enough to establish that $\dim_{\mathcal{C}} A = n$ for some $n$, and the existence of a SM functor $$ \widetilde{F}_{A}: \Rep(S_n) \to \mathcal{C}_{A}, \;\; A_n \longmapsto A.$$

Since $A$ is annihilated by an exterior power, the category $\mathcal{C}_{A}$ is Tannakian; in other words, we have an equivalence of tensor categories $\mathcal{C}_{A} \cong \Rep(G)$ for some \InnaA{affine} algebraic group $G$ (here $\Rep(G)$ \InnaA{denotes} the category of finite-dimensional algebraic representations of $G$). We fix such an equivalence, so we can now think of $A$ as a representation of $G$, with an underlying complex vector space. In fact, $A$ generates $\Rep(G)$ under taking tensor powers, direct sums and subquotients, which implies that $A$ is a faithful representation of $G$. 

Let $X:=Spec(A)$ be the \InnaA{affine} $G$-scheme corresponding to $A$. Since $A$ is a finite-dimensional \InnaA{Frobenius algebra} over $\C$, we conclude that $X$ is a finite set, \InnaA{and $|X|=n$}. The action of $G$ on $X$ induces a homomorphism $G\to Aut(X) = S_X$, which in turn induces a SM functor $$ \widetilde{F}_{A}: \Rep(S_X) \to \mathcal{C}_{A}, \;\; \C X \longmapsto A$$ as required. 

\end{proof}

\begin{theorem}\label{thrm:univ_prop_classical}
 Let $\mathcal{C}$ be a pre-Tannakian symmetric tensor category, and let $\V$ be an $\F_q$-linear Frobenius space in $\mathcal{C}$ such that $\wedge^N \V = 0$ for some $N\geq 1$. Then there exists a finite-dimensional $\F_q$-vector space $V$ such that $\dim_{\mathcal{C}} \V = \dim_{\F_q} V$ and there exists a SM functor
 $$ \widetilde{F}_{\V}: \Rep(GL(V)) \to \mathcal{C}, \;\; \V_n \longmapsto \V.$$

\end{theorem}

\begin{proof}
Let $\mathcal{C}_{\V}$, $G$ be as in the proof of Lemma \ref{lem:S_n_univ_prop}, so that $\mathcal{C}_{\V} \cong \Rep(G)$; as before, $\V$ can be considered as a faithful, finite-dimensional complex representation of the \InnaA{affine} algebraic group $G$.

It is enough to establish that $\dim_{\mathcal{C}} \V = q^n$ for some $n$, and the existence of a SM functor $$ \widetilde{F}_{\V}: \Rep(GL_n(\F_q)) \to \Rep(G) \cong \mathcal{C}_{\V}, \;\; \V_n \longmapsto \V.$$

Consider $\V$ as a commutative and cocommutative Frobenius algebra in $\Vect$; then $\V^*$ is a a commutative and cocommutative Frobenius algebra as well, and a faithful representation of $G$.

Let $V:=Spec(\V^*)$; this is a finite set with an action of $G$, and by the proof of Lemma \ref{lem:S_n_univ_prop}, the action of $G$ on $V$ induces an embedding $G \hookrightarrow S_V$.
The operators $\mu_a$ ($a\in \F_q^{\times}$), $\dot{+}$ and $z$ on $\V$ are morphisms of coalgebras by Relations \ref{rel:mu_coalg_mor}, \ref{rel:z_coalg_mor}, \ref{rel:plus_coalg_mor}, hence they endow the $G$-set $V$ with the structure of an $\F_q$-linear vector space. In particular, this gives an embedding of finite groups $G \hookrightarrow GL(V)$, and we obtain $$\dim_{\mathcal{C}} \V = \dim_{\C} \V = \dim_{\F_q}(V).$$ The embedding $G \hookrightarrow GL(V)$ induces  a SM functor $$ \widetilde{F}_{\V}: \Rep(GL_n(\F_q)) \to \Rep(G), \;\; \V_n \longmapsto \V,$$ which completes the proof. 
\end{proof}

The following corollary is proved in the same manner as Corollary \ref{cor:univ_prop_2}:

\begin{corollary}\label{cor:univ_prop_classical_2}
  The functor
 $$Fun^{\otimes}\left(\Rep(GL_n(\F_q)), \mathcal{C}\right) \longrightarrow Frob_{\F_q}(\mathcal{C}, q^n), \;\;\widetilde{F} \longmapsto \widetilde{F}(\V_n) $$
 is an equivalence of categories. Here 
 \begin{itemize}
  \item $Frob_{\F_q}(\mathcal{C}, q^n)$ is the subcategory of $\mathcal{C}$ consisting of $q^n$-dimensional $\F_q$-linear Frobenius spaces $\V$ annihilated by some exterior power, and isomorphisms between them.
  \item $Fun^{\otimes}\left(\Rep(GL_n(\F_q)), \mathcal{C}\right)$ is the category of SM functors $\Rep(GL_n(\F_q)) \to \mathcal{C}$ and natural SM isomorphisms.
 \end{itemize}
\end{corollary}

\begin{corollary} \label{cor-factor}
 Let $\mathcal{C}$ be a pre-Tannakian symmetric tensor category, and let $\V$ be an $\F_q$-linear Frobenius space in $\mathcal{C}$ annihilated by some exterior power. Let $n\in \mathbb{Z}_{\geq 0}$ such that $\dim \V = q^n$. Then the SM functor $F_{\V}: \kar{t=q^n} \to \mathcal{C}$ seen in \ref{thrm:univ_prop_Del} factors through the functor $F_n: \kar{t=q^n} \to \Rep(GL_n(\F_q))$, and we have a natural SM isomorphism $$ F_{\V} \longrightarrow \widetilde{F}_{\V} \circ F_n. $$ 
\end{corollary}

\begin{remark}
 In a subsequent article, we will show that for a pre-Tannakian  symmetric tensor category $\mathcal{C}$, and an $\F_q$-linear Frobenius space $\V$ in $\mathcal{C}$ {\it not} annihilated by any exterior power, the functor $F_{\V}$ factors through a certain tensor category $\urep$ in which $\kar{t}$ sits as a full SM subcategory. This category $\urep$ is called the {\it abelian envelope} of $\kar{t}$ and it is expected to be similar to the Deligne tensor categories $Rep^{ab}(\underline{S}_t)$, $Rep^{ab}(\underline{GL}_t(\C))$ seen in \cite{CO2, Del07, EHS}.
\end{remark}

\section{Injective representations of infinite general linear group}\label{sec:repinf}

We now study the analogue of Sam-Snowden's category of algebraic representations of $S_{\infty}$ (see \cite{SS}) for the group $GL_{\infty}(\F_q) = \bigcup_{n\geq 1} GL_n(\F_q)$. We restrict ourselves here to the full subcategory in the category of $GL_{\infty}(\F_q)$-modules whose objects are tensor powers of the natural representation $\V_{\infty} = \C\F_q^{\infty}$. Our main result is that this category is the universal SM category generated by an $\F_q$-linear semi-Frobenius space.

\subsection{A subcategory of the Deligne category}\label{ssec:infty_def}

\begin{definition}
 For any $s, k\geq 0$, denote $$Rel^{\infty}_{s, k} = \{ R \subset \F_q^{s+k} \; \text{ linear subspace, where } \; R \to \F_q^k \; \text{ is surjective }  \}.$$ 
\end{definition}
 This is a subset of $Rel_{s, k}$ as defined in Section \ref{sec:Deligne_def}. Recall the operations defined in Section \ref{sec:classical_endom}.

 \begin{lemma}\label{lem:Rel_infty_well_def}
  Let $R \in Rel^{\infty}_{s, k}$, $S \in Rel^{\infty}_{k, l}$. Then $$d(R, S) =0, \;\; S \star R \in Rel^{\infty}_{s,l}.$$
  
  Furthermore, for $R_1 \in Rel^{\infty}_{s_1, k_1}$, $R_2 \in Rel^{\infty}_{s_2, k_2}$, consider $R_1 \times R_2 \subset \F_q^{s_1+s_2+k_1+k_2}$. Then $R_1\times R_2 \in Rel^{\infty}_{s_1+s_2, k_1+k_2}$.
 \end{lemma}

\begin{proof} 
Using the notation of Definition \ref{def:composition_star_d_R_S} we have for $R \in Rel^{\infty}_{s, k}$, $S \in Rel^{\infty}_{k, l}$ \[ \dim \left(\substack{\text{ projection of } (R, 0)+ (0,S) \\ \text{ on the subspace } \F_q^k \subset \F_q^{l+k+s}}\right) = k \] which implies $d(R,S) = 0$. 

For $R \in Rel^{\infty}_{s, k}$, $S \in Rel^{\infty}_{k, l}$, there exist $A \in Mat_{k\times s}(\F_q), A'\in Mat_{d_1\times s}(\F_q), B\in Mat_{l\times k}(\F_q), B'\in Mat_{d_2\times k}(\F_q)$ such that $$R = Row\begin{bmatrix} -A &I_k\\
A' &0
\end{bmatrix}, \;\; S = Row \begin{bmatrix} -B &I_l\\
B' &0
\end{bmatrix} \;\; \InnaA{\Longrightarrow} \;\; S\circledast R = Row \begin{bmatrix}
-A &I_k &0\\
A' &0 &0\\
0 &-B &I_l \\
0 &B' &0 
\end{bmatrix}$$
Multiplying the above matrix by the invertible matrix $\begin{bmatrix}
I_{k} &0 &0 &0\\
0 &I_{d_1} &0 &0\\
B &0 &I_l &0\\
-B' &0 &0 &I_{d_2}
\end{bmatrix}$, we obtain: $$S\circledast R = Row \begin{bmatrix}
-A &I_k &0\\
A' &0 &0\\
-BA &0 &I_l \\
B'A &0 &0 
\end{bmatrix} \; \Longrightarrow \; S\star R = Row \begin{bmatrix}
-BA &I_l \\
A'&0\\
B'A  &0 
\end{bmatrix}$$
Thus the projection of $S \star R$ to $\F_q^s$ in $(\F_q^l ,0, \F_q^s)$ is surjective. The last statement of the lemma is obvious.
\end{proof}

Let us consider a ($\C$-linear) category $\mathcal{I}_{\infty} $ whose objects are $[k]$, $k\geq 0$, with morphism spaces $\Hom_{\mathcal{I}_{\infty} } ([s], [k]) := \C Rel^{\infty}_{s, k}$, where we denote the morphism corresponding to $R\in Rel^{\InnaA{\infty}}_{s,k}$ by $f_R$. The composition of morphisms defined as follows: for $R \in Rel^{\infty}_{s, k}$, $S \in Rel^{\infty}_{k, l}$, we define $$f_S \circ f_R :=  f_{S\star R}.$$

We also define the monoidal structure on $\mathcal{I}_{\infty}$ as follows: we set $[l]\otimes [k]:=[l+k]$, and tensor products of morphisms given  by $f_{R_1} \otimes f_{R_2} :=f_{R_1\times R_2}$, just like in Section \ref{sec:Deligne_def}.

The symmetric structure on $\mathcal{I}_{\infty}$ is again given by the morphisms
$s=f_{R'}: [l]\otimes [k] \to  [k]\otimes [l]$ where $$ R' = \{(a_1, \ldots, a_l, b_1, \ldots, b_k,-b_1, \ldots, -b_k,-a_1, \ldots, -a_l)|\, \forall i, j, \; a_i,  b_j \in \F_q\} \in Rel^{\infty}_{k+l, k+l}.$$

Clearly, for each $t\in \C$, we obtain a SM embedding $\Gamma_t: \mathcal{I}_{\infty} \hookrightarrow \kar{t}$ which is not full. When $t=q^n$, we can compose $\Gamma_{t=q^n}$ with the specialization functor $$F_n: \kar{t=q^n} \longrightarrow \Rep(GL_n(\F_q)), \;\;\V_{t=q^n} \longmapsto \V_n$$
to obtain a SM functor
$$\widetilde{\Gamma}_{n}: \mathcal{I}_{\infty} \to \Rep(GL_n(\F_q)), \;\; [1] \longmapsto \V_n. $$

\subsection{Representations of the infinite general linear group}\label{ssec:repinf_def}
\begin{definition}\label{def:repinf}

\mbox{}

 \begin{enumerate}
  \item 

 Let $GL_{\infty}(\F_q) = \bigcup_{n\geq 1} GL_n(\F_q)$, defined with respect to the embedding $GL_n(\F_q) \subset GL_{n+1}(\F_q)$ given by $A \mapsto \begin{bmatrix}
 A &0 \\
 0 & 1
\end{bmatrix}
$.

\item Let $\iota_n:\F_q^n \to \F_q^{n+1}$ be the natural embedding, given by adding a coordinate $0$ at the end. Let $\InnaA{V:=}\F_q^{\infty} = \bigcup_{n\neq 0} \F_q^n$; this is the countable-dimensional vector space 
consisting of infinite sequences of elements in $\F_q$ which have finite support 
(i.e. sequences $(a_i)_{i=1}^{\infty}$, $a_i \in \F_q$ and $a_i =0$ for $i>>1$). We will denote by $\F_q^n$ the subspace of sequences $(a_i)_{i=1}^{\infty}$ in  $\F_q^{\infty}$ for which $a_i=0$ for $i>n$.

We denote by $\V_{\infty}$ the representation $\C\F_q^{\infty}$ of $GL_{\infty}(\F_q)$.
 
 We will consider a natural basis for $\V_{\infty}^{\otimes s}$ given by elements of $\InnaA{V}^{\times s}$ written as $(v_1| \ldots| v_s)$, $v_i \in \InnaA{V}$.

\InnaA{
For any $s\geq 0$, consider the embedding $ \V_n^{\otimes s} \to \V_{n+1}^{\otimes s}$ given by $\iota_n:\F_q^n\to \F_q^{n+1}$. By abuse of notation, we will denote the image of this embedding again by $\V_{n}^{\otimes s}$.}

\item 
\InnaA{In this section, we will denote by $f_R^{(n)} : \V_n^{\otimes s}\to \V_n^{\otimes k}$ the map corresponding to $R\subset \F_q^{s+k}$ as defined in Section \ref{sec:classical_endom}, which we previously denoted simply by $f_R$.}
 \item 
 We denote by $L_n, H_n \subset P_n\subset GL_{\infty}(\F_q)$ the following subgroups:
 
 By $P_n$ we denote the parabolic subgroup of all matrices $ \begin{bmatrix}
 C &A \\
 0 & B
\end{bmatrix} \in GL_{n+k}(\F_q)
$, $k\geq 0$ where $C\in GL_n(\F_q), B\in GL_k(\F_q)$ and $A \in Mat_{n\times k}(\F_q)$. 

By $H_n \subset P_n$ we denote the subgroup of all matrices as above for which $C$ is the identity matrix of size $n\times n$.

We will also denote by $L_n \subset P_n$ the Levi subgroup of $P_n$: the subgroup of all matrices as above for which $A=0$.

Denote by $\Gamma_n: \repinf \to \Rep(GL_n(\F_q))$ the following functor: for $W \in \repinf$, $\Gamma_n(W)$ is given by the $H_n$-invariants in $Res_{P_n}^{GL_\infty(\F_q)} W$.

 \item We define $\mathcal{T}(GL_{\infty}(\F_q))$ to be the full subcategory of $GL_{\infty}(\F_q)-mod$ whose objects are tensor powers of $\V_{\infty}$.

 \item For $R\in Rel^{\infty}_{s,k}$, let $\overline{f}_R: \V_{\infty}^{\otimes s}\to \V_{\infty}^{\otimes k}$ be the morphism defined by 
$$ v_1\otimes \ldots \otimes v_s \mapsto \sum_{\substack{(w_1|\ldots|w_k)\in \InnaA{V}^{\times k} \\ (v_1| \ldots| v_s|w_1|\ldots|w_k) \in R^{\perp}_{\InnaA{ V}} }} w_1 \otimes\ldots\otimes w_k. $$
Since the projection $R \to \F_q^k$ is surjective, this sum is actually finite. 

 \end{enumerate}
\end{definition} 

\begin{remark}
The tensor powers of $\V_{\infty}$ are admissible $GL_{\infty}(\F_q)$-representations in the sense of Nagpal \cite{N2}: that is, every vector in such a representation is stabilized by $P_n$ for some $n>0$.

 Every admissible representation (in the sense of Nagpal \cite{N2}) is a subquotient of a finite direct sum of tensor powers of $\V_{\infty}$. 
\end{remark}

Let us begin by showing why the sets $Rel^{\infty}_{s,k}$ occur in the context of representation stabilization.

\InnaA{

\begin{lemma}\label{lem:Gamma_n_stabilization}
 Let $s, k\geq 0$ and let $n\geq  k+s$. Let $f\in \Hom_{GL_{n+1}(\F_q)}(\V_{n+1}^{\otimes s}, \V_{n+1}^{\otimes k})$, such that $f(\V_{n}^{\otimes s}) \subset \V_{n}^{\otimes k}$.
 Then $f\in span\{f_R^{(n+1)}\,|\,  R\in Rel^{\infty}_{s,k}\}$.
\end{lemma}
For the proof of this lemma, will use the following straightforward statement:
\begin{lemma}\label{lem:aux_orth_span}
Let $k\geq 0$ and let $U$ be a vector space over $\F_q$. 
\begin{enumerate}
    \item Let $R\subset \F_q^{k}$. Then $$\max_{(u_1| \ldots| u_k) \in R^{\perp}_U } \dim span\{u_1, \ldots, u_k\} = \dim R^{\perp} = k-\dim R.$$
    \item Let $\underline{u}:=(u_1| \ldots| u_k) \in U^{\times k}$. Let $\mathcal{R}_{\underline{u}}:=\{R\subset \F_q^k \,|\, (u_1| \ldots| u_k) \in R^{\perp}_U \}$. Then there exists $R\in \mathcal{R}_{\underline{u}}$ such that $S\subset R$ for any $S\in \mathcal{R}_{\underline{u}}$, and furthermore, $$\dim R = k - \dim span \{u_1, \ldots, u_k\}.$$
\end{enumerate}
\end{lemma}
\begin{proof}[Proof of \cref{lem:Gamma_n_stabilization}]
By \cref{ssec:basis_of_endomorphisms}, we may write $f=\sum_{R\subset \F_q^{s+k}} \alpha_R f_R^{(n+1)}$ for some $\alpha_R\in \C$. Consider the partially ordered set $X = \{R\subset \F_q^{s+k} \, |\, \alpha_R \neq 0\},$ the order given by inclusion. We need to prove that $X\subset Rel^{\infty}_{s,k}$; equivalently, that any minimal element of $X$ lies in $Rel^{\infty}_{s,k}$. 

Let $R \in X$ be such a minimal element. Let $\pi: \F_q^{s+k} \to \F_q^k$ be the obvious projection. We need to show that $\pi(R)= \F_q^k$. We will now choose $v_1, \ldots, v_s, w_1, \ldots, w_k \in \F_q^{n+1}$ such that $$(v_1| \ldots| v_s|w_1|\ldots|w_k) \in R^{\perp}_{\F_q^{n+1}} \;\; \text{ and }\;\;\dim span\{v_1, \ldots, v_s, w_1, \ldots, w_k\} = k+s-\dim R.$$

Since $n\geq k+s$, one may choose $v_1, \ldots, v_s \in \iota_n(\F_q^n) \subset \F_q^{n+1}$. Assume that $\pi(R)\neq \F_q^k$; then we may also choose $w_1, \ldots, w_k$ such that at least one of them does not belong to $\iota_n(\F_q^n)$. The set $\mathcal{R}:=\mathcal{R}_{(v_1| \ldots| v_s, w_1| \ldots| w_k)}$ as defined in \cref{lem:aux_orth_span} contains $R$; since $R$ is a minimal element in $X$, we actually have $ \mathcal{R} \cap X = \{R\}.$

Consider the decomposition of
$f(v_1 \otimes \ldots \otimes v_s)$ with respect to the basis $(\F_q^{n+1})^{\times k}$ of $\V_{n+1}^{\otimes k}$. The summand $w_1 \otimes \ldots \otimes w_k$ appears with coefficient 
$ \sum_{S\in \mathcal{R}} \alpha_S$ which is, by the computation above, $ \alpha_R$. But $f(v_1 \otimes \ldots \otimes v_s) \in \V_n^{\otimes k}$, so this coefficient must be $0$, contradicting that $R\in X$. Thus $R\in Rel_{s, k}^{\infty}$ and the statement is proved.

\end{proof}
}
 The functor $\Gamma_n$ is clearly left-exact. It also behaves nicely on the subcategory $\mathcal{T}(GL_{\infty}(\F_q))$:
\begin{lemma}\label{lem:Gamma_n_on_tensor_powers}
 $\Gamma_n(\V_{\infty}^{\otimes k}) = \V_n^{\otimes k} $ and $\Gamma_n(\overline{f}_R) = {f}^{\InnaA{(n)}}_R $ for any $R\in Rel^{\infty}_{s,k}$.
\end{lemma}
\begin{proof}
If a finite sum $$\sum_{v_1, \ldots, v_k \in \F_q^{\infty}} \alpha_{v_1,\ldots, v_k} v_1\otimes \ldots \otimes v_k $$ belongs to $(\V_{\infty}^{\otimes k})^{H_n}$ then $v_1, \ldots, v_k \in \F_q^n \subset \F_q^{\infty}$. This proves the first statement.

In particular, $\Gamma_n(\overline{f}_R)$ is just the restriction of $\overline{f}_R$ to the subspace $$\V_n^{\otimes k} \cong (\C \F_q^n)^{\otimes k} \subset (\C\F_q^{\infty})^{\otimes k} = \V_{\infty}^{\otimes k}.$$
\end{proof}
\begin{corollary}\label{cor:Gamma_n_SM}
 The functor $\Gamma_n: \mathcal{T}(GL_{\infty}(\F_q)) \to \Rep(GL_n(\F_q))$ is SM.
\end{corollary}
In fact, one can show (see also \cite{N2, SS}) that the functor $\Gamma_n: \repinf \to \Rep(GL_n(\F_q))$ is SM.

\begin{proposition}\label{prop:rep_infty_inj_objects}
 There exists a functor $F: \mathcal{I}_{\infty} \to \mathcal{T}(GL_{\infty}(\F_q))$ which sends $$[k] \longmapsto \V_{\infty}^{\otimes k}, \;\; f_R \longmapsto \overline{f}_R.$$ Furthermore, this functor is SM and is an equivalence of SM categories.
\end{proposition}
\begin{proof}
 We need to check that the following statements hold:
 \begin{enumerate}
 \item For $R \in Rel^{\infty}_{s, k}$, $S \in Rel^{\infty}_{k, l}$, we have $\overline{f}_S \circ \overline{f}_R = \overline{f}_{S\star R}$.
 \item For $R_1 \in Rel^{\infty}_{s_1, k_1}$, $R_2 \in Rel^{\infty}_{s_2, k_2}$, we have $\overline{f}_{R_1} \otimes \overline{f}_{R_2} = \overline{f}_{R_1 \times R_2}$.
  \item  $\{\overline{f}_R: R\in Rel^{\infty}_{s,k}\}$ form a basis for $\Hom_{GL_{\infty}(\F_q)}(\V_{\infty}^{\otimes s}, \V_{\infty}^{\otimes k})$
 \end{enumerate}
 The fact that $\Gamma_n(\overline{f}_R) = f^{\InnaA{(n)}}_R$ for any $n$ immediately implies the first two statements above; it also shows that $\{\overline{f}_R: R\in Rel^{\infty}_{s,k}\}$ are linearly independent in $\Hom_{GL_{\infty}(\F_q)}(\V_{\infty}^{\otimes s}, \V_{\infty}^{\otimes k})$, since they are linearly independent for $n>>s+k$. 
 
It remains to show that $\{\overline{f}_R: R\in Rel^{\infty}_{s,k}\}$ form a spanning set in $\Hom_{GL_{\infty}(\F_q)}(\V_{\infty}^{\otimes s}, \V_{\infty}^{\otimes k})$.
\InnaA{
Indeed, let $f\in \Hom_{GL_{\infty}(\F_q)}(\V_{\infty}^{\otimes s}, \V_{\infty}^{\otimes k})$. Then for any $n\geq 0$, $f(\V_n^{\otimes s}) \subset \V_n^{\otimes k}$ (since these are $H_n$-invariants). By Lemma \ref{lem:Gamma_n_stabilization}, for $n>s+k$, the map $f\rvert_{\V_n^{\otimes s}}: \V_n^{\otimes s} \InnaA{\to} \V_n^{\otimes k}$ can be uniquely written as $\sum_{R \in Rel^{\infty}_{s,k}} \alpha_R^{(n)} f_R^{(n)}$ for some $ \alpha_R^{(n)} \in \C$. Since $f_R^{(n+1)} \rvert_{\V_n^{\otimes s}} = f_R^{(n)}$ for any $R\in Rel^{\infty}_{s,k}$, we conclude that for any $R\in Rel^{\infty}_{s,k}$, the scalar $ \alpha^{(n)}_R$ does not depend on $n$ when $n\geq s+k$. This implies that $f\in span \{\overline{f}_R: R\in Rel^{\infty}_{s,k}\}$ and the proof is complete. 
}
\end{proof}

Thus we obtain faithful SM functors $\mathcal{T}(GL_{\infty}(\F_q)) \to \kar{t}$ for every $t\in \C$, and (non-faithful) SM functors $\mathcal{T}(GL_{\infty}(\F_q)) \to \Rep(GL_n(\F_q))$ for every $n\in \Z_{\geq 0}$. We will denote such a functor again by $\Gamma_t$ and $\widetilde{\Gamma}_{n}$ respectively.

\begin{remark} We will see in \cite{EAH} that the objects $\V_{\infty}^{\otimes k} \in \mathcal{T}(GL_{\infty}(\F_q))$ are injective objects in the category $\repinf$, and the full subcategory of $\repinf$ which is the additive idempotent (Karoubi) completion of $\mathcal{T}(GL_{\infty}(\F_q))$ in $\repinf$ is the full subcategory of injective objects in $\repinf$.
\end{remark}

\subsection{Universal property}\label{ssec:universal_property_repinf}

\begin{definition}\label{def:semi_Frob_space}
 Let $\mathcal{C}$ be a SM category. An {\it $\F_q$-linear semi-Frobenius space} in $\mathcal{C}$ is an object $\V\in\mathcal{C}$  equipped with the following 
structures:
 \begin{enumerate}[label=(\subscript{SemiStr}{{\arabic*}})]
%   \item $\V = \triv \oplus \Vc$ for some object $\Vc \in \mathcal{C}$. We denote the embedding $\triv \to \V$ by $z$.
  \item $\V$ is equipped with the structure of a coalgebra object in 
$\mathcal{C}$ given by morphisms $$m^*:\V\to \V^{\otimes 2}, \;\; \eps^*:\V \to \triv$$ and by a multiplication operation $m:\V^{\otimes 2}\to \V$  such that the following conditions hold:
\begin{enumerate}[label=(\subscript{SemiFr}{{\arabic*}})]
\item $(\V, m^*, \eps^*)$ is a 
cocommutative counital coalgebra object.
\item The 
multiplication $m$ is associative and commutative.

 \item $m, m^*$ satisfy the Frobenius Relations and the Speciality Relation as in \ref{rel:Frob_alg2}.
\end{enumerate}

  \item$\V$ is a module over the field $\F_q$. In other words, $\V$ is equipped 
with
maps \begin{align*}
      &\dot{+}: \V \otimes \V \to \V, \\ &\mu:(\F_q, 
\cdot) \to (End_{\mathcal{C}}(\V), \circ),\;\; a \mapsto \mu_a, \\ &z :\triv 
\to \V
     \end{align*}
 satisfying the Relations  \ref{rel:F_q_lin_plus_ass_comm},\ref{rel:F_q_lin_zero}, \ref{rel:F_q_lin_mu}, \ref{rel:F_q_lin_plus_lin_distr}. 
 
 Furthermore, we require that the above structures satisfy Relations \ref{rel:mu_coalg_mor}, \ref{rel:z_coalg_mor}, \ref{rel:plus_coalg_mor}, \ref{rel:cancellation_axiom}.

\end{enumerate}

\end{definition}

Clearly, any $\F_q$-linear Frobenius space is also an $\F_q$-linear semi-Frobenius space.

\begin{lemma}\label{lem:semiFrob_infty}
 The object $[1]$ in $\mathcal{I}_{\infty}$ is an $\F_q$-linear semi-Frobenius space.
\end{lemma}

\begin{proof} 
The object $[1]$ is a Frobenius space in $\kar{t}$. Hence all the required properties of Definition \ref{def:semi_Frob_space} hold, provided the maps $m, m^*, \eps^*, \dot{+}, \mu$ and $z$ are defined by relations of the form $Rel_{s,k}^{\infty}$ (morphisms in $\mathcal{I}_{\infty}$). This follows from their explicit definition given in Proposition \ref{prop:gen_morphisms}.
\end{proof}

\begin{theorem}\label{thrm:inf_univ_prop}
  Let $\mathcal{C}$ be a SM category, and let $\V $ be an $\F_q$-linear semi-Frobenius space in $\mathcal{C}$. Then there exists a SM functor
 $$ \Gamma_{\V}: \mathcal{I}_{\infty} \to \mathcal{C}, \;\; [1] \longmapsto \V.$$
\end{theorem}

\begin{proof} 
The proof is analogous to the proof of Theorem \ref{thrm:univ_prop_Del}.
We may assume $\mathcal{C}$ to be a strict rigid symmetric monoidal category. Then define a functor $\Gamma_V:\mathcal{I}_{\infty} \to \mathcal{C}$ by sending $[k]$ to $\V^{\otimes k}$ and for each linear subspace $R \subset \F_q^{k+l}$, $\Gamma_V$ sends $f_R: [k] \to [l]$ to $\hat{f}_R:\V^{\otimes k} \to \V^{\otimes l}$, as defined in Definition \ref{eq:f_R_semiFrob_def} below.  

We then extend $\Gamma_V$ by linearity to the entire space $\Hom([k], [l])$.
 
In order to ensure that this is indeed a functor, we need to check that for any $R \in Rel^{\infty}_{s, k}, S \in Rel^{\infty}_{k, l}$, $\Gamma_V(f_S \circ f_R) = \hat{f}_S \circ \hat{f}_R$. This is proved in Proposition \ref{prop:composition_f_R_infty}.
 
Now, to check that this is a monoidal functor, we only need to check that $\Gamma_V(f_{R_1} \otimes f_{R_2}) = \hat{f}_{R_1} \otimes \hat{f}_{R_2}$ for any linear subspaces $R_1 \in Rel^{\infty}_{s_1, k_1}, R_2\in Rel^{\infty}_{s_2, k_2}$. This is proved in Lemma \ref{lem:tensor_prod_f_R_infty}. It remains to show that this functor is symmetric, which follows from the fact the the symmetry morphism 
 $\sigma\in \Hom_{\mathcal{C}}(\V^{\otimes l} \otimes \V^{\otimes k}, \V^{\otimes k} \otimes \V^{\otimes l})$ is given by $\hat{f}_{R'}$ where $$ R' = \{(a_1, \ldots, a_l, b_1, \ldots, b_k,-b_1, \ldots, -b_k,-a_1, \ldots, -a_l)|\, \forall i, j, \; a_i,  b_j \in \F_q\} \subset  \F_q^{2k+2l}.$$
\end{proof}

%\Innas{uniqueness}
%\thorsten{The functor extends in an obvious way to the additive Karoubi envelopes.}

\begin{corollary}\label{cor:univ_prop_semi_2}
The functor
 $$Fun^{\otimes}_{iso}\left(\mathcal{I}_{\infty}, \mathcal{C}\right) \longrightarrow SFrob_{\F_q}(\mathcal{C}), \;\; \Gamma\longmapsto \Gamma([1])$$
 is an equivalence of categories. Here 
 \begin{itemize}
  \item $SFrob_{\F_q}(\mathcal{C})$ is the subcategory of $\mathcal{C}$ consisting of $\F_q$-linear semi-Frobenius spaces in $\mathcal{C}$ and isomorphisms between them.
  \item $Fun^{\otimes}_{iso}\left(\kar{t}, \mathcal{C}\right)$ is the category of SM functors $\mathcal{I}_{\infty} \to \mathcal{C}$ and natural SM isomorphisms.
 \end{itemize}
\end{corollary}

\begin{proof}
The proof is the same as that of Corollary \ref{cor:univ_prop_2}.
\end{proof}

\begin{remark}
If the object $\V \in \mathcal{C}$ is actually rigid, of dimension $t$ (that is, it is an $\F_q$-linear Frobenius space, and not just semi-Frobenius), then Theorem \ref{thrm:univ_prop_Del} implies that we have a symmetric monoidal functor $F_{\V}: \mathcal{T}(\underline{GL}_t(\F_q)) \longrightarrow \mathcal{C}$  whose restriction to the subcategory $\mathcal{I}_{\infty} \subset \mathcal{T}(\underline{GL}_t(\F_q))$ gives a symmetric monoidal functor $$\Gamma': \mathcal{I}_{\infty}  \longrightarrow \mathcal{C}$$ 
This functor is equal to $\Gamma_V$ by Corollary \ref{cor:univ_prop_semi_2}.
In particular, in that case the morphisms $\hat{f}_R$ in $\mathcal{C}$ described in Definition \ref{eq:f_R_semiFrob_def} will coincide with the morphisms $\widetilde{f}_R$ defined in Section \ref{sec:morph_fR}.
\end{remark}

\subsection{Diagrams for the semi-Frobenius spaces} \label{ssec:semi-diagrams}

Let $\mathcal{C}$ be a SM category, and let $\V $ be an $\F_q$-linear semi-Frobenius space in $\mathcal{C}$. 

In this setting, we still have all the constructions of Section \ref{ssec:Frob_obj} except for the morphisms $\eps$ and $coev$ (note that the morphism $$z^*:=ev \circ (z\otimes \id) = ev \circ (\id \otimes z)$$ is still defined!). 

Consequently, all the statements of Sections  \ref{ssec:muA_tens_prod}, \ref{ssec:muA_comp} and \ref{ssec:transp_mu_A} will hold. 

In this setting, we will define morphisms $\hat{f}_R $ in $\mathcal{C}$ corresponding to $R\in Rel^{\infty}_{s,k}$, which will be the analogues of the morphisms defined in Section \ref{ssec:infty_def}, \ref{ssec:repinf_def}. 

\begin{example}
Let $\mathcal{C} = \Rep(GL_n(\F_q))$ and let $\V:=\V_n$. Then for any $R\in Rel^{\infty}_{s, k}$ there exist matrices $A \in Mat_{k\times s}(\F_q), A'\in Mat_{d\times s}(\F_q)$ such that $$R = Row\begin{bmatrix} -A &I_k\\
A' &0
\end{bmatrix}, \,\, rk(A')=d.$$ 
The map $f_R \in \Hom_{GL_n(\F_q)}(\V_n^{\otimes s}, \V_n^{\otimes k})$ can now be described as follows:
$$ v_1\otimes \ldots \otimes v_s \; \mapsto \; \delta_{A'\vec{v},\, \dot{0}} \, w_1 \otimes\ldots\otimes w_k, \;\;\; \text{ where } \;\;\; \vec{v} := \begin{bmatrix}
v_1 \\v_2 \\ \vdots \\ v_s
\end{bmatrix},\;\; \begin{bmatrix}
w_1\\ w_2 \\ \vdots \\ w_k
\end{bmatrix} := A\vec{v},
$$
Here we used notation as defined in Section \ref{ssec:notn_gen_lin_finite}.
\end{example}

\begin{definition}\label{def:f_R_semiFrob_def}
Let $R\in Rel^{\infty}_{s,k}$. We define a morphism $\hat{f}_R \in \Hom_{\mathcal{C}}(\V^{\otimes s}, \V^{\otimes k})$ as follows: let $A \in Mat_{k\times s}(\F_q), A'\in Mat_{d\times s}(\F_q)$ be such that $$R = Row\begin{bmatrix} -A &I_k\\
A' &0
\end{bmatrix}, \,\, rk(A')=d$$ 
Then we set $$\hat{f}_R :=\left( \id_{\V^{\otimes k}} \otimes (z^*)^{\otimes d} \right) \,\circ\, \mu_{\begin{bmatrix} A \\
A'
\end{bmatrix}}$$
Diagrammatically, $\hat{f}_R$ is given by  
\begin{equation}\label{eq:f_R_semiFrob_def}
\begin{tikzpicture}[anchorbase,scale=1.3]
%leftmost part:
\draw[-] (-0.7,-0.6)--(-0.7,0.6)--(0.7,0.6)--(0.7,-0.6)--(-0.7,-0.6);
\node at (0,0) {$\begin{bmatrix}
A \\ A'
\end{bmatrix}$};
\draw[-] (-0.4,-0.6)--(-0.4,-0.9);
\draw[-] (0.4,-0.6)--(0.4,-0.9);
\node at (-0.2,-0.75) {$\cdot$};
\node at (-0,-0.75) {$\cdot$};
\node at (0.2,-0.75) {$\cdot$};

\draw[-] (-0.6,0.6)--(-0.6,0.9);
\draw[-] (-0.1,0.6)--(-0.1,0.9);

\draw[-] (0.1,0.6)--(0.1,0.85);
\draw[-] (0.6,0.6)--(0.6,0.85);

\node at (0.1,0.9) {$\circ$};
\node at (0.6,0.9) {$\circ$};

\node at (-0.3,0.7) {$\cdot$};
\node at (-0.35,0.7) {$\cdot$};
\node at (-0.4,0.7) {$\cdot$};

\node at (0.3,0.7) {$\cdot$};
\node at (0.35,0.7) {$\cdot$};
\node at (0.4,0.7) {$\cdot$};
\end{tikzpicture}
\quad = \quad
 \begin{tikzpicture}[anchorbase,scale=1.3]
%leftmost part:
\draw[-] (-0.7,-0.2)--(-0.7,0.4)--(0.7,0.4)--(0.7,-0.2)--(-0.7,-0.2);
\node at (0,0.1) {$A$};
\draw[-] (-0.4,-0.2)--(-0.4,-0.7);
\draw[-] (0.4,-0.2)--(0.4,-0.5);
\node at (-0.2,-0.4) {$\cdot$};
\node at (-0,-0.4) {$\cdot$};
\node at (0.2,-0.4) {$\cdot$};
\draw[-] (-0.4,0.4)--(-0.4,0.8);
\draw[-] (0.4,0.4)--(0.4,0.8);
\node at (-0.2,0.5) {$\cdot$};
\node at (-0,0.5) {$\cdot$};
\node at (0.2,0.5) {$\cdot$};%%%

%%% rightmost part:
\draw[-] (3.3,-0.2)--(3.3,0.4)--(4.7,0.4)--(4.7,-0.2)--(3.3,-0.2);
\node at (4,0.1) {$A'$};
\draw[-] (3.6,-0.2)--(3.6,-0.7);
\draw[-] (4.4,-0.2)--(4.4,-0.5);
\node at (3.8,-0.4) {$\cdot$};
\node at (4,-0.4) {$\cdot$};
\node at (4.2,-0.4) {$\cdot$};
\draw[-] (3.6,0.4)--(3.6,0.7);
\draw[-] (4.4,0.4)--(4.4,0.7);
\node at (3.6,0.75) {$\circ$};
\node at (4.4,0.75) {$\circ$};
\node at (3.8,0.5) {$\cdot$};
\node at (4,0.5) {$\cdot$};
\node at (4.2,0.5) {$\cdot$};
%%% Bottom part
\draw[-] (-0.4,-0.7)--(3.6,-0.7);
\draw[-] (0.4,-0.5)--(3.5, -0.5);
\draw[-] (3.7,-0.5)--(4.4,-0.5);
\draw[-] (3, -0.7)--(3,-1);
\draw[-] (3.4, -0.5)--(3.4, -0.65);
\draw[-] (3.4, -0.75)--(3.4, -1);
\node at (3.1,-0.95) {$\cdot$};
\node at (3.2,-0.95) {$\cdot$};
\node at (3.3,-0.95) {$\cdot$};
\end{tikzpicture}
\end{equation}
Explicitly, this means that $\hat{f}_R$ looks like this:
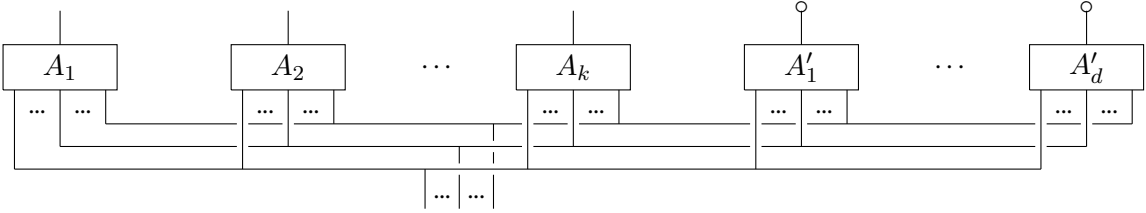
\begin{equation*}
\begin{tikzpicture}[anchorbase,scale=1.5]%leftmost part:
\draw[-] (-0.5,-0.2)--(-0.5,0.2)--(0.5,0.2)--(0.5,-0.2)--(-0.5,-0.2);
\node at (0,0) {${A_1}$};

\draw[-] (-0.4,-0.2)--(-0.4,-0.9);
\draw[-] (0.4,-0.2)--(0.4,-0.5);
\draw[-] (0,-0.2)--(0,-0.7);

\draw[-] (0,0.2)--(0,0.5);

\node at (-0.25,-0.4) {$\cdot$};
\node at (-0.2,-0.4) {$\cdot$};
\node at (-0.15,-0.4) {$\cdot$};

\node at (0.25,-0.4) {$\cdot$};
\node at (0.2,-0.4) {$\cdot$};
\node at (0.15,-0.4) {$\cdot$};

%%% left intermediate
\draw[-] (1.5,-0.2)--(1.5,0.2)--(2.5,0.2)--(2.5,-0.2)--(1.5,-0.2);
\node at (2,0) {${A_2}$};

\draw[-] (1.6,-0.2)--(1.6,-0.9);
\draw[-] (2.4,-0.2)--(2.4,-0.5);
\draw[-] (2,-0.2)--(2,-0.7);

\draw[-] (2,0.2)--(2,0.5);

\node at (1.85,-0.4) {$\cdot$};
\node at (1.8,-0.4) {$\cdot$};
\node at (1.75,-0.4) {$\cdot$};

\node at (2.15,-0.4) {$\cdot$};
\node at (2.2,-0.4) {$\cdot$};
\node at (2.25,-0.4) {$\cdot$};
%%%intermediate dots:
\node at (3.2,0) {$\cdot$};
\node at (3.3,0) {$\cdot$};
\node at (3.4,0) {$\cdot$};
%%% middle part:
\draw[-] (4,-0.2)--(4,0.2)--(5,0.2)--(5,-0.2)--(4,-0.2);
\node at (4.5,0) {${A_{k}}$};

\draw[-] (4.1,-0.2)--(4.1,-0.9);
\draw[-] (4.9,-0.2)--(4.9,-0.5);
\draw[-] (4.5,-0.2)--(4.5,-0.7);

\draw[-] (4.5,0.2)--(4.5,0.5);

\node at (4.25,-0.4) {$\cdot$};
\node at (4.3,-0.4) {$\cdot$};
\node at (4.35,-0.4) {$\cdot$};

\node at (4.65,-0.4) {$\cdot$};
\node at (4.7,-0.4) {$\cdot$};
\node at (4.75,-0.4) {$\cdot$};

%%% right intermediate:
\draw[-] (6,-0.2)--(6,0.2)--(7,0.2)--(7,-0.2)--(6,-0.2);
\node at (6.5,0) {${A'_{1}}$};

\draw[-] (6.1,-0.2)--(6.1,-0.9);
\draw[-] (6.9,-0.2)--(6.9,-0.5);
\draw[-] (6.5,-0.2)--(6.5,-0.7);

\draw[-] (6.5,0.2)--(6.5,0.5);
\node at (6.5,0.53) {$\circ$};

\node at (6.25,-0.4) {$\cdot$};
\node at (6.3,-0.4) {$\cdot$};
\node at (6.35,-0.4) {$\cdot$};

\node at (6.65,-0.4) {$\cdot$};
\node at (6.7,-0.4) {$\cdot$};
\node at (6.75,-0.4) {$\cdot$};
%%%intermediate dots:
\node at (7.7,0) {$\cdot$};
\node at (7.8,0) {$\cdot$};
\node at (7.9,0) {$\cdot$};

%%% rightmost:
\draw[-] (8.5,-0.2)--(8.5,0.2)--(9.5,0.2)--(9.5,-0.2)--(8.5,-0.2);
\node at (9,0) {${A'_{d}}$};

\draw[-] (8.6,-0.2)--(8.6,-0.9);
\draw[-] (9.4,-0.2)--(9.4,-0.5);
\draw[-] (9,-0.2)--(9,-0.7);

\draw[-] (9,0.2)--(9,0.5);
\node at (9,0.53) {$\circ$};

\node at (8.75,-0.4) {$\cdot$};
\node at (8.8,-0.4) {$\cdot$};
\node at (8.85,-0.4) {$\cdot$};

\node at (9.15,-0.4) {$\cdot$};
\node at (9.2,-0.4) {$\cdot$};
\node at (9.25,-0.4) {$\cdot$};

%%% Bottom part
\draw[-] (-0.4,-0.9)--(8.6,-0.9);%bottom horizontal line

%middle horizontal
\draw[-] (0,-0.7)--(1.55, -0.7);
\draw[-] (1.65,-0.7)--(4.05, -0.7);
\draw[-] (4.15,-0.7)--(6.05,-0.7);
\draw[-] (6.15,-0.7)--(8.55,-0.7);
\draw[-] (8.65,-0.7)--(9,-0.7);

%top horizontal line
\draw[-] (0.4,-0.5)--(1.55, -0.5);
\draw[-] (1.65, -0.5)--(1.95, -0.5);
\draw[-] (2.05, -0.5)--(4.05, -0.5);
\draw[-] (4.15,-0.5)--(4.45,-0.5);
\draw[-] (4.55,-0.5)--(6.05,-0.5);
\draw[-] (6.15,-0.5)--(6.45,-0.5);
\draw[-] (6.55,-0.5)--(8.55,-0.5);
\draw[-] (8.65,-0.5)--(8.95,-0.5);
\draw[-] (9.05,-0.5)--(9.4,-0.5);

\draw[-] (3.2, -0.9)--(3.2,-1.25);

\draw[-] (3.5, -0.7)--(3.5, -0.85);
\draw[-] (3.5, -0.95)--(3.5, -1.25);

\draw[-] (3.8, -0.5)--(3.8, -0.65);
\draw[-] (3.8, -0.75)--(3.8, -0.85);
\draw[-] (3.8, -0.95)--(3.8, -1.25);

%bottom dots
\node at (3.3,-1.15) {$\cdot$};
\node at (3.35,-1.15) {$\cdot$};
\node at (3.4,-1.15) {$\cdot$};

%bottom dots
\node at (3.6,-1.15) {$\cdot$};
\node at (3.65,-1.15) {$\cdot$};
\node at (3.7,-1.15) {$\cdot$};
\end{tikzpicture}
\end{equation*}
\end{definition}

We first check that $\hat{f}_R$ is well defined: 
\begin{lemma}
The morphism $\hat{f}_R$ doesn't depend on the choice of $A, A'$ as above.
\end{lemma}
\begin{proof}
Let $B \in Mat_{k\times s}(\F_q), B'\in Mat_{d\times s}(\F_q)$ be another pair of matrices such that $$R = Row\begin{bmatrix} B &I_k\\
B' &0
\end{bmatrix},\, rk(B')=d.$$ Then we can find some $D\in GL_d(\F_q)$ and $D'\in Mat_{k \times d}$ for which
$$\begin{bmatrix} B \\
B'
\end{bmatrix} = \begin{bmatrix} I_k &D'\\ 
0 &D
\end{bmatrix} \begin{bmatrix} A \\
A'
\end{bmatrix}. $$ 
Thus we only need to show that 
$$\left( \id_{\V^{\otimes k}} \otimes (z^*)^{\otimes d} \right) \,\circ\, \mu_{\begin{bmatrix} I_k &D'\\ 0 &D \end{bmatrix}} = \left( \id_{\V^{\otimes k}} \otimes (z^*)^{\otimes d} \right)$$
Diagrammatically, we need to prove that
\begin{equation}\label{eq:f_R_well_def_semiFrob}
 \begin{tikzpicture}[anchorbase,scale=1.3]
\draw[-] (-0.7,-0.5)--(-0.7,0.5)--(0.7,0.5)--(0.7,-0.5)--(-0.7,-0.5);
\node at (0,0) {$\begin{bmatrix}
I_k &D'\\
0 &D
\end{bmatrix}$};
\draw[-] (-0.4,-0.5)--(-0.4,-0.8);
\draw[-] (0.4,-0.5)--(0.4,-0.8);
\node at (-0.2,-0.65) {$\cdot$};
\node at (-0,-0.65) {$\cdot$};
\node at (0.2,-0.65) {$\cdot$};

\draw[-] (-0.6,0.5)--(-0.6,0.8);
\draw[-] (-0.1,0.5)--(-0.1,0.8);

\node at (-0.3,0.65) {$\cdot$};
\node at (-0.35,0.65) {$\cdot$};
\node at (-0.4,0.65) {$\cdot$};

\draw[-] (0.1,0.5)--(0.1,0.7);
\draw[-] (0.6,0.5)--(0.6,0.7);
\node at (0.3,0.65) {$\cdot$};
\node at (0.35,0.65) {$\cdot$};
\node at (0.4,0.65) {$\cdot$};

\node at (0.1,0.75) {$\circ$};
\node at (0.6,0.75) {$\circ$};
 \end{tikzpicture} \quad = \quad
  \begin{tikzpicture}[anchorbase,scale=1.3]
\draw[-] (-0.6,-0.8)--(-0.6,0.8);
\draw[-] (-0.1,-0.8)--(-0.1,0.8);

\node at (-0.3,0) {$\cdot$};
\node at (-0.35,0) {$\cdot$};
\node at (-0.4,0) {$\cdot$};

\draw[-] (0.1,-0.8)--(0.1,0.7);
\draw[-] (0.6,-0.8)--(0.6,0.7);
\node at (0.3,0) {$\cdot$};
\node at (0.35,0) {$\cdot$};
\node at (0.4,0) {$\cdot$};

\node at (0.1,0.75) {$\circ$};
\node at (0.6,0.75) {$\circ$};
 \end{tikzpicture} 
 \end{equation}

 The left hand side of \eqref{eq:f_R_well_def_semiFrob}, by definition, is
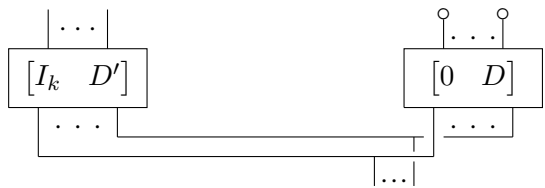
\begin{equation*}
 \begin{tikzpicture}[anchorbase,scale=1.3]
%leftmost part:
\draw[-] (-0.7,-0.2)--(-0.7,0.4)--(0.7,0.4)--(0.7,-0.2)--(-0.7,-0.2);
\node at (0,0.1) {$\begin{bmatrix}
I_k &D'
\end{bmatrix}$};
\draw[-] (-0.4,-0.2)--(-0.4,-0.7);
\draw[-] (0.4,-0.2)--(0.4,-0.5);
\node at (-0.2,-0.4) {$\cdot$};
\node at (-0,-0.4) {$\cdot$};
\node at (0.2,-0.4) {$\cdot$};
\draw[-] (-0.3,0.4)--(-0.3,0.8);
\draw[-] (0.3,0.4)--(0.3,0.8);
\node at (-0.15,0.6) {$\cdot$};
\node at (-0,0.6) {$\cdot$};
\node at (0.15,0.6) {$\cdot$};%%%

%%% rightmost part:
\draw[-] (3.3,-0.2)--(3.3,0.4)--(4.7,0.4)--(4.7,-0.2)--(3.3,-0.2);
\node at (4,0.1) {$\begin{bmatrix}
0 &D
\end{bmatrix}$};
\draw[-] (3.6,-0.2)--(3.6,-0.7);
\draw[-] (4.4,-0.2)--(4.4,-0.5);
\node at (3.8,-0.4) {$\cdot$};
\node at (4,-0.4) {$\cdot$};
\node at (4.2,-0.4) {$\cdot$};
\draw[-] (3.7,0.4)--(3.7,0.7);
\draw[-] (4.3,0.4)--(4.3,0.7);
\node at (3.7,0.75) {$\circ$};
\node at (4.3,0.75) {$\circ$};
\node at (3.8,0.5) {$\cdot$};
\node at (4,0.5) {$\cdot$};
\node at (4.2,0.5) {$\cdot$};
%%% Bottom part
\draw[-] (-0.4,-0.7)--(3.6,-0.7);
\draw[-] (0.4,-0.5)--(3.5, -0.5);
\draw[-] (3.7,-0.5)--(4.4,-0.5);
\draw[-] (3, -0.7)--(3,-1);
\draw[-] (3.4, -0.5)--(3.4, -0.65);
\draw[-] (3.4, -0.75)--(3.4, -1);
\node at (3.1,-0.95) {$\cdot$};
\node at (3.2,-0.95) {$\cdot$};
\node at (3.3,-0.95) {$\cdot$};
\end{tikzpicture}
\end{equation*}

We now use Lemma \ref{lem:horizontal stacking_with_zero} to say that 
\begin{equation*}
 \begin{tikzpicture}[anchorbase,scale=1.3]
\draw[-] (3.3,-0.2)--(3.3,0.4)--(4.7,0.4)--(4.7,-0.2)--(3.3,-0.2);
\node at (4,0.1) {$\begin{bmatrix}
0 &D
\end{bmatrix}$};
\draw[-] (3.6,-0.2)--(3.6,-0.7);
\draw[-] (4.4,-0.2)--(4.4,-0.7);
\node at (3.8,-0.4) {$\cdot$};
\node at (4,-0.4) {$\cdot$};
\node at (4.2,-0.4) {$\cdot$};
\draw[-] (3.7,0.4)--(3.7,0.7);
\draw[-] (4.3,0.4)--(4.3,0.7);
\node at (3.7,0.75) {$\circ$};
\node at (4.3,0.75) {$\circ$};
\node at (3.9,0.6) {$\cdot$};
\node at (4,0.6) {$\cdot$};
\node at (4.1,0.6) {$\cdot$};
\end{tikzpicture} \quad = \quad
 \begin{tikzpicture}[anchorbase,scale=1.3]
 \draw[-] (2.7,0.7)--(2.7,-0.7);
\draw[-] (3.1,0.7)--(3.1,-0.7);
\node at (2.85,-0.4) {$\cdot$};
\node at (2.9,-0.4) {$\cdot$};
\node at (2.95,-0.4) {$\cdot$};

\node at (2.7,0.75) {$\bullet$};
\node at (3.1,0.75) {$\bullet$};

\draw[-] (3.5,-0.2)--(3.5,0.4)--(4.5,0.4)--(4.5,-0.2)--(3.5,-0.2);
\node at (4,0.1) {$D$};
\draw[-] (3.6,-0.2)--(3.6,-0.7);
\draw[-] (4.4,-0.2)--(4.4,-0.7);
\node at (3.95,-0.4) {$\cdot$};
\node at (4,-0.4) {$\cdot$};
\node at (4.05,-0.4) {$\cdot$};
\draw[-] (3.6,0.4)--(3.6,0.7);
\draw[-] (4.4,0.4)--(4.4,0.7);
\node at (3.6,0.75) {$\circ$};
\node at (4.4,0.75) {$\circ$};
\node at (3.8,0.5) {$\cdot$};
\node at (4,0.5) {$\cdot$};
\node at (4.2,0.5) {$\cdot$};
\end{tikzpicture} 
 \quad = \quad
 \begin{tikzpicture}[anchorbase,scale=1.3]
 \draw[-] (2.7,0.7)--(2.7,-0.7);
\draw[-] (3.1,0.7)--(3.1,-0.7);
\node at (2.85,0) {$\cdot$};
\node at (2.9,0) {$\cdot$};
\node at (2.95,0) {$\cdot$};

\node at (2.7,0.75) {$\bullet$};
\node at (3.1,0.75) {$\bullet$};

\draw[-] (3.8,0.7)--(3.8,-0.7);
\draw[-] (4.2,0.7)--(4.2,-0.7);
\node at (3.8,0.75) {$\circ$};
\node at (4.2,0.75) {$\circ$};
\node at (3.95,0) {$\cdot$};
\node at (4,0) {$\cdot$};
\node at (4.05,0) {$\cdot$};
\end{tikzpicture} 
\end{equation*}
The second equality follows from the fact that $D$ is invertible and from Lemma \ref{cor:regular_system_of_eq}, stating that $(z^*)^{\otimes d} \circ \mu_D =  (z^*)^{\otimes d}$. Hence the left hand side of \eqref{eq:f_R_well_def_semiFrob} equals
\begin{equation*}
 \begin{tikzpicture}[anchorbase,scale=1.2]
%leftmost part:
\draw[-] (-0.7,-0.2)--(-0.7,0.4)--(0.7,0.4)--(0.7,-0.2)--(-0.7,-0.2);
\node at (0,0.1) {$\begin{bmatrix}
I_k &D'
\end{bmatrix}$};
\draw[-] (-0.6,-0.2)--(-0.6,-1);
\draw[-] (-0.2,-0.2)--(-0.2,-0.8);
\draw[-] (0.2,-0.2)--(0.2,-0.6);
\draw[-] (0.6,-0.2)--(0.6,-0.4);

\node at (-0.45,-0.4) {$\cdot$};
\node at (-0.4,-0.4) {$\cdot$};
\node at (-0.35,-0.4) {$\cdot$};

\node at (0.45,-0.4) {$\cdot$};
\node at (0.4,-0.4) {$\cdot$};
\node at (0.35,-0.4) {$\cdot$};

\draw[-] (-0.4,0.4)--(-0.4,0.8);
\draw[-] (0.4,0.4)--(0.4,0.8);
\node at (-0.15,0.6) {$\cdot$};
\node at (-0,0.6) {$\cdot$};
\node at (0.15,0.6) {$\cdot$};%%%

%%% rightmost part
\node at (3,0.75) {$\bullet$};
\node at (3.4,0.75) {$\bullet$};
\node at (3.15,0.5) {$\cdot$};
\node at (3.2,0.5) {$\cdot$};
\node at (3.25,0.5) {$\cdot$};

\node at (3.6,0.75) {$\circ$};
\node at (4,0.75) {$\circ$};
\node at (3.75,0.5) {$\cdot$};
\node at (3.8,0.5) {$\cdot$};
\node at (3.85,0.5) {$\cdot$};
%%% Bottom part
\draw[-] (-0.6,-1)--(3,-1) -- (3, 0.7);
\draw[-] (-0.2,-0.8)--(2.9, -0.8);
\draw[-] (3.1, -0.8)--(3.4, -0.8)-- (3.4, 0.7);
\draw[-] (0.2,-0.6)--(2.9, -0.6);
\draw[-] (3.1, -0.6)--(3.3, -0.6);
\draw[-] (3.45, -0.6)--(3.6, -0.6)-- (3.6, 0.7);
\draw[-] (0.6,-0.4)-- (2.9, -0.4);
\draw[-] (3.1, -0.4)--(3.3, -0.4);
\draw[-] (3.45, -0.4)--(3.55, -0.4);
\draw[-] (3.7, -0.4)--(4,-0.4)-- (4, 0.7);

\draw[-] (1.5,-1)--(1.5,-1.4);
\draw[-] (1.8,-0.8)--(1.8,-0.95);
\draw[-] (1.8,-1.05)--(1.8,-1.4);

\draw[-] (2.1,-0.6)--(2.1,-0.75);
\draw[-] (2.1,-0.85)--(2.1,-0.95);
\draw[-] (2.1,-1.05)--(2.1,-1.4);

\draw[-] (2.4,-0.4)--(2.4,-0.55);
\draw[-] (2.4,-0.65)--(2.4,-0.75);
\draw[-] (2.4,-0.85)--(2.4,-0.95);
\draw[-] (2.4,-1.05)--(2.4,-1.4);

\node at (1.6,-1.2) {$\cdot$};
\node at (1.65,-1.2) {$\cdot$};
\node at (1.7,-1.2) {$\cdot$};

\node at (2.2,-1.2) {$\cdot$};
\node at (2.25,-1.2) {$\cdot$};
\node at (2.3,-1.2) {$\cdot$};
\end{tikzpicture}  \InnaA{\quad \xlongequal{\text{\ref{rel:Frob_alg}}} \quad}
 \begin{tikzpicture}[anchorbase,scale=1.3]
%leftmost part:
\draw[-] (-0.7,-0.2)--(-0.7,0.4)--(0.7,0.4)--(0.7,-0.2)--(-0.7,-0.2);
\node at (0,0.1) {$\begin{bmatrix}
I_k &D'
\end{bmatrix}$};
\draw[-] (-0.6,-0.2)--(-0.6,-1);
\draw[-] (-0.2,-0.2)--(-0.2,-1);
\draw[-] (0.2,-0.2)--(0.2,-0.6);
\draw[-] (0.6,-0.2)--(0.6,-0.4);

\node at (-0.45,-0.4) {$\cdot$};
\node at (-0.4,-0.4) {$\cdot$};
\node at (-0.35,-0.4) {$\cdot$};

\node at (0.45,-0.4) {$\cdot$};
\node at (0.4,-0.4) {$\cdot$};
\node at (0.35,-0.4) {$\cdot$};

\draw[-] (-0.4,0.4)--(-0.4,0.8);
\draw[-] (0.4,0.4)--(0.4,0.8);
\node at (-0.15,0.6) {$\cdot$};
\node at (-0,0.6) {$\cdot$};
\node at (0.15,0.6) {$\cdot$};%%%

%%% rightmost part

\node at (1.6,0.75) {$\circ$};
\node at (2,0.75) {$\circ$};
\node at (1.75,0.5) {$\cdot$};
\node at (1.8,0.5) {$\cdot$};
\node at (1.85,0.5) {$\cdot$};
%%% Bottom part
\draw[-] (0.2,-0.6)--(1.6, -0.6)-- (1.6, 0.7);
\draw[-] (0.6,-0.4)--(1.55, -0.4);
\draw[-] (1.7, -0.4)--(2,-0.4)-- (2, 0.7);

\draw[-] (1.1,-0.6)--(1.1,-1);

\draw[-] (1.4,-0.4)--(1.4,-0.55);
\draw[-] (1.4,-0.65)--(1.4,-1);

\node at (1.2,-0.85) {$\cdot$};
\node at (1.25,-0.85) {$\cdot$};
\node at (1.3,-0.85) {$\cdot$};
\end{tikzpicture}
\InnaA{\quad \xlongequal{\text{\cref{lem:comparing_with_zero}}} \quad }
 \begin{tikzpicture}[anchorbase,scale=1.3]
\draw[-] (3.3,-0.2)--(3.3,0.4)--(4.7,0.4)--(4.7,-0.2)--(3.3,-0.2);
\node at (4,0.1) {$\begin{bmatrix}
I_k &D'
\end{bmatrix}$};
\draw[-] (3.4,-0.2)--(3.4,-1.3);
\draw[-] (3.8,-0.2)--(3.8,-1.3);
\draw[-] (4.2,-0.2)--(4.2,-0.5);
\draw[-] (4.6,-0.2)--(4.6,-0.5);

\node at (4.2,-0.53) {$\circ$};
\node at (4.6,-0.53) {$\circ$};

\draw[-] (4.2,-1.3)--(4.2,-0.75);
\draw[-] (4.6,-1.3)--(4.6,-0.75);
\node at (4.2,-0.72) {$\circ$};
\node at (4.6,-0.72) {$\circ$};

\node at (4.35,-0.8) {$\cdot$};
\node at (4.4,-0.8) {$\cdot$};
\node at (4.45,-0.8) {$\cdot$};

\node at (3.55,-0.4) {$\cdot$};
\node at (3.6,-0.4) {$\cdot$};
\node at (3.65,-0.4) {$\cdot$};

\node at (4.35,-0.4) {$\cdot$};
\node at (4.4,-0.4) {$\cdot$};
\node at (4.45,-0.4) {$\cdot$};

\draw[-] (3.7,0.4)--(3.7,0.7);
\draw[-] (4.3,0.4)--(4.3,0.7);
\node at (3.9,0.6) {$\cdot$};
\node at (4,0.6) {$\cdot$};
\node at (4.1,0.6) {$\cdot$};
\end{tikzpicture} 
\end{equation*}
which is just $\mu_{[I_k \, D']} \circ \left(\id_{\V}^{\otimes k} \otimes z^{\otimes d}\right)$.
By Corollary \ref{cor:horizontal stacking}, we have: $$\mu_{[I_k \, D']} \circ (\id_{\V}^{\otimes k} \otimes z^{\otimes d}) = \mu_{[I_k \, I_k]} \circ \left(\id_{\V}^{\otimes k} \otimes \left(\mu_{D'} \circ z^{\otimes d} \right)\right) $$
Now, by Relations \ref{rel:F_q_lin_mu}, we have: $$\mu_{D'} \circ z^{\otimes d} =z^{\otimes k} = \mu_{[\dot 0]} \circ z^{\otimes k},$$ where $[\dot 0]$ denotes the $k \times k$ zero matrix. So 
$$\mu_{[I_k \, D']} \circ (\id_{\V}^{\otimes k} \otimes z^{\otimes d}) = \mu_{[I_k \, I_k]} \circ \left(\id_{\V}^{\otimes k} \otimes \left(\mu_{[\dot 0]} \circ z^{\otimes k} \right)\right) =  \mu_{[I_k \, 0]} \circ (\id_{\V}^{\otimes k} \otimes z^{\otimes d})$$ In terms of diagrams, we obtained that the left hand side of \eqref{eq:f_R_well_def_semiFrob} equals
\begin{equation*}
 \begin{tikzpicture}[anchorbase,scale=1.3]
\draw[-] (3.3,-0.2)--(3.3,0.4)--(4.7,0.4)--(4.7,-0.2)--(3.3,-0.2);
\node at (4,0.1) {$\begin{bmatrix}
I_k &0
\end{bmatrix}$};
\draw[-] (3.4,-0.2)--(3.4,-1);
\draw[-] (3.8,-0.2)--(3.8,-1);
\draw[-] (4.2,-0.2)--(4.2,-0.45);
\draw[-] (4.6,-0.2)--(4.6,-0.45);

\node at (4.2,-0.5) {$\circ$};
\node at (4.6,-0.5) {$\circ$};

\draw[-] (4.2,-1)--(4.2,-0.75);
\draw[-] (4.6,-1)--(4.6,-0.75);
\node at (4.2,-0.7) {$\circ$};
\node at (4.6,-0.7) {$\circ$};

\node at (4.35,-0.8) {$\cdot$};
\node at (4.4,-0.8) {$\cdot$};
\node at (4.45,-0.8) {$\cdot$};

\node at (3.55,-0.4) {$\cdot$};
\node at (3.6,-0.4) {$\cdot$};
\node at (3.65,-0.4) {$\cdot$};

\node at (4.35,-0.4) {$\cdot$};
\node at (4.4,-0.4) {$\cdot$};
\node at (4.45,-0.4) {$\cdot$};

\draw[-] (3.7,0.4)--(3.7,0.7);
\draw[-] (4.3,0.4)--(4.3,0.7);
\node at (3.9,0.6) {$\cdot$};
\node at (4,0.6) {$\cdot$};
\node at (4.1,0.6) {$\cdot$};
\end{tikzpicture} 
\InnaA{\quad\xlongequal{\text{\cref{lem:horizontal stacking_with_zero}}}\quad}
 \begin{tikzpicture}[anchorbase,scale=1.3]
\draw[-] (3.3,-0.2)--(3.3,0.4)--(4.3,0.4)--(4.3,-0.2)--(3.3,-0.2);
\node at (3.8,0.1) {$I_k $};
\draw[-] (3.4,-0.2)--(3.4,-0.7);
\draw[-] (4.2,-0.2)--(4.2,-0.7);

\node at (3.8,-0.4) {$\cdot$};
\node at (3.85,-0.4) {$\cdot$};
\node at (3.75,-0.4) {$\cdot$};
\draw[-] (3.4,0.4)--(3.4,0.7);
\draw[-] (4.2,0.4)--(4.2,0.7);
\node at (3.8,0.6) {$\cdot$};
\node at (3.85,0.6) {$\cdot$};
\node at (3.75,0.6) {$\cdot$};

\draw[-] (4.8,-0.7)--(4.8,-0.3);
\draw[-] (5.2,-0.7)--(5.2,-0.3);
\node at (4.8,-0.25) {$\circ$};
\node at (5.2,-0.25) {$\circ$};
\draw[-] (4.8,0.2)--(4.8,0.7);
\draw[-] (5.2,0.2)--(5.2,0.7);
\node at (4.8,0.15) {$\circ$};
\node at (5.2,0.15) {$\circ$};
\node at (4.8,0.7) {$\bullet$};
\node at (5.2,0.7) {$\bullet$};

\node at (4.95,0.4) {$\cdot$};
\node at (5,0.4) {$\cdot$};
\node at (5.05,0.4) {$\cdot$};

\node at (4.95,-0.4) {$\cdot$};
\node at (5,-0.4) {$\cdot$};
\node at (5.05,-0.4) {$\cdot$};

\end{tikzpicture} 
 \quad = \quad
 \begin{tikzpicture}[anchorbase,scale=1]
 \draw[-] (2.7,0.7)--(2.7,-0.7);
\draw[-] (3.1,0.7)--(3.1,-0.7);
\node at (2.85,0) {$\cdot$};
\node at (2.9,0) {$\cdot$};
\node at (2.95,0) {$\cdot$};

\draw[-] (3.8,0.7)--(3.8,-0.7);
\draw[-] (4.2,0.7)--(4.2,-0.7);
\node at (3.8,0.75) {$\circ$};
\node at (4.2,0.75) {$\circ$};
\node at (3.95,0) {$\cdot$};
\node at (4,0) {$\cdot$};
\node at (4.05,0) {$\cdot$};
\end{tikzpicture} 
\end{equation*}
The last equality follows from Relation \ref{rel:z_coalg_mor} (see also \ref{itm:str_rel_diag_z_coalg}), which states that $\eps^*\circ z=\id_{\triv}$. This proves Equality \ref{eq:f_R_well_def_semiFrob}, and the lemma is proved.
\end{proof}

\begin{proposition}\label{prop:composition_f_R_infty}
 Let $R_1 \subset Rel^\infty_{s, k}$ and $R_2 \subset Rel^\infty_{k, l}$. Then $$\hat{f}_{R_2} \circ \hat{f}_{R_1} = \hat{f}_{R_2 \star R_1}$$ where $R_2\star R_1$ are defined in Definition \ref{def:composition_star_d_R_S}.
\end{proposition}
\begin{proof}
Let $d_i = \dim_{\F_q} R_i$, $r_i := s_i+k_i$ for $i=1,2$. 

Let $A \in Mat_{k\times s}(\F_q), A'\in Mat_{d_1\times s}(\F_q)$, $B \in Mat_{l\times k}(\F_q), B'\in Mat_{d_2\times k}(\F_q)$ be such that $$R_1 = Row\begin{bmatrix} -A &I_{k}\\
A' &0
\end{bmatrix}, \; R_2 = Row\begin{bmatrix} -B &I_{l}\\
B' &0
\end{bmatrix}, \;\; rk(A')=d_1, \, rk(B')=d_2.$$ 
As we saw in the proof of Lemma \ref{lem:Rel_infty_well_def}, we have: $$R_2\star R_1 = Row \begin{bmatrix}
-BA &I_l \\
A'&0\\
B'A  &0 
\end{bmatrix} = Row \begin{bmatrix}
-BA &I_l \\
B'A  &0 \\
A'&0
\end{bmatrix}.$$

Recall that by Proposition \ref{prop:tensor_prod_mu_B}, we have the following equalities (for equality of appropriate diagrams, see Equation \eqref{eq:f_R_semiFrob_def}):
\begin{align*}
\hat{f}_{R_1} &= \left( \id_{\V^{\otimes k}} \otimes (z^*)^{\otimes d_1} \right) \,\circ\, \mu_{\begin{bmatrix} A \\
A'
\end{bmatrix}} \;= \; \left( \id_{\V^{\otimes k}} \otimes (z^*)^{\otimes d_1} \right) \,\circ\, \mu_{\begin{bmatrix} A &0 \\
0 &A'
\end{bmatrix}} \,\circ\, \mu_{\begin{bmatrix} I_s \\
I_s
\end{bmatrix}} \\&=\;\left( \id_{\V^{\otimes k}} \otimes (z^*)^{\otimes d_1} \right) \,\circ\,\left( \mu_{A} \otimes \mu_A' \right) \,\circ\, \mu_{\begin{bmatrix} I_s \\
I_s
\end{bmatrix}}
\end{align*}

The statement which we want to prove, $\hat{f}_{R_2} \circ \hat{f}_{R_1} = \hat{f}_{R_2 \star R_1}$ can \InnaA{then} be interpreted as:
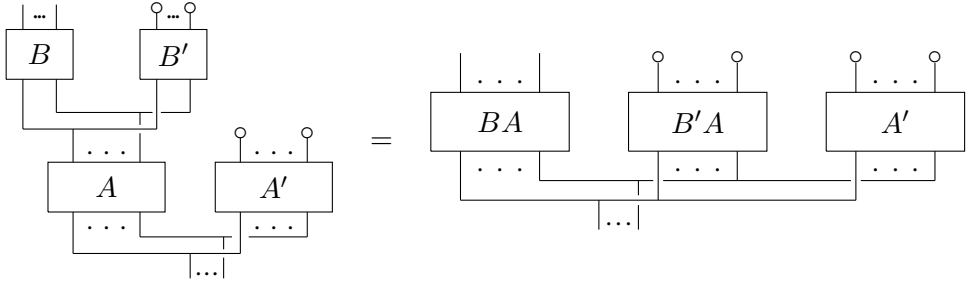
\begin{equation*}
 \begin{tikzpicture}[anchorbase,scale=1.1]
%leftmost part:
\draw[-] (1.3,-0.2)--(1.3,0.4)--(2.7,0.4)--(2.7,-0.2)--(1.3,-0.2);
\node at (2,0.1) {$A$};
\draw[-] (1.6,-0.2)--(1.6,-0.7);
\draw[-] (2.4,-0.2)--(2.4,-0.5);
\node at (1.8,-0.4) {$\cdot$};
\node at (2,-0.4) {$\cdot$};
\node at (2.2,-0.4) {$\cdot$};

\node at (1.8,0.5) {$\cdot$};
\node at (2,0.5) {$\cdot$};
\node at (2.2,0.5) {$\cdot$};%%%

\draw[-] (1.6,0.4)--(1.6,0.8);
\draw[-] (2.4,0.4)--(2.4,0.75);
\draw[-] (2.4,0.85)--(2.4,1);

\draw[-] (1,0.8)--(2.6,0.8);
\draw[-] (1.4,1)--(2.55,1);
\draw[-] (2.65,1)--(3,1);

\draw[-] (1,0.8)--(1,1.4);
\draw[-] (1.4,1)--(1.4,1.4);
\draw[-] (0.8,1.4)--(0.8,2)--(1.6,2)--(1.6,1.4)--(0.8,1.4);
\node at (1.2,1.7) {$B$};

\draw[-] (1,2)--(1,2.3);
\draw[-] (1.4,2)--(1.4,2.3);
\node at (1.15,2.15) {$\cdot$};
\node at (1.2,2.15) {$\cdot$};
\node at (1.25,2.15) {$\cdot$};

\draw[-] (2.6,0.8)--(2.6,1.4);
\draw[-] (3,1)--(3,1.4);
\draw[-] (2.4,1.4)--(2.4,2)--(3.2,2)--(3.2,1.4)--(2.4,1.4);
\node at (2.8,1.7) {$B'$};

\draw[-] (2.6,2)--(2.6,2.2);
\draw[-] (3,2)--(3,2.2);
\node at (2.75,2.15) {$\cdot$};
\node at (2.8,2.15) {$\cdot$};
\node at (2.85,2.15) {$\cdot$};
\node at (2.6,2.25) {$\circ$};
\node at (3,2.25) {$\circ$};

%%% rightmost part:
\draw[-] (3.3,-0.2)--(3.3,0.4)--(4.7,0.4)--(4.7,-0.2)--(3.3,-0.2);
\node at (4,0.1) {$A'$};
\draw[-] (3.6,-0.2)--(3.6,-0.7);
\draw[-] (4.4,-0.2)--(4.4,-0.5);
\node at (3.8,-0.4) {$\cdot$};
\node at (4,-0.4) {$\cdot$};
\node at (4.2,-0.4) {$\cdot$};
\draw[-] (3.6,0.4)--(3.6,0.7);
\draw[-] (4.4,0.4)--(4.4,0.7);
\node at (3.6,0.75) {$\circ$};
\node at (4.4,0.75) {$\circ$};
\node at (3.8,0.5) {$\cdot$};
\node at (4,0.5) {$\cdot$};
\node at (4.2,0.5) {$\cdot$};
%%% Bottom part
\draw[-] (1.6,-0.7)--(3.6,-0.7);
\draw[-] (2.4,-0.5)--(3.5, -0.5);
\draw[-] (3.7,-0.5)--(4.4,-0.5);
\draw[-] (3, -0.7)--(3,-1);
\draw[-] (3.4, -0.5)--(3.4, -0.65);
\draw[-] (3.4, -0.75)--(3.4, -1);
\node at (3.1,-0.95) {$\cdot$};
\node at (3.2,-0.95) {$\cdot$};
\node at (3.3,-0.95) {$\cdot$};
\end{tikzpicture}
\quad = \quad 
 \begin{tikzpicture}[anchorbase,scale=1.3]
%leftmost part:
\draw[-] (1.3,-0.2)--(1.3,0.4)--(2.7,0.4)--(2.7,-0.2)--(1.3,-0.2);
\node at (2,0.1) {$BA$};
\draw[-] (1.6,-0.2)--(1.6,-0.7);
\draw[-] (2.4,-0.2)--(2.4,-0.5);
\node at (1.8,-0.4) {$\cdot$};
\node at (2,-0.4) {$\cdot$};
\node at (2.2,-0.4) {$\cdot$};
\draw[-] (1.6,0.4)--(1.6,0.8);
\draw[-] (2.4,0.4)--(2.4,0.8);
\node at (1.8,0.5) {$\cdot$};
\node at (2,0.5) {$\cdot$};
\node at (2.2,0.5) {$\cdot$};%%%

%%% middle part:
\draw[-] (3.3,-0.2)--(3.3,0.4)--(4.7,0.4)--(4.7,-0.2)--(3.3,-0.2);
\node at (4,0.1) {$B'A$};
\draw[-] (3.6,-0.2)--(3.6,-0.7);
\draw[-] (4.4,-0.2)--(4.4,-0.5);
\node at (3.8,-0.4) {$\cdot$};
\node at (4,-0.4) {$\cdot$};
\node at (4.2,-0.4) {$\cdot$};
\draw[-] (3.6,0.4)--(3.6,0.7);
\draw[-] (4.4,0.4)--(4.4,0.7);
\node at (3.6,0.75) {$\circ$};
\node at (4.4,0.75) {$\circ$};
\node at (3.8,0.5) {$\cdot$};
\node at (4,0.5) {$\cdot$};
\node at (4.2,0.5) {$\cdot$};

%%% rightmost part:
\draw[-] (5.3,-0.2)--(5.3,0.4)--(6.7,0.4)--(6.7,-0.2)--(5.3,-0.2);
\node at (6,0.1) {$A'$};
\draw[-] (5.6,-0.2)--(5.6,-0.7);
\draw[-] (6.4,-0.2)--(6.4,-0.5);
\node at (5.8,-0.4) {$\cdot$};
\node at (6,-0.4) {$\cdot$};
\node at (6.2,-0.4) {$\cdot$};
\draw[-] (5.6,0.4)--(5.6,0.7);
\draw[-] (6.4,0.4)--(6.4,0.7);
\node at (5.6,0.75) {$\circ$};
\node at (6.4,0.75) {$\circ$};
\node at (5.8,0.5) {$\cdot$};
\node at (6,0.5) {$\cdot$};
\node at (6.2,0.5) {$\cdot$};
%%% Bottom part
\draw[-] (1.6,-0.7)--(5.6,-0.7);
\draw[-] (2.4,-0.5)--(3.55, -0.5);
\draw[-] (3.65,-0.5)--(5.55,-0.5);
\draw[-] (5.65,-0.5)--(6.4,-0.5);
\draw[-] (3, -0.7)--(3,-1);
\draw[-] (3.4, -0.5)--(3.4, -0.65);
\draw[-] (3.4, -0.75)--(3.4, -1);

\node at (3.1,-0.95) {$\cdot$};
\node at (3.2,-0.95) {$\cdot$};
\node at (3.3,-0.95) {$\cdot$};
\end{tikzpicture}
\end{equation*}

Let us now prove this by an explicit computation, which boils down to matrix multiplication. Proposition \ref{prop:compos_mu_A} implies
\begin{align*}
\hat{f}_{R_2} \circ \hat{f}_{R_1} &=\;  \left( \id_{\V^{\otimes l}} \otimes (z^*)^{\otimes d_2} \right) \,\circ\, \mu_{\begin{bmatrix} B \\
B'
\end{bmatrix}} \circ \left( \id_{\V^{\otimes k}} \otimes (z^*)^{\otimes d_1} \right) \,\circ\, \mu_{\begin{bmatrix} A \\
A'
\end{bmatrix}} \\
&=\; \left( \id_{\V^{\otimes l}} \otimes (z^*)^{\otimes d_2} \right) \,\circ\, \mu_{\begin{bmatrix} B \\
B'
\end{bmatrix}} \circ \left( \id_{\V^{\otimes k}} \otimes (z^*)^{\otimes d_1} \right) \,\circ\, (\mu_A 
\otimes \mu_{A'})\otimes \mu_{\begin{bmatrix} I_s \\
I_s
\end{bmatrix}} \\ &=\;
 \left( \id_{\V^{\otimes l}} \otimes (z^*)^{\otimes d_2} \right) \,\circ\, \mu_{\begin{bmatrix} BA \\
B'A
\end{bmatrix}} \circ \left( \id_{\V^{\otimes k}} \otimes \left(  (z^*)^{\otimes d_1} \,\circ\, \mu_{A'}\right) \right)  \,\circ\,  \mu_{\begin{bmatrix} I_s \\
I_s
\end{bmatrix}} \\ & =\;  \left( \id_{\V^{\otimes l}} \otimes (z^*)^{\otimes d_2} \otimes (z^*)^{\otimes d_1} \right) \,\circ\, \left(\mu_{BA} \otimes \mu_{B'A}\otimes   \mu_{A'}\right)  \,\circ\,  \left(\mu_{\begin{bmatrix} I_s \\
I_s
\end{bmatrix}} \otimes \id_{\V^{\otimes s}} \right) \,\circ\,  \mu_{\begin{bmatrix} I_s \\
I_s
\end{bmatrix}}.
\end{align*}

From Corollary \ref{cor:horizontal stacking} and Proposition \ref{prop:tensor_prod_mu_B}, we obtain:
$$ \left(\mu_{\begin{bmatrix} I_s \\
I_s
\end{bmatrix}} \otimes \id_{\V^{\otimes s}} \right) \,\circ\,  \mu_{\begin{bmatrix} I_s \\
I_s
\end{bmatrix}} = \mu_{\begin{bmatrix} I_s &0 \\
I_s &0 \\
0 &I_s
\end{bmatrix}}  \,\circ\,  \mu_{\begin{bmatrix} I_s \\
I_s
\end{bmatrix}} = \mu_{\begin{bmatrix} I_s \\
I_s \\
I_s
\end{bmatrix}}.$$
Thus we have:
\begin{align*}
\hat{f}_{R_2} \circ \hat{f}_{R_1} &=\;  \left( \id_{\V^{\otimes l}} \otimes (z^*)^{\otimes d_2} \otimes (z^*)^{\otimes d_1} \right) \,\circ\, \left(\mu_{BA} \otimes \mu_{B'A} \otimes  \mu_{A'}\right)  \,\circ\,  \mu_{\begin{bmatrix} I_s \\
I_s\\I_s
\end{bmatrix}} \\
&=\;  \left( \id_{\V^{\otimes l}} \otimes (z^*)^{\otimes d_2+ d_1} \right)  \,\circ\, \mu_{\begin{bmatrix} BA \\
B'A\\A' 
\end{bmatrix}} = \hat{f}_{R_2 \star R_1}
\end{align*}
where the equality between the two lines follows from Proposition \ref{prop:tensor_prod_mu_B}. The proposition is proved.

\end{proof}

\begin{lemma}\label{lem:tensor_prod_f_R_infty}
Let $R_1 \in Rel^\infty_{s_1, k_1}$ and $R_2\in Rel^\infty_{s_2, k_2}$. We have: $\hat{f}_{R_1}\otimes \hat{f}_{R_2} = \hat{f}_{R_1\times R_2}$.
\end{lemma}
\begin{proof}
Let $d_i = \dim_{\F_q} R_i$, $r_i := s_i+k_i$ for $i=1,2$. 

Let $A^{(i)} \in Mat_{k_i\times s_i}(\F_q), (A')^{(i)}\in Mat_{d_i\times s_i}(\F_q)$ be such that $$R_i = Row\begin{bmatrix} -A^{(i)} &I_{k_i}\\
(A')^{(i)} &0
\end{bmatrix}\;\; rk\left((A')^{(i)}\right)=d_i.$$ 

Consider $R_1 \times R_2 $ as a subset $ \F_q^{s_1+s_2+k_1+k_2}$.
Then $$ R_1 \times R_2 = Row \begin{bmatrix} 
-A^{(1)} &0 &I_{k_1} &0\\
(A')^{(1)} &0 &0 &0\\
0 &-A^{(2)} &0  &I_{k_2} \\
0 &(A')^{(2)} &0  &0
\end{bmatrix} = Row \begin{bmatrix} 
-A^{(1)} &0 &I_{k_1} &0\\
0 &-A^{(2)} &0  &I_{k_2} \\
(A')^{(1)} &0 &0 &0\\
0 &(A')^{(2)} &0  &0
\end{bmatrix}$$ so $$\hat{f}_{R_1\times R_2} = \left( \id_{\V^{\otimes k_1 + k_2}} \otimes (z^*)^{\otimes d_1+ d_2} \right) \,\circ\,\mu_{\begin{bmatrix} A^{(1)} &0 \\
0 &A^{(2)} \\(A')^{(1)}  &0 \\ 
0 &(A')^{(2)}
\end{bmatrix}}.$$
Let us compute $\hat{f}_{R_1} \otimes \hat{f}_{R_2}$. We have:
\begin{align*}
\hat{f}_{R_1}\otimes \hat{f}_{R_2} &=\left(\left( \id_{\V^{\otimes k_1}} \otimes (z^*)^{\otimes d_1} \right) \,\circ\, \mu_{\begin{bmatrix} A^{(1)} \\
(A')^{(1)}
\end{bmatrix}} \right) \otimes \left(\left( \id_{\V^{\otimes k_2}} \otimes (z^*)^{\otimes d_2}  \right) \,\circ\, \mu_{\begin{bmatrix} A^{(2)} \\
(A')^{(2)}
\end{bmatrix}} \right) \\
&= \left( \id_{\V^{\otimes k_1}} \otimes (z^*)^{\otimes d_1}  \otimes  \id_{\V^{\otimes k_2}} \otimes (z^*)^{\otimes d_2} \right) \,\circ\,\left( \mu_{\begin{bmatrix} A^{(1)} \\
(A')^{(1)}
\end{bmatrix}} \otimes  \mu_{\begin{bmatrix} A^{(2)} \\
(A')^{(2)}
\end{bmatrix}} \right) 
\end{align*}
By Proposition \ref{prop:tensor_prod_mu_B}, we have:
$$\mu_{\begin{bmatrix} A^{(1)} \\
(A')^{(1)}
\end{bmatrix}} \otimes  \mu_{\begin{bmatrix} A^{(2)} \\
(A')^{(2)}
\end{bmatrix}} = \mu_{\begin{bmatrix} A^{(1)} &0 \\
(A')^{(1)}  &0 \\ 0 &A^{(2)} \\
0 &(A')^{(2)}
\end{bmatrix}}$$
Let $\sigma:\V^{\otimes d_1} \otimes \V^{\otimes k_2}\to  \V^{\otimes k_2} \otimes \V^{\otimes d_1}$ be the symmetry morphism. Then $\sigma^2 = \id$, so
\begin{align*}
&\hat{f}_{R_1}\otimes \hat{f}_{R_2}  = \left( \id_{\V^{\otimes k_1}} \otimes (z^*)^{\otimes d_1}  \otimes  \id_{\V^{\otimes k_2}} \otimes (z^*)^{\otimes d_2} \right) \, \circ \sigma \, \circ \sigma \,\circ\,\mu_{\begin{bmatrix} A^{(1)} &0 \\
(A')^{(1)}  &0 \\ 0 &A^{(2)} \\
0 &(A')^{(2)}
\end{bmatrix}} \\ & = \left( \InnaA{\id} \otimes (z^*)^{\otimes d_1+ d_2} \right) \, \circ \sigma \,\circ\,\mu_{\begin{bmatrix} A^{(1)} &0 \\
(A')^{(1)}  &0 \\ 0 &A^{(2)} \\
0 &(A')^{(2)}
\end{bmatrix}}  = \left( \InnaA{\id} \otimes (z^*)^{\otimes d_1+ d_2} \right) \,\circ\,\mu_{\begin{bmatrix} A^{(1)} &0 \\
0 &A^{(2)} \\(A')^{(1)}  &0 \\ 
0 &(A')^{(2)}
\end{bmatrix}}
\end{align*}
(the last equality follows from the fact that $\sigma$ is a natural transformation). The statement is thus proved.
\end{proof}

\section*{Acknowledgements}
The research of I. E. and T. H. was supported by the ISF grant 711/18. The research of T. H. was partially funded by the Deutsche Forschungsgemeinschaft (DFG, German Research
Foundation) under Germany's Excellence Strategy – EXC-2047/1 – 390685813. We thank Avraham Aizenbud, Nate Harman and Andrew Snowden for helpful discussions.

%For Transformation groups
\section*{Conflict of Interest} The authors declare that they have no conflict of interest.

\end{document}